\documentclass[12pt]{article}

\usepackage{fullpage}

\usepackage{amsmath}
\usepackage{amsfonts}
\usepackage{amssymb}
\usepackage{amsthm}

\theoremstyle{plain}
\newtheorem{theorem}{Theorem}[section]
\newtheorem{lemma}[theorem]{Lemma}
\newtheorem{proposition}[theorem]{Proposition}

\newtheorem{claim}[theorem]{Claim}

\theoremstyle{definition}
\newtheorem{definition}[theorem]{Definition}
\newtheorem{example}[theorem]{Example}
\newtheorem{remark}[theorem]{Remark}
\newtheorem{question}[theorem]{Question}

\DeclareMathOperator*{\argmin}{arg\;min}
\DeclareMathOperator*{\argmax}{arg\;max}

\usepackage{graphicx}
\usepackage{color}
\usepackage[dvipsnames]{xcolor}
\usepackage{bm}
\usepackage{dsfont}
\usepackage{multirow}
\usepackage{mathrsfs}
\usepackage{booktabs}

\newcommand{\showsupplement}[1]{}
\newcommand{\customresize}[1]{\resizebox{\textwidth}{!}{#1}}

\usepackage{algorithm}
\usepackage{algorithmic}

\usepackage{hyperref}
\hypersetup{colorlinks=true,citecolor=Blue,urlcolor=Blue,linkcolor=Blue}

\usepackage[round]{natbib}
\usepackage{authblk}

\author[1]{\normalsize \textbf{Hiroki Waida}}
\author[1,2]{\normalsize \textbf{Takafumi Kanamori}}

\date{}

\title{\Large \textbf{Statistical learnability of smooth boundaries via pairwise binary classification with deep ReLU networks}}

\affil[1]{\small Institute of Science Tokyo}
\affil[2]{\small RIKEN AIP}

\begin{document}

\maketitle



\begin{abstract}
The topic of nonparametric estimation of smooth boundaries is extensively studied in the conventional setting where pairs of single covariate and response variable are observed.
However, this traditional setting often suffers from the cost of data collection.
Recent years have witnessed the consistent development of learning algorithms for binary classification problems where one can instead observe paired covariates and binary variable representing the statistical relationship between the covariates.
In this work, we theoretically study the learnability of ordered multiple smooth boundaries under a pairwise binary classification setting.
One of the challenging problems is the non-identifiability issue on the order of smooth subsets, which yields the gap between the generalizability and the learnability of smooth boundaries in the pairwise binary classification setting.
To deal with the challenges due to this non-identifiability directly, we develop a proof method using a localization argument of the given vector-valued function class.
Consequently, we prove that some ordered multiple smooth boundaries are learnable via a pairwise binary classification algorithm defined with a localized class of deep ReLU networks.
\end{abstract}



\section{Introduction}
\label{sec:introduction}

Let $(\mathcal{X},\mathcal{B}(\mathcal{X}),\mu)$ be the measure space for which $\mathcal{X}=[0,1]^{K}$, $\mathcal{B}(\mathcal{X})$ is the Borel $\sigma$-algebra, and $\mu$ is the Lebesgue measure, following the common setting in the field of nonparametric statistics (see, e.g.,~\citep{schmidt-hieber2020nonparametric,suzuki2018adaptivity,kim2021fast,bos2022convergence,imaizumi2022advantage}).
Let $\{\mathcal{K}_{i}\}_{i=1}^{d_{1}}$ be a disjoint partition of $\mathcal{X}$ for some $d_{1}\in\mathbb{N}$.
Namely, $\{\mathcal{K}_{i}\}_{i=1}^{d_{1}}$ is a sequence of disjoint subsets for which $\bigcup_{i=1}^{d_{1}}\mathcal{K}_{i}=\mathcal{X}$ holds.
A classical class of subsets with H\"{o}lder continuous boundaries is introduced by~\citep{dudley1974metric}, and this class has been commonly employed in the literature~\citep{mammen1995asymptotical,mammen1999smooth,petersen2018optimal,imaizumi2019deep,imaizumi2022advantage}.
In the current work, we study the estimation problem of a disjoint partition $\{\mathcal{K}_{i}\}_{i=1}^{d_{1}}$ that belongs to a general class of H\"{o}lder continuous boundaries introduced in~\citep{imaizumi2022advantage}.

Given a partition $\{\mathcal{K}_{i}\}_{i=1}^{d_{1}}$ for which every subset is characterized with Dudley's boundary class~\citep{dudley1974metric} in some sense, nonparametric estimation or learning of the indicator functions $\mathds{1}_{\mathcal{K}_{1}},\cdots,\mathds{1}_{\mathcal{K}_{d_{1}}}$ is extensively studied in many topics such as set estimation~\citep{mammen1995asymptotical}, discriminant analysis~\citep{mammen1999smooth}, classification~\citep{tsybakov2004optimal,tarigan2008moment,kim2021fast,meyer2023optimal,caragea2023neural}, and regression~\citep{imaizumi2019deep,imaizumi2022advantage}, under the conventional supervised learning setting.
Here, the conventional supervised learning refers to the setting where a dataset is commonly defined as pairs of independently and identically distributed (i.i.d.) random variables on a probability space $(\Omega,\Sigma,Q)$, which are drawn from the joint distribution of a $\mathcal{X}$-valued covariate $X$ and a label $Z$, namely
\begin{align*}
(X_{1},Z_{1}),\cdots,(X_{n},Z_{n})\sim_{i.i.d.} Q\circ (X,Z)^{-1}.
\end{align*}
However, in practice it is often costly to collect labels $Z_{1},\cdots,Z_{n}$ corresponding to the given covariates $X_{1},\cdots,X_{n}$.

In the present work, we investigate the topic in another setting where some statistical relationship between covariates $X$ and $X'$ is instead available.
More precisely, we consider a dataset
\begin{align}
\label{eq:self-supervised dataset}
(X_{1},X_{1}',Y_{1}),\cdots,(X_{n},X_{n}',Y_{n})\sim_{i.i.d.} P,
\end{align}
where $X_{i},X_{i}':\Omega\to\mathcal{X}$ and $Y_{i}:\Omega\to\mathcal{Y}=\{1,-1\}$ are random variables.
Here, $\mathcal{Y}$ is equipped with the counting measure, and $P$ is a Borel probability measure in $\mathcal{X}^{2}\times\mathcal{Y}$.
This problem setting is extensively studied in the context of ranking problems (see, e.g.,~\citep{robbiano2013upper}), similarity learning (see, e.g.,~\citep{chen2009similarity,jin2009regularized,cao2015generalization,bao2022pairwise}), and self-supervised learning (see, e.g.,~\citep{arora2019theoretical,tosh2021contrastivejmlr,tosh2021contrastive,chen2020simple,tsai2020neural,chuang2022robust,zhai2023sigmoid,balestriero2023cookbook}), to improve the efficiency of learning procedures.
Indeed, $Y$ always takes either $Y=1$ (i.e., $X$ and $X'$ are \emph{similar}) or $Y=-1$ (i.e., $X$ and $X'$ are \emph{dissimilar}).

Besides the aspect on data utilization, boundary estimation based on a pairwise similarity might be reasonable from some statistical learning viewpoints, provided that some additional assumption on subsets is introduced in the data space.
For instance, the standard goal of binary classification is to obtain a hypersurface that classifies the observed covariates consistently in terms of the similarity determined by the Bayes classifier (see, e.g.,~\citep{hastie2009elements}).
Similarly in the case where one can observe pairwise data $(X,X',Y)$, a similar intuition might be still valid if one can observe some suitable pairwise similarity.
In this context, \citet{bao2022pairwise} show that the generalization error of a conventional binary classification problem characterized by a single decision boundary is upper bounded by that of a similarity learning problem where $Y$ is defined as $Y=ZZ'$ for the given independent, $\mathcal{X}\times\{-1,1\}$-valued supervised data $(X,Z),(X',Z')$.
However, their method deals with a single decision boundary, and they use supervised data in the formalization of $Y$.
Furthermore, they assume that $X$ and $X'$ are independent.
Usually, similar covariates in a given pairwise data share some common latent factors~\citep{arora2019theoretical,haochen2021provable,vonkugelgenself2021,parulekar2023infonce}.
Hence, it is natural to consider the setting where $X$ and $X'$ can be dependent to each other.

In the field of self-supervised learning, it is shown by~\citet{haochen2021provable} that the learnability in a multiclass classification problem is guaranteed when both paired covariates and supervised data are used in the definition of estimators.
Nevertheless, to the best of our knowledge, the learnability of multiple smooth boundaries is not proven when one can use only pairwise data.
In the context of nonparametric statistics, it is shown by~\citet{kim2021fast} and~\citet{imaizumi2019deep,imaizumi2022advantage} that statistical learning of multiple boundaries using deep neural networks is possible under the conventional settings of supervised learning.
However, the arguments developed in~\citep{kim2021fast,imaizumi2019deep,imaizumi2022advantage} cannot directly apply to our problem setting because of a gap between the conventional and pairwise settings (see Section~\ref{subsubsec:outline}).

Therefore, we study the following problem under a general setting where nonparametric estimation of multiple smooth boundaries is considered:
\begin{question}
\label{question:main question}
Given a partition $\{\mathcal{K}_{i}\}_{i=1}^{d_{1}}$ of $\mathcal{X}$, if every subset $\mathcal{K}_{i}$ has a smooth boundary, then is it possible to estimate both the boundaries and the order of the subsets, using some learning algorithm that requires only samples generated in a pairwise binary classification setting~\eqref{eq:self-supervised dataset}?
\end{question}
Note that the estimation of both subsets and the order is frequently studied in the context of statistical classification (see, e.g.,~\citep{tsybakov2004optimal,tarigan2008moment,kim2021fast,meyer2023optimal,bao2022pairwise,haochen2021provable}).

\paragraph{Pairwise data.}
To investigate the problem, we need to ask what kind of pairwise relation is reasonable to be used in boundary estimation.
In essence, our approach builds on the pairwise relation introduced by~\citet{tsai2020neural}.
More precisely, \citet{tsai2020neural} introduce the density function of a probability measure $P$ in $\mathcal{X}^{2}\times\mathcal{Y}$ satisfying
\begin{align}
\label{eq:condition of tsai et al}
\begin{cases}
p(x,x',1)&=p_{Y}(1)q(x,x'),\\
p(x,x',-1)&=p_{Y}(-1)p_{X}(x)p_{X'}(x'),
\end{cases}
\end{align}
where $q(x,x')$ denotes a probability density function on $\mathcal{X}^{2}$, $p_{X}(x)$ and $p_{X'}(x')$ are the marginal distributions of $q(x,x')$, and $p_{Y}(y)$ is a probability function on $\mathcal{Y}$.
This relation is useful in the following two points:
First, the pairwise relation introduced in~\citep{tsai2020neural} uses the statistical independence, which is simple and flexible.
Also, the statistical independence can be checked by applying statistical independence tests (see, e.g.,~\citep{albert2022adaptive}), which is reasonable from practical aspects.
Second, this pairwise relation is commonly used in the context of contrastive learning (see, e.g.,~\citep{arora2019theoretical,chen2020simple,tosh2021contrastive,tosh2021contrastivejmlr,haochen2021provable,chuang2022robust}).
Here, contrastive learning is known as a learning method that can handle large-scaled problems (see, e.g.,~\citep{gutmann2010noise,oord2018representation,arora2019theoretical,henaff2020data,he2020momentum,chen2020simple,haochen2021provable,dwibedi2021with}).
Note that \citet{tsai2020neural} use the pairwise relation~\eqref{eq:condition of tsai et al} to study mutual information estimation and some applications including contrastive learning.
In the current work we use this pairwise relation to study nonparametric boundary estimation, different from the purpose of~\citep{tsai2020neural}.

\paragraph{Problem setting.}
We formalize a problem setting by using the indicator functions $\{\mathds{1}_{\mathcal{K}_{i}}\}_{i=1}^{d_{1}}$, as in~\citep{tsybakov2004optimal,kim2021fast,meyer2023optimal,imaizumi2019deep,imaizumi2022advantage}.
For convenience, we informally introduce the following notions, where the formal definitions are deferred to the later sections:
\begin{itemize}
\item \emph{H\"{o}lder continuous partitions.} We say that the given disjoint partition $\{\mathcal{K}_{i}\}_{i=1}^{d_{1}}$ of $\mathcal{X}$ is $\alpha$-H\"{o}lder continuous if the topological boundary of every subset in the partition is defined with some $\alpha$-H\"{o}lder continuous functions on $[0,1]^{K-1}$, following~\citep{imaizumi2022advantage} (see Definition~\ref{def:class of smooth partitions} for the formal definition).
\item \emph{Classes of probability distributions.} Given $\alpha>0$, a parameter $\tau\geq 1$ of the Tsybakov noise condition~\citep{mammen1999smooth,tsybakov2004optimal} (see Section~\ref{subsec:noise condition and smooth boundaries}), and a hyperparameter $\xi\in \Xi$ including variables $K$ and $d_{1}$, we introduce a new class $\mathcal{P}_{\alpha,\tau,\xi}$ of Borel probability measures in $\mathcal{X}^{2}\times\mathcal{Y}$ for which each probability density function satisfies the property~\eqref{eq:condition of tsai et al} due to~\citep{tsai2020neural} and is characterized by some disjoint, $\alpha$-H\"{o}lder continuous partition $\{\mathcal{K}_{i}\}_{i=1}^{d_{1}}$ of $\mathcal{X}$ (see Definition~\ref{def:main assumption}).
Roughly speaking, this class is regarded as an extension of the conventional settings (see, e.g.,~\citep{tsybakov2004optimal,kim2021fast}) to a pairwise setting.
\item \emph{$L^{2}$-risk.} Similarly to~\citep{tsybakov2004optimal,imaizumi2019deep,imaizumi2022advantage,meyer2023optimal}, for the given distribution $P\in\mathcal{P}_{\alpha,\tau,\xi}$ that is defined with some disjoint, $\alpha$-H\"{o}lder continuous partition $\{\mathcal{K}_{i}\}_{i=1}^{d_{1}}$ of $\mathcal{X}$, we consider the $L^{2}$-risk
\begin{align*}
\mathcal{R}(\widehat{g}_{n};P)=\mathbb{E}\left[\sum_{i=1}^{d_{1}}\|\widehat{g}_{n,i}(U_{1},\cdots,U_{n})-\mathds{1}_{\mathcal{K}_{i}}\|_{L^{2}(\mathcal{X},P_{X})}^{2}\right],
\end{align*}
which measures the gap between the given estimator $\widehat{g}_{n}=(\widehat{g}_{n,1},\cdots,\widehat{g}_{n,d_{1}})$ and the indicator functions $\mathds{1}_{\mathcal{K}_{1}},\cdots,\mathds{1}_{\mathcal{K}_{d_{1}}}$, where $U_{1},\cdots,U_{n}:\Omega\to\mathcal{X}^{2}\times\mathcal{Y}$ are i.i.d. random variables drawn from $P$, and $P_{X}$ is the marginal distribution of $P$ (see Definition~\ref{def:l2 risk}).
Note that under this risk function, the order of the subsets is arbitrary fixed and is to be estimated simultaneously.
This setting is consistent with Question~\ref{question:main question}.
\end{itemize}

\paragraph{Outline of the main results.}
As discussed in Section~\ref{subsec:general upper bound of main result}, we notice that the mathematical relation between the generalizability and the learnability of decision boundaries in a conventional classification problem, which is proven in~\citep[Proposition~1]{lecue2007optimal}, is not directly applicable to the pairwise setting (see also the next paragraph for an overview).
This observation is due to the non-identifiability issue of the estimation problem: namely, given any sequence $\{\mathcal{K}_{i}\}_{i=1}^{d_{1}}$ of smooth subsets and any permutation $\pi$ on $\{1,\cdots,d_{1}\}$ except the identity map, $\{\mathcal{K}_{i}\}_{i=1}^{d_{1}}$ and $\{\mathcal{K}_{\pi(i)}\}_{i=1}^{d_{1}}$ are distinct parameters in terms of the indices, although this difference cannot be incorporated in the pairwise similarity, since no supervised data is available.
To address this issue directly, we develop a proof method using a \emph{local} estimator $\widehat{g}_{n}^{\textup{local}}$, which is a map from $\mathcal{P}_{\alpha,\tau,\xi}$ to the set of all estimators.
This is a generalization of the standard \emph{global} estimator (see Definition~\ref{def:definitions of local and global estimators}).
The following statement is an informal version of the main theorem for the local estimators:
\begin{theorem}[Informal, see Theorem~\ref{thm:estimation error for deep relu networks}]
\label{thm:informal version of local theorem}
Given any $\alpha>0$, $\tau\geq 1$, $\xi\in \Xi$, and $n\in\mathbb{N}$, there are a function class $\mathcal{F}$ of deep ReLU networks, a local estimator $\widehat{g}_{n}^{\textup{local}}$ on $\mathcal{P}_{\alpha,\tau,\xi}$ whose range of $\widehat{g}_{n,P}^{\textup{local}}:=\widehat{g}_{n}^{\textup{local}}(P)$ is a set of probability-simplex-valued functions defined with $\mathcal{F}$ for each $P\in\mathcal{P}_{\alpha,\tau,\xi}$, and a constant $C>0$ independent of $n$, such that when $n$ is sufficiently large, we have
\begin{align*}
\sup_{P\in\mathcal{P}_{\alpha,\tau,\xi}}\mathcal{R}(\widehat{g}_{n,P}^{\textup{local}};P)\leq Cn^{-\frac{\alpha}{(2\tau-1)\alpha+\tau(K-1)}}\log^{3\tau^{-1}+1}{n}.
\end{align*}
\end{theorem}
This theorem implies that the smooth boundaries are learnable via the pairwise binary classification problem using a local estimator.
We construct an estimator using an Empirical Risk Minimization (ERM) algorithm, based on similarity learning~\citep{jin2009regularized,chen2009similarity,cao2015generalization,bao2022pairwise} and contrastive learning~\citep{arora2019theoretical,awasthi2022do,tsai2020neural,wang2020understanding,chen2021large} (see Section~\ref{subsec:loss function continued} and Definition~\ref{def:formal definition of erm}).

\begin{remark}
\label{rem:remarks in introduction}
We employ a learning model defined with deep ReLU networks because we can check the sufficient conditions of the excess risk bound proven in~\citep{park2009convergence,kim2021fast}, using some facts shown in~\citep{petersen2018optimal,nakada2020adaptive,imaizumi2019deep,imaizumi2022advantage,bos2022convergence} (see Appendix~B.6.4\showsupplement{in}).
Hence, one needs not to assume that the learning model must be defined with deep ReLU networks.
Meanwhile, since deep ReLU networks are widely used in the field of theoretical statistics (see, e.g.,~\citep{schmidt-hieber2020nonparametric,suzuki2018adaptivity,kim2021fast,imaizumi2019deep,bos2022convergence,meyer2023optimal}), it is reasonable to focus on deep ReLU networks to study Question~\ref{question:main question}.
\end{remark}

We provide the discussion of local estimators in Section~\ref{subsec:discussion of the formal version of the main theorem} in detail.

\paragraph{Proof method.}
We summarize the main idea in the proof method of Theorem~\ref{thm:informal version of local theorem} (see Section~\ref{subsec:general upper bound of main result}), which may be the primary, technical contribution of the current work.

The generalizability of a binary classification problem is usually quantified by the excess risk (see, e.g.,~\citep{mohri2018foundations}):
\begin{align*}
\mathbb{E}[P_{X_{0},Z}(\widehat{g}_{n}(x)\neq z)]-P_{X_{0},Z}(g^{*}(x)\neq z),
\end{align*}
where $P_{X_{0},Z}$ is a distribution on a measurable space $\mathcal{X}_{0}\times\{0,1\}$, $\widehat{g}_{n}$ is an estimator using $n$ samples, $g^{*}$ is the Bayes classifier, and the expectation is taken over the distribution of the samples.
Meanwhile, in~\citep{tsybakov2004optimal,meyer2023optimal}, the learnability of smooth decision boundaries is formalized with the risk function
\begin{align*}
\mathbb{E}[\|\mathds{1}_{\{x\in\mathcal{X}_{0}\;|\; \widehat{g}_{n}(x)=1\}}-\mathds{1}_{\{x\in\mathcal{X}_{0}\;|\; g^{*}(x)=1\}}\|_{L^{2}(\mathcal{X}_{0},P_{X_{0}})}^{2}],
\end{align*}
where $P_{X_{0}}$ is the marginal distribution of $P_{X_{0},Z}$ in $\mathcal{X}_{0}$.
It is proven in~\citep[Proposition~1]{tsybakov2004optimal} that in a conventional binary classification problem, under some condition on the conditional probability, the generalizability of a classifier directly implies the learnability of the decision boundary.
Since we use hinge loss to develop an algorithm in the current work, we focus on the property proven in~\citep[Proposition~1]{lecue2007optimal}, which is a variant of~\citep[Proposition~1]{tsybakov2004optimal} for hinge loss.

However, we find that the mathematical relation proven in~\citep[Proposition~1]{lecue2007optimal} is not directly applicable to some pairwise binary classification problem (see Example~\ref{example:not good representation}).
Due to this non-identifiability issue, some localization argument of vector-valued function classes is required (see Question~\ref{question:how to derive a lower bound}).
Note that this difficulty can also be observed when the problem setting of~\citet{bao2022pairwise} is employed, as discussed in Appendix~A.2\showsupplement{in}.
Thus, this observation is not particular to our problem setting.

While one can partially bypass this issue by transforming the problem in Question~\ref{question:main question} into an identifiable setting where a permutation-invariant $L^{2}$-risk is employed to estimate only smooth subsets, it is still challenging to derive upper bounds of such risk functions, as we discuss in Section~\ref{subsec:comparison with other methods}.
We overcome these technical difficulties by developing a localization argument (see Theorem~\ref{thm:main result}).
The method consists of two steps:
We first utilize the pairwise setting to evaluate the $L^{2}$-risk, using the sum of quantities defined with some subsets of a regular simplex (see Lemma~\ref{lem:step 3 lemma 0}).
Then, for each subset of the regular simplex, we derive a lower bound of the $L^{1}$-risk of a classifier to apply~\citep[Proposition~1]{lecue2007optimal} (see Lemma~\ref{lem:step 3 lemma 1}).

For the detailed comparison with other approaches developed in~\citep{bao2022pairwise,haochen2021provable,ge2024on}, see Section~\ref{subsec:comparison with other methods}.

\paragraph{Organization of the paper.}
The rest of this paper consists of the following sections.
In Section~\ref{sec:problem setting}, we define the notation used in this paper.
In Section~\ref{sec:main results}, we formalize the problem setting and present the main theorem of the current work.
In Section~\ref{subsec:general upper bound of main result}, we present the key ideas used in the proof of the main theorem in Section~\ref{sec:main results}.
In Section~\ref{sec:discussion of the main results}, we discuss the main theorem and its proof method in detail.
In Section~\ref{subsec:discussion of the main theorem}, we review the related literature.
In Section~\ref{sec:discussion}, we present some discussion and future work.

\section{Preliminaries}
\label{sec:problem setting}

In Section~\ref{subsec:notations}, we introduce some basic notation.
In Section~\ref{subsec:noise condition and smooth boundaries}, we review the noise condition introduced in~\citep{mammen1999smooth,tsybakov2004optimal} and the class of smooth boundaries studied in~\citep{imaizumi2022advantage}.
In Section~\ref{subsec:learning models}, we define several classes of learning models using deep ReLU networks, as mentioned in Remark~\ref{rem:remarks in introduction}.

\subsection{Notation}
\label{subsec:notations}

As in Section~\ref{sec:introduction}, we define $\mathcal{X}=[0,1]^{K}$ for $K\in\mathbb{N}$ and endow $\mathcal{X}$ with the Borel $\sigma$-algebra $\mathcal{B}(\mathcal{X})$ and the Lebesgue measure $\mu$.
Let $\mathcal{Y}=\{-1,1\}$.
Throughout this work, $K$ and $d_{1}$ are natural numbers representing the dimension of $\mathcal{X}$ and the number of subsets, respectively.
We need some additional notation to present the problem setting to study Question~\ref{question:main question}.
For some basic mathematical notions, we refer the reader to~\citep{steinwart2008support}.
Several notation lists can be found in \hyperref[appsec:notation list]{Appendix}.

For a topological space $\mathcal{A}$, let $\mathcal{B}(\mathcal{A})$ denote the Borel $\sigma$-algebra.
Given $s\in(0,\infty]$ and a non-negative, $\sigma$-finite measure $\nu$ on a measurable space $\mathcal{A}$, we define the $L^{s}(\mathcal{A},\nu)$-norm as
$\|g\|_{L^{s}(\mathcal{A},\nu)}=(\int_{\mathcal{A}}|g(x)|^{s}\nu(dx))^{1/s}$ if $s<\infty$, and $\|g\|_{L^{\infty}(\mathcal{A},\nu)}=\inf\{t\geq 0\;|\;\nu(\{x\in\mathcal{A}\;|\; |g(x)|>t\})=0\}$ if $s=\infty$.
Given a measurable, $\mathbb{R}^{t}$-valued function $f$ on $\mathcal{A}$, let $\|f\|_{\mathcal{A},\nu,s}:=\|\|f\|_{s}\|_{L^{s}(\mathcal{A},\nu)}$, where $\|\cdot\|_{s}$ denotes the $s$-norm in the Euclidean space.
Note that in the case where the Lebesgue measure $\mu$ is used, we often abbreviate as $\|g\|_{L^{s}(\mathcal{X})}:=\|g\|_{L^{s}(\mathcal{X},\mu)}$ and $\|f\|_{\mathcal{A},s}:=\|f\|_{\mathcal{A},\mu,s}$ for any given $s\in[1,\infty]$, any $g:\mathcal{A}\to\mathbb{R}$, and any function $f:\mathcal{A}\to\mathbb{R}^{t}$.
Given a vector-valued function $g:\mathcal{A}\to \mathbb{R}^{s}$ on a set $\mathcal{A}$, we often write as $g=(g_{1},\cdots,g_{s})$ with $g_{1},\cdots,g_{s}:\mathcal{A}\to\mathbb{R}$, namely, $g(x)=(g_{1}(x),\cdots,g_{s}(x))$ for each $x\in\mathcal{A}$.
We remark that any vector $b\in\mathbb{R}^{s}$ is written as $b=(b_{j}):=(b_{1},\cdots,b_{s})$.
Note also that given $s,t\in\mathbb{N}$, $\mathbb{R}^{s\times t}$ denotes the set of all linear operators from $\mathbb{R}^{t}$ to $\mathbb{R}^{s}$.
Throughout the paper, any linear operator $W\in \mathbb{R}^{s\times t}$ is identified with the corresponding matrix and is written in matrix notation, namely, $W=(W_{j_{1},j_{2}})$.
We endow any finite set with the discrete topology and consider the measurable space equipped with the Borel $\sigma$-algebra.
Also, the cardinality of a finite set $\mathcal{A}$ is denoted by $|\mathcal{A}|$.

Given a Borel probability measure $P$ in $\mathcal{X}^{2}\times\mathcal{Y}$ that is absolutely continuous for the product measure $\mu\otimes\mu\otimes\chi$, let $p(x,x',y)$ be the probability density function on $\mathcal{X}^{2}\times\mathcal{Y}$, where $\chi$ denotes the counting measure in $\mathcal{Y}$.
Define the function $p_{X,X'}(x,x')=p(x,x',1)+p(x,x',-1)$.
Let $p_{X}(x)=\int_{\mathcal{X}}p_{X,X'}(x,x')\mu(dx')$ and $p_{X'}(x')=\int_{\mathcal{X}}p_{X,X'}(x,x')\mu(dx)$.
Denote by $P_{X}$ and $P_{X'}$, the probability measures whose Lebesgue densities are $p_{X}$ and $p_{X'}$, respectively.
Let $p_{Y}(y)=\int_{\mathcal{X}^{2}}p(x,x',y)(\mu\otimes\mu)(dx,dx')$.
We define $q(x,x')=p(x,x'|y=1)$, following the pairwise relation~\eqref{eq:condition of tsai et al} due to~\citep{tsai2020neural}.
The conditional probability $p(y=1|x,x')=p(x,x',1)/p_{X,X'}(x,x')$ is denoted by $\eta(x,x'):=p(y=1|x,x')$.
Denote by $P_{X,X'}$, the probability measure whose Lebesgue density is $p_{X,X'}$.

We define the sign function as $\textup{sign}(s)=1$ if $s\geq 0$ and $\textup{sign}(s)=-1$ if $s<0$.
The domain of any random variable defined in this work is the probability space $(\Omega,\Sigma,Q)$.
For any $s_{1},s_{2}\in\mathbb{R}$, the notation $s_{1}\lesssim s_{2}$ means that there is a constant $C>0$ independent of the sample size $n$ such that $s_{1}\leq Cs_{2}$, unless otherwise specified.
The notation $s_{1}\gtrsim s_{2}$ means that $-s_{1}\lesssim -s_{2}$.
Given $s\in\mathbb{R}$, we define $\lfloor s\rfloor =\max\{t\in\mathbb{Z}\;|\; t\leq s\}$ and $\lceil s\rceil = \min\{t\in\mathbb{Z}\;|\; s\leq t\}$.
Given $s,t\in\mathbb{R}$, let $s\vee t=\max\{s,t\}$ and $s\wedge t=\min\{s,t\}$.
Given a set $\mathcal{A}$, the indicator function of $\mathcal{A}$ is denoted by $\mathds{1}_{\mathcal{A}}$.

\subsection{Noise Condition and Smooth Boundaries}
\label{subsec:noise condition and smooth boundaries}

\paragraph{Noise condition.}
Let $\mathcal{X}_{0}$ be a measurable space, and let $\nu$ be a non-negative, $\sigma$-finite measure in $\mathcal{X}_{0}$.
Let $P$ be a probability measure in $\mathcal{X}_{0}\times\mathcal{Y}$ for which it has a probability density function $p(x,y)$ on $\mathcal{X}_{0}\times\mathcal{Y}$ with respect to the product measure $\nu\otimes \chi$.
Note that the marginal distribution of $P$ in $\mathcal{X}_{0}$ is denoted by $P_{X_{0}}$.
The Tsybakov noise condition introduced by~\citet{tsybakov2004optimal} is an assumption requiring that there are a threshold $s_{0}\in (0,1]$, a parameter $\tau\geq 1$, and a constant $c>0$, such that for every $s\in (0,s_{0}]$ it holds that
\begin{align}
\label{eq:margin assumption}
P_{X_{0}}(\{x\in\mathcal{X}_{0}\;|\;|2p(y=1|x)-1|\leq s\})\leq c\cdot s^{\frac{1}{\tau-1}}.
\end{align}
See also~\citep{mammen1999smooth} for a more general condition.
The Tsybakov noise condition is commonly used in statistical learning theory (see, e.g.,~\citep{bartlett2006convexity,audibert2007fast,lecue2007optimal}).
In the current work, we consider the following setting:
\begin{definition}[$\tau$-(NC) with $\theta_{\textup{NC}}$]
\label{def:definition of noise condition}
In the case where $\mathcal{X}_{0}=\mathcal{X}^{2}$ and $\nu=\mu\otimes\mu$, we say that a probability measure $P$ in $\mathcal{X}^{2}\times\mathcal{Y}$ that is absolutely continuous for $\mu\otimes\mu\otimes \chi$ satisfies $\tau$-(NC) with $\theta_{\textup{NC}}\in (0,1]$ if either of the following conditions is satisfied:
\begin{itemize}
\item $\tau>1$, and there is a constant $c\in [1,\theta_{\textup{NC}}^{-1}]$ such that $P$ satisfies the Tsybakov noise condition~\eqref{eq:margin assumption} with $s_{0}=1$, $\tau$, and $c$.
\item $\tau=1$, and there is a threshold $s_{0}\in [\theta_{\textup{NC}},1]$ such that $P$ satisfies the Tsybakov noise condition~\eqref{eq:margin assumption} with $s_{0}$, $\tau=1$, and any $c\geq 1$.
\end{itemize}
\end{definition}
Note that we need $\theta_{\textup{NC}}$ to apply Proposition~1 in~\citep{lecue2007optimal} in the proof method (see Section~\ref{subsubsec:outline}).
Note also that 1-(NC) with any $\theta_{\textup{NC}}\in (0,1]$ is an instance of the condition of~\citet{massart2006risk}.

\paragraph{Smooth boundaries.}
We recall the definition of smooth boundaries introduced by~\citet{imaizumi2022advantage}.
For any $\alpha> 0$, $R\geq 0$, and $K\in\mathbb{N}$, denote by $\mathcal{C}^{\alpha,K-1}_{R}$, the ball of the $\alpha$-H\"{o}lder space on $[0,1]^{K-1}$ with radius $R$, namely
\begin{align*}
\mathcal{C}_{R}^{\alpha,K-1}=
\left\{
\begin{array}{c|l}
\multirow{2}{*}{$h:[0,1]^{K-1}\to\mathbb{R}$} & h\textup{ is }\lceil \alpha-1\rceil\textup{-times differentiable,}\\
& \|h\|_{\mathcal{C}^{\alpha,K-1}}\leq R
\end{array}
\right\}.
\end{align*}
Here, let $\bm{s}\in\mathbb{N}_{0}^{K-1}:=(\mathbb{N}\cup\{0\})^{K-1}$ be the multi-index, let $\partial_{\bm{s}}$ be the differential operator, and let $\|h\|_{\infty}$ and $\|x\|_{\infty}$ be respectively the uniform norms for real-valued functions and vectors.
Then, the functional $\|\cdot\|_{\mathcal{C}^{\alpha,K-1}}$ is defined as
\begin{align*}
\|h\|_{\mathcal{C}^{\alpha,K-1}}
=\sum_{\substack{\bm{s}\in (\mathbb{N}_{0})^{K-1}:\\\|\bm{s}\|_{1}\leq \lceil \alpha-1\rceil}}\|\partial_{\bm{s}}h\|_{\infty}+\sum_{\substack{\bm{s}\in (\mathbb{N}_{0})^{K-1}:\\ \|\bm{s}\|_{1}=\lceil \alpha -1\rceil}}\sup_{\substack{x,x'\in [0,1]^{K-1},\\ x\neq x'}}\frac{|\partial_{\bm{s}}h(x)-\partial_{\bm{s}}h(x')|}{\|x-x'\|_{\infty}^{\alpha-\lceil\alpha -1\rceil}}.
\end{align*}

\begin{figure}
    \centering
    \includegraphics[width=0.65\linewidth]{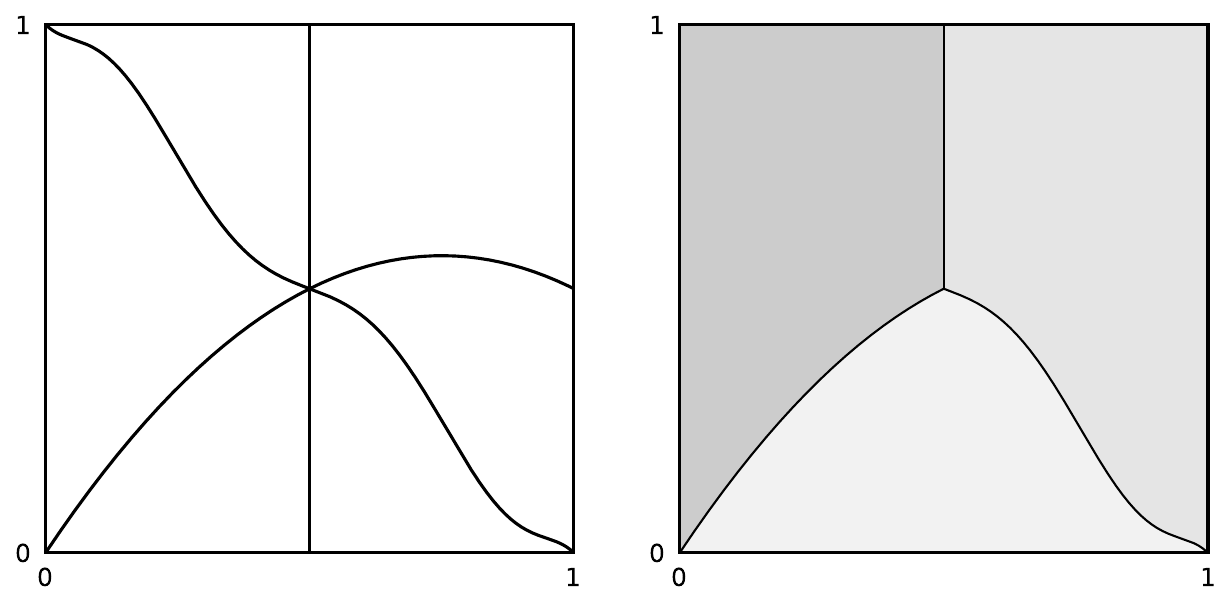}
    \caption{An example of Definition~\ref{def:class of smooth partitions}, namely the boundary class due to~\citep{imaizumi2022advantage}. In the left panel, three H\"{o}lder continuous functions divide the space $\mathcal{X}=[0,1]^{K}$ ($K=2$) into eight subsets $\{\bigcap_{k=1}^{3}\mathcal{L}_{s_{k},h_{k},j_{k}}\}_{(s_{1},s_{2},s_{3})\in \{-1,1\}^{3}}$, where note that this family includes two subsets that are empty sets. In the right panel, $\mathcal{X}$ is divided into three subsets $\{\mathcal{K}_{i}\}_{i=1}^{3}$ that are defined with $\{\bigcap_{k=1}^{3}\mathcal{L}_{s_{k},h_{k},j_{k}}\}_{(s_{1},s_{2},s_{3})\in\{-1,1\}^{3}}$ of the left panel.}
    \label{fig:smooth boundaries}
\end{figure}

\begin{definition}[Class $\mathscr{P}_{\alpha,R}^{K,d_{1},E}$]
\label{def:class of smooth partitions}
Given any $\alpha>0$, $R\geq 0$, $K,d_{1},E\in\mathbb{N}$ for which $2^{E}\geq d_{1}$ is satisfied, the class $\mathscr{P}_{\alpha,R}^{K,d_{1},E}$ is defined as
\begin{align*}
\mathscr{P}_{\alpha,R}^{K,d_{1},E}=
\left\{
\begin{array}{c|l}
\multirow{2}{*}{$\mathscr{S}$}
& \textup{The sequence }\mathscr{S}=\{\mathcal{K}_{i}\}_{i=1}^{d_{1}}\textup{ is a disjoint partition of }\mathcal{X} \\
& \textup{such that (P1) is satisfied with }\alpha,R,K,d_{1},\textup{and }E
\end{array}
\right\},
\end{align*}
where the following condition (P1) is due to~\citep{imaizumi2022advantage} (see Remark~\ref{rem:the property p2}):
\begin{enumerate}
\item[(P1)] Given $\mathscr{S}=\{\mathcal{K}_{i}\}_{i=1}^{d_{1}}$, there are a disjoint partition $\{\mathcal{I}_{i}\}_{i=1}^{d_{1}}$ of $\{1,-1\}^{E}$, functions $h_{1},\cdots,h_{E}\in\mathcal{C}_{R}^{\alpha,K-1}$, and indices $j_{1},\cdots,j_{E}\in\{1,\cdots,K\}$ such that for every $i=1,\cdots,d_{1}$,
\begin{align*}
\mathcal{K}_{i}&=\bigcup_{(s_{1},\cdots,s_{E})\in\mathcal{I}_{i}}\bigcap_{k=1}^{E}\mathcal{L}_{s_{k},h_{k},j_{k}},\\
\mathcal{L}_{s_{k},h_{k},j_{k}}&=
\begin{cases}
\{x\in\mathcal{X}\;|\; x_{j_{k}}\geq h_{k}(x_{\setminus j_{k}})\}&\quad\textup{ if }s_{k}=1,\\
\{x\in\mathcal{X}\;|\; x_{j_{k}}<h_{k}(x_{\setminus j_{k}})\}&\quad\textup{ if }s_{k}=-1,
\end{cases}
\end{align*}
where $x_{\setminus j}=(x_{1},\cdots,x_{j-1},x_{j+1},\cdots,x_{K})$.
\end{enumerate}
\end{definition}
\begin{remark}
\label{rem:the property p2}
Let $\{\mathcal{K}_{i}\}_{i=1}^{d_{1}}$ be any sequence of disjoint subsets that satisfies (P1), and let $\mathcal{K}_{i}=\bigcup_{(s_{1},\cdots,s_{E})\in\mathcal{I}_{i}}\bigcap_{k=1}^{E}\mathcal{L}_{s_{k},h_{k},j_{k}}$ for each $i=1,\cdots,d_{1}$.
In~\citep{imaizumi2022advantage}, a partition is defined as $\{\bigcup_{(s_{1},\cdots,s_{E})\in\mathcal{I}_{i}}\bigcap_{k=1}^{E}\{x\in\mathcal{X}\;|\; s_{k}x_{j_{k}}\geq s_{k}h_{k}(x_{\setminus j_{k}})\}\}_{i=1}^{d_{1}}$.
To define $\{\mathcal{K}_{i}\}_{i=1}^{d_{1}}$ as a sequence of disjoint subsets, we slightly simplify the definition of~\citep{imaizumi2022advantage}.
Note also that we often use the following useful property, which is indeed equivalent to the claim in~\citep[p.8]{imaizumi2022advantage}:
\begin{enumerate}
\item[(P2)] Given any $E\in\mathbb{N}$, $h_{1},\cdots,h_{E}\in\mathcal{C}_{R}^{\alpha,K-1}$, and $j_{1},\cdots,j_{E}\in \{1,\cdots,K\}$, it holds that $\mu(\{x\in\mathcal{X}\;|\; x_{j_{k}}=h_{k}(x_{\setminus j_{k}})\})=0$ for every $k\in\{1,\cdots,E\}$.
\end{enumerate}
In other words, for any given partition $\{\mathcal{K}_{i}\}_{i=1}^{d_{1}}\in\mathscr{P}_{\alpha,R}^{K,d_{1},E}$ and any pair $(i,j)$ of distinct indices, $\textup{cl}(\mathcal{K}_{i})\cap\textup{cl}(\mathcal{K}_{j})$ has Lebesgue measure zero, where $\textup{cl}(\cdot)$ denotes the closure in $\mathcal{X}$.
This property is an immediate consequence of Fubini's theorem.
Similarly to the analysis in~\citep{imaizumi2022advantage}, the subsequent analyses are still valid even if the original definition of~\citep{imaizumi2022advantage} is employed instead, thanks to property (P2).
\end{remark}
\begin{remark}
We provide some background of the class $\mathscr{P}_{\alpha,R}^{K,d_{1},E}$:
\label{rem:remarks on smooth boundaries}
\begin{enumerate}
\item[(i)] The set $\mathcal{L}_{s_{k},h_{k},j_{k}}$ is either the epigraph of an $\alpha$-H\"{o}lder continuous function or its complement and is usually referred to as ``boundary fragment''~\citep{mammen1999smooth}.
Note that this set is originally considered in~\citep{dudley1974metric} and is also studied in~\citep{petersen2018optimal}.
The subset $\bigcap_{k=1}^{E}\mathcal{L}_{s_{k},h_{k},j_{k}}$ may have a piecewise smooth boundary and some corners, and its statistical property is studied in~\citep{imaizumi2019deep}.
The set $\mathcal{K}_{i}$ is defined as the union of the subsets in $\{\bigcap_{k=1}^{E}\mathcal{L}_{s_{k},h_{k},j_{k}}\}_{\bm{s}\in \mathcal{I}_{i}}$ (see Figure~\ref{fig:smooth boundaries}).
As discussed in Remark~\ref{rem:the property p2}, $\{\mathcal{K}_{i}\}_{i=1}^{d_{1}}$ is a partition of $\mathcal{X}$.
In the context of binary classification, a similar definition of smooth partitions is also employed in~\citep{kim2021fast}.
\item[(ii)] While some regression problems defined with smooth boundaries are studied in~\citep{imaizumi2019deep,imaizumi2022advantage}, several classes of smooth subsets are usually considered in the context of classification problems (see, e.g.,~\citep{tsybakov2004optimal,petersen2018optimal,caragea2023neural,kim2021fast,meyer2023optimal}).
Some detailed review and discussion can be found in Appendix~A.3\showsupplement{of}.
\end{enumerate}
\end{remark}

\subsection{Learning Models Using Deep ReLU Networks}
\label{subsec:learning models}

In this section, suppose that natural numbers $K$ and $d_{1}$ are arbitrary fixed.
The purpose of this section is to introduce two classes of vector-valued functions, which are essential to present the main theorem of this work.
The vector-valued functions are defined with a regular simplex, following~\citep{awasthi2022do}.
Such functions are often considered in the context of machine learning~\citep{liu2021learning,awasthi2022do,graf2021dissecting,zhu2022balanced,chen2022perfectly,lee2024analysis,koromilas2024bridging}.
In the current work, the regular simplex plays some crucial roles in establishing the convergence rates (see Section~\ref{subsec:loss function continued} for the details).

Similarly to~\citep{awasthi2022do}, we make use of a regular simplex in the learning algorithm.
Define $d=d(d_{1}):=d_{1}-1$, for simplicity.
Denote by $\mathcal{S}^{d-1}$, the unit hypersphere in $\mathbb{R}^{d}$ centered at the origin.
Let $v_{1},\cdots,v_{d_{1}}\in\mathcal{S}^{d-1}$ be vectors satisfying $\sum_{i=1}^{d_{1}}v_{i}=\bm{0}$ and the condition that $\|v_{i}-v_{j}\|_{2}=\|v_{i'}-v_{j'}\|_{2}$ for any $i,j,i',j'\in\{1,\cdots,d_{1}\}$ such that $i\neq j$ and $i'\neq j'$ (see, e.g.,~\citep[Corollary~2.6]{conn2009introduction}).
Then, a regular simplex $\Delta^{d}$ is commonly defined as 
\begin{align*}
\Delta^{d}=\left\{c_{1}v_{1}+\cdots+c_{d_{1}}v_{d_{1}}\;|\;\bm{c}=(c_{1},\cdots,c_{d_{1}})\in[0,1]^{d_{1}}, \|\bm{c}\|_{1}=1\right\}.
\end{align*}
We endow $\Delta^{d}$ with the subspace topology from $\mathbb{R}^{d}$.
We also define 
\begin{align*}
D_{\Delta^{d}}=\max_{z,z'\in\Delta^{d}}\|z-z'\|_{2}.
\end{align*}
For some basic properties of a regular simplex used in the proofs, see Appendix~B.1\showsupplement{in}.
Then, we define the following function classes:

\paragraph{The whole class.}
Given a function $f:\mathcal{X}\to\Delta^{d}$, by the definition of $\Delta^{d}$, there are some functions $g_{1},\cdots,g_{d_{1}}:\mathcal{X}\to\mathbb{R}$ such that $f=\sum_{i=1}^{d_{1}}g_{i}v_{i}$.
Note that such functions $g_{1},\cdots,g_{d_{1}}$ are uniquely determined for each $f$, since $v_{1},\cdots,v_{d_{1}}$ are affinely independent (see, e.g.,~\citep[pp.102--104]{hatcher2002algebraic}).
Define
\begin{align*}
\mathcal{F}_{0}=\left\{f:\mathcal{X}\to\Delta^{d}\;\Big|\; f=\sum_{i=1}^{d_{1}}g_{i}v_{i},\; g_{1},\cdots,g_{d_{1}}\textup{ are measurable}\right\}.
\end{align*}

\paragraph{ReLU networks.}
Let $L\in\mathbb{N}$, $d_{\textup{NN},0},\cdots,d_{\textup{NN},L}\in\mathbb{N}$, and $\bm{d}=(d_{\textup{NN},0},\cdots,d_{\textup{NN},L})$.
Denote the ReLU function by $\sigma_{\textup{ReLU}}:\mathbb{R}\to\mathbb{R}$, $\sigma_{\textup{ReLU}}(s)=s\vee 0$.
Given $\bm{W}=(W_{1},\cdots,W_{L})\in \prod_{i=1}^{L}\mathbb{R}^{d_{\textup{NN},i}\times d_{\textup{NN},i-1}}$ and $\bm{b}=(b_{1},\cdots,b_{L})\in\prod_{i=1}^{L}\mathbb{R}^{d_{\textup{NN},i}}$, we define ReLU networks as 
\begin{align}
\label{eq:definition of relu networks}
g_{\bm{W},\bm{b}}=A_{L}\circ \bm{\sigma}_{\textup{ReLU},d_{\textup{NN},L-1}} \circ A_{L-1}\circ \cdots\circ \bm{\sigma}_{\textup{ReLU},d_{\textup{NN},1}} \circ A_{1},
\end{align}
where $A_{1},\cdots,A_{L}$ and $\bm{\sigma}_{\textup{ReLU},d_{\textup{NN},1}},\cdots,\bm{\sigma}_{\textup{ReLU},d_{\textup{NN},L-1}}$ are defined as
\begin{align*}
A_{i}(z)=W_{i}z+b_{i} \quad \textup{for every }z\in\mathbb{R}^{d_{\textup{NN},i-1}},
\end{align*}
and
\begin{align*}
\bm{\sigma}_{\textup{ReLU},d_{\textup{NN},i}}(z)=(\sigma_{\textup{ReLU}}(z_{1}),\cdots,\sigma_{\textup{ReLU}}(z_{d_{\textup{NN},i}})) \quad\textup{for every } z\in\mathbb{R}^{d_{\textup{NN},i}}.
\end{align*}
Given any $L\in\mathbb{N}$, $J,S,M\geq 0$, and $\bm{d}=(d_{\textup{NN},0},\cdots,d_{\textup{NN},L})\in\mathbb{N}^{L+1}$, we employ the following standard class $\mathcal{F}_{L,J,S,M,\bm{d}}^{\textup{NN}}$ of ReLU networks studied in~\citep{nakada2020adaptive,imaizumi2019deep,imaizumi2022advantage}:
\begin{align*}
&\mathcal{F}_{L,J,S,M,\bm{d}}^{\textup{NN}}\\
&=\left\{
\begin{array}{@{}l@{}|l@{}}
\multirow{3}{*}{$g_{\bm{W},\bm{b}}:[0,1]^{d_{\textup{NN},0}}\to\mathbb{R}^{d_{\textup{NN},L}}$}
&\bm{W}\in\prod_{i=1}^{L}\mathbb{R}^{d_{\textup{NN},i}\times d_{\textup{NN},i-1}},\bm{b}\in\prod_{i=1}^{L}\mathbb{R}^{d_{\textup{NN},i}},\\
&\|\bm{W}\|_{\infty}\vee\|\bm{b}\|_{\infty}\leq J,\|\bm{W}\|_{0}+\|\bm{b}\|_{0}\leq S,\\
&\|g_{\bm{W},\bm{b}}\|_{[0,1]^{d_{\textup{NN},0}},\infty}\leq M
\end{array}
\right\},
\end{align*}
where in this definition, $\|\bm{W}\|_{\infty}$ denotes the uniform norm of $\bm{W}$, and $\|\bm{W}\|_{0}$ and $\|\bm{b}\|_{0}$ denote the number of non-zero entries in $\bm{W}$ and $\bm{b}$, respectively.
The constraints in $\mathcal{F}_{L,J,S,M,\bm{d}}^{\textup{NN}}$ are commonly used (see, e.g.,~\citep{schmidt-hieber2020nonparametric,petersen2018optimal,bos2022convergence,nakada2020adaptive,imaizumi2019deep,imaizumi2022advantage}).
In the current work, we employ this class due to the following reasons: (i) We can use a covering number bound shown in~\citep{nakada2020adaptive}. (ii) We can apply the approximation theory of indicator functions developed in~\citep{petersen2018optimal,imaizumi2019deep,imaizumi2022advantage}.

\paragraph{Classes of $\Delta^{d}$-valued ReLU networks.}
Since the range of every vector-valued function in $\mathcal{F}_{0}$ is restricted to $\Delta^{d}$, we need to introduce additional notation of $\Delta^{d}$-valued ReLU networks.

Let $H=(H_{1},\cdots,H_{d_{1}}):\mathbb{R}^{d_{1}}\to\mathbb{R}^{d_{1}}$ be the softmax function, namely, for any $i\in\{1,\cdots,d_{1}\}$ we define $H_{i}(z)=\exp(z_{i})/(\sum_{j=1}^{d_{1}}\exp(z_{j}))$.
Given $L\in\mathbb{N}$, $J,S,M\geq 0$, $\bm{d}=(d_{\textup{NN},0},\cdots,d_{\textup{NN},L-1},d_{1})\in \mathbb{N}^{L+1}$, we define the class of $\Delta^{d}$-valued ReLU networks as
\begin{align*}
\mathcal{F}_{L,J,S,M,\bm{d}}^{\Delta^{d}\textup{-NN}}=\left\{f_{\bm{W},\bm{b}}=\sum_{i=1}^{d_{1}} (H_{i}\circ g_{\bm{W},\bm{b}})v_{i}\;\bigg|\;g_{\bm{W},\bm{b}}\in\mathcal{F}_{L,J,S,M,\bm{d}}^{\textup{NN}}\right\}.
\end{align*}
In the main results, we follow~\citep{petersen2018optimal,imaizumi2019deep,imaizumi2022advantage,bos2022convergence} to determine $L,J,S,M$, and $\bm{d}$.

\begin{remark}
\label{rem:the dimension is known}
We develop an algorithm that requires the value of $d_{1}$ to be known in advance (see Section~\ref{subsec:loss function continued}).
This condition is mild both in the contexts of nonparametric boundary estimation and machine learning:
Firstly, the definition of partitions introduced by~\citet{imaizumi2022advantage}, which is also used in this work, implicitly assumes that the number of subsets should be less than or equal to $2^{E}$.
Besides, in the context of contrastive learning, some similar assumptions on the number of subsets are considered in~\citep{haochen2021provable,awasthi2022do,parulekar2023infonce}.
\end{remark}

\section{Formalization and Main Results}
\label{sec:main results}

In this section, we present the formal statement of Theorem~\ref{thm:informal version of local theorem}.

\subsection{Problem Setting}
\label{subsec:assumptions}

\paragraph{Assumptions.}
We introduce a new class $\mathcal{P}_{\alpha,\tau,\xi}$ of probability distributions in $\mathcal{X}^{2}\times\mathcal{Y}$.
This class is parameterized with the set $\mathscr{P}_{\alpha,R}^{K,d_{1},E}$ (see Definition~\ref{def:class of smooth partitions}).
While one can interpret this class as an extension of the standard setting of binary classification (see, e.g.,~\citep[p.142]{tsybakov2004optimal} and \cite[Theorem~3.1]{kim2021fast}), the main difference is that the decision boundary is defined with the union of disjoint subsets in $\mathcal{X}^{2}$, which enables us to formalize the problem with multiple boundaries (see Remark~\ref{rem:some remarks on the main assumptions}--(iii)).
\begin{definition}[Class $\mathcal{P}_{\alpha,\tau,\xi}$]
\label{def:main assumption}
Given any $\alpha> 0$, $\tau\geq 1$, $R\geq 1$, $K,d_{1},E\in\mathbb{N}$ for which $2^{E}\geq d_{1}$, $\theta_{\textup{NC}}\in (0,1]$, $\theta_{1}\geq 1$, $0<\theta_{2}\leq \frac{1}{2}$, and $\frac{1}{2} \leq \theta_{3}<1$, we define
\begin{align*}
\mathcal{P}_{\alpha,\tau,\xi}
=\left\{
\begin{array}{c|l}
\multirow{3}{*}{$P$}&P\textup{ is a Borel probability measure in }\mathcal{X}^{2}\times\mathcal{Y}\\
& \textup{such that all of (A1) -- (A4) are satisfied with}\\
& \alpha, \tau, \textup{ and } \xi=(R,K,d_{1},E,\theta_{\textup{NC}},\theta_{1},\theta_{2},\theta_{3})
\end{array}
\right\},
\end{align*}
where conditions (A1) -- (A4) are defined as follows:
\begin{enumerate}
\item[(A1)] $P$ is absolutely continuous for the product measure $\mu\otimes\mu\otimes\chi$. The density $p(x,x',y)$ of $P$ on $\mathcal{X}^{2}\times\mathcal{Y}$ satisfies condition~\eqref{eq:condition of tsai et al} due to~\citep{tsai2020neural}.
\item[(A2)] $P$ satisfies $\tau$-(NC) with $\theta_{\textup{NC}}$ (see Definition~\ref{def:definition of noise condition}).
\item[(A3)] $q(x,x')$ is a symmetric function satisfying that $\|q\|_{L^{\infty}(\mathcal{X}^{2})}\leq \theta_{1}^{2}$. Also, $p_{X}(x),p_{X'}(x')$ are positive and continuous at every $x,x'\in\mathcal{X}$, and it holds that $\|p_{X}\|_{L^{\infty}(\mathcal{X})}\vee \|p_{X'}\|_{L^{\infty}(\mathcal{X})} \leq \theta_{1}$.
In addition, $p_{Y}(-1)\in [\theta_{2},1)$.
\item[(A4)] There is a sequence $\mathscr{S}=\{\mathcal{K}_{i}\}_{i=1}^{d_{1}}\in\mathscr{P}_{\alpha,R}^{K,d_{1},E}$ such that $P_{X}(\mathcal{K}_{i})\leq \theta_{3}$ for every $i\in\{1,\cdots,d_{1}\}$, and we have
\begin{align}
\label{eq:statistical partition of unity}
\left\{(x,x')\in\mathcal{X}^{2}\;\Big|\;\eta(x,x')\geq\frac{1}{2}\right\}=\bigcup_{i=1}^{d_{1}}\mathcal{K}_{i}\times\mathcal{K}_{i}.
\end{align}
\end{enumerate}
\end{definition}
\begin{definition}
\label{def:auxiliary notions used in the theorem}
We also introduce the following notions related to Definition~\ref{def:main assumption}:
\begin{itemize}
\item The set $\Xi$ of hyperparameters is defined as
\begin{align*}
\Xi=
\left\{
\begin{array}{@{}l|l@{}}
\multirow{3}{*}{$(R,K,d_{1},E,\theta_{\textup{NC}},\theta_{1},\theta_{2},\theta_{3})$} & R\geq 1, K,d_{1},E\in\mathbb{N} \textup{ for which } \\
& 2^{E}\geq d_{1}, \theta_{\textup{NC}}\in (0,1], \textup{ and} \\
& 0<\theta_{2}\leq \frac{1}{2} \leq \theta_{3}<1\leq \theta_{1}
\end{array}
\right\}.
\end{align*}
\item Given $\tau\geq 1$ and $\xi\in\Xi$, we define
\begin{align*}
\mathcal{P}_{\tau,\xi}=\bigcup_{\alpha>0}\mathcal{P}_{\alpha,\tau,\xi},\;\;\textup{and}\quad \mathcal{P}_{\xi}=\bigcup_{\tau\geq 1}\mathcal{P}_{\tau,\xi}.
\end{align*}
\item Given $\xi=(R,K,d_{1},E,\theta_{\textup{NC}},\theta_{1},\theta_{2},\theta_{3})\in \Xi$, let $\mathscr{S}_{P}$ denote the map
\begin{align*}
\mathcal{P}_{\xi}\ni P\mapsto \mathscr{S}_{P}\in\bigcup_{\alpha>0}\mathscr{P}_{\alpha,R}^{K,d_{1},E},
\end{align*}
for which for each $P\in\mathcal{P}_{\xi}$, some $\alpha>0$ and $\tau\geq 1$ satisfying that $P\in\mathcal{P}_{\alpha,\tau,\xi}$, and $\{\mathcal{K}_{i}\}_{i=1}^{d_{1}}\in\mathscr{P}_{\alpha,R}^{K,d_{1},E}$ satisfying condition (A4) for $P$, it holds that $\mathscr{S}_{P}=\{\mathcal{K}_{i}\}_{i=1}^{d_{1}}$.
\end{itemize}
\end{definition}

\begin{figure}
    \centering
    \includegraphics[width=0.4\linewidth]{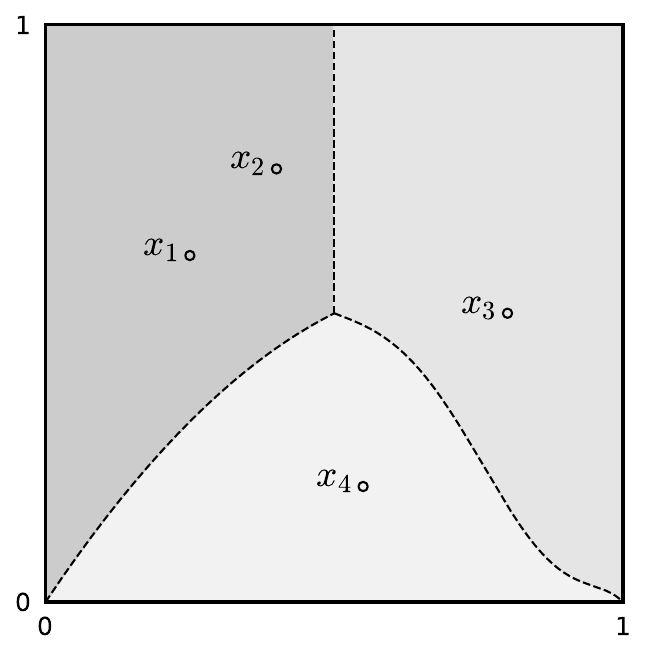}
    \caption{An illustration of condition~\eqref{eq:statistical partition of unity}, where we use the same boundaries as those in~Figure~\ref{fig:smooth boundaries}. See also Remark~\ref{rem:some remarks on the main assumptions}--(iii) for the discussion of~\eqref{eq:statistical partition of unity}. By~\eqref{eq:statistical partition of unity}, we assume that the pair $(x_{1},x_{2})$ satisfies $\eta(x_{1},x_{2})\geq \frac{1}{2}$ since these points belong to the same subset. In addition, we assume that any $(x,x')\in\{(x_{1},x_{3}),(x_{1},x_{4}),(x_{2},x_{3}),(x_{2},x_{4}),(x_{3},x_{4})\}$ satisfies $\eta(x,x')<\frac{1}{2}$. In the current work we investigate the learnability of unknown boundaries from such pairwise relations.}
    \label{fig:definition of similarity}
\end{figure}

\begin{remark}
\label{rem:some remarks on the main assumptions}
We provide additional discussion of the main assumptions:
\begin{enumerate}
\item[(i)] In condition (A3), we assume that $q(x,x')$ is symmetric, implying that $p_{X}=p_{X'}$.
This assumption is reasonable since the purpose of the current work is to study the learnability of smooth boundaries in a single partition $\{\mathcal{K}_{i}\}_{i=1}^{d_{1}}$, as in the literature~\citep{tsybakov2004optimal,kim2021fast,imaizumi2019deep,imaizumi2022advantage,meyer2023optimal}.
\item[(ii)] The thresholds $\theta_{\textup{NC}}$, $\theta_{1},\theta_{2}$, and $\theta_{3}$ are required to analyze the supremum of the risk function.
The variable $\theta_{\textup{NC}}$ is introduced in Definition~\ref{def:definition of noise condition}.
The conditions $\|q\|_{L^{\infty}(\mathcal{X}^{2})}\leq \theta_{1}^{2}$ and $\|p_{X}\|_{L^{\infty}(\mathcal{X})}\vee \|p_{X'}\|_{L^{\infty}(\mathcal{X})}\leq \theta_{1}$ are useful when evaluating the approximation errors, similarly to~\citep[Theorem~7]{imaizumi2022advantage} and~\citep[Lemma~A.3]{kim2021fast}.
Meanwhile, the conditions $p_{Y}(-1)\geq \theta_{2}$ and $P_{X}(\mathcal{K}_{i})\leq \theta_{3}$ are tailored to our proof method and are used in the proofs of Theorem~\ref{thm:main result} and Lemma~\ref{lem:step 3 lemma 0}, respectively (see Appendix~B.5 and Appendix~B.3\showsupplement{in}).
See Remark~\ref{remark:condition on theorem for local erm}--(iii) for the result under the setting where all the conditions using either $\theta_{2}$ or $\theta_{3}$ are relaxed.
\item[(iii)] For each sequence $\mathscr{S}=\{\mathcal{K}_{i}\}_{i=1}^{d_{1}}\in\mathscr{P}_{\alpha,R}^{K,d_{1},E}$ satisfying condition (A4) in Definition~\ref{def:main assumption} for some $P\in\mathcal{P}_{\alpha,\tau,\xi}$, and any $x,x'\in\mathcal{X}$, condition \eqref{eq:statistical partition of unity} means that
\begin{align*}
p(y=1|x,x')\geq p(y=-1|x,x')\;\;\iff\;\; \exists i\in\{1,\cdots,d_{1}\} \textup{ such that }x,x'\in\mathcal{K}_{i}.
\end{align*}
Intuitively, we can interpret that $x$ and $x'$ are similar points if $\eta(x,x')\geq \frac{1}{2}$.
In other words, $x$ and $x'$ are dissimilar if $\eta(x,x')<\frac{1}{2}$.
Thus, condition~\eqref{eq:statistical partition of unity} provides a connection between the pairwise relation and the smooth subsets.
See also Figure~\ref{fig:definition of similarity} for an illustration of condition~\eqref{eq:statistical partition of unity}.
Note that for any random variable $(X,X',Y)\sim P\in\mathcal{P}_{\xi}$, even if it holds that $\eta(X(\omega),X'(\omega))\geq \frac{1}{2}$ for some $\omega\in\Omega$, $Y(\omega)$ is not necessarily equal to $1$.
Note also that condition~\eqref{eq:statistical partition of unity} has some mathematical relations to some other conditions introduced by~\citep{awasthi2022do,waida2023towards,parulekar2023infonce} in the context of contrastive learning (see Section~\ref{appsec:comparison with related mathematical notions}).
\item[(iv)] By condition~\eqref{eq:statistical partition of unity}, the map $P\mapsto \mathscr{S}_{P}$ is well-defined.
Given any hyperparameter $\xi=(R,K,d_{1},E,\theta_{\textup{NC}},\theta_{1},\theta_{2},\theta_{3})\in\Xi$, this map identifies each $P\in\mathcal{P}_{\xi}$ with some parameter $\{\mathcal{K}_{i}\}_{i=1}^{d_{1}}\in\bigcup_{\alpha>0}\mathcal{P}_{\alpha,R}^{K,d_{1},E}$.
See Appendix~A.1\showsupplement{in} for a property of the map $\mathscr{S}_{P}$.
\end{enumerate}
\end{remark}

\paragraph{Risk functions.}
Let $\widehat{g}_{n}:(\mathcal{X}^{2}\times\mathcal{Y})^{n}\to\mathcal{G}_{0}$ be a map called \emph{estimator} (see Definition~\ref{def:definitions of local and global estimators}), where
\begin{align*}
\mathcal{G}_{0}=
\left\{
\begin{array}{c|l}
\multirow{2}{*}{$g:\mathcal{X}\to [0,1]^{d_{1}}$} & g=(g_{1},\cdots,g_{d_{1}}), g_{1},\cdots,g_{d_{1}}:\mathcal{X}\to[0,1]\textup{ are}\\
& \textup{measurable, and } \sum_{i=1}^{d_{1}}g_{i}(x)=1\;\textup{for each }x\in\mathcal{X}
\end{array}
\right\}.
\end{align*}
Given $\alpha>0$, $\tau\geq 1$, and $\xi\in\Xi$, let $P\in\mathcal{P}_{\alpha,\tau,\xi}$.
Then, we aim at estimating the sets in the sequence $\mathscr{S}_{P}=\{\mathcal{K}_{i}\}_{i=1}^{d_{1}}$.
Similarly to~\citep{tsybakov2004optimal,imaizumi2019deep,imaizumi2022advantage,meyer2023optimal}, we analyze the convergence rates of estimators in terms of the $L^{2}$-risk\footnote{Note that for any given estimator $\widehat{g}_{n}:(\mathcal{X}^{2}\times\mathcal{Y})^{n}\to\mathcal{G}_{0}$, we can write as $\widehat{g}_{n}=(\widehat{g}_{n,1},\cdots,\widehat{g}_{n,d_{1}})$ in the sense that $\widehat{g}_{n}(u_{1}^{n})=(\widehat{g}_{n,1}(u_{1}^{n}),\cdots,\widehat{g}_{n,d_{1}}(u_{1}^{n}))$ for any $u_{1}^{n}:=(u_{1},\cdots,u_{n})\in (\mathcal{X}^{2}\times\mathcal{Y})^{n}$.}:
\begin{definition}[$L^{2}$-risk]
\label{def:l2 risk}
Given any $\alpha>0$, $\tau\geq 1$, and $\xi=(R,K,d_{1},E,\theta_{\textup{NC}},\theta_{1},\theta_{2},\theta_{3})\in \Xi$, any probability distribution $P\in\mathcal{P}_{\alpha,\tau,\xi}$, and any estimator $\widehat{g}_{n}:(\mathcal{X}^{2}\times\mathcal{Y})^{n}\to\mathcal{G}_{0}$, the $L^{2}$-risk is defined as
\begin{align}
\label{eq:sup l2 risk}
\mathcal{R}(\widehat{g}_{n};P)=\mathbb{E}\left[\sum_{i=1}^{d_{1}}\|\widehat{g}_{n,i}(U_{1},\cdots,U_{n})-\mathds{1}_{\mathcal{K}_{i}}\|_{L^{2}(\mathcal{X},P_{X})}^{2}\right],
\end{align}
where $\widehat{g}_{n}=(\widehat{g}_{n,1},\cdots,\widehat{g}_{n,d_{1}})$, $\mathscr{S}_{P}=\{\mathcal{K}_{i}\}_{i=1}^{d_{1}}$, and for any sequence of i.i.d. $(\mathcal{X}^{2}\times\mathcal{Y})$-valued random variables $U_{1}=(X_{1},X_{1}',Y_{1}),\cdots,U_{n}=(X_{n},X_{n}',Y_{n})$ drawn from the distribution $P$, the expectation in the right-hand side of~\eqref{eq:sup l2 risk} is taken in terms of the distribution of $U_{1}^{n}:=(U_{1},\cdots,U_{n})$.
\end{definition}
Note that the order of subsets is estimated simultaneously, as in the standard settings of conventional classification problems~\citep{tsybakov2004optimal,tarigan2008moment,kim2021fast,meyer2023optimal}.

\subsection{Key Notions}
\label{subsec:core notations}

We find that some proof method is required to bypass a technical difficulty in the analysis of the $L^{2}$-risk in Definition~\ref{def:l2 risk}.
To maintain the readability, we postpone the details until Section~\ref{subsec:general upper bound of main result} and introduce several notions that will be used in the main theorem.

We define a contrastive function, which is slightly generalized from Definition~3.7 in~\citep{awasthi2022do}.
This notion identifies the parameter $\mathscr{S}_{P}$ of the given $P\in\mathcal{P}_{\xi}$ with some $\Delta^{d}$-valued function.
\begin{definition}[Contrastive function]
\label{def:contrastive representations}
For $\xi=(R,K,d_{1},E,\theta_{\textup{NC}},\theta_{1},\theta_{2},\theta_{3})\in \Xi$, $P\in\mathcal{P}_{\xi}$, and $\mathscr{S}_{P}=\{\mathcal{K}_{i}\}_{i=1}^{d_{1}}$, the \emph{contrastive function} $f^{*}:\mathcal{X}\to\Delta^{d}$ of $P$ is defined as
\begin{align*}
f^{*}(x)=\sum_{i=1}^{d_{1}}\mathds{1}_{\mathcal{K}_{i}}(x)v_{i}.
\end{align*}
\end{definition}
In~\citep[Definition~3.7]{awasthi2022do}, the function $\sum_{i=1}^{d'}\mathds{1}_{\mathcal{K}_{i}'}v_{i}$ is also employed, where $d'\leq d_{1}$, and $\{\mathcal{K}_{i}'\}_{i=1}^{d'}$ is a sequence of disjoint subsets parameterizing the distribution introduced in~\citep{arora2019theoretical} in the sense of~\citep[Assumption~3.1]{awasthi2022do}.
The difference from Definition~3.7 in~\citep{awasthi2022do} is that condition~\eqref{eq:statistical partition of unity} in Definition~\ref{def:main assumption} is weaker than the setting considered in~\citep{awasthi2022do}, as discussed in Section~\ref{appsec:comparison with related mathematical notions}.
In the context of contrastive learning, some different types of simplex-valued functions are considered in~\citep{haochen2021provable,lee2024analysis,koromilas2024bridging}, and also in~\citep{graf2021dissecting,zhu2022balanced,chen2022perfectly} under a supervised metric learning setting of~\citep{khosla2020supervised}.
Hence, in the current work we use the terminology \emph{contrastive function}.

Using the notion of contrastive function, we introduce a subclass of $\mathcal{F}_{0}$.
In a nutshell, we use this notion to address the technical difficulty due to the non-identifiability issue of the problem setting in Section~\ref{subsec:assumptions} (see Section~\ref{subsubsec:outline}).
\begin{definition}[$(\beta,\beta_{0},P)$-localized subclass]
\label{def:concentrated subclass}
Given any $\beta>0$, $\beta_{0}\geq 0$, hyperparameter $\xi=(R,K,d_{1},E,\theta_{\textup{NC}},\theta_{1},\theta_{2},\theta_{3})\in \Xi$, $P\in\mathcal{P}_{\xi}$, and the contrastive function $f^{*}$ of $P$, the $(\beta,\beta_{0},P)$-\emph{localized subclass} $\mathscr{F}_{\beta,\beta_{0},P}(\mathcal{F})$ of a set $\mathcal{F}\subset\mathcal{F}_{0}$ is defined as
\begin{align*}
\mathscr{F}_{\beta,\beta_{0},P}(\mathcal{F})=\left\{f\in\mathcal{F}\;|\; P_{X}\left(\{x\in\mathcal{X}\;|\;\|f(x)-f^{*}(x)\|_{2}<\beta\}\right)\geq 1-\beta_{0}\right\}.
\end{align*}
\end{definition}
\begin{remark}
\label{rem:remarks about localized subclass}
We provide several remarks of Definition~\ref{def:concentrated subclass}:
\begin{enumerate}
\item[(i)] Let $D_{\textup{proj}}$ denote the distance between $v_{1}$ and the simplex that does not contain $v_{1}$, namely
\begin{align*}
D_{\textup{proj}}=\inf
\left\{
\|z-v_{1}\|_{2} \;|\; z=0\cdot v_{1}+\sum_{i=2}^{d_{1}}c_{i}v_{i}\in \Delta^{d}
\right\}.
\end{align*}
In particular, we often consider the parametrization $\mathscr{F}_{\beta,\beta^{-1}\varepsilon,P}$, where $\beta\in (0,D_{\textup{proj}})$, $\varepsilon>0$, $\xi\in\Xi$, and $P\in\mathcal{P}_{\xi}$.
Note that if $\beta^{-1}\varepsilon\geq 1$, then $\mathscr{F}_{\beta,\beta^{-1}\varepsilon,P}(\mathcal{F})=\mathcal{F}$ for any $\mathcal{F}\subset\mathcal{F}_{0}$.
\item[(ii)] For any $\xi\in\Xi$, $P\in\mathcal{P}_{\xi}$, $\beta\in(0,D_{\textup{proj}})$ and $\mathcal{F}\subset\mathcal{F}_{0}$, it holds that $\mathscr{F}_{\beta,\beta^{-1}\varepsilon,P}(\mathcal{F})\subseteq \mathscr{F}_{\beta,\beta^{-1}\varepsilon',P}(\mathcal{F})$ for any $0\leq \varepsilon\leq \varepsilon'$.
Moreover, for the contrastive function $f^{*}$ of $P$ and any decreasing positive sequence $\{\varepsilon_{n}\}_{n\in\mathbb{N}}$ such that $\varepsilon_{n}\to 0$ as $n\to\infty$, the definition directly implies that
\begin{align*}
\bigcup_{n'=n}^{\infty}\bigcap_{k=n'}^{\infty}\mathscr{F}_{\beta,\beta^{-1}\varepsilon_{k},P}(\mathcal{F})\supset\{f\in\mathcal{F}\;|\; \|f-f^{*}\|_{2}<\beta\;\; P_{X}\textup{-almost surely}\}.
\end{align*}
Thus, $\mathscr{F}_{\beta,\beta^{-1}\varepsilon_{n},P}(\mathcal{F})$ contains a neighborhood of $f^{*}$ in the space $\mathcal{F}_{0}$ endowed with the topological structure induced by the semi-norm $\|\cdot\|_{\mathcal{X},P_{X},\infty}$.
\end{enumerate}
\end{remark}

\begin{figure}
    \centering
    \includegraphics[width=0.9\linewidth]{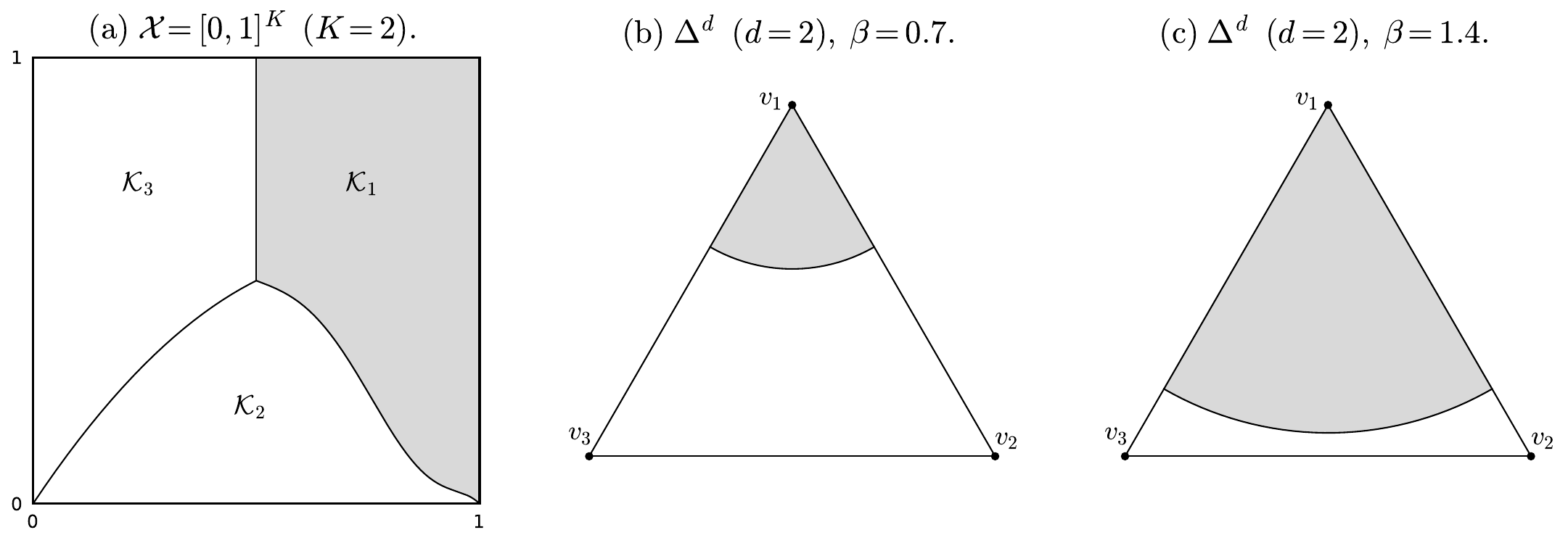}
    \caption{An illustration of the constraint in the definition of localized subclasses (see Definition~\ref{def:concentrated subclass}). In the left panel, we consider the same boundaries as those in Figure~\ref{fig:smooth boundaries}. For instance, the constraint requires the given function $f:\mathcal{X}\to\Delta^{d}$ to map points belonging to the subset $\mathcal{K}_{1}$ of the left panel in the shaded subset of $\Delta^{d}$ with probability at least $1-\beta_{0}$. Here, in the middle or right panel, the shaded area shows the subset $\{z\in\Delta^{d}\;|\; \|z-v_{1}\|_{2}<\beta\}$, where note that the contrastive function $f^{*}$ satisfies that $f^{*}(x)\in\{v_{1},\cdots,v_{d_{1}}\}$ for any $x\in\mathcal{X}$. Similar requirements apply to the subsets $\mathcal{K}_{2}$ and $\mathcal{K}_{3}$ in the left panel, with different vertices.}
    \label{fig:localized subclass}
\end{figure}

In addition to Remark~\ref{rem:remarks about localized subclass}--(ii), the localized subclass also contains vector-valued functions that may be pointwisely far apart from the true function $f^{*}$ since $\beta$ can take any value in $(0,D_{\textup{proj}})$ (see Figure~\ref{fig:localized subclass}).
For instance, given $\beta\in(1,D_{\textup{proj}})$, $\beta_{0}\geq 0$, a subset $\mathcal{A}\in\mathcal{B}(\mathcal{X})$ for which $P_{X}(\mathcal{A})\geq 1-\beta_{0}$ is satisfied, and the function $f_{0}\in\mathcal{F}_{0}$ satisfying $f_{0}(x)=\bm{0}$ if $x\in\mathcal{A}$ and $f_{0}(x)=v_{1}$ otherwise, we have $f_{0}\in\mathscr{F}_{\beta,\beta_{0},P}(\mathcal{F}_{0})$.
To estimate the boundaries, we need to construct an algorithm that outputs a function close to $f^{*}$, not $f_{0}$.

Some discussion of the comparison to several similar notions introduced in~\citep{schiebinger2015geometry,trillos2021geometric,mendelson2015learning,mendelson2017local} in some different contexts can be found in Section~\ref{appsec:comparison with related mathematical notions}.
In addition, we provide some interpretations of the ERM algorithm using a localized subclass in Section~\ref{subsec:discussion of the formal version of the main theorem}.

We formally define \emph{local} and \emph{global} estimators:
\begin{definition}
\label{def:definitions of local and global estimators}
A map $\widehat{f}_{n}:(\mathcal{X}^{2}\times\mathcal{Y})^{n}\to\mathcal{F}_{0}$ is called \emph{estimator}.
Given any estimator $\widehat{f}_{n}$, the map $\widehat{g}_{n}:(\mathcal{X}^{2}\times\mathcal{Y})^{n}\to \mathcal{G}_{0}$ satisfying $\widehat{f}_{n}=\sum_{i=1}^{d_{1}}\widehat{g}_{n,i}v_{i}$ with $\widehat{g}_{n}=(\widehat{g}_{n,1},\cdots,\widehat{g}_{n,d_{1}})$ is also called \emph{estimator} and is uniquely determined since $v_{1},\cdots,v_{d_{1}}$ are affinely independent.
In particular, we consider the following classes of maps:
\begin{itemize}
    \item Let $\xi\in\Xi$, and let $\mathcal{P}\subset \mathcal{P}_{\xi}$ be arbitrary. Given an estimator $\widehat{f}_{n,P}^{\textup{local}}:(\mathcal{X}^{2}\times\mathcal{Y})^{n}\to\mathcal{F}_{0}$ defined for each $P\in\mathcal{P}$, the \emph{local estimator} $\widehat{f}_{n}^{\textup{local}}$ of the class $\mathcal{P}$ is defined as the map
    \begin{align*}
    \mathcal{P}\ni P\mapsto \widehat{f}_{n}^{\textup{local}}(P):=\widehat{f}_{n,P}^{\textup{local}}.
    \end{align*}
    The local estimator $\widehat{g}_{n}^{\textup{local}}$, namely the map
    \begin{align*}
    \mathcal{P}\ni P\mapsto \widehat{g}_{n}^{\textup{local}}(P):=\widehat{g}_{n,P}^{\textup{local}}=(\widehat{g}_{n,P,1}^{\textup{local}},\cdots,\widehat{g}_{n,P,d_{1}}^{\textup{local}}):(\mathcal{X}^{2}\times\mathcal{Y})^{n}\to\mathcal{G}_{0},
    \end{align*}
    is defined via $\widehat{f}_{n,P}^{\textup{local}}=\sum_{i=1}^{d_{1}}\widehat{g}_{n,P,i}^{\textup{local}}v_{i}$ for each $P\in\mathcal{P}$.
    \item A \emph{global estimator} $\widehat{f}_{n}:(\mathcal{X}^{2}\times\mathcal{Y})^{n}\to \mathcal{F}_{0}$ is simply defined as an estimator.
    The global estimator $\widehat{g}_{n}:(\mathcal{X}^{2}\times\mathcal{Y})^{n}\to\mathcal{G}_{0}$ is the estimator satisfying $\widehat{f}_{n}=\sum_{i=1}^{d_{1}}\widehat{g}_{n,i}v_{i}$ with $\widehat{g}_{n}=(\widehat{g}_{n,1},\cdots,\widehat{g}_{n,d_{1}})$.
\end{itemize}
\end{definition}
\begin{remark}
\label{rem:remark of the global estimators}
Any global estimator is identified with a local estimator defined as a constant map.
Note that the definition of global estimators is the same as that of estimators.
Hence, the definition of global estimators in Definition~\ref{def:definitions of local and global estimators} might be a bit redundant.
Nevertheless, we employ this terminology because we also consider some estimator $\widehat{g}_{n,P}^{\textup{local}}$ defined for each fixed $P\in\mathcal{P}_{\xi}$.
While each $\widehat{g}_{n,P}^{\textup{local}}$ defined with $P$ is a global estimator by Definition~\ref{def:definitions of local and global estimators}, the purpose is to define a local estimator.
Hereafter, we refer to both $\widehat{f}_{n,P}^{\textup{local}}$ and $\widehat{g}_{n,P}^{\textup{local}}$ as \emph{estimators defined for each} $P$, for convenience.
\end{remark}
Given $\mathcal{F}\subset\mathcal{F}_{0}$, we often consider the local estimator $\widehat{f}_{n}^{\textup{local}}$ of estimator $\widehat{f}_{n,P}^{\textup{local}}:(\mathcal{X}^{2}\times\mathcal{Y})^{n}\to \mathscr{F}_{\beta,\beta_{0},P}(\mathcal{F})\subset \mathcal{F}_{0}$ defined for each $P\in\mathcal{P}_{\xi}$.
Note that throughout this paper, the symbols of local and global estimators are distinguished by the existence of the superscript: for instance, $\widehat{f}_{n}^{\textup{local}}$ and $\widehat{f}_{n}$ denote local and global estimators, respectively.

\subsection{Main Theorem}
\label{subsec:examples}

We are now in a position to state the main theorem of this work, which is the formal statement of Theorem~\ref{thm:informal version of local theorem}.

\begin{theorem}
\label{thm:estimation error for deep relu networks}
For any $\alpha> 0$, $\tau\geq 1$, $\xi=(R,K,d_{1},E,\theta_{\textup{NC}},\theta_{1},\theta_{2},\theta_{3})\in \Xi$, $\beta\in(0,D_{{\textup{proj}}})$, and $n\in\mathbb{N}\setminus\{1,2\}$ satisfying $\varepsilon_{n}=n^{-\tau\alpha/((2\tau-1)\alpha+\tau(K-1))}<2^{-1}$, there are
\begin{enumerate}
\item[(i)] constants $C^{*}>0$ and $N\in\mathbb{N}$ that are independent of $n$,
\item[(ii)] some $c>0$, $L^{*}\in\mathbb{N}$, $J^{*},S^{*},M^{*}\geq 0$, and $\bm{d}^{*}\in\mathbb{N}^{L^{*}+1}$ satisfying the conditions $L^{*}\lesssim \log_{2}{\varepsilon_{n}^{-1}}$, $J^{*}\lesssim \varepsilon_{n}^{-c}$, $S^{*}\lesssim\varepsilon_{n}^{-(K-1)/\alpha}\log_{2}{\varepsilon_{n}^{-1}}$, $M^{*}\lesssim |\log(4d_{1}^{-2}\theta_{1}^{-1}\varepsilon_{n})|\vee 1$, and $\bm{d}^{*}=(K,d_{\textup{NN},1}^{*},\cdots,d_{\textup{NN},L^{*}-1}^{*},d_{1})$, and
\item[(iii)] a local estimator $\widehat{f}_{n}^{\textup{local}}$ of the class $\mathcal{P}_{\alpha,\tau,\xi}$ for which $\widehat{f}_{n}^{\textup{local}}(P):=\widehat{f}_{n,P}^{\textup{local}}$ is defined with $\widehat{f}_{n,P}^{\textup{local}}:(\mathcal{X}^{2}\times\mathcal{Y})^{n}\to \mathscr{F}_{\beta,\beta^{-1}\varepsilon_{n},P}(\mathcal{F}_{L^{*},J^{*},S^{*},M^{*},\bm{d}^{*}}^{\Delta^{d}\textup{-NN}})\subset \mathcal{F}_{0}$ for each $P\in\mathcal{P}_{\alpha,\tau,\xi}$,
\end{enumerate}
such that for the estimator $\widehat{g}_{n,P}^{\textup{local}}=(\widehat{g}_{n,P,1}^{\textup{local}},\cdots,\widehat{g}_{n,P,d_{1}}^{\textup{local}}):(\mathcal{X}^{2}\times\mathcal{Y})^{n}\to\mathcal{G}_{0}$ defined by the identity $\widehat{f}_{n,P}^{\textup{local}}=\sum_{i=1}^{d_{1}}\widehat{g}_{n,P,i}^{\textup{local}}v_{i}$ for each $P\in\mathcal{P}_{\alpha,\tau,\xi}$, if $n\geq N$, then we have
\begin{align*}
\sup_{P\in\mathcal{P}_{\alpha,\tau,\xi}}\mathcal{R}(\widehat{g}_{n,P}^{\textup{local}};P)\leq C^{*}n^{-\frac{\alpha}{(2\tau-1)\alpha+\tau(K-1)}}\log^{3\tau^{-1}+1}{n}.
\end{align*}
\end{theorem}
\begin{remark}
\label{remark:condition on theorem for local erm}
We provide several comments on Theorem~\ref{thm:estimation error for deep relu networks}:
\begin{enumerate}
\item[(i)] This theorem holds for any given $\beta\in (0,D_{\textup{proj}})$.
Let $C^{*}=C^{*}(\beta)$ be the constant defined in the proof of Theorem~\ref{thm:estimation error for deep relu networks} for each $\beta\in (0,D_{\textup{proj}})$ (see Appendix~B.6.4\showsupplement{in}).
While $C^{*}(\beta)<\infty$ if $\beta\in (0,D_{\textup{proj}})$, we have $\lim_{\beta\to 0}C^{*}(\beta)=\infty$ and $\lim_{\beta\to D_{\textup{proj}}}C^{*}(\beta)=\infty$.
\item[(ii)] Given $\beta\in (0,D_{\textup{proj}})$, if the natural number $n$ in Theorem~\ref{thm:estimation error for deep relu networks} satisfies the additional condition that $\beta^{-1}\varepsilon_{n}\geq 1$, then the local estimator in this theorem is a global estimator since it holds that $\mathscr{F}_{\beta,\beta^{-1}\varepsilon_{n},P}(\mathcal{F}_{L^{*},J^{*},S^{*},M^{*},\bm{d}^{*}}^{\Delta^{d}\textup{-NN}})=\mathcal{F}_{L^{*},J^{*},S^{*},M^{*},\bm{d}^{*}}^{\Delta^{d}\textup{-NN}}$ by Definition~\ref{def:concentrated subclass}.
\item[(iii)] Suppose that all the conditions using either $\theta_{2}$ or $\theta_{3}$ in Definition~\ref{def:main assumption} are removed, and the conditions $p_{Y}(-1)\in (0,1)$ and $\max_{i=1,\cdots,d_{1}}P_{X}(\mathcal{K}_{i})<1$ are added instead.
Under this setting, for every $P$ in this modified class, there is some constant $C_{P}^{*}$ and an estimator $\widehat{f}_{n,P}^{\textup{local}}$ such that for the estimator $\widehat{g}_{n,P}^{\textup{local}}$ satisfying that $\widehat{f}_{n,P}^{\textup{local}}=\sum_{i=1}^{d_{1}}\widehat{g}_{n,P,i}^{\textup{local}}v_{i}$, we have
\begin{align*}
\mathcal{R}(\widehat{g}_{n,P}^{\textup{local}};P)\leq C_{P}^{*}n^{-\frac{\alpha}{(2\tau-1)\alpha+\tau(K-1)}}\log^{3\tau^{-1}+1}{n},
\end{align*}
where $C_{P}^{*}$ should satisfy that $\sup_{P}C_{P}^{*}=\infty$.
The proof is almost the same as that of Theorem~\ref{thm:estimation error for deep relu networks} and is thus omitted.
\item[(iv)] The best convergence rate in Theorem~\ref{thm:estimation error for deep relu networks} is $n^{-\frac{\alpha}{\alpha+K-1}}\log^{4}{n}$ when $\tau=1$, namely when $P$ satisfies the condition of~\citet{massart2006risk}.
This observation shares some similarity with the analyses of~\citep{lecue2007optimal,alquier2019estimation} under conventional supervised learning settings, since we use~\citep[Proposition~1]{lecue2007optimal} in the proof (see Section~\ref{subsubsec:outline}).
\end{enumerate}
\end{remark}
This theorem indicates the learnability of smooth boundaries via some local pairwise binary classification algorithm.

\section{Proof Outline of Theorem~\ref{thm:estimation error for deep relu networks}}
\label{subsec:general upper bound of main result}

We give an outline of the proof of Theorem~\ref{thm:estimation error for deep relu networks}.
The proof consists of several steps.
In Section~\ref{subsec:loss function continued}, we introduce a learning algorithm.
In Section~\ref{subsubsec:outline}, we show the proof strategy of the main theorem.
In Section~\ref{subsubsec:an oracle inequality}, we present an estimation bound and the remained part of the proof.
All the proofs omitted here are deferred to Appendix~B\showsupplement{in}.

\subsection{Learning Algorithm}
\label{subsec:loss function continued}

In the estimation procedure, we consider to execute an algorithm based on contrastive learning, where contrastive learning is known as an efficient, tractable methodology to learn pairwise relation (see, e.g.,~\citep{gutmann2010noise,oord2018representation,arora2019theoretical,chen2020simple}).
Contrastive learning is originally developed by~\citet{gutmann2010noise} as a parametric estimation method, while recently it has been investigated as a methodology to learn the statistical relationship between covariates using some vector-valued functions~\citep{tsai2020neural,tosh2021contrastivejmlr,tosh2021contrastive,bao2022pairwise,chuang2022robust,zhai2023sigmoid}.

Before introducing the algorithm, some justification of using vector-valued functions is required.
The following fact justifies the usefulness of the vector-valued functions in $\mathcal{F}_{0}$:
\begin{proposition}
\label{prop:simple fact for risks}
We have the following properties:
\begin{enumerate}
\item[(i)] Let $d_{1}\in\mathbb{N}\setminus \{1\}$, and let $d=d_{1}-1$.
Given any $z=\sum_{i=1}^{d_{1}}c_{i}v_{i}\in \Delta^{d}$ and $z'=\sum_{i=1}^{d_{1}}c_{i}'v_{i}\in\Delta^{d}$, it holds that $\|z-z'\|_{2}^{2}=d_{1}d^{-1}\sum_{i=1}^{d_{1}}|c_{i}-c_{i}'|^{2}$.
\item[(ii)] Let $\xi=(R,K,d_{1},E,\theta_{\textup{NC}},\theta_{1},\theta_{2},\theta_{3})\in\Xi$, and let $d=d_{1}-1$. Let $f,f'\in\mathcal{F}_{0}$, and denote by $f=\sum_{i=1}^{d_{1}}g_{i}v_{i}$ and $f'=\sum_{i=1}^{d_{1}}g_{i}'v_{i}$.
Given any $P\in\mathcal{P}_{\xi}$, we have
\begin{align*}
\|f-f'\|_{\mathcal{X},P_{X},2}^{2}=\frac{d_{1}}{d}\sum_{i=1}^{d_{1}}\|g_{i}-g_{i}'\|_{L^{2}(\mathcal{X},P_{X})}^{2}.
\end{align*}
\end{enumerate}
\end{proposition}
By Proposition~\ref{prop:simple fact for risks}--(ii), for every $P\in\mathcal{P}_{\xi}$ we have
\begin{align*}
\mathcal{R}(\widehat{g}_{n,P}^{\textup{local}};P)=d_{1}^{-1}d\cdot \mathbb{E}[\|\widehat{f}_{n,P}^{\textup{local}}(U_{1},\cdots,U_{n})-f^{*}\|_{\mathcal{X},P_{X},2}^{2}],
\end{align*}
where $\widehat{f}_{n,P}^{\textup{local}}:(\mathcal{X}^{2}\times\mathcal{Y})^{n}\to\mathcal{F}_{0}$ is any estimator for which $\widehat{f}_{n,P}^{\textup{local}}=\sum_{i=1}^{d_{1}}\widehat{g}_{n,P,i}^{\textup{local}}v_{i}$ is satisfied for each $P\in\mathcal{P}_{\xi}$, and $f^{*}$ is the contrastive function of $P$.
Thus, it suffices to show an upper bound of the $L^{2}$-risk $\mathbb{E}[\|\widehat{f}_{n,P}^{\textup{local}}(U_{1},\cdots,U_{n})-f^{*}\|_{\mathcal{X},P_{X},2}^{2}]$.

We now introduce the loss function.
Following~\citep{chen2021large,wang2020understanding}, we consider a loss function defined with the squared Euclidean distance, namely, $\rho_{f}:\mathcal{X}\times\mathcal{X}\to\mathbb{R}$ defined as 
\begin{align*}
\rho_{f}(x,x')=\|f(x)-f(x')\|_{2}^{2},
\end{align*}
for each $f\in\mathcal{F}_{0}$.
Then, following the standard approach in similarity learning and metric learning~\citep{jin2009regularized,chen2009similarity,cao2015generalization,bao2022pairwise}, we define a hinge loss as follows:
\begin{definition}[Loss function]
\label{def:contrastive loss}
Let $\psi:\mathbb{R}\to\mathbb{R}$ be the function defined as
\begin{align*}
\psi(s)=1-2D_{\Delta^{d}}^{-2}s.
\end{align*}
Then, for every $f\in\mathcal{F}_{0}$, we define the hinge loss $\ell_{f}:\mathcal{X}^{2}\times\mathcal{Y}\to\mathbb{R}$ as 
\begin{align*}
\ell_{f}(x,x',y)=\max\{0,1-y\psi\circ\rho_{f}(x,x')\}.
\end{align*}
\end{definition}
\begin{remark}
We use the function $\psi\circ\rho_{f}$ as a classifier (see Lemma~B.2\showsupplement{in} for a basic property).
This loss function can be used in a contrastive learning algorithm, as it belongs to a general class of loss functions of contrastive learning proposed in~\citep{chen2021intriguing}.
In Section~\ref{appsec:comparison with related mathematical notions} we provide some discussion of related work~\citep{arora2019theoretical,li2021self,shah2022max,waida2023towards,ji2023power,jin2009regularized,cao2015generalization,zhou2024generalization,kim2021fast,imaizumi2019deep,imaizumi2022advantage,meyer2023optimal}.
\end{remark}

In our problem setting, the motivation of using the above loss function is due to the following property, which utilizes a fact shown in~\citep{lin2002support}.
\begin{proposition}
\label{thm:main result basic case}
Given any $\xi=(R,K,d_{1},E,\theta_{\textup{NC}},\theta_{1},\theta_{2},\theta_{3})\in \Xi$ and $P\in\mathcal{P}_{\xi}$, if $f^{*}$ is the contrastive function of $P$, then we have
\begin{align*}
\mathbb{E}_{P}[\ell_{f^{*}}]=\inf_{f\in\mathcal{F}_{0}}\mathbb{E}_{P}[\ell_{f}].
\end{align*}
\end{proposition}
The key point is that the codomain of any vector-valued function in $\mathcal{F}_{0}$ is restricted to $\Delta^{d}$.
This property might be natural since the relation between regular simplices and empirical risk minimization in representation learning has been shown in~\citep{liu2021learning,lee2024analysis,koromilas2024bridging}.
Note that this proposition slightly generalizes some part of Theorem~3.8 in~\citep{awasthi2022do} for hinge loss since we deal with a more general setting.
Note also that some relation between the probability simplex and the population risk minimizer of a similarity learning problem is proven in~\citep[Theorem~1]{zhou2024generalization}, while the setting in Proposition~\ref{thm:main result basic case} is based on a contrastive learning problem.

\subsection{Proof Strategy}
\label{subsubsec:outline}

It is proven in~\citep[Proposition~1]{lecue2007optimal} that under the Tsybakov noise condition~\citep{mammen1999smooth,tsybakov2004optimal}, the excess risk of hinge loss gives the following upper bound of the $L^{1}$-risk between classifiers:
\begin{lemma}[{Proposition~1 in~\citep{lecue2007optimal}}]
\label{citelem:lecue lemma}
Let $\mathcal{X}_{0}$ be a measurable space with a non-negative, $\sigma$-finite measure $\nu$.
Let $g:\mathcal{X}_{0}\to[-1,1]$ be a measurable function.
Let $P$ be a probability measure in $\mathcal{X}_{0}\times\mathcal{Y}$ that has a probability density function $p(x,y)$ on $\mathcal{X}_{0}\times\mathcal{Y}$ with respect to $\nu\otimes\chi$, where the marginal distribution of $P$ in $\mathcal{X}_{0}$ is denoted by $P_{X_{0}}$.
Also, denote by $g^{*}(x)=\textup{sign}(2p(y=1|x)-1)$, the Bayes classifier for $P$.
Suppose that either of the following conditions is satisfied:
\begin{itemize}
\item $P$ satisfies the Tsybakov noise condition~\eqref{eq:margin assumption} with $s_{0}=1$, $\tau> 1$, and $c>0$.
\item $P$ satisfies the Tsybakov noise condition~\eqref{eq:margin assumption} with $s_{0}\in (0,1]$, $\tau=1$, and $c>0$.
\end{itemize}
Then, for the coefficient $C_{0}=C_{0}(\tau,c):=((\tau-1)/(2c\tau))^{1-\tau}\tau s_{0}^{-1}$, it holds that
\begin{align*}
\mathbb{E}_{P_{X_{0}}}[|g-g^{*}|]^{\tau}\leq C_{0}(\mathbb{E}_{P}[\max\{0,1-yg(x)\}]-\mathbb{E}_{P}[\max\{0,1-yg^{*}(x)\}]).
\end{align*}
\end{lemma}

In general, the above property is known as \emph{Bernstein condition}~\citep{bartlett2005empirical} (see also~\citep{tsybakov2004optimal,steinwart2007fast,tarigan2008moment,alquier2019estimation}).
Here, note that any $P\in\mathcal{P}_{\tau,\xi}$ satisfies all the conditions in Lemma~\ref{citelem:lecue lemma}.
Applying Lemma~\ref{citelem:lecue lemma}, we can obtain an upper bound of the quantity $\mathbb{E}_{P_{X,X'}}[|\psi\circ\rho_{f}-\psi\circ\rho_{f^{*}}|]$ for any $f\in\mathcal{F}_{0}$, where $f^{*}$ is the contrastive function of $P$.
Hence, it might suffice to show a lower bound of this quantity.
In other words, we consider the following question:
\begin{question}
\label{question:inequality between l2 norm and l2 distance between classifiers}
For all $f\in\mathcal{F}_{0}$, is there a constant $c=c(f)>0$ such that for every $x,x'\in\mathcal{X}$ it holds that
\begin{align*}
\|f(x)-f^{*}(x)\|_{2}^{2}\leq c|\psi\circ\rho_{f}(x,x')-\psi\circ\rho_{f^{*}}(x,x')|\;?
\end{align*}
\end{question}

However, we can find a counterexample, which is due to the non-identifiability of the problem setting in Section~\ref{subsec:assumptions}:
\begin{example}
\label{example:not good representation}
Given any $\xi=(R,K,d_{1},E,\theta_{\textup{NC}},\theta_{1},\theta_{2},\theta_{3})\in \Xi$, $P\in\mathcal{P}_{\xi}$, and any permutation $\pi$ on $\{1,\cdots,d_{1}\}$ for which $\pi(i)\neq i$ for any $i\in\{1,\cdots,d_{1}\}$, we define the function $f_{\pi}^{*}:\mathcal{X}\to\Delta^{d}$ as $f_{\pi}^{*}=\sum_{i=1}^{d_{1}}g_{\pi(i)}^{*}v_{i}$, where $f^{*}=\sum_{i=1}^{d_{1}}g_{i}^{*}v_{i}$ is the contrastive function of $P$.
For instance, for the permutation $\pi_{1}$ satisfying that $\pi_{1}(i)=i-1$ for every $i\in\{2,\cdots,d_{1}\}$ and $\pi_{1}(1)=d_{1}$, it holds that
\begin{align*}
f_{\pi_{1}}^{*}(x)=
\begin{cases}
v_{i+1} &\textup{if }x\in\mathcal{K}_{i},\; i\in\{1,\cdots,d_{1}-1\},\\
v_{1} &\textup{if }x\in\mathcal{K}_{d_{1}}.
\end{cases}
\end{align*}
Then, we have $\|f_{\pi}^{*}(x)-f^{*}(x)\|_{2}^{2}>0$ for every $x\in\mathcal{X}$.
In contrast, by the definition of vertices $v_{1},\cdots,v_{d_{1}}$, we have that $\|v_{\pi(i)}-v_{\pi(j)}\|_{2}=\|v_{i}-v_{j}\|_{2}$ for any $i,j\in\{1,\cdots,d_{1}\}$.
This implies that $\|f_{\pi}^{*}(x)-f_{\pi}^{*}(x')\|_{2}=\|f^{*}(x)-f^{*}(x')\|_{2}$ for any $x,x'\in\mathcal{X}$, which is equivalent to the claim that $|\psi\circ\rho_{f_{\pi}^{*}}(x,x')-\psi\circ\rho_{f^{*}}(x,x')|=0$ for every $x,x'\in\mathcal{X}$.
\end{example}
This example implies that the order of subsets $\{\mathcal{K}_{i}\}_{i=1}^{d_{1}}$ causes the non-identifiability issue.
One can also observe this by Lemma~B.2--(iii)\showsupplement{in}.
See also Appendix~A.2\showsupplement{in} for a similar observation under a different, similarity learning setting of~\citet{bao2022pairwise}.
Thus, we need to ask the following question, instead of Question~\ref{question:inequality between l2 norm and l2 distance between classifiers}:
\begin{question}
\label{question:how to derive a lower bound}
Let $\mathcal{F}\subset\mathcal{F}_{0}$, $n\in\mathbb{N}$, $\xi\in\Xi$, $P\in\mathcal{P}_{\xi}$, and $f^{*}$ be the contrastive function of $P$.
\begin{enumerate}
    \item[(Q1)] Is there a sufficient condition on $\mathcal{F}$ such that for any given $f\in\mathcal{F}$, if it holds that $\mathbb{E}_{P_{X,X'}}[|\psi\circ\rho_{f}-\psi\circ\rho_{f^{*}}|]=0$, then $\|f-f^{*}\|_{\mathcal{X},P_{X},2}^{2}=0$?
    \item[(Q2)] Is there a function $\mathcal{U}_{n}:[0,\infty)\to [0,\infty)$ such that $\limsup_{n\to\infty}\mathcal{U}_{n}(0)=0$, and for any estimator $\widehat{f}_{n}:(\mathcal{X}^{2}\times\mathcal{Y})^{n}\to\mathcal{F}\subset \mathcal{F}_{0}$ it holds that
    \begin{align*}
    \mathbb{E}[\|\widehat{f}_{n}-f^{*}\|_{\mathcal{X},P_{X},2}^{2}]\leq \mathcal{U}_{n}(\mathbb{E}[\mathbb{E}_{P_{X,X'}}[|\psi\circ\rho_{\widehat{f}_{n}}-\psi\circ\rho_{f^{*}}|]]) \; ?
    \end{align*}
\end{enumerate}
\end{question}
Now, we first consider (Q1).
We notice that Example~\ref{example:not good representation} implies that some condition on vector-valued functions in $\mathcal{F}_{0}$ is required to exclude the functions introduced in the example.
Hence, we may assume that the given function $f\in\mathcal{F}$ should satisfy that $\|f(x)-f^{*}(x)\|_{2}<\beta$ for any $x\in\mathcal{X}$, where $\beta<D_{\textup{proj}}$, and $f^{*}$ is the contrastive function of the given $P\in\mathcal{P}_{\xi}$.
However, this condition is not suitable when $f^{*}$ is estimated using neural networks since we additionally need to consider the approximation errors.
To address this issue, we use the notion of localized subclasses defined in Definition~\ref{def:concentrated subclass}.

To address (Q2), we evaluate the gap $|\psi\circ\rho_{f}-\psi\circ\rho_{f^{*}}|$ in the next subsection.

\subsection{General Estimation Bounds and Further Analyses}
\label{subsubsec:an oracle inequality}

\paragraph{General estimation bounds.}
Given any $P\in\mathcal{P}_{\xi}$, define $P_{X,X'}^{-}$ as the probability measure whose Lebesgue density is $p_{X}(x)p_{X'}(x')$.
Note that by condition (A3), we have that $p_{X}(x)p_{X'}(x')=p_{X}(x)p_{X}(x')$ for any $x,x'\in\mathcal{X}$.
\begin{lemma}
\label{lem:step 3 lemma 2}
Let $\xi=(R,K,d_{1},E,\theta_{\textup{NC}},\theta_{1},\theta_{2},\theta_{3})\in \Xi$, $P\in\mathcal{P}_{\xi}$, and $\mathscr{S}_{P}=\{\mathcal{K}_{i}\}_{i=1}^{d_{1}}$.
Denote by $f^{*}$, the contrastive function of $P$.
Let $\beta\in(0,D_{{\textup{proj}}})$, $\beta_{0}\geq 0$, and $\mathcal{F}\subset\mathcal{F}_{0}$.
For any distinct $i,j\in\{1,\cdots,d_{1}\}$ and every $f\in\mathscr{F}_{\beta,\beta_{0},P}(\mathcal{F})$, we have 
\begin{align*}
P_{X,X'}^{-}((\{x\in\mathcal{X}\;|\;\|f(x)-f^{*}(x)\|_{2}\geq \beta\}\cap \mathcal{K}_{i})\times \mathcal{K}_{j})\leq \beta_{0}.
\end{align*}
\end{lemma}
\begin{proof}
By the definition of localized subclasses (Definition~\ref{def:concentrated subclass}), we have the claim.
\end{proof}

The main idea of the proof is to construct a sequence of subsets in $\Delta^{d}$ to establish inequalities similar to that in Question~\ref{question:inequality between l2 norm and l2 distance between classifiers}.
The following lemma enables us to employ this idea:
\begin{lemma}
\label{lem:step 3 lemma 0}
Let $\xi=(R,K,d_{1},E,\theta_{\textup{NC}},\theta_{1},\theta_{2},\theta_{3})\in \Xi$, $\beta\in(0,D_{{\textup{proj}}})$, $\beta_{0}\geq 0$, $n\in\mathbb{N}\setminus\{1\}$, $\mathcal{F}\subset\mathcal{F}_{0}$, and $P\in\mathcal{P}_{\xi}$.
Let $f^{*}$ be the contrastive function of $P$, and let $\mathscr{S}_{P}=\{\mathcal{K}_{i}\}_{i=1}^{d_{1}}$.
Then, there is a constant $C>0$ independent of $n$ and $P$ such that for any $f\in\mathscr{F}_{\beta,\beta_{0},P}(\mathcal{F})$, we have
\begin{align}
\label{eq:step 3 eq 0}
&\mathbb{E}_{P_{X}}[\|f-f^{*}\|_{2}^{2}]\nonumber\\
&\leq C\left((D_{\Delta^{d}}-\beta)\beta_{0}+\sum_{i\neq j}\sum_{w=0}^{\lfloor\log_{2}{n}\rfloor}\left(\frac{1}{2}\right)^{2w+1}\beta^{2}P_{i,j}^{-}(2^{-(w+1)}\beta)+\frac{\beta^{2}}{n}\right),
\end{align}
where $P_{i,j}^{-}(r)=P_{X,X'}^{-}((\{x\in\mathcal{X}\;|\;\|f(x)-f^{*}(x)\|_{2}> r\}\cap \mathcal{K}_{i})\times \mathcal{K}_{j})$.
\end{lemma}
In the right-hand side of~\eqref{eq:step 3 eq 0}, we truncate the infinite series to prevent it from diverging.
This argument makes it possible to proceed the subsequent analysis at the cost of additional factor $\log{n}$ in the final convergence rate.

Now, it remains to show an upper bound of $P_{i,j}^{-}(2^{-(w+1)}\beta)$, $i\neq j$, in~\eqref{eq:step 3 eq 0}.
For any $r\in(0,2^{-1}\beta]$ we decompose the subset $\{x\in\mathcal{X}\;|\;\|f(x)-f^{*}(x)\|_{2}>r\}$ as
\begin{align*}
\{x\in\mathcal{X}\;|\; r<\|f(x)-f^{*}(x)\|_{2}<\beta\}\cup\{x\in\mathcal{X}\;|\;\beta\leq \|f(x)-f^{*}(x)\|_{2}\leq D_{\Delta^{d}}\}.
\end{align*}
The probability of the latter subset is evaluated by Lemma~\ref{lem:step 3 lemma 2}.
Thus, we investigate the former subset.
\begin{lemma}
\label{lem:step 3 lemma 1}
Let $\xi=(R,K,d_{1},E,\theta_{\textup{NC}},\theta_{1},\theta_{2},\theta_{3})\in \Xi$, $\beta\in (0,D_{{\textup{proj}}})$, $\beta_{0}\geq 0$, $\mathcal{F}\subset\mathcal{F}_{0}$, and $P\in\mathcal{P}_{\xi}$.
Let $f^{*}$ be the contrastive function of $P$.
Let $\mathscr{S}_{P}=\{\mathcal{K}_{i}\}_{i=1}^{d_{1}}$.
For every $i,j\in\{1,\cdots,d_{1}\}$ such that $i\neq j$, any $r\in(0,2^{-1}\beta]$, and any $f\in\mathscr{F}_{\beta,\beta_{0},P}(\mathcal{F})$, there is a constant $C_{i,j}>0$ independent of $f$, $r$ and $P$ such that
\begin{align*}
&P_{X,X'}^{-}((\{x\in\mathcal{X}\;|\; r<\|f(x)-f^{*}(x)\|_{2}\leq \beta\}\cap \mathcal{K}_{i})\times \mathcal{K}_{j})\nonumber\\
&\leq C_{i,j}(r\wedge (D_{\Delta^{d}}(1-\beta/D_{{\textup{proj}}})))^{-2}\mathbb{E}_{P_{X,X'}^{-}}[|\psi\circ\rho_{f}-\psi\circ\rho_{f^{*}}|]+\beta_{0}.
\end{align*}
\end{lemma}
In the proof, the condition that $\|f(x)-f^{*}(x)\|_{2}<\beta$ with probability at least $1-\beta_{0}$ (see Definition~\ref{def:concentrated subclass}) is utilized to show lower bounds of the quantity $|\psi\circ\rho_{f}(x,x')-\psi\circ\rho_{f^{*}}(x,x')|$.
The local condition in Definition~\ref{def:concentrated subclass} is introduced to deal with this technical difficulty.

Incorporating the above lemmas in the approach shown in Section~\ref{subsubsec:outline}, we obtain the estimation bound of the $L^{2}$-risk in Definition~\ref{def:l2 risk} for a general estimator.
\begin{theorem}
\label{thm:main result}
Let $\mathcal{F}\subset\mathcal{F}_{0}$, $\tau\geq 1$, $\xi=(R,K,d_{1},E,\theta_{\textup{NC}},\theta_{1},\theta_{2},\theta_{3})\in \Xi$, $\beta\in (0,D_{{\textup{proj}}})$, $\beta_{0}\geq 0$, and $n\in\mathbb{N}\setminus\{1\}$.
Let $P\in\mathcal{P}_{\tau,\xi}$, and let $U_{1},\cdots,U_{n}$ be i.i.d. random variables following $P$.
Then, there are positive constants $C$, $C'$, and $C''$ that are independent of $n$ and $P$, such that
for any estimator $\widehat{f}_{n,P}^{\textup{local}}:(\mathcal{X}^{2}\times\mathcal{Y})^{n}\to\mathscr{F}_{\beta,\beta_{0},P}(\mathcal{F})\subset\mathcal{F}_{0}$ and $\widehat{g}_{n,P}^{\textup{local}}=(\widehat{g}_{n,P,1}^{\textup{local}},\cdots,\widehat{g}_{n,P,d_{1}}^{\textup{local}}):(\mathcal{X}^{2}\times\mathcal{Y})^{n}\to\mathcal{G}_{0}$ for which $\widehat{f}_{n,P}^{\textup{local}}=\sum_{i=1}^{d_{1}}\widehat{g}_{n,P,i}^{\textup{local}}v_{i}$ is satisfied, we have
\begin{align}
\label{eq:main inequality}
\mathcal{R}(\widehat{g}_{n,P}^{\textup{local}};P)\leq C(\log{n})\mathbb{E}[\mathcal{E}(\widehat{f}_{n,P}^{\textup{local}}(U_{1}^{n});P)]^{\frac{1}{\tau}}+C'\beta_{0}+\frac{C''}{n},
\end{align}
where $\mathcal{E}(f;P)=\mathbb{E}_{P}[\ell_{f}]-\mathbb{E}_{P}[\ell_{f^{*}}]$ denotes the excess risk of the given $f\in\mathscr{F}_{\beta,\beta_{0},P}(\mathcal{F})$ with respect to the contrastive function $f^{*}=\sum_{i=1}^{d_{1}}g_{i}^{*}v_{i}$ of $P$.
\end{theorem}
Namely, the combination of Lemmas~\ref{lem:step 3 lemma 2} --~\ref{lem:step 3 lemma 1} provides a solution to (Q2) in Question~\ref{question:how to derive a lower bound}.

\paragraph{Further analyses.}
To apply Theorem~\ref{thm:main result}, we introduce a local ERM estimator.
\begin{definition}[Local ERM]
\label{def:formal definition of erm}
Given $\xi=(R,K,d_{1},E,\theta_{\textup{NC}},\theta_{1},\theta_{2},\theta_{3})\in \Xi$, let $\beta\in (0,D_{\textup{proj}})$, $n\in\mathbb{N}\setminus\{1,2\}$, $\varepsilon>0$, 
$\mathcal{P}\subset\mathcal{P}_{\xi}$, and $\mathcal{F}\subset\mathcal{F}_{0}$.
For each $P\in\mathcal{P}$, consider the $(\beta,\beta^{-1}\varepsilon,P)$-localized subclass $\mathscr{F}_{\beta,\beta^{-1}\varepsilon,P}(\mathcal{F})$ of $\mathcal{F}$.
In addition, let $\ell_{f}$ be the hinge loss in Definition~\ref{def:contrastive loss}.
Here, define $\widehat{f}_{n,P}^{\textup{LERM}}:(\mathcal{X}^{2}\times\mathcal{Y})^{n}\to\mathscr{F}_{\beta,\beta^{-1}\varepsilon,P}(\mathcal{F})$ as a map satisfying
\begin{align*}
\widehat{f}_{n,P}^{\textup{LERM}}(u_{1},\cdots,u_{n})\in\argmin_{f\in\mathscr{F}_{\beta,\beta^{-1}\varepsilon,P}(\mathcal{F})}\frac{1}{n}\sum_{i=1}^{n}\ell_{f}(u_{i}),
\end{align*}
where $u_{i}=(x_{i},x_{i}',y_{i})\in\mathcal{X}^{2}\times\mathcal{Y}$ for each $i=1,\cdots,n$.
Then, the $(\beta,\varepsilon,n,\mathcal{P},\mathcal{F})$\emph{-local ERM estimator} $\widehat{g}_{n}^{\textup{LERM}}$ is defined as the local estimator of the class $\mathcal{P}$ for which the estimator $\widehat{g}_{n}^{\textup{LERM}}(P):=\widehat{g}_{n,P}^{\textup{LERM}}:(\mathcal{X}^{2}\times\mathcal{Y})^{n}\to\mathcal{G}_{0}$ of each $P\in\mathcal{P}$ satisfies that
\begin{align*}
\widehat{f}_{n,P}^{\textup{LERM}}=\sum_{i=1}^{d_{1}}\widehat{g}_{n,P,i}^{\textup{LERM}}v_{i},
\end{align*}
where $\widehat{g}_{n,P}^{\textup{LERM}}=(\widehat{g}_{n,P,1}^{\textup{LERM}},\cdots,\widehat{g}_{n,P,d_{1}}^{\textup{LERM}})$.
\end{definition}
Note that without loss of generality we may assume the existence of the local ERM estimator at every sample $(u_{1},\cdots,u_{n})\in (\mathcal{X}^{2}\times\mathcal{Y})^{n}$, for simplicity.
When this assumption is violated for some sample, it suffices to modify the definition of $\mathcal{F}_{L,J,S,M,\bm{d}}^{\Delta^{d}\textup{-NN}}$ so that the modified class becomes a finite set of ReLU networks, similarly to~\citep[Definition~2.9]{petersen2018optimal} (see Remark~B.29\showsupplement{in}).

The remained steps of the proof, which are deferred to Appendix~B.6\showsupplement{in}, can be summarized as follows:
\begin{itemize}
\item To prove Theorem~\ref{thm:estimation error for deep relu networks}, we additionally need to analyze the excess risk.
We show that the analyses developed in~\citep{park2009convergence,kim2021fast} are applicable to a pairwise binary classification setting.
\item To do so, we need to evaluate the approximation errors, where we follow similar arguments to the approximation theorems on deep neural networks with the softmax function developed by~\citet{bos2022convergence}.
We also use several approximation theorems of indicator functions using deep ReLU networks developed in the previous work~\citep{petersen2018optimal,imaizumi2019deep,imaizumi2022advantage}.
See Appendix B.6.1 and B.6.2\showsupplement{in} for the details.
\item A remained technical issue is about to what extent one can reduce the value of $\beta_{0}$ in~\eqref{eq:main inequality}.
In fact, the approximation theorems in~\citep{petersen2018optimal,imaizumi2019deep,imaizumi2022advantage,bos2022convergence} are developed for the class $\mathcal{F}_{L,J,S,M,\bm{d}}^{\textup{NN}}$, while in our analysis we consider the approximation property of the localized subclass $\mathscr{F}_{\beta,\beta^{-1}\varepsilon_{n},P}(\mathcal{F}_{L,J,S,M,\bm{d}}^{\Delta^{d}\textup{-NN}})$.
In Proposition~B.24 in Appendix B.6.3\showsupplement{of}, we prove that this issue can be resolved by showing that some function approximating the true function within a small error belongs to a localized subclass.
\item Combining Lemma~21 in~\citep{nakada2020adaptive}, Proposition~1 in~\citep{lecue2007optimal}, and the other arguments mentioned above, we can apply Theorem~A.1 in~\citep{kim2021fast} to the excess risk in~\eqref{eq:main inequality} of Theorem~\ref{thm:main result}.
See Appendix B.6.4\showsupplement{in} for the details.
\end{itemize}
Combining all the steps, we can prove the claim of Theorem~\ref{thm:estimation error for deep relu networks}.

\section{Discussion}
\label{sec:discussion of the main results}

We provide the detailed discussion of Theorem~\ref{thm:estimation error for deep relu networks} and its proof method in Theorem~\ref{thm:main result}.

\subsection{Discussion of Theorem~\ref{thm:estimation error for deep relu networks}}
\label{subsec:discussion of the formal version of the main theorem}

\paragraph{Minimax lower bound for global estimators.}
For the definition of global estimators, see Definition~\ref{def:definitions of local and global estimators}.
We prove a minimax lower bound of the pairwise binary classification problem in Section~\ref{subsec:assumptions}, when $\tau=1$.
\begin{theorem}
\label{thm:minimax lower bound}
Given any $n\in\mathbb{N}$, $\alpha>0$, and $\xi=(R,K,d_{1},E,\theta_{\textup{NC}},\theta_{1},\theta_{2},\theta_{3})\in\Xi$ for which the conditions $\theta_{1}(1-\theta_{3})^{\frac{1}{2}}\geq 1$, $\theta_{3}>\frac{1}{2}$, and $\theta_{\textup{NC}}<\frac{1-\theta_{3}}{2(1+\theta_{3})}$ are satisfied, there is a constant $C>0$ independent of $n$ such that
\begin{align*}
\inf_{\widehat{g}_{n}}\sup_{P\in\mathcal{P}_{\alpha,1,\xi}}\mathcal{R}(\widehat{g}_{n};P)
\geq Cn^{-\frac{\alpha}{\alpha+K-1}},
\end{align*}
where the infimum is taken over the set of all global estimators.
\end{theorem}

Note that each $P$ in the class $\mathcal{P}_{\alpha,1,\xi}$ is defined in $\mathcal{X}^{2}\times\mathcal{Y}$, and the dimension of $\mathcal{X}^{2}$ is $2K$.
However, since each $P\in\mathcal{P}_{\alpha,1,\xi}$ has parameter $\mathscr{S}_{P}=\{\mathcal{K}_{i}\}_{i=1}^{d_{1}}$ for which each $\mathcal{K}_{i}$ is a subset of $\mathcal{X}=[0,1]^{K}$, it might be natural that this lower bound is observed.
In the proof of Claim~B.34 in Appendix~B.7\showsupplement{of}, condition~\eqref{eq:statistical partition of unity} in (A4) of Definition~\ref{def:main assumption} enables us to observe this result.

The conditions $\theta_{1}(1-\theta_{3})^{\frac{1}{2}}\geq 1$ and $\theta_{\textup{NC}}<\frac{1-\theta_{3}}{2(1+\theta_{3})}$ are satisfied if both $\theta_{1}$ and $\theta_{\textup{NC}}^{-1}$ are sufficiently large for the given $\theta_{3}$.
This sufficient condition is reasonable, since conditions (A2) and (A3) in Definition~\ref{def:main assumption} become weaker as both $\theta_{1}$ and $\theta_{\textup{NC}}^{-1}$ increase.

We note that the proof of Theorem~\ref{thm:minimax lower bound} is based on Assouad's lemma~\citep{assouad1983deux}, and specifically we use the version shown in~\citep{yu1997assouad}.
While this approach is standard in the field of set estimation (see, e.g.,~\citep{mammen1995asymptotical,mammen1999smooth,tsybakov2004optimal,meyer2023optimal}), both the construction of probability distributions and the derivation of bounds are complicated due to the pairwise binary classification setting.
Therefore, the proof of Theorem~\ref{thm:minimax lower bound} may be of an independent interest.
One of the key ideas is to construct a finite set of Borel probability measures, using useful notions developed in~\citep{arora2019theoretical,awasthi2022do}.
The proof of Theorem~\ref{thm:minimax lower bound} is given in Appendix~B.7\showsupplement{of}.

\begin{remark}
\label{rem:open question for general noise condition in minimax lower bounds}
We comment on several limitations of Theorem~\ref{thm:minimax lower bound}.
\begin{enumerate}
    \item[(i)] Currently, we do not know how to prove a minimax lower bound when $\tau>1$.
    The proof method of Theorem~\ref{thm:minimax lower bound} is specific to the case where $\tau=1$, and it might be required to extend the method in a non-trivial way.
    \item[(ii)] It is well known that $n^{-\frac{\alpha}{\alpha+K-1}}$ is the minimax optimal rate in some conventional binary classification problems~\citep{tsybakov2004optimal,kim2021fast,meyer2023optimal} (see also~\citep{mammen1995asymptotical,mammen1999smooth,imaizumi2019deep,imaizumi2022advantage} for some results in other learning problems).
    As discussed in Remark~\ref{remark:condition on theorem for local erm}--(ii) and (iv), Theorem~\ref{thm:estimation error for deep relu networks} implies that the $L^{2}$-risk of a global estimator is upper bounded by $C^{*}n^{-\frac{\alpha}{\alpha+K-1}}\log^{4}{n}$ for some constant $C^{*}>0$ if $\tau=1$ and $\beta^{-1}\varepsilon_{n}\geq 1$.
    While the minimax rate is not determined by this result (since $n$ must satisfy $\varepsilon_{n}\geq \beta$), Theorem~\ref{thm:estimation error for deep relu networks} might provide some clues to prove this open question.
    For instance, as discussed in Appendix~C\showsupplement{of}, one may consider to analyze the approximation property of the ERM algorithm, although doing so might be challenging.
\end{enumerate}
\end{remark}

\paragraph{Comparison theorem.}
Let $\alpha>0$, $\tau\geq 1$, $\xi\in\Xi$, and $n\in\mathbb{N}$.
By the definitions of local and global estimators (see Definition~\ref{def:definitions of local and global estimators}), it holds that
\begin{align}
\label{eq:comparison of local and global minimax risks}
\inf_{\widehat{g}_{n}^{\textup{local}}}\sup_{P\in\mathcal{P}_{\alpha,\tau,\xi}}\mathcal{R}(\widehat{g}_{n,P}^{\textup{local}};P)\leq \inf_{\widehat{g}_{n}}\sup_{P\in\mathcal{P}_{\alpha,\tau,\xi}}\mathcal{R}(\widehat{g}_{n};P),
\end{align}
where the infimum of the left-hand side is taken over the set of all local estimators of the class $\mathcal{P}_{\alpha,\tau,\xi}$, while in the right-hand side the infimum is taken over the set of all global estimators.
By~\eqref{eq:comparison of local and global minimax risks}, there is a local estimator $\widehat{g}_{n}^{\textup{local}}$ of the class $\mathcal{P}_{\alpha,\tau,\xi}$ such that for any global estimator $\widehat{g}_{n}$, we have
\begin{align}
\label{eq:comparison inequality of local estimators}
\sup_{P\in\mathcal{P}_{\alpha,\tau,\xi}}\mathcal{R}(\widehat{g}_{n,P}^{\textup{local}};P)\lesssim \sup_{P\in\mathcal{P}_{\alpha,\tau,\xi}}\mathcal{R}(\widehat{g}_{n};P).
\end{align}
In other words, there is a local estimator such that its convergence rate is not slower than that of any global estimator.
Clearly, this is a statistical property, which does not hold for any local estimator.
Therefore, one can test the effectiveness of the given proof method by checking whether the obtained local estimator achieves the inequality~\eqref{eq:comparison inequality of local estimators}.

The following theorem implies the existence of a local estimator defined with a localized subclass for which it attains~\eqref{eq:comparison inequality of local estimators} up to some logarithmic factor:
\begin{theorem}
\label{thm:comparison theorem of local estimators}
For any $\alpha>0$, $\xi=(R,K,d_{1},E,\theta_{\textup{NC}},\theta_{1},\theta_{2},\theta_{3})\in\Xi$, $\beta\in (0,D_{\textup{proj}})$, and $n\in\mathbb{N}\setminus \{1,2\}$ for which $\varepsilon_{n}=n^{-\alpha/(\alpha+K-1)}<2^{-1}$, $\theta_{1}(1-\theta_{3})^{\frac{1}{2}}\geq 1$, $\theta_{3}>\frac{1}{2}$, and $\theta_{\textup{NC}}<\frac{1-\theta_{3}}{2(1+\theta_{3})}$ are satisfied, there are
\begin{enumerate}
\item[(i)] $L^{*}\in\mathbb{N}$, $J^{*},S^{*},M^{*}\geq 0$, and $\bm{d^{*}}=(K,d_{\textup{NN},1}^{*},\cdots,d_{\textup{NN},L^{*}-1}^{*},d_{1})\in\mathbb{N}^{L^{*}+1}$ depending on $n$, and
\item[(ii)] a local estimator $\widehat{f}_{n}^{\textup{local}}$ of the class $\mathcal{P}_{\alpha,1,\xi}$ for which $\widehat{f}_{n}^{\textup{local}}(P):=\widehat{f}_{n,P}^{\textup{local}}:(\mathcal{X}^{2}\times\mathcal{Y})^{n}\to \mathcal{F}_{L^{*},J^{*},S^{*},M^{*},\bm{d}^{*}}^{\Delta^{d}\textup{-NN}}\subset \mathcal{F}_{0}$ is satisfied for every $P\in\mathcal{P}_{\alpha,1,\xi}$,
\end{enumerate}
such that for any global estimator $\widehat{g}_{n}:(\mathcal{X}^{2}\times\mathcal{Y})^{n}\to \mathcal{G}_{0}$ and the local estimator $\widehat{g}_{n}^{\textup{local}}$ satisfying $\widehat{f}_{n,P}^{\textup{local}}=\sum_{i=1}^{d_{1}}\widehat{g}_{n,P,i}^{\textup{local}}v_{i}$ for each $P\in\mathcal{P}_{\alpha,1,\xi}$, we have
\begin{align*}
\sup_{P\in\mathcal{P}_{\alpha,1,\xi}}\mathcal{R}(\widehat{g}_{n,P}^{\textup{local}};P)\lesssim (\log^{4}{n})\sup_{P\in\mathcal{P}_{\alpha,1,\xi}}\mathcal{R}(\widehat{g}_{n};P).
\end{align*}
\end{theorem}
\begin{proof}
This claim is the combination of Theorem~\ref{thm:estimation error for deep relu networks} and Theorem~\ref{thm:minimax lower bound}.
\end{proof}
Note that the logarithmic factor $\log^{4}{n}$ in this theorem may be less important if one particularly focuses on the relation between the convergence rates of the given local and global estimators, similarly to the standard argument on minimax optimality (see, e.g.,~\citep{schmidt-hieber2020nonparametric,bos2022convergence,kim2021fast,meyer2023optimal,imaizumi2019deep,imaizumi2022advantage}).

\paragraph{Discussion of minimax lower bounds of local estimators.}
In connection with Theorem~\ref{thm:comparison theorem of local estimators}, we discuss some technical problems of the local minimax risk $\inf_{\widehat{g}_{n}^{\textup{local}}}\sup_{P\in\mathcal{P}_{\alpha,\tau,\xi}}\mathcal{R}(\widehat{g}_{n,P}^{\textup{local}};P)$ defined with the set of all local estimators of $\mathcal{P}_{\alpha,\tau,\xi}$ using deep ReLU networks.
It might be natural to ask the applicability of the standard minimax lower bounds, such as Le Cam's method~\citep{lecam1973convergence}, to the local minimax risk (see~\citep[Lemma~1]{yu1997assouad} and~\citep[Theorem~2.1]{tsybakov2009introduction} for the details of Le Cam's method, and see also~\citep[Lemma~2]{yu1997assouad} for the connection between Le Cam's method and Assouad's lemma~\citep{assouad1983deux}).
Let $\delta$ be a pseudo-distance on a parameter set $\Theta$.
According to~\citep[Eq.~(2.8)]{tsybakov2009introduction}, Le Cam's method builds on the triangle inequality $\delta(\vartheta,\vartheta')\leq \delta(\vartheta,\widehat{\vartheta}_{n})+\delta(\vartheta',\widehat{\vartheta}_{n})$ for any parameters $\vartheta,\vartheta'\in\Theta$ and any estimator $\widehat{\vartheta}_{n}$.
A similar argument is also used in~\citep[p.425]{yu1997assouad} under a general setting.
However, this triangle inequality is not necessarily applicable for the local estimators due to the dependence on the true parameter.
In fact, if there are at least two distinct parameters $\vartheta,\vartheta'$ such that $\delta(\vartheta,\vartheta')>0$, one can construct a local estimator $\vartheta\mapsto \widehat{\vartheta}_{n,\vartheta}$ for which $\delta(\vartheta,\vartheta')>\delta(\vartheta,\widehat{\vartheta}_{n,\vartheta})+\delta(\vartheta',\widehat{\vartheta}_{n,\vartheta'})$ is satisfied (e.g., the trivial estimator $\vartheta\mapsto \widehat{\vartheta}_{n,\vartheta}:=\vartheta$).
Note that it is well known that every ReLU network is a piecewise linear function (see, e.g.,~\citep[Theorem~2.1]{arora2018understanding}), and thus in the setting of Theorem~\ref{thm:estimation error for deep relu networks} the existence of such a trivial estimator might not be guaranteed.
However, it is still unclear whether this inequality holds for any local estimators using deep ReLU networks.

Therefore, it is required to develop a new general theory of local minimax lower bounds to study the optimality in terms of local minimax risks, although this might be a challenging problem.
Since the purpose of the this work is to develop a proof method of the learnability of smooth boundaries via pairwise binary classification, this topic is beyond the scope of the current work.

\paragraph{Interpretation of the local ERM estimator.}
The reader may wonder whether some local estimators have been studied so far.
Interestingly, several local ERM estimators are introduced in~\citep{mendelson2015learning,mendelson2017local}.
In~\citep{mendelson2015learning}, an ERM estimator depending on the given data distribution is introduced to develop the theory of regression under some mild assumptions on data distributions.
In~\citep{mendelson2017local}, a local ERM estimator is employed to establish tight upper bounds of the $L^{2}$-risk of a regression problem.
For a general learning theory based on a sub-Gaussian condition, see~\citep{alquier2019estimation}.
Note that in the current work, we use the local ERM estimator in Definition~\ref{def:formal definition of erm} to address the non-identifiability issue shown in Example~\ref{example:not good representation}, different from the purposes of~\citep{mendelson2015learning,mendelson2017local,alquier2019estimation}.

While the study of local estimators in Theorem~\ref{thm:estimation error for deep relu networks} is of mathematical interest, the reader may wonder about the practical implication of the theory.
Since providing a formal definition of what is practical is impossible without elucidating the practical nature of standard ERM algorithms, the current work cannot give a complete answer to this question.
Instead, we comment on some ideas that might be useful to fill the gap between the theory of local estimators and the practical implementation.

The notion of local estimators might be less common than the standard global estimators, in the field of statistics.
Meanwhile, in the field of machine learning, one may interpret the local ERM algorithm in Definition~\ref{def:formal definition of erm} as a learning algorithm influenced by some inductive biases.
For instance, it is well known that solving the optimization problem of the commonly-used ERM algorithm defined over the whole function class of deep neural networks is usually intractable due to the non-convexity with respect to the parameters in the networks.
Usually, stochastic optimization algorithms are used instead (see, e.g.,~\citep{oord2018representation,henaff2020data,he2020momentum,chen2020simple,dwibedi2021with}).
It is well known that stochastic optimization algorithms contain various types of inductive biases that control the optimization dynamics, such as regularization, initialization of parameters in neural networks, and sampling noise (see, e.g.,~\citep{suzuki2020generalization}).
Hence, one of the possible interpretations is to assume that the local ERM estimator defined with the localized subclass is realized by some other inductive biases depending on the prior knowledge of the data distribution.
Recently, in the context of self-supervised learning some structural conditions on inductive biases of function classes have been introduced to overcome the limitations of global estimators from several viewpoints, such as misclassification risk bounds in~\citep{saunshi2022understanding,haochen2023a}, designs of algorithms in~\citep{cabannes23ssl}, and optimal solutions in~\citep{parulekar2023infonce}.

An open question is about how to provide the mathematical definitions of such inductive biases precisely.
A possible future direction is to generalize the statistical optimization theory of deep learning in~\citep{suzuki2020generalization} to apply it to a pairwise binary classification problem, although the problem setting and assumptions in~\citep{suzuki2020generalization} are quite different from ours, and thus the further analysis might be highly challenging.

\paragraph{Application to nonparametric multiclass classification.}
To maintain the readability, we defer some additional results on the application of Theorem~\ref{thm:estimation error for deep relu networks} to a nonparametric multiclass classification problem, to Appendix~D\showsupplement{of}.

\subsection{Comparison with Other Proof Methods}
\label{subsec:comparison with other methods}

We compare the proof method developed in Section~\ref{subsubsec:outline} and Section~\ref{subsubsec:an oracle inequality} with some other approaches, including the methods developed in~\citep{bao2022pairwise,haochen2021provable,ge2024on}.

\paragraph{Permutation-invariant risks.}
We discuss the estimation problem based on a permutation-invariant risk using a permutation on $\{1,\cdots,d_{1}\}$.
In particular, we focus on the risk function $\mathbb{E}[\min_{\pi}\|\widehat{f}_{n}-f_{\pi}^{*}\|_{\mathcal{X},P_{X},2}^{2}]$, where the minimum is taken over all permutations on $\{1,\cdots,d_{1}\}$, and let $f_{\pi}^{*}:=\sum_{i=1}^{d_{1}}g_{\pi(i)}^{*}v_{i}$ for the given contrastive function $f^{*}=\sum_{i=1}^{d_{1}}g_{i}^{*}v_{i}$.

It is clear that the consistency of a given estimator under this permutation-invariant risk does not always imply the consistency under the $L^{2}$-risk in Definition~\ref{def:l2 risk}, as implied by Example~\ref{example:not good representation}.
Thus, the usage of the permutation-invariant risk is not suitable in the case where the index of each subset has a specific meaning\footnote{For instance, in a standard binary classification problem of medical diagnosis, one may suppose that the indices $i=1$ and $i=2$ represent positive and negative, respectively. In this case, one needs to estimate both the decision boundary and the order of the subsets simultaneously to control both the Type I and Type II errors.}, particularly, in Question~\ref{question:main question}.
A similar situation is also considered in~\citep{bao2022pairwise}.

In the context of representation learning, this permutation-invariant risk might be reasonable, as long as one can use supervised data in a downstream task to estimate the optimal permutation (see, e.g.,~\citep{arora2019theoretical,chen2020simple,haochen2021provable} for downstream tasks).
This setting is different from that in Question~\ref{question:main question}.

Here, we discuss some relations between this permutation-invariant risk and Question~\ref{question:how to derive a lower bound}.
In the case where $\mathbb{E}[\min_{\pi}\|\widehat{f}_{n}-f_{\pi}^{*}\|_{\mathcal{X},P_{X},2}^{2}]$ is employed instead, it is clear that (Q1) in Question~\ref{question:how to derive a lower bound} is resolved immediately.
Indeed, if $f\in\mathcal{F}_{0}$ satisfies that $\mathbb{E}_{P_{X,X'}}[|\psi\circ\rho_{f}-\psi\circ\rho_{f^{*}}|]=0$, then there is a permutation $\pi^{*}$ on $\{1,\cdots,d_{1}\}$ such that $f=f_{\pi^{*}}^{*}$, $P_{X}$-almost surely.
Hence, in this case we have $\min_{\pi}\|f-f_{\pi}^{*}\|_{\mathcal{X},P_{X},2}^{2}=\|f-f_{\pi^{*}}^{*}\|_{\mathcal{X},P_{X},2}^{2}=0$.
To address (Q2) in Question~\ref{question:how to derive a lower bound}, one may consider a generalization of the localized subclass in Definition~\ref{def:concentrated subclass}.
For instance, given $\xi\in\Xi$, $P\in\mathcal{P}_{\xi}$, $\beta\in (0,D_{\textup{proj}})$, $\beta_{0}\geq 0$, and $\mathcal{F}\subset \mathcal{F}_{0}$, we define
\begin{align*}
\mathscr{G}_{\beta,\beta_{0},P}(\mathcal{F})=\{f\in\mathcal{F}\;|\; \max_{\pi}P_{X}(\{x\in\mathcal{X}\;|\; \|f(x)-f_{\pi}^{*}(x)\|_{2}<\beta\})\geq 1-\beta_{0}\},
\end{align*}
where the maximum is taken over all permutations on $\{1,\cdots,d_{1}\}$.
Let $f\in\mathscr{G}_{\beta,\beta_{0},P}(\mathcal{F})$.
Then, there is a permutation $\pi^{*}$ on $\{1,\cdots,d_{1}\}$ such that we have $P_{X}(\|f-f_{\pi^{*}}^{*}\|_{2}<\beta)\geq 1-\beta_{0}$.
Also, one can apply almost the same arguments as in the proofs of Lemmas~\ref{lem:step 3 lemma 2} -- \ref{lem:step 3 lemma 1} to the case where the map $P\mapsto \mathscr{S}_{P}:=\{\mathcal{K}_{i}\}_{i=1}^{d_{1}}$ and the contrastive function $f^{*}$ introduced in Definitions~\ref{def:contrastive representations},~\ref{def:concentrated subclass}, and the lemmas are respectively generalized to $P\mapsto \mathscr{S}_{\pi,P}:=\{\mathcal{K}_{\pi(i)}\}_{i=1}^{d_{1}}$ and $f_{\pi}^{*}$ with any given permutation $\pi$.
Thus, using almost the same argument as the proof of Theorem~\ref{thm:main result}, one can verify that there are positive constants $C,C'$, and $C''$ that are independent of $n$ and $P$ such that for any estimator $\widehat{f}_{n,P}^{\textup{local}}:(\mathcal{X}^{2}\times\mathcal{Y})^{n}\to \mathscr{G}_{\beta,\beta_{0},P}(\mathcal{F})\subset \mathcal{F}_{0}$,
\begin{align*}
\mathbb{E}[\min_{\pi}\|\widehat{f}_{n,P}^{\textup{local}}(U_{1}^{n})-f_{\pi}^{*}\|_{\mathcal{X},P_{X},2}^{2}]\leq
C(\log{n})\mathbb{E}[\mathcal{E}(\widehat{f}_{n,P}^{\textup{local}}(U_{1}^{n});P)]^{1/\tau}+C'\beta_{0}+C''/n,
\end{align*}
where $P\in\mathcal{P}_{\tau,\xi}$, and $\tau,n$, and $U_{1},\cdots,U_{n}$ are defined as in Theorem~\ref{thm:main result}.

\paragraph{Comparison with~\citep{bao2022pairwise}.}
For convenience, we review several claims proven in~\citet{bao2022pairwise}.
The fundamental part of the method of~\citep{bao2022pairwise} is the following claim proven in~\citep[Theorem~1]{bao2022pairwise}:
\begin{theorem}[{Theorem~1 in~\citep{bao2022pairwise}}]
\label{citelem:identity of bao et al}
Given $K\in\mathbb{N}$, let $\mathcal{X}_{0}\subset \mathbb{R}^{K}$.
Given i.i.d. pairs $(X,Z)$ and $(X',Z')$ of covariates $X,X':\Omega\to\mathcal{X}_{0}$ and binary labels $Z,Z':\Omega\to\{-1,1\}$, define the random variable $Y:\Omega\to \{-1,1\}$ as $Y:=ZZ'$.
Let $P_{X_{0},Z}$ and $P$ be the distributions of $(X,Z)$ and $(X,X',Y)$, respectively.
Then, for the function $h:\mathbb{R}\to\mathbb{R}$ defined as $h(s)=\frac{1}{2}-\frac{1}{2}\sqrt{1-2s}$ and any measurable map $g:\mathcal{X}_{0}\to\{-1,1\}$, it holds that
\begin{align*}
P_{X_{0},Z}(g(x)\neq z)\wedge P_{X_{0},Z}(-g(x)\neq z)=h(P(g(x)g(x')\neq y)).
\end{align*}
\end{theorem}
Here, let $\widehat{g}_{n}^{\textup{SL}}$ denote the empirical risk minimizer of a similarity learning problem considered in~\citep[Eq.~(6)]{bao2022pairwise}, and let $\widehat{s}_{n'}$ denote an estimator of the sign of the given binary classifier introduced in~\citep[Eq.~(7)]{bao2022pairwise}, for convenience.
\citet[Theorem~3]{bao2022pairwise} prove an upper bound of the excess risk $P_{X_{0},Z}(\widehat{s}_{n'}\textup{sign}(\widehat{g}_{n}^{\textup{SL}}(x))\neq z)-\inf_{g^{*}}P_{X_{0},Z}(g^{*}(x)\neq z)$ for a given distribution $P_{X_{0},Z}$ on a measurable space $\mathcal{X}_{0}\times\{-1,1\}$.
Note that in~\citep[Theorem~3]{bao2022pairwise} the consistency of the given estimator is not proven;
We refer the reader to~\citep{bao2022pairwise} for the formal statements.

Regarding the results, the main differences are as follows:
(i) Comparing to~\citep[Theorem~1]{bao2022pairwise} (see Theorem~\ref{citelem:identity of bao et al}), one can see that the problem setting of the current work is not similar to that in~\citep{bao2022pairwise}.
Specifically, in the formalization of Definition~\ref{def:main assumption} we use no supervised data.
In addition, we do not assume that $X$ and $X'$ are always independent.
(ii) Theorem~3 in~\citep{bao2022pairwise} focuses on the generalizability of a learning algorithm, while the learnability of smooth boundaries is considered in Theorem~\ref{thm:estimation error for deep relu networks}.
(iii) In~\citep[Section~4]{bao2022pairwise}, some implications of Theorem~3 of~\citep{bao2022pairwise} to parametric models are discussed, while nonparametric estimation of multiple boundaries is not considered.
The differences in terms of the proof methods can be described as follows:
\begin{itemize}
\item Note that the well-known argument of set estimation used in~\citep{tsybakov2004optimal,meyer2023optimal} is applicable to the result of~\citep{bao2022pairwise}: Namely, using Proposition~1 in~\citep{tsybakov2004optimal}, one can see that under the Tsybakov noise condition~\citep{mammen1999smooth,tsybakov2004optimal}, Theorem~3 in~\citep{bao2022pairwise} implies an upper bound of the $L^{2}$-risk of $\widehat{s}_{n'}\cdot (\textup{sign}\circ \widehat{g}_{n}^{\textup{SL}})$.
In a nutshell, the method of~\citet{bao2022pairwise} is an algebraic approach.
According to~\citep[Appendix~A.2]{bao2022pairwise}, the key idea of their method is to define the pairwise response variable $Y$ as $Y=ZZ'$ using i.i.d. supervised data $(X,Z),(X',Z'):\Omega\to \mathcal{X}_{0}\times \{-1,1\}$, which makes it possible to use an identity proven in~\citep[Corollary~1, p.1242]{shimada2021classification}.
While their method does not rely on any localization argument, the algebraic property of labels in $\{-1,1\}$ is essential in~\citep[Theorem~3]{bao2022pairwise}.
Thus, it might be challenging to extend their method to a setting of multiclass classification, as discussed in Appendix~A.2\showsupplement{of}.
\item The method presented in Section~\ref{subsec:general upper bound of main result} is a geometric approach, and the key idea is to use the notion of localized subclasses in Definition~\ref{def:concentrated subclass} to bypass the technical obstacle shown in Example~\ref{example:not good representation}.
While this method relies on a localization argument, this approach is applicable even when the estimation of multiple smooth boundaries is considered.
\end{itemize}

\paragraph{Comparison with~\citep{haochen2021provable}.}
It is shown in~\citep{haochen2021provable} that multiple decision boundaries of a downstream linear multiclass classification problem are learnable if one can observe both supervised and pairwise data.
Hence, the problem setting considered in~\citep{haochen2021provable} is distinct from that in Question~\ref{question:main question}.
For the differences in terms of the results of multiclass classification, see Appendix~D.1\showsupplement{in}.

\paragraph{Comparison with~\citep{ge2024on}.}
\citet{ge2024on} prove an inequality that shares some similarity with (Q2) in Question~\ref{question:how to derive a lower bound}, while their purpose is to find some connection between pairwise binary classification and a downstream supervised learning problem and is quite different from ours.
In Lemma~D.2 of~\citep{ge2024on}, they prove an upper bound of $\|f-f'\|_{\mathcal{X}_{0},P_{X_{0}},2}^{2}$ based on the assumption that $\mathbb{E}_{P_{X_{0}}}[f(x)f'(x)^{\top}]$ is a symmetric matrix for the given pair of vector-valued functions $(f,f')$, where $\mathcal{X}_{0}$ is a measurable space, and $P_{X_{0}}$ is a probability distribution in $\mathcal{X}_{0}$.
Note that the assumption in~\citep[Lemma~D.2]{ge2024on} is usually violated in the problem setting of Theorem~\ref{thm:estimation error for deep relu networks}.
Note also that \citet{ge2024on} consider the setting where the paired covariates $X,X'$ are independent.
On the other hand, in our setting, the covariates $X$ and $X'$ are not necessarily independent, following~\citep{tsai2020neural}.
In the field of contrastive learning, it is common to assume that $X$ and $X'$ can be dependent, both in practice and in theory (see, e.g,~\citep{arora2019theoretical,chen2020simple,haochen2021provable}).

\section{Related Literature}
\label{subsec:discussion of the main theorem}

We provide the discussion of related work.

\subsection{Related Work on Learnability of Smooth Boundaries}
\label{subsec:regarding nonparametric estimation of smooth boundaries}
The statistical learnability of H\"{o}lder continuous boundaries is studied in many works~\citep{mammen1995asymptotical,mammen1999smooth,tsybakov2004optimal,kim2021fast,meyer2023optimal,imaizumi2019deep,imaizumi2022advantage}.
For the results of classical estimators, we focus on the most related work by~\citep{tsybakov2004optimal};
see~\citep{mammen1995asymptotical} for some results of set estimation, and see also~\citep{mammen1999smooth} for discriminant analysis.
Regarding the results using deep neural networks, we consider to compare to the related works~\citep{kim2021fast,meyer2023optimal} and~\citep{imaizumi2019deep,imaizumi2022advantage}.
See also~\citep[Theorem~1]{imaizumi2019deep} for some results of least-squares method using deep ReLU networks.

\begin{table}[t]
    \centering
    \caption{A summary of comparison to the previous work. In each row, we summarize the information from the corresponding reference. Regarding the column ``Convergence rate,'' the notation $\zeta=(K-1)/\alpha$ is used for convenience. The value $\tau\geq 1$ is the parameter of the Tsybakov noise condition~\citep{mammen1999smooth,tsybakov2004optimal}. Note that this noise condition is employed in~\citep{tsybakov2004optimal,kim2021fast,meyer2023optimal} and in the current work. Note also that \citet[Corollary~3.8]{meyer2023optimal} also shows results for $L^{2}$-risk using~\citep[Proposition~1]{tsybakov2004optimal}, under this condition. In~\citep{kim2021fast}, the parameter is defined with an affine transformation $A$. $\{\mathcal{A}_{i}\}_{i}$ and $\mathcal{A}$ denote some subsets defined with $\alpha$-H\"{o}lder continuous functions, where $\alpha> 0$. In~\citep{imaizumi2019deep,imaizumi2022advantage}, it is assumed that $g_{i}$ is $\gamma$-H\"{o}lder continuous. ``DNN'' is the abbreviation of ``Deep Neural Networks.'' Some ReLU networks are employed in~\citep{kim2021fast,imaizumi2019deep,meyer2023optimal} and also in the current work, while neural networks with general activation functions are considered in~\citep{imaizumi2022advantage}.}
    \customresize{
    \begin{tabular}{@{}l@{~~}l@{~~}l@{~~}l@{~~}l@{~~}l@{}}
        \hline
        Reference & Algorithm & Covariate & Parameter & Convergence rate \\
        \hline
        \citep[Thm.~1]{tsybakov2004optimal} & Binary classification & Single & $\mathds{1}_{\mathcal{A}}$ & $n^{-\frac{\tau}{\zeta+2\tau-1}}$ \\
        & (0-1 loss, sieve estimator) & & & (Excess risk) \\
        \hline
        \citep[Thm.~3.1]{kim2021fast} & Binary classification & Single & $A\circ \mathds{1}_{\mathcal{A}}$ & $n^{-\frac{\tau}{\tau\zeta+2\tau-1}}\log^{\frac{3\tau}{\tau\zeta+2\tau-1}}{n}$ \\
         & (hinge loss, DNN estimator) & & & (Excess risk) \\
        \hline
        \citep[Thm.~5.1]{kim2021fast} & Multiclass classification & Single & $\{A\circ\mathds{1}_{\mathcal{A}_{i}}\}_{i}$ & $n^{-\frac{\tau}{\tau\zeta+2\tau-1}}\log^{\frac{3\tau}{\tau\zeta+2\tau-1}}{n}$ \\
        & (DNN estimator) & & & (Excess risk) \\
        \hline
        \citep[Cor.~3.8]{meyer2023optimal} & Binary classification & Single & $\mathds{1}_{\mathcal{A}}$ & $n^{-\frac{s\tau}{\zeta+2\tau-1}}\log^{\frac{2s\tau}{\zeta}}{n}$, $s\geq 1$ \\
        & (0-1 loss, DNN estimator) & & & (Excess risk) \\
        \hline
        \citep[Thm.~2]{imaizumi2019deep} & Bayes estimation & Single & $\sum_{i}g_{i}\mathds{1}_{\mathcal{A}_{i}}$ & $n^{-(\frac{2\gamma}{2\gamma+K}\wedge \frac{1}{\zeta+1})}\log^{2}{n}$ \\
        & (DNN estimator) & & & ($L^{2}$-risk) \\
        \hline
        \citep[Thm.~7]{imaizumi2022advantage} & Least-squares method & Single & $\sum_{i} g_{i}\mathds{1}_{\mathcal{A}_{i}}$ & $n^{-(\frac{2\gamma}{2\gamma+K}\wedge\frac{1}{\zeta+1})}\log^{2}{n}$ \\
        & (DNN estimator) & & & ($L^{2}$-risk) \\
        \hline
        This work & Pairwise binary classification & Paired & $\sum_{i} \mathds{1}_{\mathcal{A}_{i}}v_{i}$ & $n^{-\frac{1}{\tau\zeta+2\tau-1}}\log^{3\tau^{-1}+1}{n}$ \\
        (Thm.~\ref{thm:estimation error for deep relu networks}) & (DNN estimator) & & & ($L^{2}$-risk) \\
        \hline
    \end{tabular}
    }
    \label{tab:comparison to related literature}
\end{table}

The comparison is shown in Table~\ref{tab:comparison to related literature}.
The main differences can be seen in four points, namely methods, types of data (either conventional data $(X,Z)$ or pairwise data $(X,X',Y)$), how boundary estimation is carried out (namely the true parameter and the risk function), and the convergence rates of upper bounds.
We can summarize the main points of Table~\ref{tab:comparison to related literature} as follows:
\begin{itemize}
\item \emph{Convergence rates.} When $\tau=1$, the convergence rate obtained in Theorem~\ref{thm:estimation error for deep relu networks} is consistent with the results shown in~\citep{tsybakov2004optimal,kim2021fast,meyer2023optimal,imaizumi2019deep,imaizumi2022advantage}, up to some logarithmic factors.
When $\tau\geq 1$, the convergence rate in Theorem~\ref{thm:estimation error for deep relu networks} is similar to that in~\citep[Theorem~3.1]{kim2021fast}.
This observation might be natural since the problem setting in Definition~\ref{def:main assumption} aligns with the standard setting employed in~\citep{kim2021fast}.
In addition similarly to~\citep[Corollary~3.8]{meyer2023optimal}, the combination of Theorem~3.1 in~\citep{kim2021fast} and Proposition~1 in~\citep{tsybakov2004optimal} implies the rate $n^{-\alpha/((2\tau-1)\alpha+\tau(K-1))}$ up to a logarithmic factor under the $L^{2}$-risk.
\item \emph{Problem settings.} Note that it is proven in~\citep[Theorem~3.2]{kim2021fast} that the rate in Theorem~3.1 of~\citep{kim2021fast} is improved under some additional assumption on data distributions.
Since another additional condition is also assumed in~\citep[Corollary~3.8]{meyer2023optimal}, the result in~\citep{meyer2023optimal} might not be directly comparable to Theorem~\ref{thm:estimation error for deep relu networks}.
While the regression problems considered in~\citep{imaizumi2019deep,imaizumi2022advantage} are more general than the setting defined with the $L^{2}$-risk in Definition~\ref{def:l2 risk}, it suffices to consider this $L^{2}$-risk to study Question~\ref{question:main question}.
\item \emph{Data.} The most clear difference is that we use a contrastive learning algorithm that requires data observed in a pairwise binary classification setting.
\item \emph{Estimators.} Another difference is that some local estimator is considered in Theorem~\ref{thm:estimation error for deep relu networks}, while the global ERM estimators are employed in~\citep{kim2021fast,meyer2023optimal,imaizumi2019deep,imaizumi2022advantage}.
This observation is due to the mathematical difference of learnabilities between the conventional and pairiwise binary classification problems, as discussed in Section~\ref{subsubsec:outline}.
\end{itemize}

\subsection{Related Work on Pairwise Binary Classification}
\label{subsec:on learning theory}

Similarity learning is an instance of pairwise binary classification, and its generalizability via the excess risks has been discussed in the literature~\citep{cao2015generalization,bao2022pairwise,zhou2024generalization}.
The most related work is~\citep{bao2022pairwise}, and the comparison has been shown in Section~\ref{subsec:comparison with other methods}.
The theoretical performance of deep neural networks in similarity learning involving nonparametric estimation is investigated in~\citep{zhou2024generalization}.
\citet{zhou2024generalization} show several upper bounds of the excess risk of similarity learning in the setting where conditional probability functions belong to the Sobolev space.
Meanwhile, in our work, we are mainly interested in the $L^{2}$-risk~\eqref{eq:sup l2 risk}.
Furthermore, we focus on the smoothness of boundaries, different from~\citep{zhou2024generalization}.
For the mathematical differences of classification problems in terms of the smoothness of target functions, we refer the reader to~\citep{audibert2007fast,kim2021fast}.

Contrastive learning is often interpreted as an instance of pairwise binary classification in the literature~\citep{gutmann2010noise,tsai2020neural,tosh2021contrastivejmlr,tosh2021contrastive,bao2022pairwise,chuang2022robust,zhai2023sigmoid}.
The theory of some general contrastive learning frameworks is extensively studied in much previous work (see, e.g., \citep{arora2019theoretical,haochen2021provable,wang2022chaos,bao2022on,saunshi2022understanding,huang2023towards,waida2023towards}).
In particular, several properties of the population risk minimizers of contrastive learning algorithms are investigated in~\citep{wang2020understanding,haochen2021provable,awasthi2022do,parulekar2023infonce,johnson2023contrastive,zhai2023understanding,koromilas2024bridging}.
However, the estimation performance in terms of the $L^{2}$-risk is less studied.
In the context of contrastive learning, the work by~\citet{tosh2021contrastivejmlr,tosh2021contrastive} is relevant to our analyses.
Indeed, in~\citep[Theorem~4]{tosh2021contrastivejmlr} and \citep[Theorem~11]{tosh2021contrastive}, they show several estimation error bounds of contrastive learning defined with some specific classification losses, where they consider the setting in which the estimation error is measured by downstream regression tasks and derive upper bounds in terms of an excess risk.
On the other hand, in Theorem~\ref{thm:estimation error for deep relu networks} the estimation error is measured by the $L^{2}$-risk for $\Delta^{d}$-valued functions, which enables us to study the performance of boundary estimation.
Furthermore, in Theorem~\ref{thm:estimation error for deep relu networks} the convergence rate is shown.
For the comparison to~\citep{haochen2021provable,ge2024on} in terms of the proof methods, see Section~\ref{subsec:comparison with other methods}.

\subsection{Discussion of Definitions}
\label{appsec:comparison with related mathematical notions}

\paragraph{Class $\mathscr{P}_{\alpha,R}^{K,d_{1},E}$ in Definition~\ref{def:class of smooth partitions}.}
The arguments developed in the proof of Theorem~\ref{thm:estimation error for deep relu networks} can be modified so that other smooth functions are employed.
For instance, the approximation theorems for boundaries defined with Barron functions~\citep{barron1993universal} and some applications to the conventional binary classification problems are studied in~\citep{caragea2023neural}.
In addition, the Besov space can also be employed if we instead apply the analyses of~\citep{suzuki2018adaptivity}, following the proofs of approximation theorems for indicator functions developed in~\citep{petersen2018optimal,imaizumi2019deep,imaizumi2022advantage}.

\paragraph{Condition (A3) in Definition~\ref{def:main assumption}.}
As discussed in Remark~\ref{rem:some remarks on the main assumptions}--(i), we consider a single-modal setting, and thus we assume that $q$ is symmetric.
Meanwhile, in the context of multimodal learning, a learning problem is often defined under the setting where the marginal distributions $p_{X}$ and $p_{X'}$ are distinct (see, e.g.,~\citep{radford2021learning} and~\citep[Section~4.2]{balestriero2023cookbook}).
The further study of multimodal learning settings is beyond the scope of this work, since we consider a single-modal setting as a method to discuss Question~\ref{question:main question}.
However, the proof method might be applicable under some modification.
For instance, it might be worth studying the setting where a localized subclass of the given function class is defined as the set of all functions satisfying the constraint in Definition~\ref{def:concentrated subclass} for both $P_{X}$ and $P_{X'}$.

\paragraph{Condition~\eqref{eq:statistical partition of unity} in Definition~\ref{def:main assumption}.}
In the context of contrastive learning theory, condition~\eqref{eq:statistical partition of unity} is weaker than the formulation of~\citep{awasthi2022do}.
Indeed, \citet{awasthi2022do} use a joint distribution introduced by~\citep{arora2019theoretical} and assume that the joint probability at any two data points belonging to the different disjoint subsets is zero.
On the other hand, the condition is stronger than the settings studied in~\citep{waida2023towards} and~\citep{parulekar2023infonce}.
The main difference to the formulation of~\citep{waida2023towards} is that we consider a setting of binary classification rather than contrastive learning, which requires to use disjoint partitions of the space $\mathcal{X}$.
Also, \eqref{eq:statistical partition of unity} is understood as an example of the formulation introduced by~\citep{parulekar2023infonce} since~\eqref{eq:statistical partition of unity} defines an equivalence relation in $\mathcal{X}$.
Meanwhile \citet{parulekar2023infonce} use a general equivalence relation using latent variables introduced by~\citep{vonkugelgenself2021}.

\paragraph{Localized subclasses in Definition~\ref{def:concentrated subclass}.}
For any given $\beta>0$, $\beta_{0}\geq 0$, $\xi\in\Xi$, $P\in\mathcal{P}_{\xi}$, and $\mathcal{F}\subset\mathcal{F}_{0}$, if $f\in\mathscr{F}_{\beta,\beta_{0},P}(\mathcal{F})$, then
\begin{align*}
P_{X}\circ f^{-1}\left(\bigcup_{i=1}^{d_{1}}\{z\in\Delta^{d}\;|\; \|z-v_{i}\|_{2}<\beta\}\right)\geq 1-\beta_{0}.
\end{align*}
Hence, Definition~\ref{def:concentrated subclass} implies how features are embedded.
In this sense, Definition~\ref{def:concentrated subclass} shares some similarity with some related concepts introduced in~\citep{schiebinger2015geometry,trillos2021geometric} in the context of clustering theory.
\citet[Definition~1]{schiebinger2015geometry} introduce an embedding condition based on finite samples.
\citet[Definition~8]{trillos2021geometric} consider a population setting and use the angles between a vector $z$ and orthonormal bases in the Euclidean space.
Meanwhile, we use the distance from each vertex.
Also, we use the notion of localized subclasses to address Question~\ref{question:main question}.

The notion of localized subclasses is also related to the condition referred to as ``small-ball condition'' by~\citep[p.10]{mendelson2015learning}.
Specifically, \citet[Assumption~3.1]{mendelson2015learning} additionally normalizes the difference $g-g'$ ($g,g'\in L^{2}(\mathcal{A})$ for a measure space $\mathcal{A}$) and considers a condition on the tail probability of the normalized gap.
Also, \citet[Definition~2.1]{mendelson2017local} employs a sub-Gaussian condition.
On the other hand, in our analysis it suffices to use an unnormalized gap as defined in Definition~\ref{def:concentrated subclass}.

\paragraph{Hinge loss in Definition~\ref{def:contrastive loss}.}
The function $\psi\circ\rho_{f}$ shares some similarity with a classifier introduced in~\citep{jin2009regularized}.
In~\citep{jin2009regularized}, a classifier is defined with the Mahalanobis distance.
On the other hand, in Definition~\ref{def:contrastive loss}, we employ the Euclidean distance and additionally embed the variables $x$ and $x'$ in the regular simplex.
Furthermore, \citet{jin2009regularized} define the binary variable by the information showing whether the supervised labels of two covariates coincide or not.
Meanwhile, we use the statistical dependence to define the distribution of $(X,X',Y)$, as in~\citep{tsai2020neural}.

Hinge loss is often studied in the literature on statistical learning (e.g., \citep{boser1992training,lin2002support,zhang2004statistical,bartlett2006convexity,lecue2007optimal,steinwart2008support,kim2021fast}).
In particular, in a setting of conventional binary classification, \citet[Theorem~3.2]{kim2021fast} study nonparametric estimation of smooth boundaries using empirical risk minimization with hinge loss.
Note that it is common in the literature~\citep{kim2021fast,meyer2023optimal,imaizumi2019deep,imaizumi2022advantage} that nonparametric estimation of smooth boundaries is studied with some specific loss functions, such as 0-1 loss in~\citep{meyer2023optimal} and squared loss in~\citep{imaizumi2019deep,imaizumi2022advantage}.
Furthermore, hinge loss and its variants are also studied in the context of self-supervised learning~\citep{arora2019theoretical,li2021self,shah2022max,waida2023towards,ji2023power} and similarity learning~\citep{jin2009regularized,cao2015generalization,zhou2024generalization}.

\section{Conclusion}
\label{sec:discussion}

In this work, we develop a method to prove that under several conditions, both the disjoint partition of smooth boundaries introduced in~\citep{imaizumi2022advantage} and its order are jointly learnable using a local estimator generated by a pairwise binary classification problem.

In addition to the discussion in Section~\ref{sec:discussion of the main results}, we provide some concluding remarks.

\paragraph{Other data and distances.}
While condition~\eqref{eq:condition of tsai et al} due to~\citep{tsai2020neural} is employed in the problem setting, it might be interesting to consider some other conditions to use another data generating process.
In the learning algorithm, the Euclidean distance is used for simplicity, while it might be interesting to consider other distance functions (see, e.g.,~\citep{dovgoshey2013weak}).

\paragraph{The choice of vector-valued networks.}
We note that in Definition~\ref{def:formal definition of erm}, the softmax function is used.
In the proof of Theorem~\ref{thm:estimation error for deep relu networks}, the class $\mathcal{F}_{L,J,S,M,\bm{d}}^{\Delta^{d}\textup{-NN}}$ ensures that the range of any network $f$ in this class is always included in $\Delta^{d}$, and we can thus apply several properties, such as Proposition~\ref{prop:simple fact for risks} and Lemma~\ref{lem:step 3 lemma 1}.
Additionally, by this property and the continuity of ReLU networks, the classifier $\psi\circ \rho_{f}:\mathcal{X}^{2}\to [-1,1]$ is continuous, similarly to the case where the conventional classification problem using the hinge loss is considered (see, e.g.,~\citep{lin2002support,lecue2007optimal}).
In this sense, it is a natural choice to use the softmax function to develop a pairwise binary classification algorithm using the hinge loss.

Let us recall that in addition to the result of~\citet{kim2021fast} for set estimation using the conventional classification algorithm with the hinge loss, the statistical property of set estimation using the 0-1 loss is also studied by~\citet{meyer2023optimal} (see Section~\ref{subsec:regarding nonparametric estimation of smooth boundaries}).
Hence, it is reasonable to consider another approach using the 0-1 loss, in the context of pairwise binary classification.
For instance, let us consider the case where one first estimates the smooth boundaries by some neural networks and then compose them with the indicator function to produce the estimator $\widehat{f}_{n}^{\textup{set}}:(\mathcal{X}^{2}\times\mathcal{Y})^{n}\to\mathcal{F}_{0}$ defined as $\widehat{f}_{n}^{\textup{set}}:=\sum_{i=1}^{d_{1}}\mathds{1}_{\widehat{\mathcal{K}}_{n,i}}v_{i}$, where $\widehat{\mathcal{K}}_{n,1},\cdots,\widehat{\mathcal{K}}_{n,d_{1}}$ denotes some set estimators based on the networks (for some examples of set estimators using neural networks, see~\citep{meyer2023optimal}).
Note that this is a variant of the estimator using the softmax function, defined in a similar way to Definition~\ref{def:formal definition of erm}.
In the case where such an estimator is considered, the applicability of the ideas in Section~\ref{subsec:general upper bound of main result} may depend on the property of the estimator, such as the range of $\widehat{f}_{n}^{\textrm{set}}(u_{1},\cdots,u_{n})$ for each $(u_{1},\cdots,u_{n})\in(\mathcal{X}^{2}\times\mathcal{Y})^{n}$.
Additionally, since $\psi\circ \rho_{\widehat{f}_{n}^{\textup{set}}(u_{1},\cdots,u_{n})}$ is not a continuous function in general, some additional discussion of the practical implementation is needed.
Thus, the study of the statistical property of pairwise binary classification using the 0-1 loss is an independent, interesting future work.

\paragraph{Other loss functions.}
Since the study of the learnability of smooth boundaries using pairwise binary classification algorithms has been lacking, in the current work we focus on the hinge loss as a tractable approach that allows us to examine several technical issues, including the existence of the Bayes classifier and Question~\ref{question:how to derive a lower bound}.
Meanwhile, in the field of machine learning, some other loss functions, such as \emph{InfoNCE}~\citep{oord2018representation}, its variants~\citep{henaff2020data,he2020momentum,chen2020simple,dwibedi2021with}, and other self-supervised learning algorithms (see, e.g.,~\citep{ermolov2021whitening,dufumier23integrating,huang2023towards,balestriero2023cookbook}), have recently been used.
However, the statistical properties of such loss functions, including the relations to the \emph{Bernstein condition}~\citep{bartlett2005empirical} and the Bayes classifier, have not been fully elucidated.
Additionally, while the statistical properties of several classical loss functions of binary classification are well understood (see, e.g.,~\citep{zhang2004statistical,bartlett2006convexity,alquier2019estimation}), the study of the optimal classifiers in the pairwise binary classification problems is also lacking.
To discuss whether one can replace the hinge loss with other loss functions in the proof method of Section~\ref{subsec:general upper bound of main result}, some detailed analyses of the loss functions are required in advance, and thus this topic is an independent, important future direction.


\subsubsection*{Acknowledgments}
\addcontentsline{toc}{section}{Acknowledgments}


We would like to thank the anonymous referees, the Associate Editor, and the Editor for many insightful suggestions and comments on the presentation of the paper, the main theorems, and the discussion of the main results.
We used \texttt{NumPy}~\citep{harris2020array} and \texttt{Matplotlib}~\citep{hunter2007matplotlib} to plot Figures~\ref{fig:smooth boundaries} -- \ref{fig:localized subclass}.
The visualization was performed using TSUBAME 4.0 of Institute of Science Tokyo.


\subsubsection*{Funding}
\addcontentsline{toc}{section}{Funding}


This work was partially supported by JSPS KAKENHI Grant Number 20H00576, 23H03460, and 24K14849.
We also acknowledge partial support by JST BOOST, Japan Grant Number JPMJBS2417 and JPMJBS2430.

\appendix



\section*{Notation Lists}
\label{appsec:notation list}

We show the main notation in Table~\ref{tab:some general notations} -- Table~\ref{tab:notations of algorithm}.

\begin{table}[H]
    \centering
    \caption{Some general notation.}
    \begin{tabular}{p{0.18\columnwidth}p{0.56\columnwidth}l}
    \hline
    Notation & Definition(s) & Section \\
    \hline
    $\mathcal{B}(\mathcal{A})$ & Borel $\sigma$-algebra of a topological space $\mathcal{A}$. & Section~\ref{subsec:notations} \\
    $\|\cdot\|_{L^{s}(\mathcal{A},\nu)}$ & $L^{s}$-norm of the given measure space $(\mathcal{A},\nu)$. & Section~\ref{subsec:notations} \\
    $\|\cdot\|_{s}$ & $s$-norm in the Euclidean space. & Section~\ref{subsec:notations} \\
    $\|\cdot\|_{\mathcal{A},\nu,s}$ & $\|f\|_{\mathcal{A},\nu,s}=\|\|f\|_{s}\|_{L^{s}(\mathcal{A},\nu)}$ for the given $f:\mathcal{A}\to\mathbb{R}^{t}$. & Section~\ref{subsec:notations} \\
    $\textup{sign}$ & $\textup{sign}(s)=1$ if $s\geq 0$ and $-1$ if $s<0$. & Section~\ref{subsec:notations} \\
    $g=(g_{1},\cdots,g_{s})$ & Given a set $\mathcal{A}$, $g_{1},\cdots,g_{s}:\mathcal{A}\to\mathbb{R}$, where $g:\mathcal{A}\to\mathbb{R}^{s}$. & Section~\ref{subsec:notations} \\
    $b=(b_{j})$ & Given $b\in\mathbb{R}^{s}$, $b=(b_{j}):=(b_{1},\cdots,b_{s})$. & Section~\ref{subsec:notations} \\
    $\mathbb{R}^{s\times t}$ & The set of all linear operators from $\mathbb{R}^{t}$ to $\mathbb{R}^{s}$. & Section~\ref{subsec:notations} \\
    $W=(W_{j_{1},j_{2}})$ & Given $W\in\mathbb{R}^{s\times t}$, $(W_{j_{1},j_{2}})$ is the matrix identified with $W$. & Section~\ref{subsec:notations} \\
    $\lesssim$ & Given $s_{1},s_{2}\in\mathbb{R}$, $s_{1}\lesssim s_{2}$ if there is $C>0$ independent of $n$ for which $s_{1}\leq Cs_{2}$, unless otherwise specified. & Section~\ref{subsec:notations} \\
    $\lceil\cdot\rceil, \lfloor\cdot\rfloor$ & Given $s\in\mathbb{R}$, $\lceil s\rceil=\min\{t\in\mathbb{Z}\;|\; s\leq t\}$ and $\lfloor s\rfloor=\max\{t\in\mathbb{Z}\;|\; t\leq s\}$. & Section~\ref{subsec:notations} \\
    $\mathds{1}_{\mathcal{A}}$ & Indicator function of the given set $\mathcal{A}$. & Section~\ref{subsec:notations} \\
    \hline
    \end{tabular}
    \label{tab:some general notations}
\end{table}

\begin{table}[H]
    \centering
    \caption{Notation of sets and related notions.}
    \begin{tabular}{p{0.18\columnwidth}p{0.56\columnwidth}l}
    \hline
    Notation & Definition(s) & Section \\
    \hline
    $\mathcal{X}$ & $\mathcal{X}=[0,1]^{K}$, where $K\in\mathbb{N}$. & Section~\ref{subsec:notations}\\
    $\mu$ & Lebesgue measure in $(\mathcal{X},\mathcal{B}(\mathcal{X}))$. & Section~\ref{subsec:notations}\\
    $\|\cdot\|_{\mathcal{X},s}$ & $\|f\|_{\mathcal{X},s}=\|f\|_{\mathcal{X},\mu,s}$ for any vector-valued function $f:\mathcal{X}\to\mathbb{R}^{t}$, where $s\in [1,\infty]$. & Section~\ref{subsec:notations}\\
    $\mathcal{Y}$ & $\mathcal{Y}=\{1,-1\}$. & Section~\ref{subsec:notations}\\
    $\chi$ & Counting measure in $\mathcal{Y}$. & Section~\ref{subsec:notations} \\
    $d_{1}$ & Number of subsets. & Section~\ref{subsec:notations} \\
    $(\Omega,\Sigma,Q)$ & A probability space. & Section~\ref{subsec:notations}\\
    $\mathcal{C}_{R}^{\alpha,K-1}$ & $\alpha$-H\"{o}lder ball on $[0,1]^{K-1}$ centered at the origin, where the radius is $R$. & Section~\ref{subsec:noise condition and smooth boundaries} \\
    $\|\cdot\|_{\mathcal{C}^{\alpha,K-1}}$ & H\"{o}lder norm of the $\alpha$-H\"{o}lder space on $[0,1]^{K-1}$. & Section~\ref{subsec:noise condition and smooth boundaries} \\
    $\mathscr{P}_{\alpha,R}^{K,d_{1},E}$ & A class of disjoint partitions of $\mathcal{X}$, where the smoothness condition is due to~\citep{imaizumi2022advantage} (see Definition~\ref{def:class of smooth partitions}). & Section~\ref{subsec:noise condition and smooth boundaries} \\
    $\tau,\theta_{\textup{NC}}$ & $\tau\geq 1$ is a parameter of the Tsybakov noise condition~\eqref{eq:margin assumption} due to~\citep{mammen1999smooth,tsybakov2004optimal}, and $\theta_{\textup{NC}}\in (0,1]$ is a threshold (see Definition~\ref{def:definition of noise condition}). & Section~\ref{subsec:noise condition and smooth boundaries} \\
    $d$ & $d=d_{1}-1$. & Section~\ref{subsec:learning models} \\
    $\mathcal{S}^{d-1}$ & Unit hypersphere in $\mathbb{R}^{d}$. & Section~\ref{subsec:learning models} \\
    $\Delta^{d}$ & A regular simplex in $\mathbb{R}^{d}$. & Section~\ref{subsec:learning models}\\
    $v_{1},\cdots,v_{d_{1}}$ & The vertices of $\Delta^{d}$. & Section~\ref{subsec:learning models}\\
    $D_{\Delta^{d}}$ & Diameter of the regular simplex $\Delta^{d}$. & Section~\ref{subsec:learning models} \\
    $\theta_{1},\theta_{2},\theta_{3}$ & Some thresholds used in Definition~\ref{def:main assumption}. & Section~\ref{subsec:assumptions} \\
    $\Xi$ & A set of hyperparameters (see Definition~\ref{def:auxiliary notions used in the theorem}). & Section~\ref{subsec:assumptions} \\
    $\mathcal{P}_{\alpha,\tau,\xi}$ & A class of Borel probability measures in $\mathcal{X}^{2}\times\mathcal{Y}$ (see Definition~\ref{def:main assumption}). & Section~\ref{subsec:assumptions} \\
    $\mathcal{P}_{\tau,\xi},\mathcal{P}_{\xi}$ & $\mathcal{P}_{\tau,\xi}=\bigcup_{\alpha>0}\mathcal{P}_{\alpha,\tau,\xi}$ and $\mathcal{P}_{\xi}=\bigcup_{\tau\geq 1}\mathcal{P}_{\tau,\xi}$ (see Definition~\ref{def:auxiliary notions used in the theorem}). & Section~\ref{subsec:assumptions} \\
    $\mathscr{S}_{P}$ & See Definition~\ref{def:auxiliary notions used in the theorem}. & Section~\ref{subsec:assumptions} \\
    \hline
    \end{tabular}
    \label{tab:notations of sets}
\end{table}

\begin{table}[H]
    \centering
    \caption{Notation of probability distributions.}
    \begin{tabular}{p{0.18\columnwidth}p{0.56\columnwidth}l}
    \hline
    Notation & Definition(s) & Section \\
    \hline
    $p(x,x',y)$ & Probability density function satisfying~\eqref{eq:condition of tsai et al} due to~\citep{tsai2020neural}. & Section~\ref{subsec:notations}\\
    $p_{X}(x)$ & $p_{X}(x)=\int_{\mathcal{X}}(p(x,x',1)+p(x,x',-1))\mu(dx')$. & Section~\ref{subsec:notations} \\
    $p_{X'}(x')$ & $p_{X'}(x')=\int_{\mathcal{X}}(p(x,x',1)+p(x,x',-1))\mu(dx)$. & Section~\ref{subsec:notations}\\
    $p_{X,X'}(x,x')$ & $p_{X,X'}(x,x')=p(x,x',1)+p(x,x',-1)$. & Section~\ref{subsec:notations}\\
    $p_{Y}(y)$ & $p_{Y}(y)=\int_{\mathcal{X}\times\mathcal{X}}p(x,x',y)\mu(dx)\mu(dx')$. & Section~\ref{subsec:notations} \\
    $q(x,x')$ & $q(x,x')=p(x,x'|y=1)$, following~\citep{tsai2020neural}. & Section~\ref{subsec:notations}\\
    $\eta(x,x')$ & $\eta(x,x')=p(y=1|x,x')$. & Section~\ref{subsec:notations}\\
    $P_{X,X'}$ & probability measure whose density is $p_{X,X'}$. & Section~\ref{subsec:notations} \\
    $P_{X},P_{X'}$ & probability measures with densities $p_{X}$ and $p_{X'}$. & Section~\ref{subsec:notations} \\
    $P_{X,X'}^{-}$ & probability measure with density $p_{X}\otimes p_{X'}$. & Section~\ref{subsubsec:an oracle inequality} \\
    \hline
    \end{tabular}
    \label{tab:notations of probability distributions}
\end{table}

\begin{table}[H]
    \centering
    \caption{Notation of function classes and risk functions.}
    \begin{tabular}{p{0.18\columnwidth}p{0.56\columnwidth}l}
    \hline
    Notation & Definition(s) & Section \\
    \hline
    $\mathcal{F}_{0}$ & A set of $\Delta^{d}$-valued functions on $\mathcal{X}$. & Section~\ref{subsec:learning models} \\
    $\sigma_{\textup{ReLU}}$ & ReLU function. & Section~\ref{subsec:learning models} \\
    $g_{\bm{W},\bm{b}}$ & ReLU networks. & Section~\ref{subsec:learning models} \\
    $\mathcal{F}_{L,J,S,M,\bm{d}}^{\textup{NN}}$ & A class of ReLU networks introduced in~\citep{nakada2020adaptive,imaizumi2019deep,imaizumi2022advantage}. & Section~\ref{subsec:learning models} \\
    $H$ & Softmax function. & Section~\ref{subsec:learning models} \\
    $\mathcal{F}_{L,J,S,M,\bm{d}}^{\Delta^{d}\textup{-NN}}$ & A set of $\Delta^{d}$-valued ReLU networks. & Section~\ref{subsec:learning models} \\
    $f_{\bm{W},\bm{b}}$ & $\Delta^{d}$-valued ReLU networks in $\mathcal{F}_{L,J,S,M,\bm{d}}^{\Delta^{d}\textup{-NN}}$. & Section~\ref{subsec:learning models} \\
    $\mathcal{G}_{0}$ & A set of probability-simplex-valued functions on $\mathcal{X}$. & Section~\ref{subsec:assumptions} \\
    $\mathcal{R}(\widehat{g}_{n};P)$ & $L^{2}$-risk of the given estimator $\widehat{g}_{n}$ (see Definition~\ref{def:l2 risk}). & Section~\ref{subsec:assumptions} \\
    $\mathcal{E}(f;P)$ & An excess risk (see Theorem~\ref{thm:main result}). & Section~\ref{subsubsec:an oracle inequality} \\
    \hline
    \end{tabular}
    \label{tab:notations of function classes}
\end{table}

\begin{table}[H]
    \centering
    \caption{Notation of estimators and algorithms.}
    \begin{tabular}{p{0.18\columnwidth}p{0.56\columnwidth}l}
    \hline
    Notation & Definition(s) & Section \\
    \hline
    $U_{1}^{n}$ & $U_{1}^{n}=(U_{1},\cdots,U_{n})$, where $U_{1},\cdots,U_{n}:\Omega\to \mathcal{X}^{2}\times\mathcal{Y}$ are random variables (see Definition~\ref{def:l2 risk}). & Section~\ref{subsec:assumptions} \\
    $f^{*}$ & Contrastive function (see Definition~\ref{def:contrastive representations}). & Section~\ref{subsec:core notations} \\
    $\mathscr{F}_{\beta,\beta_{0},P}(\mathcal{F})$ & Localized subclass of the given class $\mathcal{F}\subset\mathcal{F}_{0}$ (see Definition~\ref{def:concentrated subclass}). & Section~\ref{subsec:core notations} \\
    $D_{\textup{proj}}$ & See Remark~\ref{rem:remarks about localized subclass}--(i). & Section~\ref{subsec:core notations} \\
    $\widehat{f}_{n},\widehat{g}_{n}$ & Global estimators (see Definition~\ref{def:definitions of local and global estimators}). & Section~\ref{subsec:core notations} \\
    $\widehat{f}_{n}^{\textup{local}},\widehat{g}_{n}^{\textup{local}}$ & Local estimators (see Definition~\ref{def:definitions of local and global estimators}). & Section~\ref{subsec:core notations} \\
    $\widehat{f}_{n,P}^{\textup{local}},\widehat{g}_{n,P}^{\textup{local}}$ & Estimators defined with each $P\in\mathcal{P}\subset\mathcal{P}_{\xi}$ (see Remark~\ref{rem:remark of the global estimators}). & Section~\ref{subsec:core notations} \\
    $\rho_{f}$ & A squared Euclidean distance between vectors $f(x)$ and $f(x')$. & Section~\ref{subsec:loss function continued} \\
    $\psi$ & Function $\psi(s)=1-2D_{\Delta^{d}}^{-2}s$ (see Definition~\ref{def:contrastive loss}). & Section~\ref{subsec:loss function continued} \\
    $\ell_{f}$ & Hinge loss in Definition~\ref{def:contrastive loss}. & Section~\ref{subsec:loss function continued} \\
    $\widehat{g}_{n}^{\textup{LERM}}$ & $(\beta,\varepsilon,n,\mathcal{P},\mathcal{F})$-local ERM estimator (Definition~\ref{def:formal definition of erm}). & Section~\ref{subsubsec:an oracle inequality} \\
    $\widehat{f}_{n}^{\textup{LERM}}$ & Local estimator corresponding to $\widehat{g}_{n}^{\textup{LERM}}$ (see Definition~\ref{def:formal definition of erm}). & Section~\ref{subsubsec:an oracle inequality} \\
    \hline
    \end{tabular}
    \label{tab:notations of algorithm}
\end{table}

\addcontentsline{toc}{section}{Notion Lists}

\counterwithin{figure}{section}



\section{Additional Discussion}
\label{appsec:technical discussion}

\subsection{Validity of the Estimation Problem}
\label{subsec:well-conditionedness of the stimation problem}

It is natural to ask whether (A1) -- (A4) in Definition~\ref{def:main assumption} are well-conditioned for estimating smooth partitions introduced by~\citep{imaizumi2022advantage}.
In what follows, given any $\theta_{4}\in (0,1/2]$, we show that the range of the map $P\mapsto \mathscr{S}_{P}$ is related to the class
\begin{align*}
\mathscr{P}_{\alpha,R,+}^{K,d_{1},E}=\{\{\mathcal{K}_{i}\}_{i=1}^{d_{1}}\in\mathscr{P}_{\alpha,R}^{K,d_{1},E}\;|\; \mu(\mathcal{K}_{i})\in [\theta_{4},\theta_{3}]\textup{ for any }i=1,\cdots,d_{1}\}.
\end{align*}
\begin{proposition}
\label{lem:universal partition}
Let $\alpha>0$, $\tau\geq 1$, $\xi=(R,K,d_{1},E,\theta_{\textup{NC}},\theta_{1},\theta_{2},\theta_{3})\in \Xi$, and $\theta_{4}\in (0,1/2]$.
If $\theta_{1}\theta_{4}^{\frac{1}{2}}\geq 1$ and $\theta_{\textup{NC}}$ satisfies either $\theta_{\textup{NC}}\leq (\frac{1-\theta_{3}}{2(1+\theta_{3})})^{\frac{1}{\tau-1}}$ (if $\tau>1$) or $\theta_{\textup{NC}}<\frac{1-\theta_{3}}{2(1+\theta_{3})}$ (if $\tau=1$), then for every $\mathscr{S}=\{\mathcal{K}_{i}\}_{i=1}^{d_{1}}\in\mathscr{P}_{\alpha,R,+}^{K,d_{1},E}$, there are a Borel probability measure $P\in\mathcal{P}_{\alpha,\tau,\xi}$ and a permutation $\pi$ on $\{1,\cdots,d_{1}\}$ such that $\mathscr{S}_{P}=\{\mathcal{K}_{\pi(i)}\}_{i=1}^{d_{1}}$.
\end{proposition}

In other words, this proposition implies that statistical learning with samples drawn from any distribution in $\mathcal{P}_{\alpha,\tau,\xi}$ may cover the smooth partitions belonging to the class $\mathscr{P}_{\alpha,R,+}^{K,d_{1},E}$, as long as $\theta_{1}$ and $\theta_{\textup{NC}}$ satisfy these conditions.
Thus, it is reasonable to focus on the class $\mathcal{P}_{\alpha,\tau,\xi}$.

To show the above proposition, we utilize some notions developed in~\citep{arora2019theoretical,awasthi2022do}.

\begin{proof}[Proof of Proposition~\ref{lem:universal partition}]
Denote by $\mathscr{S}=\{\mathcal{K}_{i}\}_{i=1}^{d_{1}}$.
Let $P_{X}$ be an arbitrary Borel probability measure in $(\mathcal{X},\mathcal{B}(\mathcal{X}))$ such that $P_{X}$ is absolutely continuous for the Lebesgue measure $\mu$, and the Lebesgue density $p_{X}$ is continuous and positive at every point in $\mathcal{X}$ and satisfies $\|p_{X}\|_{L^{\infty}(\mathcal{X})}\leq \theta_{1}\theta_{4}^{\frac{1}{2}}$ and $P_{X}(\mathcal{K}_{i})\in [\theta_{4},\theta_{3}]$ for every $i\in\{1,\cdots,d_{1}\}$.
Let $B$ and $B'$ be i.i.d. random variables drawn from a distribution on the set $\{1,\cdots,d_{1}\}$ such that $Q(B=i)=Q(B'=i)=P_{X}(\mathcal{K}_{i})$ for any $i=1,\cdots,d_{1}$.
Also, let $V_{1},\cdots,V_{d_{1}}$ and $V_{1}',\cdots,V_{d_{1}}'$ be i.i.d. random variables on $\Omega$ such that both $V_{i}$ and $V_{i}'$ follow the conditional distribution $P_{X}(\cdot|\mathcal{K}_{i})=P_{X}(\cdot\cap\mathcal{K}_{i})/P_{X}(\mathcal{K}_{i})$ for every $i=1,\cdots,d_{1}$.
Here, consider a random variable $\Gamma:\Omega\to\{0,1\}$ drawn from the Bernoulli distribution with parameter $Q(\Gamma=1)=2^{-1}\in (0,1-\theta_{2}]$.
Following~\citep{arora2019theoretical,awasthi2022do}, we define
\begin{align*}
Y&=2\mathds{1}_{\left\{\Gamma=1\right\}}-1,\\
(X,X')|Y&=
\begin{cases}
(\sum_{i=1}^{d_{1}}V_{i}\mathds{1}_{\{B=i\}},\sum_{i=1}^{d_{1}}V_{i}'\mathds{1}_{\{B=i\}})&\quad\textup{if }Y=1,\\
(\sum_{i=1}^{d_{1}}V_{i}\mathds{1}_{\{B=i\}},\sum_{i=1}^{d_{1}}V_{i}'\mathds{1}_{\{B'=i\}})&\quad\textup{if }Y=-1.
\end{cases}
\end{align*}
Note that the distribution of $(X,X')|Y$ is indeed an example of joint distributions introduced by~\citep[Eq.~(1), (2)]{arora2019theoretical}.
In particular, the above construction satisfies Assumption~3.1 in~\citep{awasthi2022do}.
Hence, the above example expresses the setting developed in~\citep{arora2019theoretical,awasthi2022do} as a pairwise binary classification problem.
Thus, it suffices to show that the example satisfies (A1)--(A4).

By the construction, the distribution of $(X,X',Y)$ satisfies both conditions (A1) and (A3), where note that $\|q\|_{L^{\infty}(\mathcal{X}^{2})}\leq \theta_{4}^{-1}(\theta_{1}\theta_{4}^{\frac{1}{2}})^{2}=\theta_{1}^{2}$.
In addition, for any $i\in\{1,\cdots,d_{1}\}$ and every $(x,x')\in \mathcal{K}_{i}\times\mathcal{K}_{i}$, the condition $\theta_{3}<1$ implies that we have
\begin{align*}
\eta(x,x')=\frac{1}{1+P_{X}(\mathcal{K}_{i})} >\frac{1}{2}.
\end{align*}
Note that $\eta(x,x')=0$ otherwise.
Thus, the distribution satisfies condition (A4).
Here, let $t_{i}=1/(1+P_{X}(\mathcal{K}_{i}))-2^{-1}$ for each $i\in\{1,\cdots,d_{1}\}$.
When $\tau>1$, condition (A2) is satisfied if
\begin{align*}
\theta_{\textup{NC}}^{-1}\min_{i\in\{1,\cdots,d_{1}\}}t_{i}^{\frac{1}{\tau-1}}\geq 1.
\end{align*}
This sufficient condition is satisfied if $\theta_{\textup{NC}}\leq ((1-\theta_{3})/(2(1+\theta_{3})))^{1/(\tau-1)}$, when $\tau>1$.
When $\tau=1$, condition (A2) is satisfied with $\theta_{\textup{NC}}< (1-\theta_{3})/(2(1+\theta_{3}))$ since it holds that $\max_{i\in\{1,\cdots,d_{1}\}}P_{X}(\mathcal{K}_{i})\leq \theta_{3}$.
Therefore, the distribution of $(X,X',Y)$, denoted by $P$, belongs to $\mathcal{P}_{\alpha,\tau,\xi}$.

By the construction of $P$ and condition~\eqref{eq:statistical partition of unity}, for the partition $\mathscr{S}_{P}=\{\mathcal{K}_{i}'\}_{i=1}^{d_{1}}\in\mathscr{P}_{\alpha,R}^{K,d_{1},E}$ satisfying condition (A4) in Definition~\ref{def:main assumption}, there is a permutation $\pi$ on the set $\{1,\cdots,d_{1}\}$ such that $\mathcal{K}_{i}'=\mathcal{K}_{\pi(i)}$ for any $i\in\{1,\cdots,d_{1}\}$.
We obtain the claim.
\end{proof}

\subsection{Discussion of a Similarity Learning Problem}
\label{appsec:similarity learning and localized subclass}

We discuss the learnability of smooth boundaries via the learning problem studied in~\citep{bao2022pairwise}, where some additional discussion can be found in Section~\ref{subsec:comparison with other methods}.
In the subsequent paragraphs, we focus on the following two topics: (i) the investigation of the problem setting of~\citet{bao2022pairwise} in terms of (Q1) in Question~\ref{question:how to derive a lower bound}, and (ii) the discussion of the extensibility of the theoretical results in~\citep{bao2022pairwise}.

Before proceeding, we define some notation, following~\citep{bao2022pairwise}.
Recall that~\citet{bao2022pairwise} consider the following data generating process of the binary variable $Y$:
\begin{align}
\label{eq:similarity learning data sampling process}
Y=ZZ',
\end{align}
where $Z,Z':\Omega\to \{-1,1\}$ are given random variables.
Recall also that in~\citep{bao2022pairwise}, the pairwise binary classifier $g(x)g(x')$ defined with the given binary classifier $g:\mathcal{X}_{0}\to\mathcal{Y}$ is trained in their algorithm, to predict the decision boundary in the given measurable space $\mathcal{X}_{0}\subset\mathbb{R}^{K}$ via the trained classifier $g$.
For convenience, we rewrite the pairwise classifier used in~\citep{bao2022pairwise} as the functional $g\mapsto h_{g}$ satisfying
\begin{align*}
h_{g}(x,x')=g(x)g(x'),
\end{align*}
for the given measurable map $g:\mathcal{X}_{0}\to\mathcal{Y}$.

\paragraph{(i) Discussion of problem settings.}
In connection with (Q1) in Question~\ref{question:how to derive a lower bound}, we consider the setting where the hinge loss is used.
Then, as shown in the following proposition, an example similar to Example~\ref{example:not good representation} is observed:
\begin{proposition}
\label{prop:excess risk and l2 risk of similarity learning}
Let $\mathcal{X}_{0}\subset\mathbb{R}^{K}$ be a measurable space endowed with a non-negative, $\sigma$-finite measure $\nu$.
Given random variables $(X,Z),(X',Z'):\Omega\to \mathcal{X}_{0}\times \{-1,1\}$ that are i.i.d. sampled from the distribution $P_{X_{0},Z}$, where $P_{X_{0},Z}$ is supposed to be absolutely continuous for $\nu\otimes \chi$, let $g^{*}:\mathcal{X}_{0}\to \{-1,1\}$ be the Bayes classifier of $P_{X_{0},Z}$, and define the random variable $Y:\Omega\to \{-1,1\}$ to satisfy~\eqref{eq:similarity learning data sampling process} for $Z$ and $Z'$.
Let $P$ be the distribution of $(X,X',Y)$, and denote the marginal distribution of $P_{X_{0},Z}$ with respect to the space $\mathcal{X}_{0}$ by $P_{X_{0}}$.
Then, we have
\begin{align*}
\mathbb{E}_{P_{X_{0},X_{0}'}}[|h_{-g^{*}}-h_{g^{*}}|]=0,
\end{align*}
and
\begin{align*}
&\|\mathds{1}_{\{x\in\mathcal{X}_{0}\;|\; -g^{*}(x)=1\}}-\mathds{1}_{\{x\in\mathcal{X}_{0}\;|\; g^{*}(x)=1\}}\|_{L^{2}(\mathcal{X}_{0},P_{X_{0}})}^{2}>0,\\
&\|(-g^{*})-g^{*}\|_{L^{2}(\mathcal{X}_{0},P_{X_{0}})}^{2}>0,
\end{align*}
where $P_{X_{0},X_{0}'}$ denotes the marginal distribution of $P$ with respect to the space $\mathcal{X}_{0}^{2}$.
\end{proposition}
\begin{proof}
The identity is due to the property that $h_{-g^{*}}(x,x')=h_{g^{*}}(x,x')$ for any $x,x'\in\mathcal{X}_{0}$.
Since $\|\mathds{1}_{\{x\in\mathcal{X}_{0}|-g^{*}(x)=1\}}-\mathds{1}_{\{x\in\mathcal{X}_{0}|g^{*}(x)=1\}}\|_{L^{2}(\mathcal{X}_{0},P_{X_{0}})}^{2}=P_{X_{0}}(\mathcal{X}_{0})=1$, we obtain the first inequality.
The second inequality is trivial.
\end{proof}
Note that under the additional condition that $P_{X_{0},X_{0}'}(p(y=1|x,x')=\frac{1}{2})=0$, Theorem~1 in~\citep{bao2022pairwise} (see Theorem~\ref{citelem:identity of bao et al}) directly implies that the Bayes classifier of $P$ defined in Proposition~\ref{prop:excess risk and l2 risk of similarity learning} is $h_{g^{*}}$.
Thus, under the setting of Proposition~\ref{prop:excess risk and l2 risk of similarity learning}, the argument presented in Section~\ref{subsubsec:outline} is applicable:
Given $\tau\geq 1$, if $P$ satisfies the Tsybakov noise condition~\citep{mammen1999smooth,tsybakov2004optimal} with some suitable constants, by Proposition~1 in~\citep{lecue2007optimal} (see Lemma~\ref{citelem:lecue lemma}), there is a constant $C>0$ such that
\begin{align*}
&\mathbb{E}_{P_{X_{0},X_{0}'}}[|h_{-g^{*}}-h_{g^{*}}|]^{\tau}\\
&\leq C(\mathbb{E}_{P}[\max\{0,1-yh_{-g^{*}}(x,x')\}]-\mathbb{E}_{P}[\max\{0,1-yh_{g^{*}}(x,x')\}]).
\end{align*}
Therefore, Proposition~\ref{prop:excess risk and l2 risk of similarity learning} implies that under the similarity learning problem of~\citet{bao2022pairwise}, one needs to overcome a problem similar to (Q1) in Question~\ref{question:how to derive a lower bound}.
In the next paragraph, we review how~\citet{bao2022pairwise} develop a method that enables one to address this problem.

\paragraph{(ii) Discussion of the extensibility.}
In~\citep[Section~A.2]{bao2022pairwise}, several equivalent identities of the expected binary misclassification loss in~\citep[Theorem~1 and Corollary~1]{shimada2021classification} are utilized to prove a core idea of a sign estimator.
Formally, it is claimed in the proof of~\citep[Corollary~1, p.1242]{shimada2021classification} that in the setting where the label is a $\{-1,1\}$-valued random variable, for any probability measure $P_{X_{0},Z}$ in any given measurable space $\mathcal{X}_{0}\times \{-1,1\}$ that is absolutely continuous for non-negative $\sigma$-finite product measure $\nu\otimes\chi$ and for any measurable function $g:\mathcal{X}_{0}\to\{-1,1\}$, it holds that
\begin{align}
\label{eq:shimada bao identity}
&P_{X_{0},Z}(g(x)\neq z)\nonumber\\
&=\mathbb{E}_{(x,z),(x',z')\sim_{i.i.d.} P_{X_{0},Z}}\left[\frac{\mathds{1}_{\{(x,z,z')\;|\; g(x)\neq zz'\}}+\mathds{1}_{\{(x',z,z')\;|\; g(x')\neq zz'\}}}{2(p(z=1)-p(z=-1))}\right]\nonumber \\
&\quad -\frac{p(z=-1)}{p(z=1)-p(z=-1)},
\end{align}
where $p(z=1):=P_{X_{0},Z}(\mathcal{X}_{0}\times \{1\})$ and $p(z=-1):=1-p(z=1)$.
\citet[Theorem~2]{bao2022pairwise} use this identity to prove a formula for calculating the optimal sign $s^{*}$ of the given measurable function $g:\mathcal{X}_{0}\to\{-1,1\}$, which is defined as
\begin{align*}
s^{*}=\argmin_{s\in\{-1,1\}}P_{X_{0},Z}(sg(x)\neq z).
\end{align*}
See~\citep[Section~A.2]{bao2022pairwise} for the detailed derivations.
Consequently, \citet[Theorem~3]{bao2022pairwise} use the formula to prove an excess risk bound of binary classification.
Therefore, the proof method of~\citep{bao2022pairwise} does not rely on a localization argument.

As mentioned in~\citep[Section~6]{bao2022pairwise}, their proof method is for binary classification, not for multiclass classification.
One can notice that in~\citep[Section~3.1.1, p.1240]{shimada2021classification} the derivation of the identity is based on the fact that the set $\{-1,1\}$ with standard multiplication is a cyclic group generated by $-1$.
From this observation, it might be useful to consider some other cyclic group to extend the method of~\citet{bao2022pairwise} to a multiclass classification problem.
In particular, the following open question remains:
(i) Is there a cyclic group such that both Theorem~1 and Theorem~2 in~\citep{bao2022pairwise} are extensible to a multiclass classification problem where the multiclass label takes values in the group?
(ii) Supposing that such a cyclic group exists, is it possible to construct a practical similarity learning algorithm that produces some minimax optimal ERM estimator?
A natural idea might be to use some cyclic group of complex numbers, since $\{-1,1\}$ can be defined as a set of complex numbers.
For instance, given $d_{1}\in\mathbb{N}$ such that $d_{1}\geq 3$, one can take a complex number $z_{1}$ for which it generates a cyclic group $\{z_{1},z_{1}^{2},\cdots,z_{1}^{d_{1}}\}$.
Another idea might be to use some quotient group to define the labels.
However, it might be highly non-trivial to prove or disprove these questions with such cyclic groups.

\subsection{Discussion of Neural Networks in Pairwise Binary Classification}
\label{subsec:discussion of neural networks in pairwise binary classification}

In~\citep{imaizumi2019deep,imaizumi2022advantage}, under a nonparametric regression problem where the regression function is defined with the product of a smooth function and a discontinuous function characterized by some smooth boundaries, the following two claims are proven: (i) It is proven that the least squares estimator using deep neural networks is minimax optimal up to some logarithmic factor (see~\citep[Theorems~1, 2, and 3]{imaizumi2019deep} and \citep[Theorems~7 and 13]{imaizumi2022advantage}). (ii) The sub-optimality of some classical linear estimators in the regression problem is proven (see~\citep[Corollary~1]{imaizumi2019deep} and~\citep[Corollary~20]{imaizumi2022advantage}).
It is well known that the classical linear estimators are minimax optimal in some regression problems defined with smooth regression functions (see, e.g.,~\citep{tsybakov2009introduction,imaizumi2019deep,imaizumi2022advantage}).
\citet{imaizumi2019deep,imaizumi2022advantage} focus on this background and consider another setting to prove a mathematical advantage of deep learning over some classical linear estimators in a regression problem.

In~\citep[Section~1.1]{imaizumi2022advantage}, it is mentioned that their problem setting using piecewise smooth functions follows that of~\citep{petersen2018optimal} (see also~\citep[Remark~1]{imaizumi2022advantage}).
According to~\citep[Section~1.1]{petersen2018optimal}, \citet{petersen2018optimal} consider piecewise continuous functions with smooth boundaries, motivated by the fact that some of the usual classification problems are defined with such functions.
Indeed, in the context of nonparametric statistics, \citet{tsybakov2004optimal} provides some analyses for a class similar to that of~\citep[Definition~3.3]{petersen2018optimal} in the conventional binary classification problem and also proves the minimax optimality of some classical estimators (see~\citep[Theorem~1 and Theorem~2]{tsybakov2004optimal}).
Meanwhile, \citet{imaizumi2019deep,imaizumi2022advantage} employ the setting of~\citep{petersen2018optimal} to prove an advantage of deep learning in a standard regression problem.

In the current work, we studied the theoretical properties of smooth boundaries via a nonparametric classification problem, as in~\citep{petersen2018optimal} and the related work~\citep{tsybakov2004optimal,kim2021fast,meyer2023optimal}.
Thus, both the motivation and the goal are different from those in~\citep{imaizumi2019deep,imaizumi2022advantage}.
In addition, deep ReLU networks have been often employed to prove the consistency of estimators in classification problems~\citep{kim2021fast,bos2022convergence,meyer2023optimal}.
Therefore, the further investigation of the superiority of deep learning in (pairwise) binary classification problems in connection with the contributions of~\citet{imaizumi2019deep,imaizumi2022advantage} is beyond the scope of the current work.

\section{Proofs}
\label{appsec:auxiliary lemmas for the proof of theorem main result}

\subsection{Useful Properties of a Regular Simplex}
\label{appsubsec:proof of proposition simple fact for risks}

We provide a proof of Proposition~\ref{prop:simple fact for risks}:
\begin{proof}[Proof of Proposition~\ref{prop:simple fact for risks}]
Since the first claim is a special case of the second claim, we prove claim (ii).
Let $f_{1}=\sum_{i=1}^{d_{1}}g_{1,i}v_{i},f_{2}=\sum_{i=1}^{d_{1}}g_{2,i}v_{i}\in\mathcal{F}_{0}$.
By Corollary~2 in~\citep{alexander1977width}, the set $\{v_{1},\cdots,v_{d_{1}}\}$ satisfying the definition of vertices in $\Delta^{d}$ (see Section~\ref{subsec:learning models}) is uniquely determined up to rotation.
Hence, by~\citep[Corollary~2.6]{conn2009introduction} it holds that
\begin{align}
\label{eq:same inner product}
\langle v_{i},v_{j}\rangle=-1/d \textup{ for any } i,j\in\{1,\cdots,d_{1}\} \textup{ such that } i\neq j.
\end{align}
Then, we have
\begin{align}
&\|f_{1}(x)-f_{2}(x)\|_{2}^{2}\nonumber\\
&=
\left\|\sum_{i=1}^{d_{1}}(g_{1,i}(x)-g_{2,i}(x))v_{i}\right\|_{2}^{2}\nonumber\\
&=
\sum_{i=1}^{d_{1}}|g_{1,i}(x)-g_{2,i}(x)|^{2}-\frac{1}{d}\sum_{\substack{i,j=1,\cdots,d_{1},\\i\neq j}}(g_{1,i}(x)-g_{2,i}(x))(g_{1,j}(x)-g_{2,j}(x)).\nonumber\\
&=\frac{d_{1}}{d}\sum_{i=1}^{d_{1}}|g_{1,i}(x)-g_{2,i}(x)|^{2}-\frac{1}{d}\left(\sum_{i=1}^{d_{1}}g_{1,i}(x)-\sum_{i=1}^{d_{1}}g_{2,i}(x)\right)^{2}\nonumber\\
\label{eq:proof of prop simple fact for risks eq 1}
&=\frac{d_{1}}{d}\sum_{i=1}^{d_{1}}|g_{1,i}(x)-g_{2,i}(x)|^{2}.
\end{align}
Here, in the second equality we use~\eqref{eq:same inner product}.
In the last inequality we use the fact $\sum_{i=1}^{d_{1}}g_{1,i}(x)=\sum_{i=1}^{d_{1}}g_{2,i}(x)=1$, which is due to the definition of $\mathcal{F}_{0}$.
\end{proof}

Given a subset $\mathcal{A}\subset \mathbb{R}^{s}$, denote by $\textup{conv}(\mathcal{A})$, the convex hull of $\mathcal{A}$.
Also, given a subset $\mathcal{A}$ in the Euclidean space equipped with the standard distance, denote by $\textup{diam}(\mathcal{A})$, the diameter of $\mathcal{A}$.
We recall a well-known fact on the diameter of any convex hull (see, e.g.,~\citep[Lemma~5--17]{hocking1961topology} and~\citep{alexander1977width}):
\begin{lemma}[{Lemma~5--17 in~\citep{hocking1961topology}}]
\label{citelem:diameter of any convex hull}
Let $s,t\in\mathbb{N}$.
For any $z_{1},\cdots,z_{s}\in\mathbb{R}^{t}$, it holds that
\begin{align*}
\textup{diam}(\textup{conv}(\{z_{1},\cdots,z_{s}\}))=\textup{diam}(\{z_{1},\cdots,z_{s}\}).
\end{align*}
\end{lemma}

The following properties of a regular simplex are often used in this paper.
Note that the properties in Lemma~\ref{lem:d proj and d diameter}-(i) and (ii) are well-known, and we give some proofs for completeness.
\begin{lemma}
\label{lem:d proj and d diameter}
We have the following properties:
\begin{enumerate}
\item[(i)] It holds that $D_{\textup{proj}}=\frac{d_{1}}{d}$.
\item[(ii)] It holds that $D_{\Delta^{d}}=\sqrt{\frac{2d_{1}}{d}}$.
\item[(iii)] For any $f\in\mathcal{F}_{0}$ and $x,x'\in\mathcal{X}$, $\psi\circ\rho_{f}(x,x')=1$ if $f(x)=f(x')$.
In addition, $\psi\circ\rho_{f}(x,x')=-1$ if $f(x),f(x')\in\{v_{1},\cdots,v_{d_{1}}\}$ and $f(x)\neq f(x')$.
\end{enumerate}
\end{lemma}
\begin{proof}
By the definition of $D_{\textup{proj}}$, we have
\begin{align}
\label{eq:d proj eq 0}
D_{\textup{proj}}^{2}
&=
\inf_{c_{2},\cdots,c_{d_{1}}\in [0,1], \sum_{i=2}^{d_{1}}c_{i}=1}\frac{d_{1}}{d}+\frac{d_{1}}{d}\sum_{i=2}^{d_{1}}c_{i}^{2}\\
&=
\frac{d_{1}}{d}+\frac{d_{1}}{d^{2}}\nonumber\\
&=
\frac{d_{1}^{2}}{d^{2}},\nonumber
\end{align}
where in~\eqref{eq:d proj eq 0} we used Proposition~\ref{prop:simple fact for risks}--(i).

As in the proof of Proposition~\ref{prop:simple fact for risks}, by Corollary~2 in~\citep{alexander1977width} and Corollary~2.6 in~\citep{conn2009introduction}, it holds that
\begin{align}
\label{eq:d proj eq 0.0001}
\langle v_{1},v_{2}\rangle=-1/d.
\end{align}
We have
\begin{align}
\label{eq:d proj eq 1}
D_{\Delta^{d}}
&=
\|v_{1}-v_{2}\|_{2}\\
&=\sqrt{2-2\langle v_{1},v_{2}\rangle}\nonumber\\
\label{eq:d proj eq 2}
&=\sqrt{\frac{2d_{1}}{d}},
\end{align}
where~\eqref{eq:d proj eq 1} is due to Lemma~\ref{citelem:diameter of any convex hull}, and~\eqref{eq:d proj eq 2} is due to~\eqref{eq:d proj eq 0.0001}.

The first claim in (iii) is due to the definition of $\psi$ and $\rho_{f}$, namely, $\rho_{f}(x,x')=\|f(x)-f(x')\|_{2}^{2}$ and $\psi(s)=1-2D_{\Delta^{d}}^{-2}s$.
For the second claim, if $f(x),f(x')\in \{v_{1},\cdots,v_{d_{1}}\}$ and $f(x)\neq f(x')$, then by Lemma~\ref{citelem:diameter of any convex hull}, the property that $\|v_{i}-v_{j}\|_{2}=\|v_{i'}-v_{j'}\|_{2}$ for every $i,j,i',j'\in\{1,\cdots,d_{1}\}$ such that $i\neq j$ and $i'\neq j'$, and the convexity of the regular simplex $\Delta^{d}$, we have that $\psi\circ\rho_{f}(x,x')=-1$.
\end{proof}

\subsection{Proof of Proposition~\ref{thm:main result basic case}}
\label{subsec:proof of proposition main result basic case modified}

The following condition guarantees that the Bayes classifier $\textup{sign}\circ(2\eta-1)$ is representable by some $f\in\mathcal{F}_{0}$.
\begin{definition}[$(\psi_{0},\mathcal{F})$-representability]
\label{def:psi, F constructable}
Given $\xi=(R,K,d_{1},E,\theta_{\textup{NC}},\theta_{1},\theta_{2},\theta_{3})\in\Xi$, a measurable function $\psi_{0}:\mathbb{R}\to\mathbb{R}$, and $\mathcal{F}\subset \mathcal{F}_{0}$, a Borel probability measure $P\in\mathcal{P}_{\xi}$ is said as $(\psi_{0},\mathcal{F})$\textit{-representable} if there is a vector-valued function $f\in\mathcal{F}$ such that $\psi_{0}\circ \rho_{f}=\textup{sign}\circ (2\eta-1)$ holds $P_{X,X'}$-almost surely.
\end{definition}

Recall that the function $\psi(s)=1-2D_{\Delta^{d}}^{-2}s$ is defined in Definition~\ref{def:contrastive loss}.
\begin{lemma}
\label{prop:constructability for F0}
Given any $\xi=(R,K,d_{1},E,\theta_{\textup{NC}},\theta_{1},\theta_{2},\theta_{3})\in\Xi$, any $P\in \mathcal{P}_{\xi}$ is $(\psi,\mathcal{F}_{0})$-representable with the contrastive function $f^{*}$ of $P$.
\end{lemma}
\begin{proof}
By Lemma~\ref{lem:d proj and d diameter}--(iii), we have $\psi\circ\rho_{f^{*}}(x,x')=-1$ for any $i,j\in\{1,\cdots,d_{1}\}$ such that $i\neq j$, any $x\in\mathcal{K}_{i}$, and any $x'\in\mathcal{K}_{j}$.
Note also that $\psi\circ\rho_{f^{*}}(x,x')=1$ for any $i\in\{1,\cdots,d_{1}\}$ and any $x,x'\in\mathcal{K}_{i}$.
Hence, by condition (A4) in Definition~\ref{def:main assumption}, we have
\begin{align}
\label{eq:identity in the representability lemma}
\psi\circ\rho_{f^{*}}(x,x')=\textup{sign}\circ(2\eta-1)(x,x'),
\end{align}
for any $(x,x')\in\mathcal{X}\times\mathcal{X}$.
This fact indicates that $P$ is $(\psi,\mathcal{F}_{0})$-representable.
\end{proof}

We now provide the proof of Proposition~\ref{thm:main result basic case}.
To this end, we recall a well-known fact on hinge loss:
\citet[Lemma~3.1]{lin2002support} shows that a real-valued classifier minimizes expected hinge loss if and only if it is equal to the Bayes classifier.
See also Section~3.3 in~\citep{zhang2004statistical} and Example~4 in~\citep{bartlett2006convexity}.
Then, we can show the existence of the minimizers, namely, Proposition~\ref{thm:main result basic case}.
\begin{proof}[Proof of Proposition~\ref{thm:main result basic case}]
Let $h:\mathcal{X}^{2}\to [-1,1]$ be an arbitrary measurable function.
By Lemma~3.1 in~\citep{lin2002support}, the expected hinge loss $\mathbb{E}_{P}[\max\{0,1-yh(x,x')\}]$ is minimized at the Bayes classifier $h=\textup{sign}\circ (2\eta-1)$.
Note also that Lemma~\ref{prop:constructability for F0} implies that the Bayes classifier is equal to $\psi\circ\rho_{f^{*}}$.
Therefore, the combination of Lemma~3.1 in~\citep{lin2002support} and Lemma~\ref{prop:constructability for F0} implies that $\mathbb{E}_{P}[\ell_{f}]$ is minimized at $f=f^{*}$.
\end{proof}

\subsection{Proof of Lemma~\ref{lem:step 3 lemma 0}}
\label{appsubsec:proof of lemma step 3 lemma 0}

\begin{proof}
Noting the fact that $P_{X}(\mathcal{A})=P_{X,X'}^{-}(\mathcal{A}\times\mathcal{X})$ for any measurable $\mathcal{A}\subseteq\mathcal{X}$, we have
\begin{align}
&\mathbb{E}_{P_{X}}[\|f-f^{*}\|_{2}^{2}]\nonumber\\
&=\int_{0}^{\infty}P_{X}(\{x\in\mathcal{X}\;|\;\|f(x)-f^{*}(x)\|_{2}^{2}> r\})dr\nonumber\\
&=\int_{0}^{\infty}P_{X,X'}^{-}\left(\bigcup_{i,j=1}^{d_{1}}(\{x\in\mathcal{X}\;|\;\|f(x)-f^{*}(x)\|_{2}^{2}> r\}\cap \mathcal{K}_{i})\times \mathcal{K}_{j}\right)dr\nonumber\\
&\leq \int_{0}^{\infty}\sum_{i\neq j}P_{i,j}^{-}(r^{1/2})dr\nonumber\\
&\quad +\int_{0}^{\infty}\sum_{i=1}^{d_{1}}P_{X,X'}^{-}((\{x\in\mathcal{X}\;|\;\|f(x)-f^{*}(x)\|_{2}^{2}>r\}\cap \mathcal{K}_{i})\times \mathcal{K}_{i})dr\nonumber\\
\label{eq:step 3 lemma 0 eq 1}
&\leq 2 \int_{0}^{\infty}\sum_{i\neq j}rP_{i,j}^{-}(r)dr+\max_{i\in\{1,\cdots,d_{1}\}}P_{X'}(\mathcal{K}_{i})\mathbb{E}_{P_{X}}[\|f-f^{*}\|_{2}^{2}],
\end{align}
where the last inequality is obtained as follows: we note that
\begin{align*}
&\sum_{i=1}^{d_{1}}P_{X,X'}^{-}((\{x\in\mathcal{X}\;|\;\|f(x)-f^{*}(x)\|_{2}^{2}>r\}\cap \mathcal{K}_{i})\times \mathcal{K}_{i})\\
&=\sum_{i=1}^{d_{1}}P_{X'}(\mathcal{K}_{i})P_{X}(\{x\in\mathcal{X}\;|\;\|f(x)-f^{*}(x)\|_{2}^{2}>r\}\cap \mathcal{K}_{i})\\
&\leq \sum_{i=1}^{d_{1}}\max_{j\in\{1,\cdots,d_{1}\}}P_{X'}(\mathcal{K}_{j})P_{X}(\{x\in\mathcal{X}\;|\;\|f(x)-f^{*}(x)\|_{2}^{2}>r\}\cap \mathcal{K}_{i})\\
&=\max_{j\in\{1,\cdots,d_{1}\}}P_{X'}(\mathcal{K}_{j})P_{X}(\{x\in\mathcal{X}\;|\;\|f(x)-f^{*}(x)\|_{2}^{2}>r\}),
\end{align*}
where in the last equality we note that $\mathcal{K}_{1},\cdots,\mathcal{K}_{d_{1}}$ are disjoint.
Integrating both the sides in the above calculation yields~\eqref{eq:step 3 lemma 0 eq 1}.
Here, we note that by condition (A4) in Definition~\ref{def:main assumption}, we have
\begin{align*}
\max_{j\in\{1,\cdots,d_{1}\}}P_{X}(\mathcal{K}_{j})=\max_{j\in\{1,\cdots,d_{1}\}}P_{X'}(\mathcal{K}_{j})\leq\theta_{3}<1,
\end{align*}
where in the equality we use the assumption that $q(x,x')$ is symmetric, which is introduced in condition (A3) of Definition~\ref{def:main assumption} and implies that $p_{X}=p_{X'}$.
Combining this fact and \eqref{eq:step 3 lemma 0 eq 1}, we obtain
\begin{align}
\label{eq:step 3 lemma 0 eq 2}
\mathbb{E}_{P_{X}}[\|f-f^{*}\|_{2}^{2}]\leq 2(1-\theta_{3})^{-1}\int_{0}^{\infty}\sum_{i\neq j}rP_{i,j}^{-}(r)dr.
\end{align}
Let us set $C_{2}'=2(1-\theta_{3})^{-1}$.
We next show an upper bound of the integral in the right-hand side of~\eqref{eq:step 3 lemma 0 eq 2}.
We can proceed as follows:
\begin{align}
\label{eq:step 3 lemma 0 eq 3}
&\int_{0}^{\infty}r P_{i,j}^{-}(r)dr\nonumber\\
&=\int_{0}^{\frac{\beta}{n}}rP_{i,j}^{-}(r)dr+\sum_{w=0}^{\infty}\int_{(\frac{\beta}{n})\vee((\frac{1}{2})^{w+1}\beta)}^{(\frac{\beta}{n})\vee((\frac{1}{2})^{w}\beta)}rP_{i,j}^{-}(r)dr+\int_{\beta}^{D_{\Delta^{d}}}rP_{i,j}^{-}(r)dr.
\end{align}
By Lemma~\ref{lem:step 3 lemma 2}, the third term in the right-hand side of~\eqref{eq:step 3 lemma 0 eq 3} is evaluated as
\begin{align*}
\int_{\beta}^{D_{\Delta^{d}}}rP_{i,j}^{-}(r)dr\leq D_{\Delta^{d}}(D_{\Delta^{d}}-\beta)\beta_{0}.
\end{align*}
For the first term in the right-hand side of~\eqref{eq:step 3 lemma 0 eq 3}, we have
\begin{align*}
\int_{0}^{\frac{\beta}{n}}rP_{i,j}^{-}(r)dr\leq \frac{\beta^{2}}{n}.
\end{align*}
Regarding the second term, we note the bound
\begin{align}
\label{eq:step 3 lemma 0 eq 5}
&\sum_{w=0}^{\infty}\int_{(\frac{\beta}{n})\vee((\frac{1}{2})^{w+1}\beta)}^{(\frac{\beta}{n})\vee((\frac{1}{2})^{w}\beta)}rP_{i,j}^{-}(r)dr\nonumber\\
&= \int_{\frac{\beta}{n}}^{(\frac{1}{2})^{\lfloor\log_{2}{n}\rfloor}\beta} rP_{i,j}^{-}(r)dr + \sum_{w=0}^{\lfloor\log_{2}{n}\rfloor-1}\int_{(\frac{1}{2})^{w+1}\beta}^{(\frac{1}{2})^{w}\beta} rP_{i,j}^{-}(r)dr \nonumber\\
&\leq \sum_{w=0}^{\lfloor \log_{2}{n}\rfloor} \int_{(\frac{1}{2})^{w+1}\beta}^{(\frac{1}{2})^{w}\beta} rP_{i,j}^{-}(r)dr\nonumber\\
&\leq \sum_{w=0}^{\lfloor \log_{2}{n}\rfloor}\left(\frac{1}{2}\right)^{2w+1}\beta^{2}P_{i,j}^{-}(2^{-(w+1)}\beta).
\end{align}
Here, in~\eqref{eq:step 3 lemma 0 eq 5}, we used the fact that $P_{i,j}^{-}(r)$ is a non-increasing function of $r$.
Define $C_{2}=C_{2}'(1\vee (d+1)dD_{\Delta^{d}})$.
Combining~\eqref{eq:step 3 lemma 0 eq 2}, \eqref{eq:step 3 lemma 0 eq 3}, and \eqref{eq:step 3 lemma 0 eq 5}, we obtain
\begin{align*}
&\mathbb{E}_{P_{X}}[\|f-f^{*}\|_{2}^{2}] \\
&\leq C_{2}\left((D_{\Delta^{d}}-\beta)\beta_{0}+\sum_{i\neq j}\sum_{w=0}^{\lfloor \log_{2}{n}\rfloor}\left(\frac{1}{2}\right)^{2w+1}\beta^{2}P_{i,j}^{-}(2^{-(w+1)}\beta)+\frac{\beta^{2}}{n}\right),
\end{align*}
which shows the claim.
\end{proof}

\subsection{Proof of Lemma~\ref{lem:step 3 lemma 1}}
\label{appsubsec:proof of lemma step 3 lemma 1}

To prove Lemma~\ref{lem:step 3 lemma 1}, we use the following lemma:
\begin{lemma}
\label{lem:lemma of lemma computing diameter}
Let $\xi=(R,K,d_{1},E,\theta_{\textup{NC}},\theta_{1},\theta_{2},\theta_{3})\in\Xi$, $\beta\in (0,D_{\textup{proj}})$, $\beta_{0}\geq 0$, $\mathcal{F}\subset \mathcal{F}_{0}$, and $P\in\mathcal{P}_{\xi}$.
Let $f^{*}$ be the contrastive function of $P$, and let $\mathscr{S}_{P}=\{\mathcal{K}_{i}\}_{i=1}^{d_{1}}$.
Given any $i,j\in\{1,\cdots,d_{1}\}$ for which $i\neq j$ is satisfied, any $r\in (0,2^{-1}\beta]$, and any $f\in\mathscr{F}_{\beta,\beta_{0},P}(\mathcal{F})$, suppose that the point $(x,x')\in\mathcal{X}^{2}$ satisfies
\begin{align}
\label{eq:lemma of lemma computing diameter eq 1}
(x,x')\in&(\{x\in\mathcal{X}\;|\; r<\|f(x)-f^{*}(x)\|_{2}\leq \beta\}\cap \mathcal{K}_{i})\nonumber\\
&\times (\mathcal{K}_{j}\cap\{x'\in\mathcal{X}\;|\; \|f(x')-f^{*}(x')\|_{2}<\beta\}).
\end{align}
Then, we have
\begin{align*}
&\|f(x)-f(x')\|_{2}\\
&\leq ((\sqrt{3}D_{\Delta^{d}}/2)^{2}+((D_{\Delta^{d}}/2-r)\vee (\beta D_{\Delta^{d}}/D_{\textup{proj}}-D_{\Delta^{d}}/2))^{2})^{1/2}.
\end{align*}
\end{lemma}

The proof is deferred to Appendix~\ref{appsubsubsec:proof of lemma diameter of convex polytope}.
We now prove Lemma~\ref{lem:step 3 lemma 1}.

\begin{proof}[Proof of Lemma~\ref{lem:step 3 lemma 1}]
Let $(x,x')\in\mathcal{X}^{2}$ be any point such that~\eqref{eq:lemma of lemma computing diameter eq 1} is satisfied.
\begin{itemize}
\item Let us consider the case that $D_{\Delta^{d}}/2-r\geq \beta D_{\Delta^{d}}/D_{{\textup{proj}}}-D_{\Delta^{d}}/2$.
Let $r_{0}=(\sqrt{3}D_{\Delta^{d}}/2)^{2}+(D_{\Delta^{d}}/2-r)^{2}$.
Since $\psi$ is a monotonically decreasing function, \eqref{eq:lemma of lemma computing diameter eq 1} implies that $\psi\circ\rho_{f}(x,x')\geq \psi(r_{0})$.
Here note that $f^{*}(x)=v_{i}$ and $f^{*}(x')=v_{j}$.
Note also that by Lemma~\ref{lem:d proj and d diameter}--(iii), we have $\psi\circ\rho_{f^{*}}(x,x')=-1$.
Thus, it holds that 
\begin{align}
|\psi\circ\rho_{f}(x,x')-\psi\circ\rho_{f^{*}}(x,x')|\geq 2-2\widetilde{C}_{i,j}r_{0}
\end{align}
for every $(x,x')$ satisfying~\eqref{eq:lemma of lemma computing diameter eq 1}, where $\widetilde{C}_{i,j}=D_{\Delta^{d}}^{-2}$.
Here, the inequality $2-2\widetilde{C}_{i,j}r_{0}\geq \widetilde{C}_{i,j}'r^{2}$ is satisfied for every $r\in [0,D_{\Delta^{d}}/2]$, where $\widetilde{C}_{i,j}'=2\widetilde{C}_{i,j}$.
Hence, we have
\begin{align}
\label{eq:error bound lemma eq 6.1}
|\psi\circ\rho_{f}(x,x')-\psi\circ\rho_{f^{*}}(x,x')|\geq \widetilde{C}_{i,j}'r^{2}=\widetilde{C}_{i,j}'(r\wedge (D_{\Delta^{d}}(1-\beta/D_{{\textup{proj}}})))^{2}.
\end{align}
\item On the other hand, in the case that $D_{\Delta^{d}}/2-r<\beta D_{\Delta^{d}}/D_{{\textup{proj}}}-D_{\Delta^{d}}/2$, by the same arguments as~\eqref{eq:error bound lemma eq 6.1}, we have
\begin{align*}
|\psi\circ\rho_{f}(x,x')-\psi\circ\rho_{f^{*}}(x,x')|&\geq \widetilde{C}_{i,j}'(D_{\Delta^{d}}(1-\beta/D_{{\textup{proj}}}))^{2}\\
&=\widetilde{C}_{i,j}'(r\wedge (D_{\Delta^{d}}(1-\beta/D_{{\textup{proj}}})))^{2},
\end{align*}
where we note that $\beta D_{\Delta^{d}}/D_{\textup{proj}}-D_{\Delta^{d}}/2=D_{\Delta^{d}}/2-(D_{\Delta^{d}}-\beta D_{\Delta^{d}}/D_{{\textup{proj}}})$ and $D_{\Delta^{d}}-\beta D_{\Delta^{d}}/D_{{\textup{proj}}} <r \leq D_{\Delta^{d}}/2$ in this case.
\end{itemize}
Therefore, taking a constant $\widetilde{C}_{i,j}'$ to satisfy the above condition, we have
\begin{align}
\label{eq:error bound lemma eq 6.20}
&\quad\{(x,x')\in\mathcal{X}^{2}\;|\; (x,x')\textup{ satisfies \eqref{eq:lemma of lemma computing diameter eq 1}}\}\nonumber \\
&\subseteq \{(x,x')\in\mathcal{X}^{2} | \widetilde{C}_{i,j}'(r\wedge (D_{\Delta^{d}}(1-\frac{\beta}{D_{{\textup{proj}}}})))^{2}\leq |\psi\circ\rho_{f}(x,x')-\psi\circ\rho_{f^{*}}(x,x')|\}.
\end{align}

Therefore, we have
\begin{align}
\label{eq:error bound lemma eq 6.20001}
&\mathbb{E}_{P_{X,X'}^{-}}[\mathds{1}_{(\{x\in\mathcal{X}\;|\; r< \|f(x)-f^{*}(x)\|_{2}\leq\beta\}\cap \mathcal{K}_{i})\times \mathcal{K}_{j}}]\nonumber\\
&\leq\mathbb{E}_{P_{X,X'}^{-}}[\mathds{1}_{(\{x\in\mathcal{X}\;|\; r<\|f(x)-f^{*}(x)\|_{2}\leq \beta\}\cap \mathcal{K}_{i})\times (\mathcal{K}_{j}\cap \{x\in\mathcal{X}\;|\; \|f(x)-f^{*}(x)\|_{2}<\beta\})}]+\beta_{0}\\
&\leq \mathbb{E}_{P_{X,X'}^{-}}[\mathds{1}_{\{(x,x')\in\mathcal{X}^{2}\;|\; \widetilde{C}_{i,j}'(r\wedge (D_{\Delta^{d}}(1-\beta/D_{{\textup{proj}}})))^{2}\leq |\psi\circ\rho_{f}(x,x')-\psi\circ\rho_{f^{*}}(x,x')|\}}]+\beta_{0} \nonumber \\
&\leq \widetilde{C}_{i,j}'^{-1}(r\wedge(D_{\Delta^{d}}(1-\beta/D_{{\textup{proj}}})))^{-2}\mathbb{E}_{P_{X,X'}^{-}}[|\psi\circ\rho_{f}-\psi\circ\rho_{f^{*}}|]+\beta_{0} \nonumber \\
&= C_{i,j}(r\wedge (D_{\Delta^{d}}(1-\beta/D_{{\textup{proj}}})))^{-2}\mathbb{E}_{P_{X,X'}^{-}}[|\psi\circ\rho_{f}-\psi\circ\rho_{f^{*}}|]+\beta_{0}, \nonumber
\end{align}
where the first inequality is due to Lemma~\ref{lem:step 3 lemma 2} and the assumption that $p_{X}=p_{X'}$ (see condition (A3) in Definition~\ref{def:main assumption}), the second inequality is due to~\eqref{eq:error bound lemma eq 6.20}, and in the third inequality we used Markov's inequality.
In the last equality, we set $C_{i,j}=\widetilde{C}_{i,j}'^{-1}$.
Thus, we obtain the claim.
\end{proof}

\subsubsection{Proof of Lemma~\ref{lem:lemma of lemma computing diameter}}
\label{appsubsubsec:proof of lemma diameter of convex polytope}

\begin{proof}[Proof of Lemma~\ref{lem:lemma of lemma computing diameter}]
Note that $i,j,r$, and $\beta$ are fixed in this proof.
Define the sets
\begin{align*}
B_{i}&=\{z\in \Delta^{d} \;|\; r\leq \|z-v_{i}\|_{2}\leq \beta\},\\
B_{j}&=\{z\in\Delta^{d} \;|\; \|z-v_{j}\|_{2}\leq \beta\}.
\end{align*}
Here, note that by the definition of $f^{*}$, $f^{*}(x)=v_{i}$ if $x\in\mathcal{K}_{i}$, and $f^{*}(x)=v_{j}$ if $x\in\mathcal{K}_{j}$.
Hence, we notice that
\begin{align}
\label{eq:lemma of lemma computing diameter eq 0}
f(x)\in B_{i}\quad\textup{and}\quad f(x')\in B_{j}.
\end{align}
Thus, to show the claim, it suffices to evaluate the diameter of the set $B_{i}\cup B_{j}$.

To this end, we first construct a convex polytope that includes both $B_{i}$ and $B_{j}$.
Then, we evaluate the diameter of the convex polytope to conclude the proof.
Note that the diameter of the given set is denoted by $\textup{diam}(\cdot)$.
We divide the proof in several steps\footnote{The visualization of Figures~\ref{fig: lemma b.5 figure 1} -- \ref{fig: lemma b.5 figure 5} are performed using \texttt{NumPy}~\citep{harris2020array} and \texttt{Matplotlib}~\citep{hunter2007matplotlib} on TSUBAME 4.0 of Institute of Science Tokyo.}.
    \paragraph{Step 1.}
    In this step, we consider $B_{j}$.
    We show the following claim:
    \begin{claim}
    \label{claim:lemma of lemma computing diameter subclaim 1}
    We have
    \begin{align*}
    B_{j}\subset \left\{z\in\Delta^{d}\;\bigg|\; z=\sum_{h=1}^{d_{1}}c_{h}v_{h},\; c_{j}\geq 1-D_{\textup{proj}}^{-1}\beta\right\}.
    \end{align*}
    \end{claim}
    \begin{proof}[Proof of Claim~\ref{claim:lemma of lemma computing diameter subclaim 1}]
    Let $z=\sum_{h=1}^{d_{1}}c_{h}v_{h}\in B_{j}$.
    Note that
    \begin{align}
    \label{eq:subclaim 1 eq 1}
    \left\|\sum_{h=1}^{d_{1}}c_{h}v_{h}-v_{j}\right\|_{2}^{2}
    &=
    \frac{d_{1}}{d}(1-c_{j})^{2}+\frac{d_{1}}{d}\sum_{\substack{h=1,\cdots,d_{1},\\ h\neq j}}c_{h}^{2}\\
    \label{eq:subclaim 1 eq 2}
    &\geq
    \frac{d_{1}}{d}(1-c_{j})^{2}+\frac{d_{1}}{d^{2}}(1-c_{j})^{2}\\
    \label{eq:subclaim 1 eq 2.001}
    &=
    \frac{d_{1}^{2}}{d^{2}}(1-c_{j})^{2},
    \end{align}
    where Proposition~\ref{prop:simple fact for risks}--(i) is used in~\eqref{eq:subclaim 1 eq 1}, and the Cauchy-Schwarz inequality applies in~\eqref{eq:subclaim 1 eq 2}.
    By~\eqref{eq:subclaim 1 eq 2.001}, the condition $\|z-v_{j}\|_{2}\leq \beta$ implies
    \begin{align}
    \label{eq:subclaim 1 eq 3}
    \frac{d_{1}^{2}}{d^{2}}(1-c_{j})^{2}\leq \beta^{2}
    \iff 
    |1-c_{j}|\leq \frac{d}{d_{1}}\beta
    \iff
    c_{j} \geq 1-\frac{d}{d_{1}}\beta,
    \end{align}
    where the condition $c_{j}\in [0,1]$ is used in the second relationship.
    Note that $d/d_{1}=D_{\textup{proj}}^{-1}$ by Lemma~\ref{lem:d proj and d diameter}--(i).
    Thus, the relationships in~\eqref{eq:subclaim 1 eq 3} show the claim.
    \end{proof}
    \begin{figure}
        \centering
        \includegraphics[width=0.95\linewidth]{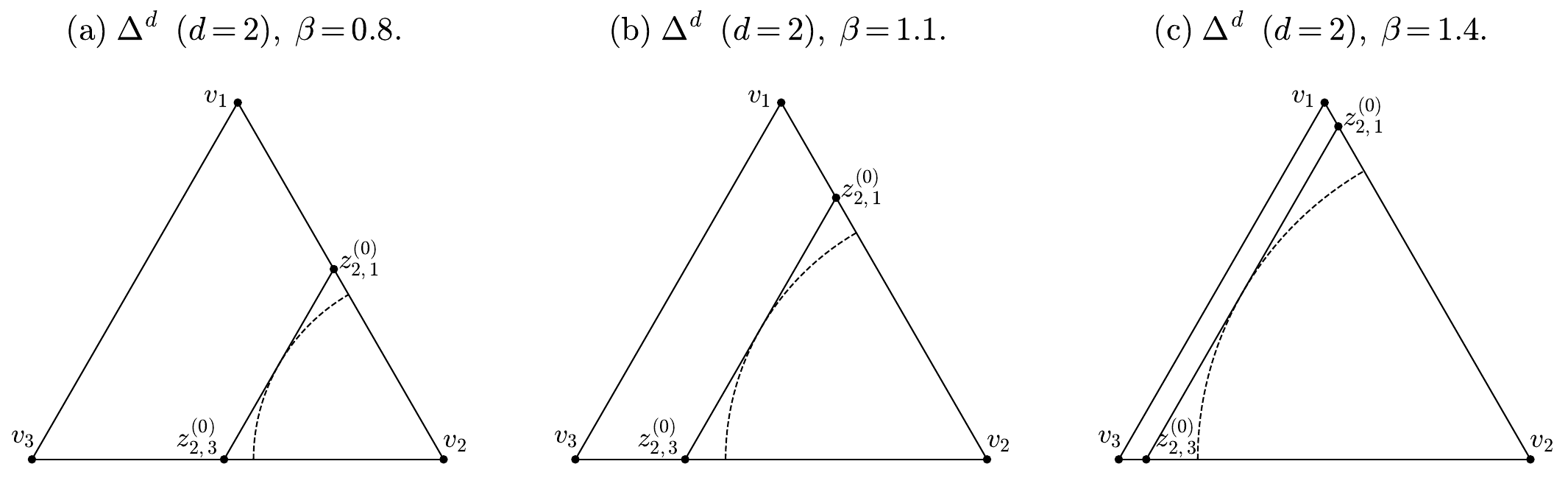}
        \caption{Examples of $z_{j,h}^{(0)}$, where we set $d=2$, $i=1$, and $j=2$. In each panel, the dashed curve shows the set $\{z\in\Delta^{d}\;|\; \|z-v_{j}\|_{2}=\beta\}$.}
        \label{fig: lemma b.5 figure 1}
    \end{figure}
    In addition, we notice the following fact, where $\textup{conv}(\cdot)$ denotes the convex hull of the given set (see Figure~\ref{fig: lemma b.5 figure 1}).
    \begin{claim}
    \label{claim:lemma of lemma computing diameter subclaim 2}
    For every $h\in \{1,\cdots,d_{1}\}\setminus \{j\}$, let $z_{j,h}^{(0)}=\sum_{k=1}^{d_{1}}c_{j,h,k}^{(0)}v_{k}\in \Delta^{d}$ be the point such that
    \begin{align*}
    c_{j,h,k}^{(0)}=
    \begin{cases}
    1-D_{\textup{proj}}^{-1}\beta &\quad \textup{if }k=j,\\
    D_{\textup{proj}}^{-1}\beta &\quad \textup{if }k=h,\\
    0&\quad \textup{otherwise}.
    \end{cases}
    \end{align*}
    Then, we have
    \begin{align*}
    \left\{z\in \Delta^{d}\;\bigg|\; z=\sum_{h=1}^{d_{1}}c_{h}v_{h},\; c_{j}\geq 1-D_{\textup{proj}}^{-1}\beta\right\}\subset \textup{conv}(\{v_{j}\}\cup \bigcup_{\substack{h=1,\cdots,d_{1},\\ h\neq j}}\{z_{j,h}^{(0)}\}).
    \end{align*}
    \end{claim}
    \begin{proof}[Proof of Claim~\ref{claim:lemma of lemma computing diameter subclaim 2}]
    Let
    \begin{align*}
    z=\sum_{h=1}^{d_{1}}c_{h}v_{h}\in \left\{z\in\Delta^{d}\;|\; z=\sum_{h=1}^{d_{1}}c_{h}v_{h},\; c_{j}\geq 1-D_{\textup{proj}}^{-1}\beta\right\}.
    \end{align*}
    Given an arbitrary $h_{0}\in\{1,\cdots,d_{1}\}\setminus \{j\}$, we can decompose $z$ as
    \begin{align*}
    z&=(c_{j}-(1-c_{j})\frac{c_{j,h_{0},j}^{(0)}}{1-c_{j,h_{0},j}^{(0)}})v_{j}+\sum_{\substack{h=1,\cdots,d_{1},\\ h\neq j}}\frac{c_{h}}{1-c_{j,h,j}^{(0)}}(c_{j,h,j}^{(0)}v_{j}+(1-c_{j,h,j}^{(0)})v_{h})\\
    &=\frac{c_{j}-c_{j,h_{0},j}^{(0)}}{1-c_{j,h_{0},j}^{(0)}}v_{j}+\sum_{\substack{h=1,\cdots,d_{1},\\ h\neq j}}\frac{c_{h}}{1-c_{j,h,j}^{(0)}}z_{j,h}^{(0)},
    \end{align*}
    where note that $c_{j,h,j}^{(0)}=1-D_{\textup{proj}}^{-1}\beta$ and $c_{j,h,j}^{(0)}/(1-c_{j,h,j}^{(0)})=(1-D_{\textup{proj}}^{-1}\beta)/(D_{\textup{proj}}^{-1}\beta)$ for any $h\in\{1,\cdots,d_{1}\}\setminus \{j\}$.
    Since $c_{j}\geq 1-D_{\textup{proj}}^{-1}\beta$ by the definition of $z$, we have
    \begin{align*}
    0\leq \frac{c_{j}-c_{j,h_{0},j}^{(0)}}{1-c_{j,h_{0},j}^{(0)}}\leq \frac{1-c_{j,h_{0},j}^{(0)}}{1-c_{j,h_{0},j}^{(0)}}= 1,
    \end{align*}
    and
    \begin{align*}
    0\leq \frac{c_{h}}{1-c_{j,h,j}^{(0)}}\leq \frac{D_{\textup{proj}}^{-1}\beta}{1-(1-D_{\textup{proj}}^{-1}\beta)}=1,
    \end{align*}
    which implies
    \begin{align*}
    z\in \textup{conv}(\{v_{j}\}\cup \bigcup_{\substack{h=1,\cdots,d_{1},\\ h\neq j}}\{z_{j,h}^{(0)}\}).
    \end{align*}
    Thus, we obtain the claim.
    \end{proof}
    By Claim~\ref{claim:lemma of lemma computing diameter subclaim 1} and Claim~\ref{claim:lemma of lemma computing diameter subclaim 2}, we obtain
    \begin{align}
    \label{eq:lemma of lemma computing diameter eq 2}
    B_{j}\subset \textup{conv}(\{v_{j}\}\cup\bigcup_{\substack{h=1,\cdots,d_{1},\\ h\neq j}}\{z_{j,h}^{(0)}\}).
    \end{align}
    \paragraph{Step 2.}
    We next focus on the set $B_{i}$.
    \begin{claim}
    \label{claim:lemma of lemma computing diameter subclaim 3}
    We have
    \begin{align*}
        B_{i}\subset\left\{z\in\Delta^{d}\;\bigg|\; z=\sum_{h=1}^{d_{1}}c_{h}v_{h},\; 1-D_{\textup{proj}}^{-1}\beta\leq c_{i}\leq 1-D_{\Delta^{d}}^{-1}r \right\}.
    \end{align*}
    \end{claim}
    \begin{proof}[Proof of Claim~\ref{claim:lemma of lemma computing diameter subclaim 3}]
    Let $z=\sum_{h=1}^{d_{1}}c_{h}v_{h}\in B_{i}$.
    We have
    \begin{align}
    \label{eq:subclaim 3 eq 1}
    \left\|\sum_{h=1}^{d_{1}}c_{h}v_{h}-v_{i}\right\|_{2}^{2}
    &=
    \frac{d_{1}}{d}(1-c_{i})^{2}+\frac{d_{1}}{d}\sum_{\substack{h=1,\cdots,d_{1},\\ h\neq i}}c_{h}^{2}\\
    \label{eq:subclaim 3 eq 2}
    &\leq
    \frac{d_{1}}{d}(1-c_{i})^{2}+\frac{d_{1}}{d}(1-c_{i})^{2}\\
    \label{eq:subclaim 3 eq 3}
    &=
    \frac{2d_{1}}{d}(1-c_{i})^{2},
    \end{align}
    where Proposition~\ref{prop:simple fact for risks}--(i) applies in~\eqref{eq:subclaim 3 eq 1}, and the Cauchy-Schwarz inequality is used in~\eqref{eq:subclaim 3 eq 2}.
    Hence, the condition $\|z-v_{i}\|_{2}\geq r$ implies
    \begin{align}
    \label{eq:subclaim 3 eq 4}
    \frac{2d_{1}}{d}(1-c_{i})^{2}\geq r^{2}
    \iff
    c_{i}\leq 1-\sqrt{\frac{d}{2d_{1}}}r,
    \end{align}
    where~\eqref{eq:subclaim 3 eq 3} and the constraint $c_{i}\in [0,1]$ are used.
    Also, combining~\eqref{eq:subclaim 3 eq 1} with the same arguments as in~\eqref{eq:subclaim 1 eq 2} -- \eqref{eq:subclaim 1 eq 3}, we obtain
    \begin{align}
    \label{eq:subclaim 3 eq 5}
    c_{i}\geq 1-\frac{d}{d_{1}}\beta.
    \end{align}
    Note that $\sqrt{d/(2d_{1})}=D_{\Delta^{d}}^{-1}$ and $d/d_{1}=D_{\textup{proj}}^{-1}$ by Lemma~\ref{lem:d proj and d diameter}--(ii) and (i).
    Thus, combining~\eqref{eq:subclaim 3 eq 4} and \eqref{eq:subclaim 3 eq 5}, we obtain the assertion.
    \end{proof}
    \begin{figure}
        \centering
        \includegraphics[width=0.95\linewidth]{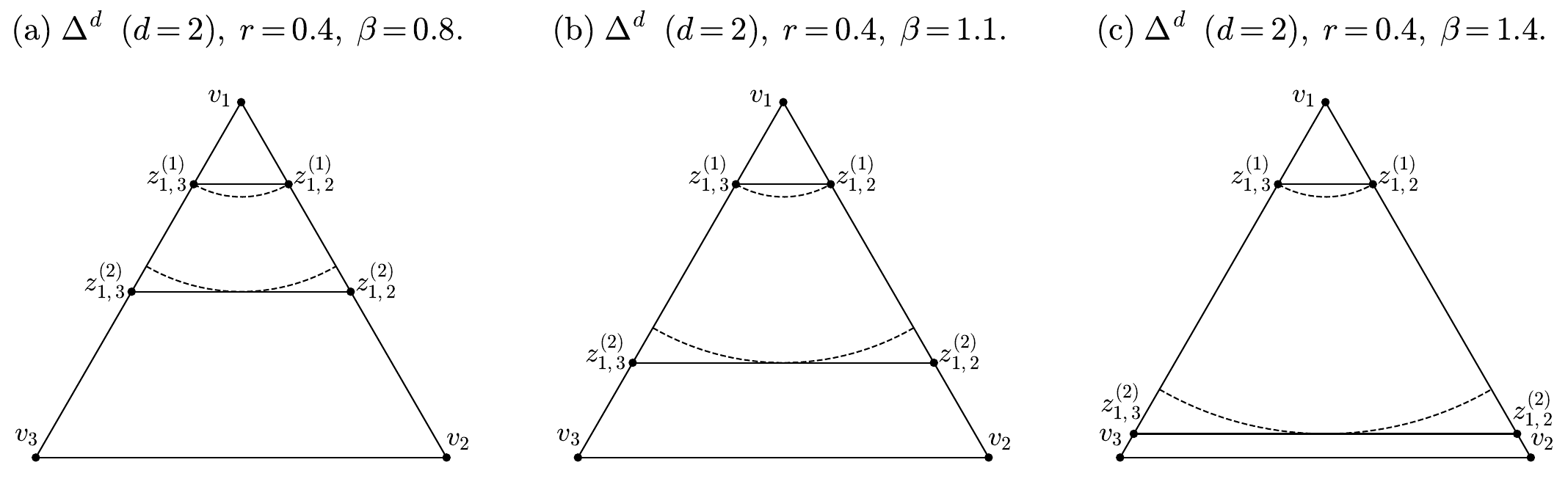}
        \caption{Examples of the points $z_{i,h}^{(1)}$ and $z_{i,h}^{(2)}$, where $d=2$, $i=1$, and $j=2$. The dashed curves in each panel show the subsets $\{z\in\Delta^{d}\;|\; \|z-v_{i}\|_{2}=r\}$ and $\{z\in\Delta^{d}\;|\; \|z-v_{i}\|_{2}=\beta\}$.}
        \label{fig: lemma b.5 figure 2}
    \end{figure}
    We also have the following claim (see Figure~\ref{fig: lemma b.5 figure 2}):
    \begin{claim}
    \label{claim:lemma of lemma computing diameter subclaim 4}
    For every $h\in \{1,\cdots,d_{1}\}\setminus \{i\}$, define $z_{i,h}^{(1)}=\sum_{k=1}^{d_{1}}c_{i,h,k}^{(1)}v_{k}\in \Delta^{d}$ and $z_{i,h}^{(2)}=\sum_{k=1}^{d_{1}}c_{i,h,k}^{(2)}v_{k}\in \Delta^{d}$ so that
    \begin{align*}
    c_{i,h,k}^{(1)}&=
    \begin{cases}
    1-D_{\Delta^{d}}^{-1}r &\quad \textup{if }k=i,\\
    D_{\Delta^{d}}^{-1}r &\quad \textup{if }k=h,\\
    0 &\quad \textup{otherwise},
    \end{cases}
    \\
    c_{i,h,k}^{(2)}&=
    \begin{cases}
    1-D_{\textup{proj}}^{-1}\beta &\quad \textup{if }k=i,\\
    D_{\textup{proj}}^{-1}\beta &\quad \textup{if }k=h,\\
    0 & \quad \textup{otherwise}.
    \end{cases}
    \end{align*}
    Then, we have
    \begin{align*}
    &\left\{z\in\Delta^{d}\;\bigg|\; z=\sum_{h=1}^{d_{1}}c_{h}v_{h},\; 1-D_{\textup{proj}}^{-1}\beta \leq c_{i}\leq 1-D_{\Delta^{d}}^{-1}r\right\}\\
    &\subset \textup{conv}(\bigcup_{\substack{h=1,\cdots,d_{1},\\ h\neq i}}\{z_{i,h}^{(1)},z_{i,h}^{(2)}\}).
    \end{align*}
    \end{claim}
    \begin{proof}[Proof of Claim~\ref{claim:lemma of lemma computing diameter subclaim 4}]
    Let
    \begin{align*}
    z=\sum_{h=1}^{d_{1}}c_{h}v_{h}\in \left\{z\in\Delta^{d}\;\bigg |\; z=\sum_{h=1}^{d_{1}}c_{h}v_{h},\; 1-D_{\textup{proj}}^{-1}\beta\leq c_{i}\leq 1-D_{\Delta^{d}}^{-1}r\right\}.
    \end{align*}
    Given an arbitrary $h\in\{1,\cdots,d_{1}\}\setminus\{i\}$, let $\lambda_{i}\in\mathbb{R}$ be the solution of the equation
    \begin{align}
    \label{eq:subclaim 4 eq 1}
    \lambda_{i}z_{i,h}^{(1)}+(1-\lambda_{i})z_{i,h}^{(2)}=c_{i}v_{i}+\{\lambda_{i} c_{i,h,h}^{(1)}+(1-\lambda_{i})c_{i,h,h}^{(2)}\}v_{h}.
    \end{align}
    Note that $\langle v_{h},v_{i}\rangle =-1/d$ by Corollary~2 in~\citep{alexander1977width} and Corollary~2.6 in~\citep{conn2009introduction}.
    Using this property, we can solve~\eqref{eq:subclaim 4 eq 1} by calculating the inner product with $v_{i}$, and consequently we have
    \begin{align}
    \label{eq:subclaim 4 eq 2}
    \lambda_{i}=\frac{c_{i}-c_{i,h,i}^{(2)}}{c_{i,h,i}^{(1)}-c_{i,h,i}^{(2)}}.
    \end{align}
    Note that $\lambda_{i}$ does not depend on the choice of $h\in\{1,\cdots,d_{1}\}\setminus\{i\}$, by the definitions of $c_{i,h,i}^{(1)}$ and $c_{i,h,i}^{(2)}$.
    We next consider the equation of $\lambda_{h}\in\mathbb{R}$ for every $h\in\{1,\cdots,d_{1}\}\setminus \{i\}$,
    \begin{align}
    \label{eq:subclaim 4 eq 3}
    \lambda_{h}\{\lambda_{i}c_{i,h,h}^{(1)}+(1-\lambda_{i})c_{i,h,h}^{(2)}\}=c_{h}.
    \end{align}
    Solving~\eqref{eq:subclaim 4 eq 3}, we have
    \begin{align}
    \label{eq:subclaim 4 eq 4}
    \lambda_{h}=\frac{c_{h}}{-c_{i}+c_{i,h,i}^{(2)}+c_{i,h,h}^{(2)}}=\frac{c_{h}}{1-c_{i}},
    \end{align}
    where we used the identity $c_{i,h,i}^{(1)}-c_{i,h,i}^{(2)}=-c_{i,h,h}^{(1)}+c_{i,h,h}^{(2)}$.
    Using~\eqref{eq:subclaim 4 eq 2} and \eqref{eq:subclaim 4 eq 4}, define $\widetilde{z}$ as
    \begin{align}
    \label{eq:subclaim 4 eq 5}
    \widetilde{z}=\sum_{\substack{h=1,\cdots,d_{1},\\ h\neq i}}\{\lambda_{i}\lambda_{h}z_{i,h}^{(1)}+(1-\lambda_{i})\lambda_{h}z_{i,h}^{(2)}\}.
    \end{align}
    By the definitions of $z_{i,h}^{(1)}$ and $z_{i,h}^{(2)}$ and equations \eqref{eq:subclaim 4 eq 1} and \eqref{eq:subclaim 4 eq 3}, we have $z=\widetilde{z}$.
    Furthermore, since $c_{i,h,i}^{(2)}\leq c_{i}\leq c_{i,h,i}^{(1)}$ by the definition of $z$, we have
    \begin{align*}
    0\leq \lambda_{i}\lambda_{h}
    &=\frac{c_{h}}{1-c_{i}}\frac{c_{i}-c_{i,h,i}^{(2)}}{c_{i,h,i}^{(1)}-c_{i,h,i}^{(2)}}\\
    &\leq \frac{c_{h}}{1-c_{i}}\frac{c_{i,h,i}^{(1)}-c_{i,h,i}^{(2)}}{c_{i,h,i}^{(1)}-c_{i,h,i}^{(2)}}=\frac{c_{h}}{1-c_{i}}=\frac{c_{h}}{\sum_{\substack{k=1,\cdots,d_{1},\\ k\neq i}}c_{k}}\leq 1.
    \end{align*}
    Similarly, we also have $0\leq (1-\lambda_{i})\lambda_{h}\leq 1$.
    Therefore, we obtain
    \begin{align*}
    z\in \textup{conv}(\bigcup_{\substack{h=1,\cdots,d_{1},\\ h\neq i}}\{z_{i,h}^{(1)},z_{i,h}^{(2)}\}).
    \end{align*}
    This shows the claim.
    \end{proof}
    We consider the following claim (see Figure~\ref{fig: lemma b.5 figure 3}):
    \begin{claim}
    \label{claim:lemma of lemma computing diameter subclaim 5}
    For every $h\in\{1,\cdots,d_{1}\}\setminus \{i\}$, let $z_{i,h}^{(3)}=\sum_{k=1}^{d_{1}}c_{i,h,k}^{(3)}v_{h}\in\Delta^{d}$ be the point such that
    \begin{align*}
    c_{i,h,k}^{(3)}=
    \begin{cases}
    c_{i,h,i}^{(1)} &\quad \textup{if }k=i\textup{ and }r\leq D_{\Delta^{d}}(1-D_{\textup{proj}}^{-1}\beta), \\
    c_{i,h,h}^{(1)} &\quad \textup{if }k=h\textup{ and }r\leq D_{\Delta^{d}}(1-D_{\textup{proj}}^{-1}\beta), \\
    D_{\textup{proj}}^{-1}\beta &\quad \textup{if }k=i\textup{ and }r>D_{\Delta^{d}}(1-D_{\textup{proj}}^{-1}\beta), \\
    1-D_{\textup{proj}}^{-1}\beta &\quad \textup{if }k=h\textup{ and }r>D_{\Delta^{d}}(1-D_{\textup{proj}}^{-1}\beta), \\
    0 &\quad \textup{otherwise}.
    \end{cases}
    \end{align*}
    Then, we have
    \begin{align*}
    \textup{conv}(\bigcup_{\substack{h=1,\cdots,d_{1},\\ h\neq i}}\{z_{i,h}^{(1)},z_{i,h}^{(2)}\})\subset \textup{conv}(\bigcup_{\substack{h=1,\cdots,d_{1},\\ h\neq i}}\{z_{i,h}^{(2)},z_{i,h}^{(3)}\}).
    \end{align*}
    \end{claim}

    \begin{figure}
        \centering
        \includegraphics[width=0.95\linewidth]{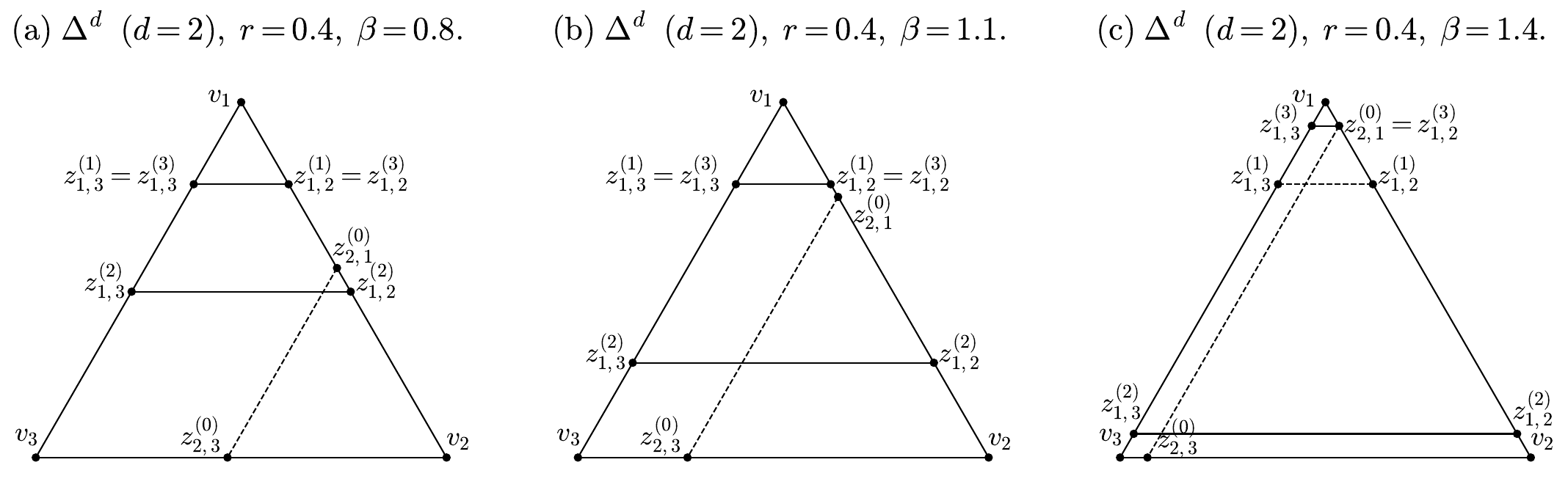}
        \caption{Examples of the points $z_{i,h}^{(3)}$, where $d=2$, $i=1$, and $j=2$.}
        \label{fig: lemma b.5 figure 3}
    \end{figure}
    
    \begin{proof}[Proof of Claim~\ref{claim:lemma of lemma computing diameter subclaim 5}]
    When $r\leq D_{\Delta^{d}}(1-D_{\textup{proj}}^{-1}\beta)$, the claim is straightforward from the definition of $c_{i,h,k}^{(3)}$.
    When $\beta\leq D_{\textup{proj}}/2$, we have $r\leq D_{\Delta^{d}}/2\leq D_{\Delta^{d}}(1-D_{\textup{proj}}^{-1}\beta)$.

    Suppose that $r>D_{\Delta^{d}}(1-D_{\textup{proj}}^{-1}\beta)$. Then, we have $\beta>D_{\textup{proj}}/2$.
    Let
    \begin{align*}
    z=\sum_{\substack{h=1,\cdots,d_{1},\\ h\neq i}}(\lambda_{1,h}z_{i,h}^{(1)}+\lambda_{2,h}z_{i,h}^{(2)})\in \textup{conv}(\bigcup_{\substack{h=1,\cdots,d_{1},\\ h\neq i}}\{z_{i,h}^{(1)},z_{i,h}^{(2)}\}).
    \end{align*}
    Note that $c_{i,h,i}^{(1)}\leq c_{i,h,i}^{(3)}$ by the definition of $c_{i,h,i}^{(3)}$, and $c_{i,h,i}^{(2)}=1-D_{\textup{proj}}^{-1}\beta<\frac{1}{2}\leq 1-D_{\Delta^{d}}^{-1}r$ since $r\leq \beta/2\leq D_{\textup{proj}}/2\leq D_{\Delta^{d}}/2$.
    Thus, we can decompose $z_{i,h}^{(1)}$ as
    \begin{align}
    \label{eq:subclaim 5 eq 1}
    z_{i,h}^{(1)}= \frac{c_{i,h,i}^{(1)}-c_{i,h,i}^{(2)}}{c_{i,h,i}^{(3)}-c_{i,h,i}^{(2)}}z_{i,h}^{(3)}+\frac{c_{i,h,i}^{(3)}-c_{i,h,i}^{(1)}}{c_{i,h,i}^{(3)}-c_{i,h,i}^{(2)}}z_{i,h}^{(2)}.
    \end{align}
    Hence, we have
    \begin{align*}
    z=
    \sum_{\substack{h=1,\cdots,d_{1},\\ h\neq i}}\left\{\left(\lambda_{1,h}\frac{c_{i,h,i}^{(3)}-c_{i,h,i}^{(1)}}{c_{i,h,i}^{(3)}-c_{i,h,i}^{(2)}}+\lambda_{2,h}\right)z_{i,h}^{(2)}+\lambda_{1,h}\frac{c_{i,h,i}^{(1)}-c_{i,h,i}^{(2)}}{c_{i,h,i}^{(3)}-c_{i,h,i}^{(2)}}z_{i,h}^{(3)}\right\}.
    \end{align*}
    Thus, we have $z\in \textup{conv}(\bigcup_{h=1,\cdots,d_{1},\; h\neq i}\{z_{i,h}^{(2)},z_{i,h}^{(3)}\})$, which shows the claim.
    \end{proof}
    By Claim~\ref{claim:lemma of lemma computing diameter subclaim 3}, Claim~\ref{claim:lemma of lemma computing diameter subclaim 4}, and Claim~\ref{claim:lemma of lemma computing diameter subclaim 5}, we have
    \begin{align}
    \label{eq:lemma of lemma computing diameter eq 3}
    B_{i}\subset \textup{conv}(\bigcup_{\substack{h=1,\cdots,d_{1},\\ h\neq i}}\{z_{i,h}^{(2)},z_{i,h}^{(3)}\}).
    \end{align}
    \paragraph{Step 3.}
    We need the following claim to proceed (see Figure~\ref{fig: lemma b.5 figure 4}):
    \begin{claim}
    \label{claim:lemma of lemma computing diameter subclaim 6}
    We have
    \begin{align*}
    &\textup{conv}(\{v_{j}\}\cup \bigcup_{\substack{h=1,\cdots,d_{1},\\ h\neq j}}\{z_{j,h}^{(0)}\})\cup \textup{conv}(\bigcup_{\substack{h=1,\cdots,d_{1},\\ h\neq i}}\{z_{i,h}^{(2)},z_{i,h}^{(3)}\})\\
    &\subset
    \textup{conv}(\{v_{j}\}\cup \bigcup_{\substack{h=1,\cdots,d_{1},\\ h\neq i,j}}\{z_{j,h}^{(0)}\}\cup \bigcup_{\substack{h=1,\cdots,d_{1},\\ h\neq i,j}}\{z_{i,h}^{(2)},z_{i,h}^{(3)}\}\cup \{z_{i,j}^{(3)}\}).
    \end{align*}
    \end{claim}
    \begin{figure}
        \centering
        \includegraphics[width=0.95\linewidth]{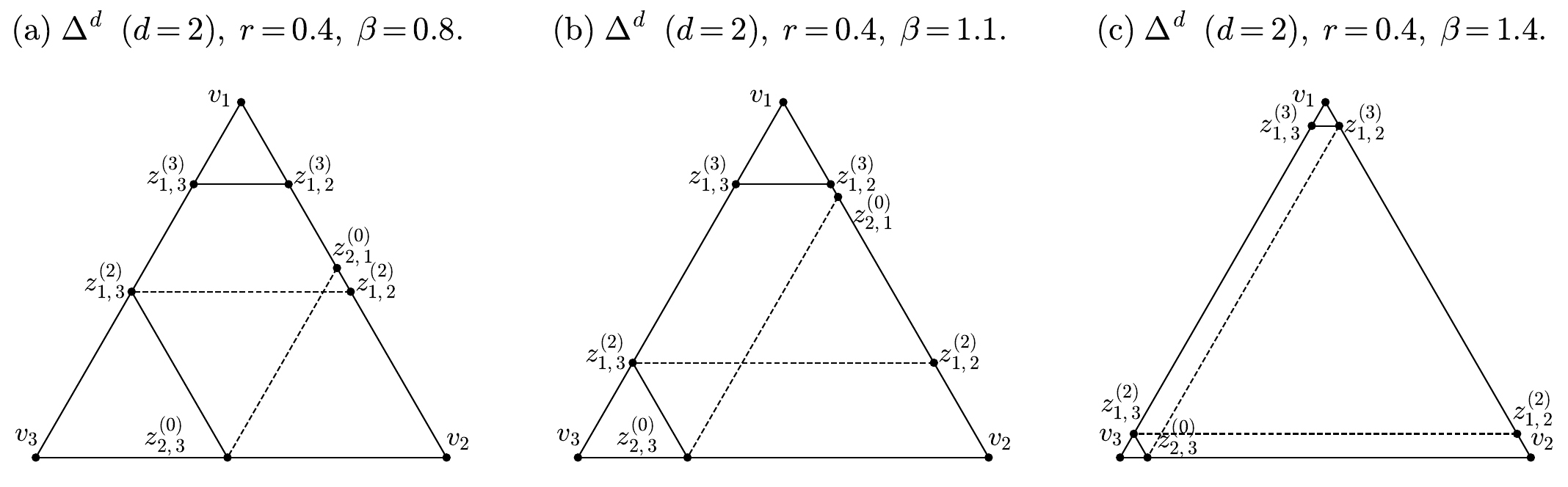}
        \caption{An illustration of Claim~\ref{claim:lemma of lemma computing diameter subclaim 6}, where $d=2$, $i=1$, and $j=2$.}
        \label{fig: lemma b.5 figure 4}
    \end{figure}
    \begin{proof}[Proof of Claim~\ref{claim:lemma of lemma computing diameter subclaim 6}]
    First, let
    \begin{align*}
    z=\lambda_{j}v_{j}+\sum_{\substack{h=1,\cdots,d_{1},\\ h\neq j}}\lambda_{h}z_{j,h}^{(0)}\in\textup{conv}(\{v_{j}\}\cup \bigcup_{\substack{h=1,\cdots,d_{1},\\ h\neq j}}\{z_{j,h}^{(0)}\}).
    \end{align*}
    Here, we notice that $z_{j,i}^{(0)}$ is decomposed as
    \begin{align*}
    z_{j,i}^{(0)}=\frac{c_{j,i,j}^{(0)}-c_{i,j,j}^{(3)}}{1-c_{i,j,j}^{(3)}}v_{j}+\frac{1-c_{j,i,j}^{(0)}}{1-c_{i,j,j}^{(3)}}z_{i,j}^{(3)}.
    \end{align*}
    Since $c_{j,i,j}^{(0)}=1-D_{\textup{proj}}^{-1}\beta$ and $c_{i,j,j}^{(3)}=\min\{c_{i,j,j}^{(1)},1-D_{\textup{proj}}^{-1}\beta\}\leq c_{j,i,j}^{(0)}$ by the definitions, we have
    \begin{align*}
    0\leq \frac{c_{j,i,j}^{(0)}-c_{i,j,j}^{(3)}}{1-c_{i,j,j}^{(3)}}\leq \frac{1-c_{i,j,j}^{(3)}}{1-c_{i,j,j}^{(3)}}=1.
    \end{align*}
    Hence,
    \begin{align*}
    z&=(\lambda_{i}\frac{c_{j,i,j}^{(0)}-c_{i,j,j}^{(3)}}{1-c_{i,j,j}^{(3)}}+\lambda_{j})v_{j}+\sum_{\substack{h=1,\cdots,d_{1},\\ h\neq i,j}} \lambda_{h}z_{j,h}^{(0)}+\lambda_{i}\frac{1-c_{j,i,j}^{(0)}}{1-c_{i,j,j}^{(3)}}z_{i,j}^{(3)}\\
    &\in
    \textup{conv}(\{v_{j}\}\cup \bigcup_{\substack{h=1,\cdots,d_{1},\\ h\neq i,j}}\{z_{j,h}^{(0)}\}\cup \bigcup_{\substack{h=1,\cdots,d_{1},\\ h\neq i,j}}\{z_{i,h}^{(2)},z_{i,h}^{(3)}\}\cup \{z_{i,j}^{(3)}\}),
    \end{align*}
    which implies
    \begin{align}
    \label{eq:subclaim 6 eq 1}
    &\textup{conv}(\{v_{j}\}\cup \bigcup_{\substack{h=1,\cdots,d_{1},\\ h\neq j}}\{z_{j,h}^{(0)}\})\nonumber\\
    &\subset
    \textup{conv}(\{v_{j}\}\cup \bigcup_{\substack{h=1,\cdots,d_{1},\\ h\neq i,j}}\{z_{j,h}^{(0)}\}\cup \bigcup_{\substack{h=1,\cdots,d_{1},\\ h\neq i,j}}\{z_{i,h}^{(2)},z_{i,h}^{(3)}\}\cup \{z_{i,j}^{(3)}\}).
    \end{align}

    Next, let
    \begin{align*}
    z=\sum_{\substack{h=1,\cdots,d_{1},\\ h\neq i}}(\lambda_{2,h}z_{i,h}^{(2)}+\lambda_{3,h}z_{i,h}^{(3)})\in \textup{conv}(\bigcup_{\substack{h=1,\cdots,d_{1},\\ h\neq i}}\{z_{i,h}^{(2)},z_{i,h}^{(3)}\}).
    \end{align*}
    We decompose $z_{i,j}^{(2)}$ as
    \begin{align*}
    z_{i,j}^{(2)}=\frac{c_{i,j,j}^{(2)}-c_{i,j,j}^{(3)}}{1-c_{i,j,j}^{(3)}}v_{j}+\frac{1-c_{i,j,j}^{(2)}}{1-c_{i,j,j}^{(3)}}z_{i,j}^{(3)}.
    \end{align*}
    Note that $c_{i,j,j}^{(2)}=D_{\textup{proj}}^{-1}\beta$ and $c_{i,j,j}^{(3)}=\min\{D_{\Delta^{d}}^{-1}r,1-D_{\textup{proj}}^{-1}\beta\}\leq D_{\Delta^{d}}^{-1}r$ by the definitions.
    Since $\sqrt{d_{1}/(2d)}\leq 1$ and $r\leq \beta$, we have
    \begin{align*}
    c_{i,j,j}^{(3)}\leq D_{\Delta^{d}}^{-1}r=\sqrt{\frac{d}{2d_{1}}}r\leq \sqrt{\frac{d}{2d_{1}}}\beta=\sqrt{\frac{d_{1}}{2d}}c_{i,j,j}^{(2)}\leq c_{i,j,j}^{(2)},
    \end{align*}
    where the first and second equalities are due to Lemma~\ref{lem:d proj and d diameter}--(ii) and (i), respectively.
    Hence, we obtain
    \begin{align*}
    0\leq \frac{c_{i,j,j}^{(2)}-c_{i,j,j}^{(3)}}{1-c_{i,j,j}^{(3)}}\leq \frac{1-c_{i,j,j}^{(3)}}{1-c_{i,j,j}^{(3)}}=1.
    \end{align*}
    Thus, we have
    \begin{align*}
    &z=\\
    &\lambda_{2,j}\frac{c_{i,j,j}^{(2)}-c_{i,j,j}^{(3)}}{1-c_{i,j,j}^{(3)}}v_{j}+\sum_{\substack{h=1,\cdots,d_{1},\\ h\neq i,j}}(\lambda_{2,h}z_{i,h}^{(2)}+\lambda_{3,h}z_{i,h}^{(3)})+(\lambda_{2,j}\frac{1-c_{i,j,j}^{(2)}}{1-c_{i,j,j}^{(3)}}+\lambda_{3,j})z_{i,j}^{(3)}\\
    &\in \textup{conv}(\{v_{j}\}\cup \bigcup_{\substack{h=1,\cdots,d_{1},\\ h\neq i,j}}\{z_{j,h}^{(0)}\}\cup \bigcup_{\substack{h=1,\cdots,d_{1},\\ h\neq i,j}}\{z_{i,h}^{(2)},z_{i,h}^{(3)}\}\cup \{z_{i,j}^{(3)}\}),
    \end{align*}
    which implies the relationship
    \begin{align}
    \label{eq:subclaim 6 eq 2}
    &\textup{conv}(\bigcup_{\substack{h=1,\cdots,d_{1},\\ h\neq i}}\{z_{i,h}^{(2)},z_{i,h}^{(3)}\})\nonumber\\
    &\subset
    \textup{conv}(\{v_{j}\}\cup \bigcup_{\substack{h=1,\cdots,d_{1},\\ h\neq i,j}}\{z_{j,h}^{(0)}\}\cup \bigcup_{\substack{h=1,\cdots,d_{1},\\ h\neq i,j}}\{z_{i,h}^{(2)},z_{i,h}^{(3)}\}\cup \{z_{i,j}^{(3)}\}).
    \end{align}

    By \eqref{eq:subclaim 6 eq 1} and \eqref{eq:subclaim 6 eq 2}, we obtain the claim.
    \end{proof}
    \paragraph{Step 4.}
    We now prove Lemma~\ref{lem:lemma of lemma computing diameter}.
    By~\eqref{eq:lemma of lemma computing diameter eq 2}, \eqref{eq:lemma of lemma computing diameter eq 3}, and Claim~\ref{claim:lemma of lemma computing diameter subclaim 6}, we have
    \begin{align}
    \label{eq:lemma of lemma computing diameter eq 4}
    B_{i}\cup B_{j}\subset
    \textup{conv}(\{v_{j}\}\cup \bigcup_{\substack{h=1,\cdots,d_{1},\\ h\neq i,j}}\{z_{j,h}^{(0)}\}\cup \bigcup_{\substack{h=1,\cdots,d_{1},\\ h\neq i,j}}\{z_{i,h}^{(2)},z_{i,h}^{(3)}\}\cup \{z_{i,j}^{(3)}\}).
    \end{align}
    Thus, we have
    \begin{align}
    \label{eq:lemma of lemma computing diameter eq 5}
    &\textup{diam}(B_{i}\cup B_{j})\nonumber \\
    &\leq 
    \textup{diam}(\textup{conv}(\{v_{j}\}\cup \bigcup_{\substack{h=1,\cdots,d_{1},\\ h\neq i,j}}\{z_{j,h}^{(0)}\}\cup \bigcup_{\substack{h=1,\cdots,d_{1},\\ h\neq i,j}}\{z_{i,h}^{(2)},z_{i,h}^{(3)}\}\cup \{z_{i,j}^{(3)}\})).
    \end{align}
    By Lemma~\ref{citelem:diameter of any convex hull}, we have
    \begin{align}
    \label{eq:lemma of lemma computing diameter eq 6}
    &\textup{diam}(\textup{conv}(\{v_{j}\}\cup \bigcup_{\substack{h=1,\cdots,d_{1},\\ h\neq i,j}}\{z_{j,h}^{(0)}\}\cup \bigcup_{\substack{h=1,\cdots,d_{1},\\ h\neq i,j}}\{z_{i,h}^{(2)},z_{i,h}^{(3)}\}\cup \{z_{i,j}^{(3)}\}))\nonumber\\
    &=
    \textup{diam}(\{v_{j}\}\cup \bigcup_{\substack{h=1,\cdots,d_{1},\\ h\neq i,j}}\{z_{j,h}^{(0)}\}\cup \bigcup_{\substack{h=1,\cdots,d_{1},\\ h\neq i,j}}\{z_{i,h}^{(2)},z_{i,h}^{(3)}\}\cup \{z_{i,j}^{(3)}\}).
    \end{align}
    By~\eqref{eq:lemma of lemma computing diameter eq 5} and \eqref{eq:lemma of lemma computing diameter eq 6}, we obtain
    \begin{align}
    \label{eq:lemma of lemma computing diameter eq 7}
    &\textup{diam}(B_{i}\cup B_{j})\nonumber \\
    &\leq
    \textup{diam}(\{v_{j}\}\cup \bigcup_{\substack{h=1,\cdots,d_{1},\\ h\neq i,j}}\{z_{j,h}^{(0)}\}\cup \bigcup_{\substack{h=1,\cdots,d_{1},\\ h\neq i,j}}\{z_{i,h}^{(2)},z_{i,h}^{(3)}\}\cup \{z_{i,j}^{(3)}\}).
    \end{align}

    It remains to compute the right-hand side of~\eqref{eq:lemma of lemma computing diameter eq 7}.
    To this end, we show all the combinations of points in the set
    \begin{align*}
    \{v_{j}\}\cup \bigcup_{\substack{h=1,\cdots,d_{1},\\ h\neq i,j}}\{z_{j,h}^{(0)}\}\cup \bigcup_{\substack{h=1,\cdots,d_{1},\\ h\neq i,j}}\{z_{i,h}^{(2)},z_{i,h}^{(3)}\}\cup \{z_{i,j}^{(3)}\}
    \end{align*}
    and the squared distance between the points in every pair, multiplied by $d/d_{1}$, using the formula in Proposition~\ref{prop:simple fact for risks}--(i).
    For instance, in (D1) we compute
    \begin{align*}
    \mathscr{D}_{1}=\frac{d}{d_{1}}\|v_{j}-z_{j,h}^{(0)}\|_{2}^{2}.
    \end{align*}
    The other parts (D2)--(D16) follow the same way as that in (D1).
    Note that we abbreviate as
    \begin{align*}
    c_{i,h,k}:=c_{i,h,k}^{(3)},
    \end{align*}
    for every $h\in \{1,\cdots,d_{1}\}\setminus \{i\}$ and $k\in \{1,\cdots,d_{1}\}$, for convenience.
    \begin{align*}
    \begin{array}{@{}l@{~}l@{~}l@{~}l@{}}
        \textup{(D1)} & (v_{j}, z_{j,h}^{(0)}) & \forall h\neq i,j, & \mathscr{D}_{1}=\frac{2\beta^{2}}{D_{\textup{proj}}^{2}}. \\
        \textup{(D2)} & (v_{j}, z_{i,h}^{(2)}) & \forall h\neq i,j, & \mathscr{D}_{2}=1+(1-\frac{\beta}{D_{\textup{proj}}})^{2}+\frac{\beta^{2}}{D_{\textup{proj}}^{2}}. \\
        \textup{(D3)} & (v_{j}, z_{i,h}^{(3)}) & \forall h\neq i,j, & \mathscr{D}_{3}=1+c_{i,h,i}^{2}+c_{i,h,h}^{2}. \\
        \textup{(D4)} & (v_{j}, z_{i,j}^{(3)}), &  & \mathscr{D}_{4}=(1-c_{i,j,j})^{2}+c_{i,j,i}^{2}. \\
        \textup{(D5)} & (z_{j,h_{1}}^{(0)}, z_{j,h_{2}}^{(0)}) & \forall h_{1},h_{2}\neq i,j\textup{ s.t. }h_{1}<h_{2}, & \mathscr{D}_{5}=\frac{2\beta^{2}}{D_{\textup{proj}}^{2}}. \\
        \textup{(D6)} & (z_{j,h}^{(0)}, z_{i,h}^{(2)}) & \forall h\neq i,j, & \mathscr{D}_{6}=2(1-\frac{\beta}{D_{\textup{proj}}})^{2}. \\
        \textup{(D7)} & (z_{j,h_{1}}^{(0)}, z_{i,h_{2}}^{(2)}) & \forall h_{1},h_{2}\neq i,j\textup{ s.t. }h_{1}\neq h_{2}, & \mathscr{D}_{7}=2(1-\frac{\beta}{D_{\textup{proj}}})^{2}+\frac{2\beta^{2}}{D_{\textup{proj}}^{2}}. \\
        \textup{(D8)} & (z_{j,h}^{(0)}, z_{i,h}^{(3)}) & \forall h\neq i,j, & \mathscr{D}_{8}=(1-\frac{\beta}{D_{\textup{proj}}})^{2}+c_{i,h,i}^{2}+(\frac{\beta}{D_{\textup{proj}}}-c_{i,h,h})^{2}. \\
        \textup{(D9)} & (z_{j,h_{1}}^{(0)}, z_{i,h_{2}}^{(3)}) & \forall h_{1},h_{2}\neq i,j\textup{ s.t. }h_{1}\neq h_{2}, & \mathscr{D}_{9}=(1-\frac{\beta}{D_{\textup{proj}}})^{2}+c_{i,h_{2},i}^{2}+\frac{\beta^{2}}{D_{\textup{proj}}^{2}}+c_{i,h_{2},h_{2}}^{2}. \\
        \textup{(D10)} & (z_{j,h}^{(0)}, z_{i,j}^{(3)}) & \forall h\neq i,j, & \mathscr{D}_{10}=(1-\frac{\beta}{D_{\textup{proj}}}-c_{i,j,j})^{2}+\frac{\beta^{2}}{D_{\textup{proj}}^{2}}+c_{i,j,i}^{2}. \\
        \textup{(D11)} & (z_{i,h_{1}}^{(2)}, z_{i,h_{2}}^{(2)}) & \forall h_{1},h_{2}\neq i,j\textup{ s.t. }h_{1}<h_{2}, & \mathscr{D}_{11}=\frac{2\beta^{2}}{D_{\textup{proj}}^{2}}. \\
        \textup{(D12)} & (z_{i,h}^{(2)}, z_{i,h}^{(3)}) & \forall h\neq i,j, & \mathscr{D}_{12}=(1-\frac{\beta}{D_{\textup{proj}}}-c_{i,h,i})^{2}+(\frac{\beta}{D_{\textup{proj}}}-c_{i,h,h})^{2}. \\
        \textup{(D13)} & (z_{i,h_{1}}^{(2)}, z_{i,h_{2}}^{(3)}) & \forall h_{1},h_{2}\neq i,j\textup{ s.t. }h_{1}\neq h_{2}, & \mathscr{D}_{13}=(1-\frac{\beta}{D_{\textup{proj}}}-c_{i,h_{2},i})^{2}+\frac{\beta^{2}}{D_{\textup{proj}}^{2}}+c_{i,h_{2},h_{2}}^{2}. \\
        \textup{(D14)} & (z_{i,h}^{(2)}, z_{i,j}^{(3)}) & \forall h\neq i,j, & \mathscr{D}_{14}=(1-\frac{\beta}{D_{\textup{proj}}}-c_{i,j,i})^{2}+\frac{\beta^{2}}{D_{\textup{proj}}^{2}}+c_{i,j,j}^{2}. \\
        \textup{(D15)} & (z_{i,h_{1}}^{(3)}, z_{i,h_{2}}^{(3)}) & \forall h_{1},h_{2}\neq i,j\textup{ s.t. }h_{1}<h_{2}, & \mathscr{D}_{15}=(c_{i,h_{1},i}-c_{i,h_{2},i})^{2}+c_{i,h_{1},h_{1}}^{2}+c_{i,h_{2},h_{2}}^{2}. \\
        \textup{(D16)} & (z_{i,h}^{(3)}, z_{i,j}^{(3)}) & \forall h\neq i,j, & \mathscr{D}_{16}=(c_{i,h,i}-c_{i,j,i})^{2}+c_{i,h,h}^{2}+c_{i,j,j}^{2}. \\
    \end{array}
    \end{align*}
    Then, noting the inequality $\lambda^{2}+(1-\lambda)^{2}\leq 1$ for any $\lambda\in [0,1]$, we have
    \begin{align*}
    \begin{cases}
    &\mathscr{D}_{1}=\mathscr{D}_{5}=\mathscr{D}_{11}\leq \mathscr{D}_{2},\\
    &\mathscr{D}_{6}\leq \mathscr{D}_{7}\leq \mathscr{D}_{2},\\
    &\mathscr{D}_{8}\leq \mathscr{D}_{9}\leq \mathscr{D}_{2},\\
    &\mathscr{D}_{4}\leq \mathscr{D}_{3},\\
    &\mathscr{D}_{10}\leq \mathscr{D}_{3},\\
    &\mathscr{D}_{12}\leq \mathscr{D}_{13}=\mathscr{D}_{14}\leq \mathscr{D}_{3},\\
    &\mathscr{D}_{15}=\mathscr{D}_{16}\leq \mathscr{D}_{3}.
    \end{cases}
    \end{align*}

    \begin{figure}
        \centering
        \includegraphics[width=0.95\linewidth]{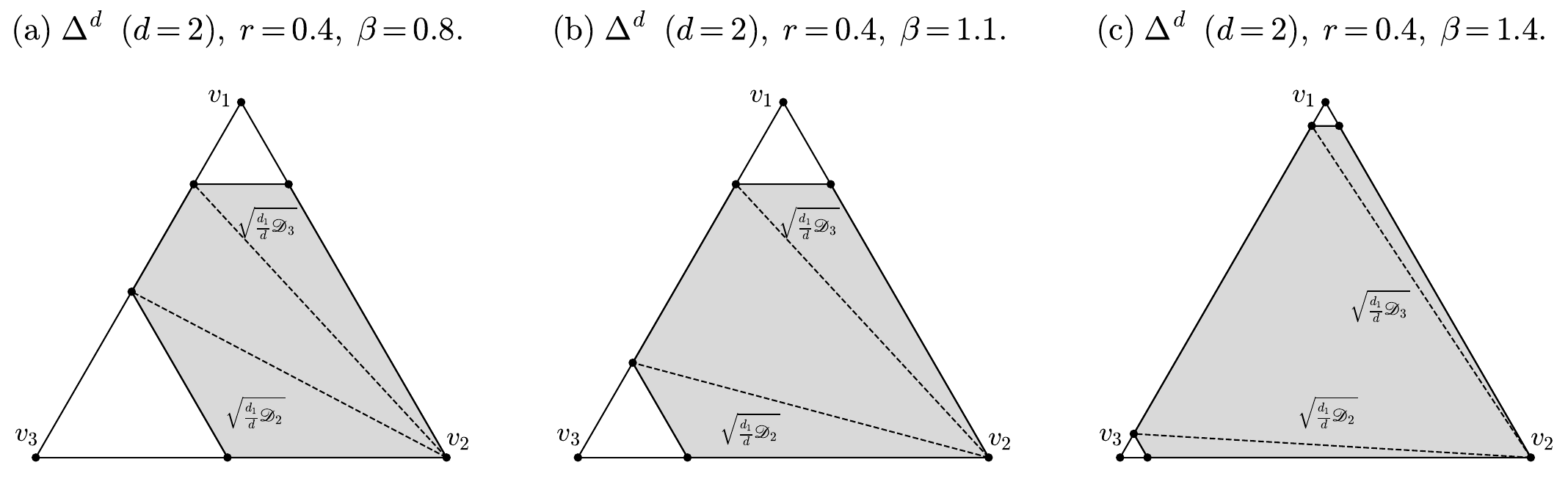}
        \caption{Examples of the convex hulls in Step 4, where $d=2$, $i=1$, and $j=2$. In each panel, the length of the dashed line above is $\sqrt{\frac{d_{1}}{d}\mathscr{D}_{3}}$, and the length of the dashed line below is $\sqrt{\frac{d_{1}}{d}\mathscr{D}_{2}}$, where we used Proposition~\ref{prop:simple fact for risks}--(i).}
        \label{fig: lemma b.5 figure 5}
    \end{figure}
    
    Regarding $\mathscr{D}_{2}$ and $\mathscr{D}_{3}$ (see Figure~\ref{fig: lemma b.5 figure 5}), we notice that the condition $r\leq D_{\Delta^{d}}(1-\beta/D_{\textup{proj}})$ implies $|1/2-D_{\Delta^{d}}^{-1}r|\geq |\beta/D_{\textup{proj}}-1/2|$, where we also used $1/2-D_{\Delta^{d}}^{-1}r\geq 0$ and $D_{\Delta^{d}}^{-1}r\leq D_{\textup{proj}}^{-1}\beta$ (note that $r\leq \beta$ and $D_{\textup{proj}}\leq D_{\Delta^{d}}$).
    In addition, $\lambda_{1}^{2}+(1-\lambda_{1})^{2}\leq \lambda_{2}^{2}+(1-\lambda_{2})^{2}$ if $|\lambda_{1}-1/2|\leq |\lambda_{2}-1/2|$ holds for the given $\lambda_{1},\lambda_{2}\in [0,1]$.
    Hence, the condition $r\leq D_{\Delta^{d}}(1-\beta/D_{\textup{proj}})$ implies that
    \begin{align*}
    (1-D_{\textup{proj}}^{-1}\beta)^{2}+(D_{\textup{proj}}^{-1}\beta)^{2}\leq (1-D_{\Delta^{d}}^{-1}r)^{2}+(D_{\Delta^{d}}^{-1}r)^{2}.
    \end{align*}
    In addition, if $r>D_{\Delta^{d}}(1-\beta/D_{\textup{proj}})$, then we have $c_{i,h,i}^{(3)}=1-c_{i,h,i}^{(2)}$.
    Thus, we always have
    \begin{align*}
        \mathscr{D}_{2}\leq \mathscr{D}_{3}.
    \end{align*}
    Therefore, for any $h_{0}\in\{1,\cdots,d_{1}\}\setminus \{i,j\}$, we have
    \begin{align}
    &\textup{diam}(\{v_{j}\}\cup \bigcup_{\substack{h=1,\cdots,d_{1},\\ h\neq i,j}}\{z_{j,h}^{(0)}\}\cup \bigcup_{\substack{h=1,\cdots,d_{1},\\ h\neq i,j}}\{z_{i,h}^{(2)},z_{i,h}^{(3)}\}\cup \{z_{i,j}^{(3)}\})\nonumber\\
    &=\sqrt{\frac{d_{1}}{d}\mathscr{D}_{3}}\nonumber \\
    &=
    \|v_{j}-z_{i,h_{0}}^{(3)}\|_{2}\nonumber\\
    &=((\sqrt{3}D_{\Delta^{d}}/2)^{2}+(D_{\Delta^{d}}/2-r)^{2}\vee (\beta D_{\Delta^{d}}/D_{\textup{proj}}-D_{\Delta^{d}}/2)^{2})^{1/2}\nonumber\\
    \label{eq:lemma of lemma computing diameter eq 8}
    &=
    ((\sqrt{3}D_{\Delta^{d}}/2)^{2}+((D_{\Delta^{d}}/2-r)\vee (\beta D_{\Delta^{d}}/D_{\textup{proj}}-D_{\Delta^{d}}/2))^{2})^{1/2},
    \end{align}
    where in the equality~\eqref{eq:lemma of lemma computing diameter eq 8} we use the following properties:
    \begin{itemize}
    \item When $0<\beta\leq D_{\textup{proj}}/2$, we have $0\leq D_{\Delta^{d}}/2-\beta D_{\Delta^{d}}/D_{\textup{proj}}\leq D_{\Delta^{d}}/2-r$, where we use the inequalities $0\leq r\leq \beta\leq \beta D_{\Delta^{d}}/D_{\textup{proj}}$.
    \item When $D_{\textup{proj}}/2<\beta\leq D_{\textup{proj}}-rD_{\textup{proj}}/D_{\Delta^{d}}$, we have $0\leq \beta D_{\Delta^{d}}/D_{\textup{proj}}-D_{\Delta^{d}}/2\leq D_{\Delta^{d}}/2-r$.
    \item When $D_{\textup{proj}}-rD_{\textup{proj}}/D_{\Delta^{d}}<\beta < D_{\textup{proj}}$, we have $0\leq D_{\Delta^{d}}/2-r\leq \beta D_{\Delta^{d}}/D_{\textup{proj}}-D_{\Delta^{d}}/2$.
    \end{itemize}
    By~\eqref{eq:lemma of lemma computing diameter eq 7} and \eqref{eq:lemma of lemma computing diameter eq 8}, we have
    \begin{align}
    \label{eq:lemma of lemma computing diameter eq 9}
    &\textup{diam}(B_{i}\cup B_{j})\nonumber\\
    &\leq ((\sqrt{3}D_{\Delta^{d}}/2)^{2}+((D_{\Delta^{d}}/2-r)\vee (\beta D_{\Delta^{d}}/D_{\textup{proj}}-D_{\Delta^{d}}/2))^{2})^{1/2}.
    \end{align}
    Combining~\eqref{eq:lemma of lemma computing diameter eq 0} and \eqref{eq:lemma of lemma computing diameter eq 9}, we obtain
    \begin{align*}
    &\|f(x)-f(x')\|_{2}\\
    &\leq \textup{diam}(B_{i}\cup B_{j})\\
    &\leq
    ((\sqrt{3}D_{\Delta^{d}}/2)^{2}+((D_{\Delta^{d}}/2-r)\vee (\beta D_{\Delta^{d}}/D_{\textup{proj}}-D_{\Delta^{d}}/2))^{2})^{1/2}.
    \end{align*}
We obtain the claim.
\end{proof}

\subsection{Proof of Theorem~\ref{thm:main result}}
\label{sec:estimation error bounds}

\begin{proof}
Let $C_{2}$ and $C_{i,j}$ ($i,j\in\{1,\cdots,d_{1}\}$ for which $i\neq j$ is satisfied) be the constants satisfying the conditions in Lemma~\ref{lem:step 3 lemma 0} and Lemma~\ref{lem:step 3 lemma 1}, respectively.
Let $f\in\mathscr{F}_{\beta,\beta_{0},P}(\mathcal{F})$ be arbitrary.
Applying Lemma~\ref{lem:step 3 lemma 2} and Lemma~\ref{lem:step 3 lemma 1} to Lemma~\ref{lem:step 3 lemma 0}, we have
\begin{align*}
&\mathbb{E}_{P_{X}}[\|f-f^{*}\|_{2}^{2}]\\
&\leq C_{2}(D_{\Delta^{d}}-\beta)\beta_{0}\\
&\quad + C_{2}\sum_{i\neq j}C_{i,j}\sum_{w=0}^{\lfloor\log_{2}{n}\rfloor}2^{-(2w+1)}\beta^{2}((2^{-(w+1)}\beta)\wedge (D_{\Delta^{d}}(1-\beta/D_{{\textup{proj}}})))^{-2}\\
&\quad \cdot\mathbb{E}_{P_{X,X'}^{-}}[|\psi\circ\rho_{f}-\psi\circ\rho_{f^{*}}|] +2C_{2}\sum_{i\neq j}\sum_{w=0}^{\infty}\left(\frac{1}{2}\right)^{2w+1}\beta^{2}\beta_{0} +C_{2}\frac{\beta^{2}}{n}\\
&\leq C_{2}'(\log{n})C_{2}\sum_{i\neq j}C_{i,j}\cdot\mathbb{E}_{P_{X,X'}^{-}}[|\psi\circ\rho_{f}-\psi\circ\rho_{f^{*}}|]\\
&\quad +C_{2}(D_{\Delta^{d}}-\beta+\frac{4}{3}d(d+1)\beta^{2})\beta_{0}+C_{2}\frac{\beta^{2}}{n},
\end{align*}
where $C_{2}'=(\log{2})^{-1}(4+2\beta^{2}(D_{\Delta^{d}}(1-\beta/D_{{\textup{proj}}}))^{-2})$, and in the last inequality we use the following inequalities valid for any $n\in\mathbb{N}\setminus\{1\}$,
\begin{align*}
&\sum_{w=0}^{\lfloor \log_{2}{n}\rfloor}2^{-(2w+1)}\beta^{2}((2^{-(w+1)}\beta)\wedge (D_{\Delta^{d}}(1-\beta/D_{{\textup{proj}}})))^{-2}\\
&\leq
4\log_{2}{n}+2(\log_{2}{n})\beta^{2}(D_{\Delta^{d}}(1-\beta/D_{{\textup{proj}}}))^{-2}\\
&\leq
(\log{2})^{-1}(4+2\beta^{2}(D_{\Delta^{d}}(1-\beta/D_{{\textup{proj}}}))^{-2})\log{n}.
\end{align*}
Note that $p_{Y}(-1)^{-1}\leq \theta_{2}^{-1}$ holds by condition (A3) in Definition~\ref{def:main assumption}.
Define
\begin{align}
\label{eq:threshold for constant c3}
C_{3}&=C_{2}'C_{2}\sum_{i\neq j}C_{i,j}\theta_{2}^{-1},\\
C_{4}&=C_{2}(D_{\Delta^{d}}-\beta+(4/3)d(d+1)\beta^{2}),\nonumber\\
C_{5}&=C_{2}\beta^{2}.\nonumber
\end{align}
Then, we have
\begin{align}
\label{eq:thm estimation bound eq 1}
\mathbb{E}_{P_{X}}[\|f-f^{*}\|_{2}^{2}]\leq C_{3}(\log{n})\mathbb{E}_{P_{X,X'}}[|\psi\circ\rho_{f}-\psi\circ\rho_{f^{*}}|]+C_{4}\beta_{0}+\frac{C_{5}}{n}.
\end{align}
Here, by Lemma~\ref{prop:constructability for F0}, $\psi\circ\rho_{f^{*}}$ is the Bayes classifier.
Note that if condition $\tau$-(NC) is satisfied, then Proposition~1 in~\citep{lecue2007optimal} (see Lemma~\ref{citelem:lecue lemma}) is applicable.
By Proposition~1 in~\citep{lecue2007optimal} (see Lemma~\ref{citelem:lecue lemma}), there is a universal constant $C_{0}>0$ that does not depend on the choice of $P$, such that
\begin{align}
\label{eq:thm estimation bound eq 2}
\mathbb{E}_{P_{X,X'}}[|\psi\circ\rho_{f}-\psi\circ\rho_{f^{*}}|]^{\tau}\leq C_{0}(\mathbb{E}_{P}[\ell_{f}]-\mathbb{E}_{P}[\ell_{f^{*}}]).
\end{align}
Applying~\eqref{eq:thm estimation bound eq 2} to~\eqref{eq:thm estimation bound eq 1}, we have
\begin{align}
\label{eq:thm estimation bound eq 3}
\mathbb{E}_{P_{X}}[\|f-f^{*}\|_{2}^{2}]\leq C_{0}^{\frac{1}{\tau}}C_{3}(\log{n})(\mathbb{E}_{P}[\ell_{f}]-\mathbb{E}_{P}[\ell_{f^{*}}])^{\frac{1}{\tau}}+C_{4}\beta_{0}+\frac{C_{5}}{n}.
\end{align}
Finally, set $f=\widehat{f}_{n,P}^{\textup{local}}(U_{1}^{n}(\omega))$ with arbitrary $\omega\in\Omega$, and integrate~\eqref{eq:thm estimation bound eq 3} on $\Omega$.
Then, by Proposition~\ref{prop:simple fact for risks}--(ii) and Jensen's inequality, we obtain the claim.
\end{proof}

\subsection{Proof of Theorem~\ref{thm:estimation error for deep relu networks}}
\label{appsec:proofs for section examples}

In this section, we prove Theorem~\ref{thm:estimation error for deep relu networks}.
The proof consists of five steps:
\begin{itemize}
\item In Step 0, we present Theorem~\ref{thm:main result} in Appendix~\ref{sec:estimation error bounds}.
\item In Step 1, we investigate the approximation error measured by the excess risk of hinge loss, following~\citep{park2009convergence,kim2021fast}. See~Appendix~\ref{appsubsubsec:step 1 of the proof of the main theorem}.
\item In Step 2, we approximate the true function $f^{*}$, according to the notion defined in Step 1. We basically follow the arguments developed by~\citet{bos2022convergence}. See Appendix~\ref{appsubsubsec:step 2 of the proof of the main theorem}.
\item In Step 3, we provide an analysis of the parameter $\beta_{0}$ in the definition of localized subclasses. See Appendix~\ref{appsubsubsec:step 3 of the proof of the main theorem}.
\item In Step 4, we complete the proof of Theorem~\ref{thm:estimation error for deep relu networks}. See Appendix~\ref{appsubsubsec:step 4 of the proof of the main theorem}.
\end{itemize}

\subsubsection{Step 1}
\label{appsubsubsec:step 1 of the proof of the main theorem}

We aim at deriving an upper bound of the excess risk.
We consider to apply the results shown in~\citep{park2009convergence,kim2021fast}.
To this end, we need to investigate the approximation error bounds.

First, we use a weaker notion of $(\psi_{0},\mathcal{F})$-representability:
\begin{definition}
\label{def:weak constructability}
Given any $\xi=(R,K,d_{1},E,\theta_{\textup{NC}},\theta_{1},\theta_{2},\theta_{3})\in\Xi$, $\mathcal{F}\subseteq\mathcal{F}_{0}$, $\varepsilon\geq 0$, $V\in [0,1]$, and a measurable function $\psi_{0}:\mathbb{R}\to\mathbb{R}$, a Borel probability measure $P\in\mathcal{P}_{\xi}$ is $(\psi_{0},\mathcal{F},\varepsilon,V)$\emph{-weak representable}, if there is a function $f\in\mathcal{F}$ such that the following inequality holds:
\begin{align}
\label{eq:condition on weak realizability}
\sup\left\{
P_{X,X'}(\mathcal{A})\;|\;\mathcal{A}\in\mathscr{W}_{f,\varepsilon}
\right\}
\geq V,
\end{align}
where we define
\begin{align*}
\mathscr{W}_{f,\varepsilon}=
\left\{
\mathcal{A}\in\mathcal{B}(\mathcal{X}^{2})\;|\;
\|\psi_{0}\circ \rho_{f}-\textup{sign}\circ(2\eta-1)\|_{\mathcal{A},P_{X,X'},1}\leq \varepsilon
\right\}.
\end{align*}
\end{definition}
Hereafter, in this section we define $\mathcal{F}\subseteq\mathcal{F}_{0}$, a measurable function $\psi_{0}:\mathbb{R}\to\mathbb{R}$, $\varepsilon\geq 0$, and $V\in[0,1]$.

Using the notion of the weak representability in Definition~\ref{def:weak constructability}, we can evaluate the approximation error in terms of the excess risk of hinge loss, which will be used when applying the results of~\citep{park2009convergence,kim2021fast}.
\begin{proposition}
\label{lem:connecting the upper bound to the excess risk}
Let $\xi=(R,K,d_{1},E,\theta_{\textup{NC}},\theta_{1},\theta_{2},\theta_{3})\in\Xi$, $\mathcal{H}\subset\mathcal{F}_{0}$, $\varepsilon\geq 0$, $V\in[0,1]$, and $\psi$ be the function defined in Definition~\ref{def:contrastive loss}.
For every $(\psi,\mathcal{H},\varepsilon,V)$-weak representable $P\in\mathcal{P}_{\xi}$, if there is a vector-valued function $f^{*}\in\mathcal{F}_{0}$ such that $\psi \circ \rho_{f^{*}}=\textup{sign}\circ (2\eta-1)$, $P_{X,X'}$-almost surely, then there is a function $f\in\mathcal{H}$ such that we have
\begin{align*}
    \mathbb{E}_{P}[\ell_{f}]-\mathbb{E}_{P}[\ell_{f^{*}}]
    \leq \varepsilon+2(1-V).
\end{align*}
\end{proposition}
\begin{proof}
Let $f_{0}\in\mathcal{H}$ be a function satisfying the condition of the weak representability,
\begin{align*}
    \sup\{P_{X,X'}(\mathcal{A})\;|\; \mathcal{A}\in\mathscr{W}_{f_{0},\varepsilon}\}\geq V.
\end{align*}
Let $\delta>0$ be arbitrary.
Then, there is a subset $\mathcal{A}\in\mathscr{W}_{f_{0},\varepsilon}$ such that $P_{X,X'}(\mathcal{A})\geq V-\delta$ and $\|\psi \circ \rho_{f_{0}}-\textup{sign}\circ (2\eta-1)\|_{\mathcal{A},P_{X,X'},1}\leq \varepsilon$.
We have
\begin{align}
    &\mathbb{E}_{P}[\ell_{f_{0}}]-\mathbb{E}_{P}[\ell_{f^{*}}]\nonumber\\
    \label{eq:prop connecting approximation error eq 40}
    &\leq \int_{\mathcal{A}}|\psi \circ \rho_{f_{0}}-\psi \circ\rho_{f^{*}}|dP_{X,X'}+\int_{\mathcal{X}\times\mathcal{X}\setminus\mathcal{A}}|\psi \circ\rho_{f_{0}}-\psi \circ\rho_{f^{*}}|dP_{X,X'}\\
    \label{eq:prop connecting approximation error eq 5}
    &\leq \varepsilon+\int_{\mathcal{X}\times\mathcal{X}\setminus\mathcal{A}}|\psi \circ\rho_{f_{0}}-\psi \circ\rho_{f^{*}}|dP_{X,X'}\\
    \label{eq:prop connecting approximation error eq 6}
    &\leq \varepsilon+2P_{X,X'}(\mathcal{X}\times\mathcal{X}\setminus\mathcal{A})\\
    &\leq \varepsilon + 2(1-V)+2\delta,\nonumber
\end{align}
where in~\eqref{eq:prop connecting approximation error eq 40} we used the well-known fact that hinge loss is Lipschitz continuous (see, e.g.,~\citep[Example~2.27]{steinwart2008support}), and in \eqref{eq:prop connecting approximation error eq 5} we used both the condition that $\psi \circ \rho_{f^{*}}=\textup{sign}\circ (2\eta -1)$, $P_{X,X'}$-almost surely, and the property that $\|\psi \circ \rho_{f_{0}}-\textup{sign}\circ (2\eta-1)\|_{\mathcal{A},P_{X,X'},1}\leq \varepsilon$ holds.
In~\eqref{eq:prop connecting approximation error eq 6}, we use the triangle inequality and $|\psi\circ\rho_{f}|\leq 1$ on $\mathcal{X}^{2}$ for every $f\in\mathcal{F}_{0}$.
Since $\delta$ is arbitrary, we obtain the claim.
\end{proof}

Next, we show several properties of weak representability.
\begin{proposition}
\label{prop:properties of weak constructablity}
Given any $\xi=(R,K,d_{1},E,\theta_{\textup{NC}},\theta_{1},\theta_{2},\theta_{3})\in\Xi$, $\mathcal{F}\subset \mathcal{F}_{0}$, and any measurable function $\psi_{0}:\mathbb{R}\to\mathbb{R}$, if $P\in\mathcal{P}_{\xi}$ is $(\psi_{0},\mathcal{F})$-representable, then $P$ is $(\psi_{0},\mathcal{F},0,V)$-weak representable for any $V\in[0,1]$.
\end{proposition}
\begin{proof}
By the definition of $(\psi_{0},\mathcal{F})$-representability, there exists a function $f\in\mathcal{F}$ such that the following identity holds:
\begin{align*}
    \|\psi_{0}\circ \rho_{f}-\textup{sign}\circ (2\eta-1)\|_{\mathcal{X}^{2},P_{X,X'},1}=0.
\end{align*}
Using the notation
\begin{align}
\label{eq:a useful notation for set of sets}
\mathscr{W}_{f,\varepsilon,\psi_{0}}=\{\mathcal{A}\in\mathcal{B}(\mathcal{X}^{2})\;|\; \|\psi_{0}\circ\rho_{f}-\textup{sign}\circ(2\eta-1)\|_{\mathcal{A},P_{X,X'},1}\leq \varepsilon\},
\end{align}
it holds that
\begin{align*}
    \sup\{P_{X,X'}(\mathcal{A})\;|\;\mathcal{A}\in\mathscr{W}_{f,0,\psi_{0}}\}=1\geq V.
\end{align*}
This shows the claim.
\end{proof}

We next see several useful properties of the weak representability:
\begin{lemma}
\label{lem:trade-off for parameters}
Let $\xi=(R,K,d_{1},E,\theta_{\textup{NC}},\theta_{1},\theta_{2},\theta_{3})\in\Xi$, $\mathcal{F}\subset \mathcal{F}_{0}$, $\varepsilon\geq 0$, $V\in [0,1]$, and let $\psi_{0}:\mathbb{R}\to\mathbb{R}$ be a measurable function.
Suppose that $\mathcal{F}'\subset\mathcal{F}_{0}$, $\varepsilon'\geq 0$, and $V'\in [0,1]$ satisfy $\mathcal{F}'\subseteq \mathcal{F}$, $\varepsilon'\leq \varepsilon$, and $V'\geq V$.
Then, any $(\psi_{0},\mathcal{F}',\varepsilon',V')$-weak representable $P\in\mathcal{P}_{\xi}$ is $(\psi_{0},\mathcal{F},\varepsilon,V)$-weak representable.
\end{lemma}
\begin{proof}
For convenience, we use the notation $\mathscr{W}_{f,\varepsilon,\psi_{0}}$ of~\eqref{eq:a useful notation for set of sets} in this proof.
Let $P$ be an arbitrary $(\psi_{0},\mathcal{F}',\varepsilon',V')$-weak representable distribution in $\mathcal{P}_{\xi}$.
Recall that by the definition of $P$, there exists some $f\in\mathcal{F}'$ such that
\begin{align*}
\sup\{P_{X,X'}(\mathcal{A})\;|\;\mathcal{A}\in\mathscr{W}_{f,\varepsilon',\psi_{0}}\}\geq V'.
\end{align*}
Here notice that from the condition that $\mathcal{F}'\subseteq \mathcal{F}$, it holds that $f\in\mathcal{F}$.
Besides, if $\mathcal{A}\in\mathcal{B}(\mathcal{X}^{2})$ satisfies the condition $\|\psi_{0}\circ \rho_{f}-\textup{sign}\circ (2\eta-1)\|_{\mathcal{A},P_{X,X'},1}\leq \varepsilon'$, then by the condition $\varepsilon'\leq \varepsilon$ we have $\|\psi_{0}\circ \rho_{f}-\textup{sign}\circ (2\eta-1)\|_{\mathcal{A},P_{X,X'},1}\leq \varepsilon$.
Hence, we have
\begin{align*}
    \sup\{P_{X,X'}(\mathcal{A})\;|\;\mathcal{A}\in\mathscr{W}_{f,\varepsilon,\psi_{0}}\}
    \geq \sup\{P_{X,X'}(\mathcal{A})\;|\;\mathcal{A}\in\mathscr{W}_{f,\varepsilon',\psi_{0}}\}\geq V'.
\end{align*}
Since $V'\geq V$, we obtain
\begin{align*}
    \sup\{P_{X,X'}(\mathcal{A})\;|\;\mathcal{A}\in\mathscr{W}_{f,\varepsilon,\psi_{0}}\}\geq V,
\end{align*}
which implies the assertion.
\end{proof}
The following lemma is a fundamental part in the proof of Proposition~\ref{prop:approximation property of H}.
\begin{lemma}
\label{lem:approximation and weak realizability}
Let $\xi=(R,K,d_{1},E,\theta_{\textup{NC}},\theta_{1},\theta_{2},\theta_{3})\in\Xi$, $\mathcal{F}\subseteq \mathcal{F}_{0}$, $\varepsilon\geq 0$, $V\in [0,1]$, and $P\in\mathcal{P}_{\xi}$.
Let $\psi$ be the function defined in Definition~\ref{def:contrastive loss}.
Suppose that $P$ is $(\psi,\mathcal{F},\varepsilon,V)$-weak representable, and $f\in\mathcal{F}$ satisfies condition~\eqref{eq:condition on weak realizability}.
Then for any subset $\mathcal{G}\subseteq\mathcal{F}$, $P$ is $(\psi,\mathcal{G},(2\varepsilon)\vee (32D_{\Delta^{d}}^{-1}\mu(\mathcal{X})\|p_{X,X'}\|_{\mathcal{X}^{2},\infty}\varepsilon_{0}),V)$-weak representable, where $\varepsilon_{0}=\inf_{g\in\mathcal{G}}\|f-g\|_{\mathcal{X},1}$.
\end{lemma}
\begin{proof}
In the proof we use the notation $\mathscr{W}_{f,\varepsilon,\psi}$ in~\eqref{eq:a useful notation for set of sets}.
Let $\mathcal{A}\in\mathscr{W}_{f,\varepsilon,\psi}$.
Let $g\in\mathcal{G}$ be arbitrary.
Then we notice that
\begin{align}
    &\|\psi\circ \rho_{g}-\textup{sign}\circ (2\eta-1)\|_{\mathcal{A},P_{X,X'},1}\nonumber\\
    &\leq 
    \|\psi\circ \rho_{f}-\psi\circ \rho_{g}\|_{\mathcal{A},P_{X,X'},1}+\|\psi\circ \rho_{f}-\textup{sign}\circ (2\eta-1)\|_{\mathcal{A},P_{X,X'},1}\nonumber\\
    \label{eq:theorem basic properties proof part3 eq 2.1}
    &\leq 
    \|\psi\circ \rho_{f}-\psi\circ \rho_{g}\|_{\mathcal{A},P_{X,X'},1}+\varepsilon,
\end{align}
where in the second inequality we used the condition $\mathcal{A}\in\mathscr{W}_{f,\varepsilon,\psi}$.
Let us bound the first term in the last line of the above inequality.
It holds that
\begin{align*}
    &\|\psi\circ \rho_{f}-\psi\circ \rho_{g}\|_{\mathcal{A},P_{X,X'},1}\\
    &= \int_{\mathcal{A}}|\psi\circ\rho_{f}(x,x')-\psi\circ\rho_{g}(x,x')|P_{X,X'}(dx,dx')\\
    &\leq 2D_{\Delta^{d}}^{-2}\int_{\mathcal{X}\times\mathcal{X}}|\rho_{f}(x,x')-\rho_{g}(x,x')|P_{X,X'}(dx,dx')\\
    &\leq 2D_{\Delta^{d}}^{-2}\|p_{X,X'}\|_{\mathcal{X}^{2},\infty}\|\rho_{f}-\rho_{g}\|_{\mathcal{X}^{2},\mu\otimes\mu,1},
\end{align*}
where in the first inequality we used the definition of $\psi$ and $\mathcal{A}\subset\mathcal{X}\times\mathcal{X}$, and in the last inequality we used H\"{o}lder's inequality.
Besides, we also notice that for any $(x,x')\in\mathcal{X}^{2}$,
\begin{align}
\label{eq:theorem basic properties proof part3 eq 3}
    &\quad|\rho_{f}(x,x')-\rho_{g}(x,x')|\nonumber\\
    &\leq (\|f(x)-f(x')\|_{2}+\|g(x)-g(x')\|_{2})\cdot(\|f(x)-g(x)\|_{2}+\|f(x')-g(x')\|_{2})\\
    \label{eq:theorem basic properties proof part3 eq 4}
    &\leq 2D_{\Delta^{d}}(\|f(x)-g(x)\|_{2}+\|f(x')-g(x')\|_{2}),
\end{align}
where we use the triangle inequality in \eqref{eq:theorem basic properties proof part3 eq 3} and \eqref{eq:theorem basic properties proof part3 eq 4}.
Thus,
\begin{align*}
    \|\rho_{f}-\rho_{g}\|_{\mathcal{X}^{2},\mu\otimes\mu,1}\leq 4D_{\Delta^{d}}\mu(\mathcal{X})\|f-g\|_{\mathcal{X},1}.
\end{align*}
Using this fact, we have
\begin{align}
\label{eq:theorem basic properties proof part3 eq 40}
\|\psi\circ\rho_{f}-\psi\circ\rho_{g}\|_{\mathcal{A},P_{X,X'},1}
\leq 8D_{\Delta^{d}}^{-1}\mu(\mathcal{X})\|p_{X,X'}\|_{\mathcal{X}^{2},\infty}\|f-g\|_{\mathcal{X},1}.
\end{align}

Let $g_{0}\in\mathcal{G}$ be a function such that $\|f-g_{0}\|_{\mathcal{X},1}\leq 2\varepsilon_{0}$.
By~\eqref{eq:theorem basic properties proof part3 eq 40}, we have
\begin{align}
\label{eq:theorem basic properties proof part3 eq 6}
    \|\psi\circ\rho_{f}-\psi\circ\rho_{g_{0}}\|_{\mathcal{A},P_{X,X'},1}\leq 
    16D_{\Delta^{d}}^{-1}\mu(\mathcal{X})\|p_{X,X'}\|_{\mathcal{X}^{2},\infty}\varepsilon_{0}.
\end{align}
From \eqref{eq:theorem basic properties proof part3 eq 2.1} and \eqref{eq:theorem basic properties proof part3 eq 6}, we have
\begin{align*}
    \|\psi\circ \rho_{g_{0}}-\textup{sign}\circ (2\eta-1)\|_{\mathcal{A},P_{X,X'},1}\leq 
    (2\varepsilon)\vee (32D_{\Delta^{d}}^{-1}\mu(\mathcal{X})\|p_{X,X'}\|_{\mathcal{X}^{2},\infty}\varepsilon_{0})
\end{align*}
Let $\varepsilon_{1}=(2\varepsilon)\vee (32D_{\Delta^{d}}^{-1}\mu(\mathcal{X})\|p_{X,X'}\|_{\mathcal{X}^{2},\infty}\varepsilon_{0})$.
Therefore, we have
\begin{align*}
    \sup\{P_{X,X'}(\mathcal{A})\;|\;\mathcal{A}\in\mathscr{W}_{g_{0},\varepsilon_{1},\psi}\}
    \geq\sup\{P_{X,X'}(\mathcal{A})\;|\;\mathcal{A}\in\mathscr{W}_{f,\varepsilon,\psi}\}
    \geq V.
\end{align*}
This indicates that $P$ is $(\psi,\mathcal{G},(2\varepsilon)\vee (32D_{\Delta^{d}}^{-1}\mu(\mathcal{X})\|p_{X,X'}\|_{\mathcal{X}^{2},\infty}\varepsilon_{0}),V)$-weak representable.
\end{proof}

\begin{proposition}
\label{prop:approximation property of H}
Let $\xi=(R,K,d_{1},E,\theta_{\textup{NC}},\theta_{1},\theta_{2},\theta_{3})\in\Xi$, $\mathcal{H}\subset \mathcal{F}_{0}$, and $\psi$ be the function defined in Definition~\ref{def:contrastive loss}.
For every $P\in\mathcal{P}_{\xi}$, there is some $C_{1}>0$ independent of $P$ such that $P$ is $(\psi,\mathcal{H},C_{1}\varepsilon_{\mathcal{H},f^{*}},1)$-weak representable, where $f^{*}=\sum_{i=1}^{d_{1}}g_{i}^{*}v_{i}$ is the contrastive function of $P$, and $\varepsilon_{\mathcal{H},f^{*}}$ is defined as
\begin{align}
\label{eq:inf sup approximation error}
\varepsilon_{\mathcal{H},f^{*}}=\inf_{f=\sum_{i=1}^{d_{1}}g_{i}v_{i}\in\mathcal{H}}\max_{i=1,\cdots,d_{1}}\|g_{i}-g_{i}^{*}\|_{L^{1}(\mathcal{X})}.
\end{align}
\end{proposition}
\begin{proof}
By Lemma~\ref{prop:constructability for F0} and Proposition~\ref{prop:properties of weak constructablity}, $P$ is $(\psi,\mathcal{F}_{0},0,V)$-weak representable for any $V\in [0,1]$.

Here let $f_{1},f_{2}\in\mathcal{F}_{0}$ be arbitrary, and denote by $f_{1}=\sum_{i=1}^{d_{1}}g_{i}^{(1)}v_{i}$ and $f_{2}=\sum_{i=1}^{d_{1}} g_{i}^{(2)}v_{i}$.
Then, we have
\begin{align}
    \label{eq:step 0 prop eq 00}
    \|f_{1}-f_{2}\|_{\mathcal{X},1}&\leq d^{\frac{1}{2}}\|\|f_{1}(x)-f_{2}(x)\|_{2}\|_{L^{1}(\mathcal{X})} \\
    &= d^{\frac{1}{2}}\left\|\left\|\sum_{i=1}^{d_{1}} g_{i}^{(1)}(x)v_{i}-\sum_{i=1}^{d_{1}} g_{i}^{(2)}(x)v_{i}\right\|_{2}\right\|_{L^{1}(\mathcal{X})}\nonumber\\
    \label{eq:step 0 prop eq 1}
    &\leq d^{\frac{1}{2}}\left\|\sum_{i=1}^{d_{1}}|g_{i}^{(1)}(x)-g_{i}^{(2)}(x)|\left\|v_{i}\right\|_{2}\right\|_{L^{1}(\mathcal{X})}\\
    \label{eq:step 0 prop eq 2}
    &\leq d^{\frac{1}{2}}\sum_{i=1}^{d_{1}}\left\|g_{i}^{(1)}-g_{i}^{(2)}\right\|_{L^{1}(\mathcal{X})},
\end{align}
where in~\eqref{eq:step 0 prop eq 00} we used the Cauchy-Schwarz inequality, and in \eqref{eq:step 0 prop eq 1} and \eqref{eq:step 0 prop eq 2} we used the triangle inequality.

By Lemma~\ref{lem:approximation and weak realizability}, $P$ is $(\psi,\mathcal{H},32D_{\Delta^{d}}^{-1}\|p_{X,X'}\|_{\mathcal{X}^{2},\infty}\varepsilon_{0},1)$-weak representable, where $\varepsilon_{0}=\inf_{f\in\mathcal{H}}\|f-f^{*}\|_{\mathcal{X},1}$.
Denote by $f^{*}=\sum_{i=1}^{d_{1}}g_{i}^{*}v_{i}$.
Then, by~\eqref{eq:step 0 prop eq 2} we have
\begin{align*}
\varepsilon_{0}\leq d^{\frac{1}{2}}d_{1}\inf_{f=\sum_{i=1}^{d_{1}}g_{i}v_{i}\in\mathcal{H}}\max_{i=1,\cdots,d_{1}}\|g_{i}-g_{i}^{*}\|_{L^{1}(\mathcal{X})}=d^{\frac{1}{2}}d_{1}\varepsilon_{\mathcal{H},f^{*}}.
\end{align*}
Here, note that by condition (A3) in Definition~\ref{def:main assumption}, we have
\begin{align}
\label{eq:threshold for constant c1}
\|p_{X,X'}\|_{\mathcal{X}^{2},\infty}\leq \|q\|_{\mathcal{X}^{2},\infty}+\|p_{X}p_{X'}\|_{\mathcal{X}^{2},\infty}\leq \theta_{1}^{2}+\theta_{1}^{2}\leq 2\theta_{1}^{2}.
\end{align}
Hence, define
\begin{align}
\label{eq:threshold for constant c1 eq 2}
C_{1}=64d^{\frac{1}{2}}d_{1}D_{\Delta^{d}}^{-1}\theta_{1}^{2}.
\end{align}
By Lemma~\ref{lem:trade-off for parameters}, we have that $P$ is $(\psi,\mathcal{H},C_{1}\varepsilon_{\mathcal{H},f^{*}},1)$-weak representable.
This shows the claim.
\end{proof}

\subsubsection{Step 2}
\label{appsubsubsec:step 2 of the proof of the main theorem}

We then show the approximation error bound for the indicator functions $g_{i}^{*}=\mathds{1}_{\mathcal{K}_{i}}$, $i=1,\cdots,d_{1}$.
Note that in this section, the notation $\lesssim$ abbreviates a coefficient that is a universal constant independent of the given error $\varepsilon>0$.

While the following inequality is originally shown in~\citep[proof of Lemma~4.3, p.2741]{bos2022convergence} for the $L^{\infty}(\mathcal{X})$ norm, its generalization to the $L^{s}(\mathcal{X})$-norm ($s\in[1,\infty]$) is straightforward, as it suffices to replace the $L^{\infty}(\mathcal{X})$-norm with the $L^{s}(\mathcal{X})$-norm in the original proof.
We provide the proof, for the reader's convenience.
\begin{lemma}[{Generalization of the inequality of~\citep{bos2022convergence}} for $\|\cdot\|_{L^{s}(\mathcal{X})}$]
\label{lem:bos lemma}
Let $g=(g_{1},\cdots,g_{d_{1}}):\mathcal{X}\to\mathbb{R}^{d_{1}}$ be a function such that $\|g\|_{\mathcal{X},\infty}\leq M$ for some $M\geq 0$.
Let $g^{*}:\mathcal{X}\to\mathbb{R}^{d_{1}}$, $g^{*}=(g_{1}^{*},\cdots,g_{d_{1}}^{*})$, be a function satisfying $\sum_{i=1}^{d_{1}}g_{i}^{*}(x)=1$ on $\mathcal{X}$.
Then, for any $i=1,\cdots,d_{1}$ and $s\in [1,\infty]$, it holds that
\begin{align*}
\left\|H_{i}\circ g-g_{i}^{*}\right\|_{L^{s}(\mathcal{X})}\leq \left\|\exp\circ{g_{i}}-g_{i}^{*}\right\|_{L^{s}(\mathcal{X})}+ \sum_{j=1}^{d_{1}}\left\|\exp\circ{g_{j}}-g_{j}^{*}\right\|_{L^{s}(\mathcal{X})}.
\end{align*}
\end{lemma}
\begin{proof}
Similarly to~\citep[proof of Lemma~4.3, p.2741]{bos2022convergence}, we note that
\begin{align*}
\|H_{i}\circ g-g_{i}^{*}\|_{L^{s}(\mathcal{X})}
&\leq
\|H_{i}\circ g-e^{g_{i}}\|_{L^{s}(\mathcal{X})}+\|e^{g_{i}}-g_{i}^{*}\|_{L^{s}(\mathcal{X})}\\
&\leq
\left(\sum_{j=1}^{d_{1}}\|H_{i}\circ g\|_{L^{\infty}(\mathcal{X})}\|e^{g_{j}}-g_{j}^{*}\|_{L^{s}(\mathcal{X})}\right)+\|e^{g_{i}}-g_{i}^{*}\|_{L^{s}(\mathcal{X})}\\
&\leq
\left(\sum_{j=1}^{d_{1}}\|e^{g_{j}}-g_{j}^{*}\|_{L^{s}(\mathcal{X})}\right)+\|e^{g_{i}}-g_{i}^{*}\|_{L^{s}(\mathcal{X})},
\end{align*}
where note that in the second inequality the property $\sum_{j=1}^{d_{1}}g_{j}^{*}=1$ is used.
\end{proof}

The following useful approximation property is proven by~\citet{bos2022convergence}:
\begin{lemma}[{Theorem~4.1 in~\citep{bos2022convergence}}]
\label{citetheorem:bos theorem}
Let $0<\varepsilon\leq \frac{1}{2}$, and let $\alpha>0$.
Then, there are constant $c'>0$ and $g_{\bm{W}^{*},\bm{b}^{*}}\in\mathcal{F}_{L_{\textup{log}},J_{\textup{log}},S_{\textup{log}},M_{\textup{log}},\bm{d}_{\textup{log}}}^{\textup{NN}}$ with parameters $L_{\textup{log}}\lesssim \log_{2}\varepsilon^{-1}$, $J_{\textup{log}}\leq 1$, $S_{\textup{log}}\lesssim \varepsilon^{-1/\alpha}\log_{2}\varepsilon^{-1}$, $M_{\textup{log}}\leq |\log(4\varepsilon)|\vee \log(1+4\varepsilon)$, and $\bm{d}_{\textup{log}}=(1,\lfloor c'\varepsilon^{-1/\alpha}\rfloor,\cdots,\lfloor c'\varepsilon^{-1/\alpha}\rfloor, 1)$ such that for any $s\in[0,1]$,
\begin{align}
\label{eq:bos inequality for approximation error}
|\exp\circ g_{\bm{W}^{*},\bm{b}^{*}}(s)-s|\leq 4\varepsilon.
\end{align}
\end{lemma}

To approximate indicator functions, we adapt the analyses of~\citet{petersen2018optimal} and \citet{imaizumi2019deep,imaizumi2022advantage}.
A similar approach is considered in~\citep{kim2021fast}, although we need to deal with the softmax function.
This additional step is done by applying the analyses shown by~\citet{bos2022convergence}.

The following properties of deep ReLU networks are well-known (see, e.g,~\citep{yarotsky2017error,schmidt-hieber2020nonparametric,petersen2018optimal,nakada2020adaptive,bos2022convergence,imaizumi2019deep,imaizumi2022advantage}).
We provide a proof, for the reader's convenience.
\begin{lemma}
\label{lem:basic operations of neural networks}
We have the following properties:
\begin{enumerate}
    \item[(i)] Given $L,L'\in\mathbb{N}$, $J,S,M,J',S',M'\geq 0$, $\bm{d}=(d_{\textup{NN},0},\cdots,d_{\textup{NN},L})\in \mathbb{N}^{L+1}$, and $\bm{d}'=(d_{\textup{NN},0}',\cdots,d_{\textup{NN},L'}')\in\mathbb{N}^{L'+1}$, if $L=L'$ and $d_{\textup{NN},0}=d_{\textup{NN},0}'$, then for each $g_{\bm{W},\bm{b}}\in \mathcal{F}_{L,J,S,M,\bm{d}}^{\textup{NN}}$ and $g_{\bm{W}',\bm{b}'}\in\mathcal{F}_{L',J',S',M',\bm{d}'}^{\textup{NN}}$, there is $g_{\bm{W}'',\bm{b}''}\in\mathcal{F}_{L,J\vee J',S+S',M\vee M',\widetilde{\bm{d}}}^{\textup{NN}}$ with $\widetilde{\bm{d}}=(d_{\textup{NN},0},d_{\textup{NN},1}+d_{\textup{NN},1}',\cdots,d_{\textup{NN},L}+d_{\textup{NN},L}')$ such that
    \begin{align}
    \label{eq:well-known property of deep relu networks property one}
    g_{\bm{W}'',\bm{b}''}(x)=(g_{\bm{W},\bm{b}}(x),g_{\bm{W}',\bm{b}'}(x))\quad\textup{ for every }x\in\mathbb{R}^{d_{\textup{NN},0}}.
    \end{align}
    \item[(ii)] Given $L,L'\in\mathbb{N}$, $J,S,M,J',S',M'\geq 0$, $\bm{d}=(d_{\textup{NN},0},\cdots,d_{\textup{NN},L})\in\mathbb{N}^{L+1}$, and $\bm{d}'=(d_{\textup{NN},0}',\cdots,d_{\textup{NN},L'}')\in\mathbb{N}^{L'+1}$, if $d_{\textup{NN},L}=d_{\textup{NN},0}'$ and $M\leq 1$, then for each $g_{\bm{W},\bm{b}}\in\mathcal{F}_{L,J,S,M,\bm{d}}^{\textup{NN}}$ and $g_{\bm{W}',\bm{b}'}\in\mathcal{F}_{L',J',S',M',\bm{d}'}^{\textup{NN}}$, there is $g_{\bm{W}'',\bm{b}''}\in\mathcal{F}_{L+L',J\vee J',S+S',M',\widetilde{\bm{d}}}^{\textup{NN}}$ with $\widetilde{\bm{d}}=(d_{\textup{NN},0},\cdots,d_{\textup{NN},L},d_{\textup{NN},1}',\cdots,d_{\textup{NN},L'}')$ such that
    \begin{align}
    \label{eq:well-known property of deep relu networks property two}
    g_{\bm{W}'',\bm{b}''}(x)=g_{\bm{W}',\bm{b}'}\circ \bm{\sigma}_{\textup{ReLU},d_{\textup{NN},L}}\circ g_{\bm{W},\bm{b}}(x)\quad \textup{ for every }x\in\mathbb{R}^{d_{\textup{NN},0}}.
    \end{align}
\end{enumerate}
\end{lemma}
\begin{proof}
For the first claim, let $\bm{W}=(W_{1},\cdots,W_{L})$, $\bm{b}=(b_{1},\cdots,b_{L})$, $\bm{W}'=(W_{1}',\cdots,W_{L}')$, and $\bm{b}'=(b_{1}',\cdots,b_{L}')$.
Similarly to~\citep[Definition~2.7]{petersen2018optimal} and~\citep[Appendix~B.1.1]{nakada2020adaptive}, one can construct networks $g_{\bm{W}'',\bm{b}''}$ for which in each layer, the weight and the bias are defined as
\begin{align*}
W_{1}''&=
\left(
\begin{matrix}
W_{1}\\
W_{1}'
\end{matrix}
\right),\quad
b_{1}''=
\left(
b_{1},
b_{1}'
\right),
\\
W_{i}''&=
\left(
\begin{matrix}
W_{i} & \bm{O} \\
\bm{O} & W_{i}'
\end{matrix}
\right),\quad
b_{i}''=
\left(
b_{i},
b_{i}'
\right),\quad i=2,\cdots,L,
\end{align*}
where $\bm{O}$ denotes the zero matrix.
This function satisfies~\eqref{eq:well-known property of deep relu networks property one}.
The second claim (ii) is proven by constructing networks as in~\eqref{eq:well-known property of deep relu networks property two}.
\end{proof}
Note that the composition operation in Lemma~\ref{lem:basic operations of neural networks}--(ii) is slightly different from~\citep[Definition~2.5]{petersen2018optimal} and~\citep[Appendix~B.1.1]{nakada2020adaptive}, where the difference is that the implementation of identity function in~\citep[Lemma~2.3]{petersen2018optimal} is not used in Lemma~\ref{lem:basic operations of neural networks}--(ii).
This difference is due to the setting of the estimation problem considered in the current work.
Indeed, it suffices to approximate indicator functions, which are always non-negative functions.
In particular, we use the following property of $\sigma_{\textup{ReLU}}$:
For any $s\in\mathbb{R}$ and $s'\geq 0$, it holds that
\begin{align}
\label{eq:inequality of relu function}
|\sigma_{\textup{ReLU}}(s)-s'|\leq |s-s'|.
\end{align}

\citet{petersen2018optimal} show the following fact.
\begin{lemma}[{Lemma~3.4 in~\citep{petersen2018optimal}}]
\label{lem:petersen voigtlaender lemma}
Given any $\alpha>0$, $R> 0$, $K\in\mathbb{N}\setminus \{1\}$, denote by $\mathcal{C}_{R}^{\alpha,K-1}([-2^{-1},2^{-1}]^{K-1})$, the ball of H\"{o}lder space on $[-2^{-1},2^{-1}]^{K-1}$ for which its center is the origin, and its radius is $R$.
Let $\mathcal{K}\subset[-2^{-1},2^{-1}]^{K}$ be any set such that there are a function $h\in\mathcal{C}_{R}^{\alpha,K-1}([-2^{-1},2^{-1}]^{K-1})$ and a permutation $\pi$ on $\{1,\cdots,K\}$ satisfying
\begin{align*}
\mathcal{K}=\{x\in[-2^{-1},2^{-1}]^{K}\;|\; x_{\pi(1)}\geq -h(x_{\setminus \pi(1)})\}.
\end{align*}
Then, for each $\varepsilon\in (0,2^{-1})$ and $s>0$, there are a constant $c\in\mathbb{N}$ independent of $\varepsilon$, a finite subset $\mathcal{W}\subset \mathbb{R}$, and deep ReLU networks $g_{\bm{W},\bm{b}}\in\mathcal{F}_{L,J,S,M,\bm{d}_{1}}^{\textup{NN}}$ with $L\leq L_{0}\in\mathbb{N}$ ($L_{0}$ is a universal constant), $J\leq \varepsilon^{-c}$, $S\lesssim \varepsilon^{-s(K-1)/\alpha}$, $M=1$, and $\bm{d}_{1}=(K,d_{\textup{NN},1},\cdots,d_{\textup{NN},L-1},1)\in\mathbb{N}^{L+1}$, such that all of $c$, $\mathcal{W}$, $L,J,S,M$, and $\bm{d}_{1}$ are independent of $\mathcal{K}$, every entry in $\bm{W}$ or $\bm{b}$ belongs to $\mathcal{W}$, and it holds that
\begin{align*}
\|g_{\bm{W},\bm{b}}-\mathds{1}_{\mathcal{K}}\|_{L^{s}([-2^{-1},2^{-1}]^{K})}<\varepsilon.
\end{align*}
\end{lemma}
The following Lemma~\ref{lem:imaizumi fukumizu lemma} is a generalization of the fact proven in~\citep[Appendix~B.1]{imaizumi2019deep} for the $L^{2}(\mathcal{X})$-norm.
In particular, \citet[Appendix~B.1]{imaizumi2019deep} prove it by combining~\citep[Proposition~3]{yarotsky2017error} and~\citep[Lemma~A.3 and Lemma~3.4]{petersen2018optimal} for the $L^{2}(\mathcal{X})$-norm.
Therefore, its generalization to the $L^{s}(\mathcal{X})$-norm with $s\geq 1$ is straightforward, as it suffices to apply Lemma~\ref{lem:petersen voigtlaender lemma} with any $s\geq 1$ in the proof of~\citep[Appendix~B.1]{imaizumi2019deep}.
For the reader's convenience, we provide the proof.
\begin{lemma}[{Generalization of~\citep[Appendix~B.1]{imaizumi2019deep} for $\|\cdot\|_{L^{s}(\mathcal{X})}$}]
\label{lem:imaizumi fukumizu lemma}
Let $\alpha>0$, $R> 0$, and $K\in\mathbb{N}\setminus \{1\}$.
Let $\mathcal{K}\subset\mathcal{X}$ be any subset such that there are functions $h_{1},\cdots,h_{E}\in\mathcal{C}_{R}^{\alpha,K-1}$, $j_{1},\cdots,j_{E}\in\{1,\cdots,K\}$, and $s_{1},\cdots,s_{E}\in\{1,-1\}$ satisfying
\begin{align*}
\mathcal{K}=\bigcap_{i=1}^{E}\{x\in\mathcal{X}\;|\; s_{i}x_{j_{i}}\geq s_{i} h_{i}(x_{\setminus j_{i}})\}.
\end{align*}
Then, for any $\varepsilon\in(0,2^{-1})$ and $s\geq 1$, there are a constant $c>0$, a universal constant $L_{0}\in\mathbb{N}$, and $g_{\bm{W},\bm{b}}\in\mathcal{F}_{L,J,S,M,\bm{d}_{1}}^{\textup{NN}}$ with $L\leq L_{0}$, $J\lesssim \varepsilon^{-c}$, $S\lesssim \varepsilon^{-s(K-1)/\alpha}$, $M\lesssim 1$, and $\bm{d}_{1}=(K,d_{\textup{NN},1},\cdots,d_{\textup{NN},L-1},1)\in\mathbb{N}^{L+1}$, such that all of $c$, $L,J,S,M$, and $\bm{d}_{1}$ are independent of $\mathcal{K}$, and it holds that
\begin{align*}
\|g_{\bm{W},\bm{b}}-\mathds{1}_{\mathcal{K}}\|_{L^{s}(\mathcal{X})}<\varepsilon.
\end{align*}
\end{lemma}
\begin{proof}
For each $i=1,\cdots,E$, define
\begin{align*}
\mathcal{J}_{i}=\{x\in\mathcal{X}\;|\; s_{i}x_{j_{i}}\geq s_{i}h_{i}(x_{\setminus j_{i}})\}.
\end{align*}
Similarly to~\citep[Appendix~B.1]{imaizumi2019deep}, the proof consists of the following two steps: First, for any function $g=g_{0}\circ (g_{1},\cdots,g_{E})$ defined with continuous functions $g_{0}:[0,1]^{E}\to\mathbb{R}$ and $g_{1},\cdots,g_{E}:\mathcal{X}\to [0,1]$, we note that
\begin{align}
\label{eq:imaizumi fukumizu lemma proof eq 1}
&\|g-\mathds{1}_{\mathcal{K}}\|_{L^{s}(\mathcal{X})}\nonumber\\
&\leq
\|g_{1}g_{2}\cdots g_{E}-\mathds{1}_{\mathcal{K}}\|_{L^{s}(\mathcal{X})}+\|g_{0}(t_{1},\cdots,t_{E})-t_{1}t_{2}\cdots t_{E}\|_{L^{\infty}([0,1]^{E})}\nonumber\\
&\leq
\sum_{i=1}^{E}\|g_{i}-\mathds{1}_{\mathcal{J}_{i}}\|_{L^{s}(\mathcal{X})}+\|g_{0}(t_{1},\cdots,t_{E})-t_{1}t_{2}\cdots t_{E}\|_{L^{\infty}([0,1]^{E})},
\end{align}
where $t_{1}t_{2}\cdots t_{E}$ denotes the function $(t_{1},t_{2},\cdots,t_{E})\mapsto t_{1}t_{2}\cdots t_{E}$.
In the next step, we apply Lemma~\ref{lem:petersen voigtlaender lemma} and Lemma~1 in~\citep{imaizumi2019deep} to conclude the proof.
To this end, we note that when $s_{i}=1$, Lemma~\ref{lem:petersen voigtlaender lemma} is directly applicable by using transform $[-2^{-1},2^{-1}]^{K}\ni x\mapsto x+2^{-1}\bm{1}_{K}\in \mathcal{X}$, where $\bm{1}_{K}=(1,\cdots,1)\in\mathcal{X}$.
When $s_{i}=-1$, it suffices to note that by property (P2) in Remark~\ref{rem:the property p2}, for any ReLU networks $g_{\bm{W},\bm{b}}$, it holds that
\begin{align}
\label{eq:imaizumi fukumizu lemma proof eq 2}
&\|1-g_{\bm{W},\bm{b}}-\mathds{1}_{\{x\in\mathcal{X}\;|\; -x_{j_{i}}\geq -h_{i}(x_{\setminus j_{i}})\}}\|_{L^{s}(\mathcal{X})}\nonumber \\
&=\|g_{\bm{W},\bm{b}}-\mathds{1}_{\{x\in\mathcal{X}\;|\; x_{j_{i}}\geq h_{i}(x_{\setminus j_{i}})\}}\|_{L^{s}(\mathcal{X})}.
\end{align}
Thus, by Lemma~\ref{lem:petersen voigtlaender lemma}, for each $i=1,\cdots,E$, we can take some deep ReLU networks $g_{\bm{W}_{i},\bm{b}_{i}}\in\mathcal{F}_{L_{1},J_{1},S_{1},M_{1},\bm{d}_{1,1}}^{\textup{NN}}$ satisfying $\|g_{\bm{W}_{i},\bm{b}_{i}}-\mathds{1}_{\mathcal{J}_{i}}\|_{L^{s}(\mathcal{X})}<\varepsilon/(2E)$ with some $c>0$, $L_{1}\in\mathbb{N}$, $J_{1}\lesssim \varepsilon^{-c}$, $S_{1}\lesssim \varepsilon^{-s(K-1)/\alpha}$, $M_{1}=1$, and $\bm{d}_{1,1}=(K,d_{\textup{NN},1}^{(1)},\cdots,d_{\textup{NN},L_{1}-1}^{(1)},1)$.
By Lemma~1 in~\citep{imaizumi2019deep}, one can take some deep ReLU networks $g_{\bm{W}_{0},\bm{b}_{0}}\in\mathcal{F}_{L_{2},J_{2},S_{2},M_{2},\bm{d}_{2}}^{\textup{NN}}$ satisfying that
\begin{align*}
\|g_{\bm{W}_{0},\bm{b}_{0}}(t_{1},\cdots,t_{E})-t_{1}t_{2}\cdots t_{E}\|_{L^{\infty}([0,1]^{E})}<\frac{\varepsilon}{2},
\end{align*}
with some $c'>0$ and parameters $L_{2}\in\mathbb{N}$, $J_{2}\leq \varepsilon^{-c'}$, $S_{2}\lesssim \varepsilon^{-s(K-1)/\alpha}$, $M_{2}\lesssim 1$, and $\bm{d}_{2}=(E,d_{\textup{NN},1}^{(2)},\cdots,d_{\textup{NN},L_{2}-1}^{(2)},1)$.
By Lemma~\ref{lem:basic operations of neural networks} and~\eqref{eq:inequality of relu function}, we can construct networks $g_{\bm{W},\bm{b}}\in\mathcal{F}_{L,J,S,M,\bm{d}_{1}}^{\textup{NN}}$ using $g_{\bm{W}_{0},\bm{b}_{0}},\cdots,g_{\bm{W}_{E},\bm{b}_{E}}$ to obtain the claim.
\end{proof}

We obtain the following approximation error bounds:
\begin{proposition}
\label{prop:approximation error bound}
Let $\alpha>0$, $\tau\geq 1$, $\xi=(R,K,d_{1},E,\theta_{\textup{NC}},\theta_{1},\theta_{2},\theta_{3})\in \Xi$, $0<\varepsilon<2^{-1}$, $P\in\mathcal{P}_{\alpha,\tau,\xi}$, and let $\mathscr{S}_{P}=\{\mathcal{K}_{i}\}_{i=1}^{d_{1}}$.
Denote by $g_{i}^{*}=\mathds{1}_{\mathcal{K}_{i}}$, $i=1,\cdots,d_{1}$.
Then, there are ReLU networks $g_{\bm{W}^{*},\bm{b}^{*}}\in\mathcal{F}_{L^{*},J^{*},S^{*},M^{*},\bm{d}^{*}}^{\textup{NN}}$ with a constant $c\geq 0$ and parameters $L^{*}\lesssim \log_{2}\varepsilon^{-1}$, $1\leq J^{*}\lesssim \varepsilon^{-c}$, $S^{*}\lesssim \varepsilon^{-(K-1)/\alpha}\log_{2}\varepsilon^{-1}$, $M^{*}\lesssim |\log(4\varepsilon)|\vee 1$, and $\bm{d}^{*}=(K,d_{\textup{NN},1}^{*},\cdots,d_{\textup{NN},L^{*}-1}^{*},d_{1})\in\mathbb{N}^{L^{*}+1}$ for which all of $c$, $L^{*},J^{*},S^{*},M^{*}$, and $\bm{d}^{*}$ are independent of $P$, such that for every $i=1,\cdots,d_{1}$,
\begin{align*}
    \|H_{i}\circ g_{\bm{W}^{*},\bm{b}^{*}}-g_{i}^{*}\|_{L^{1}(\mathcal{X})}\leq \varepsilon.
\end{align*}
\end{proposition}
\begin{proof}
We basically follow the proof of~Lemma~4.3 in~\citep{bos2022convergence}.

For some neural networks $g_{\bm{W}_{1,i},\bm{b}_{1,i}}:[0,1]^{K}\to [-1,1]$ and $g_{\bm{W}_{2,i},\bm{b}_{2,i}}:[0,1]\to\mathbb{R}$ for $i=1,\cdots,d_{1}$, define $g_{\bm{W}_{i},\bm{b}_{i}}=g_{\bm{W}_{2,i},\bm{b}_{2,i}}\circ\sigma_{\textup{ReLU}}\circ g_{\bm{W}_{1,i},\bm{b}_{1,i}}$.
Following the approach of~Lemma~4.3 in~\citep{bos2022convergence}, we evaluate every quantity $\|\exp\circ g_{\bm{W}_{i},\bm{b}_{i}}-g_{i}^{*}\|_{L^{1}(\mathcal{X})}$, $i=1,\cdots,d_{1}$, as follows:
\begin{align}
\label{eq:lemma coordinate wise approximation eq 1}
&\left\|\exp\circ g_{\bm{W}_{i},\bm{b}_{i}}-g_{i}^{*}\right\|_{L^{1}(\mathcal{X})}\nonumber\\
&\leq 
\left\|\exp\circ g_{\bm{W}_{i},\bm{b}_{i}}-\sigma_{\textup{ReLU}}\circ g_{\bm{W}_{1,i},\bm{b}_{1,i}}\right\|_{L^{1}(\mathcal{X})}
+
\left\|\sigma_{\textup{ReLU}}\circ g_{\bm{W}_{1,i},\bm{b}_{1,i}}-g_{i}^{*}\right\|_{L^{1}(\mathcal{X})}.
\end{align}
Let $g:=(g_{\bm{W}_{1},\bm{b}_{1}},\cdots,g_{\bm{W}_{d_{1}},\bm{b}_{d_{1}}})$.
Lemma~\ref{lem:bos lemma} and \eqref{eq:lemma coordinate wise approximation eq 1} assert that
\begin{align}
\label{eq:lemma coordinate wise approximation eq 2}
&\left\|H_{i}\circ g-g_{i}^{*}\right\|_{L^{1}(\mathcal{X})}\nonumber\\
&\leq 2\sum_{j=1}^{d_{1}}\left\|\exp\circ g_{\bm{W}_{j},\bm{b}_{j}}-g_{j}^{*}\right\|_{L^{1}(\mathcal{X})}\\
\label{eq:lemma coordinate wise approximation eq 3}
&\leq
2\sum_{j=1}^{d_{1}}(\left\|\exp\circ g_{\bm{W}_{j},\bm{b}_{j}}-\sigma_{\textup{ReLU}}\circ g_{\bm{W}_{1,j},\bm{b}_{1,j}}\right\|_{L^{1}(\mathcal{X})}
+
\left\|g_{\bm{W}_{1,j},\bm{b}_{1,j}}-g_{j}^{*}\right\|_{L^{1}(\mathcal{X})}),
\end{align}
where \eqref{eq:lemma coordinate wise approximation eq 2} follows from Lemma~\ref{lem:bos lemma}.
\eqref{eq:lemma coordinate wise approximation eq 3} is due to the combination of~\eqref{eq:lemma coordinate wise approximation eq 1} and~\eqref{eq:inequality of relu function}.

Regarding the second term in~\eqref{eq:lemma coordinate wise approximation eq 3}, similarly to~\citep[Lemma~5]{imaizumi2022advantage}, we note that for any disjoint partition $\{\mathcal{I}_{i}\}_{i=1}^{d_{1}}$ of $\{1,-1\}^{E}$, functions $h_{1},\cdots,h_{E}\in\mathcal{C}_{R}^{\alpha,K-1}$, and indices $j_{1},\cdots,j_{E}\in\{1,\cdots,K\}$ satisfying condition (P1) in Definition~\ref{def:class of smooth partitions}, and the sum of ReLU networks $g_{\widetilde{\bm{W}}_{1,j},\widetilde{\bm{b}}_{1,j}}=\sum_{\bm{s}\in\mathcal{I}_{j}}\sigma_{\textup{ReLU}}\circ g_{\bm{W}_{1,j,\bm{s}},\bm{b}_{1,j,\bm{s}}}$ defined with any $g_{\bm{W}_{1,j,\bm{s}},\bm{b}_{1,j,\bm{s}}}\in \mathcal{F}_{L,J,S,M,\bm{d}}^{\textup{NN}}$, $L\in\mathbb{N}$, $J,S,M\geq 0$, and $\bm{d}=(K,d_{\textup{NN},1},\cdots,d_{\textup{NN},L-1},1)$, the following inequality holds:
\begin{align}
\label{eq:neural net approximation partition eq 1}
\|g_{\widetilde{\bm{W}}_{1,j},\widetilde{\bm{b}}_{1,j}}-g_{j}^{*}\|_{L^{1}(\mathcal{X})}&\leq \sum_{\bm{s}=(s_{1},\cdots,s_{E})\in\mathcal{I}_{j}}\|g_{\bm{W}_{1,j,\bm{s}},\bm{b}_{1,j,\bm{s}}}-\mathds{1}_{\bigcap_{k=1}^{E}\mathcal{L}_{s_{k},h_{k},j_{k}}}\|_{L^{1}(\mathcal{X})}\\
\label{eq:neural net approximation partition eq 2}
&=
\sum_{\bm{s}=(s_{1},\cdots,s_{E})\in\mathcal{I}_{j}}\|g_{\bm{W}_{1,j,\bm{s}},\bm{b}_{1,j,\bm{s}}}-\mathds{1}_{\bigcap_{k=1}^{E}\textup{cl}(\mathcal{L}_{s_{k},h_{k},j_{k}})}\|_{L^{1}(\mathcal{X})},
\end{align}
where in~\eqref{eq:neural net approximation partition eq 1} we use the triangle inequality and~\eqref{eq:inequality of relu function}, and in~\eqref{eq:neural net approximation partition eq 2} $\textup{cl}(\cdot)$ denotes the closure of the given set, and property (P2) in Remark~\ref{rem:the property p2} is used.
Here, note that $|\mathcal{I}_{j}|\leq |\{1,-1\}^{E}|=2^{E}$ for any $j\in\{1,\cdots,d_{1}\}$.
By Lemma~\ref{lem:imaizumi fukumizu lemma}, for every $\bm{s}\in\mathcal{I}_{j}$, there are some $c\in\mathbb{N}$ and deep ReLU networks $g_{\widetilde{\bm{W}}_{1,j,\bm{s}}^{*},\widetilde{\bm{b}}_{1,j,\bm{s}}^{*}}\in\mathcal{F}_{\widetilde{L}_{1},\widetilde{J}_{1},\widetilde{S}_{1},\widetilde{M}_{1},\widetilde{\bm{d}}_{1,1}}^{\textup{NN}}$ with $\widetilde{L}_{1}\in \mathbb{N}$, $\widetilde{J}_{1}\lesssim \varepsilon^{-c}$, $\widetilde{S}_{1}\lesssim \varepsilon^{-(K-1)/\alpha}$, $\widetilde{M}_{1}\lesssim 1$, and $\widetilde{\bm{d}}_{1,1}=(K,\widetilde{d}_{\textup{NN},1},\cdots,\widetilde{d}_{\textup{NN},\widetilde{L}_{1}-1},1)$ such that
\begin{align}
\label{eq:lemma application of petersen lemma}
\|g_{\widetilde{\bm{W}}_{1,j,\bm{s}}^{*},\widetilde{\bm{b}}_{1,j,\bm{s}}^{*}}-\mathds{1}_{\bigcap_{k=1}^{E}\textup{cl}(\mathcal{L}_{s_{k},h_{k},j_{k}})}\|_{L^{1}(\mathcal{X})}< \frac{\varepsilon}{2^{E+1}8d_{1}}.
\end{align}
Similarly to~\citep[proof of Theorem~4.1, p.2761]{bos2022convergence}, we use the projection function $\mathbb{R}\ni s\mapsto \sigma_{\textup{ReLU}}(s)-\sigma_{\textup{ReLU}}(s-1)\in [0,1]$ to define the networks $g_{\bm{W}_{1,j,\bm{s}}^{*},\bm{b}_{1,j,\bm{s}}^{*}}$ as
\begin{align*}
g_{\bm{W}_{1,j,\bm{s}}^{*},\bm{b}_{1,j,\bm{s}}^{*}}(x)=\sigma_{\textup{ReLU}}(g_{\widetilde{\bm{W}}_{1,j,\bm{s}}^{*},\widetilde{\bm{b}}_{1,j,\bm{s}}^{*}}(x))-\sigma_{\textup{ReLU}}(g_{\widetilde{\bm{W}}_{1,j,\bm{s}}^{*},\widetilde{\bm{b}}_{1,j,\bm{s}}^{*}}(x)-1).
\end{align*}
Since the range of the indicator function is included in $\{0,1\}$, we have
\begin{align}
&\|g_{\bm{W}_{1,j,\bm{s}}^{*},\bm{b}_{1,j,\bm{s}}^{*}}-\mathds{1}_{\bigcap_{k=1}^{E}}\textup{cl}(\mathcal{L}_{s_{k},h_{k},j_{k}})\|_{L^{1}(\mathcal{X})}\nonumber\\
&\leq
\|\sigma_{\textup{ReLU}}\circ g_{\widetilde{\bm{W}}_{1,j,\bm{s}}^{*},\widetilde{\bm{b}}_{1,j,\bm{s}}^{*}}-\sigma_{\textup{ReLU}}\circ \mathds{1}_{\bigcap_{k=1}^{E}\textup{cl}(\mathcal{L}_{s_{k},h_{k},j_{k}})}\|_{L^{1}(\mathcal{X})}\nonumber\\
&\quad
+\|\sigma_{\textup{ReLU}}\circ (g_{\widetilde{\bm{W}}_{1,j,\bm{s}}^{*},\widetilde{\bm{b}}_{1,j,\bm{s}}^{*}}-1)-\sigma_{\textup{ReLU}}\circ (\mathds{1}_{\bigcap_{k=1}^{E}\textup{cl}(\mathcal{L}_{s_{k},h_{k},j_{k}})}-1)\|_{L^{1}(\mathcal{X})}\nonumber\\
\label{eq:approximation 1 j s eq 1}
&\leq
2\|g_{\widetilde{\bm{W}}_{1,j,\bm{s}}^{*},\widetilde{\bm{b}}_{1,j,\bm{s}}^{*}}-\mathds{1}_{\bigcap_{k=1}^{E}\textup{cl}(\mathcal{L}_{s_{k},h_{k},j_{k}})}\|_{L^{1}(\mathcal{X})}\\
\label{eq:approximation 1 j s eq 2}
&<
\frac{\varepsilon}{2^{E}8d_{1}},
\end{align}
where~\eqref{eq:approximation 1 j s eq 1} is due to the Lipschitz continuity of $\sigma_{\textup{ReLU}}$, and~\eqref{eq:approximation 1 j s eq 2} is due to~\eqref{eq:lemma application of petersen lemma}.
By the definition, it holds that $g_{\bm{W}_{1,j,\bm{s}}^{*},\bm{b}_{1,j,\bm{s}}^{*}}\in \mathcal{F}_{L_{1}',J_{1}',S_{1}',M_{1}',\bm{d}_{1,1}'}^{\textup{NN}}$ with $L_{1}'\in\mathbb{N}$, $J_{1}'\lesssim \varepsilon^{-c}$, $S_{1}'\lesssim \varepsilon^{-(K-1)/\alpha}$, $M_{1}'\leq 1$, and $\bm{d}_{1,1}'=(K,d_{\textup{NN},1}',\cdots,d_{\textup{NN},L_{1}'-1}',1)\in\mathbb{N}^{L_{1}'+1}$, for every $\bm{s}\in\mathcal{I}_{j}$.
By Lemma~\ref{lem:basic operations of neural networks}, there is a network $g_{\widetilde{\bm{W}}_{1,j}^{*},\widetilde{\bm{b}}_{1,j}^{*}}\in\mathcal{F}_{\widetilde{L}_{1}',\widetilde{J}_{1}',\widetilde{S}_{1}',\widetilde{M}_{1}',\widetilde{\bm{d}}_{1,1}'}^{\textup{NN}}$ with $\widetilde{L}_{1}'\in\mathbb{N}$, $\widetilde{J}_{1}'\lesssim \varepsilon^{-c}$, $\widetilde{S}_{1}'\lesssim \varepsilon^{-(K-1)/\alpha}$, $\widetilde{M}_{1}'\lesssim 1$, and $\widetilde{\bm{d}}_{1,1}'=(K,\widetilde{d}_{\textup{NN},1}',\cdots,\widetilde{d}_{\textup{NN},\widetilde{L}_{1}'-1}',1)\in\mathbb{N}^{\widetilde{L}_{1}'+1}$ such that
\begin{align*}
g_{\widetilde{\bm{W}}_{1,j}^{*},\widetilde{\bm{b}}_{1,j}^{*}}=\sum_{\bm{s}\in\mathcal{I}_{j}}\sigma_{\textup{ReLU}}\circ g_{\bm{W}_{1,j,\bm{s}}^{*},\bm{b}_{1,j,\bm{s}}^{*}}.
\end{align*}
Combining~\eqref{eq:neural net approximation partition eq 1} -- \eqref{eq:approximation 1 j s eq 2}, we have
\begin{align}
\label{eq:approximation property 1 j}
\|g_{\widetilde{\bm{W}}_{1,j}^{*},\widetilde{\bm{b}}_{1,j}^{*}}-g_{j}^{*}\|_{L^{1}(\mathcal{X})}<\frac{\varepsilon}{8d_{1}}.
\end{align}
We again use the projection function $\mathbb{R}\ni s\mapsto \sigma_{\textup{ReLU}}(s)-\sigma_{\textup{ReLU}}(s-1)\in [0,1]$ to define $g_{\bm{W}_{1,j}^{*},\bm{b}_{1,j}^{*}}$, following~\citep[proof of Theorem~4.1, p.2761]{bos2022convergence}:
\begin{align*}
g_{\bm{W}_{1,j}^{*},\bm{b}_{1,j}^{*}}(x)=\sigma_{\textup{ReLU}}(g_{\widetilde{\bm{W}}_{1,j}^{*},\widetilde{\bm{b}}_{1,j}^{*}}(x))-\sigma_{\textup{ReLU}}(g_{\widetilde{\bm{W}}_{1,j}^{*},\widetilde{\bm{b}}_{1,j}^{*}}(x)-1).
\end{align*}
Note that $g_{j}^{*}(x)=\sigma_{\textup{ReLU}}(g_{j}^{*}(x))-\sigma_{\textup{ReLU}}(g_{j}^{*}(x)-1)$ for any $x\in\mathcal{X}$, since $g_{j}^{*}(x)\in [0,1]$.
Hence, we have
\begin{align*}
&\|g_{\bm{W}_{1,j}^{*},\bm{b}_{1,j}^{*}}-g_{j}^{*}\|_{L^{1}(\mathcal{X})}\\
&\leq 
\|\sigma_{\textup{ReLU}}\circ g_{\widetilde{\bm{W}}_{1,j}^{*},\widetilde{\bm{b}}_{1,j}^{*}}-\sigma_{\textup{ReLU}}\circ g_{j}^{*}\|_{L^{1}(\mathcal{X})}\\
&\quad
+\|\sigma_{\textup{ReLU}}\circ (g_{\widetilde{\bm{W}}_{1,j}^{*},\widetilde{\bm{b}}_{1,j}^{*}}-1)-\sigma_{\textup{ReLU}}\circ (g_{j}^{*}-1)\|_{L^{1}(\mathcal{X})}\\
&\leq
2\|g_{\widetilde{\bm{W}}_{1,j}^{*},\widetilde{\bm{b}}_{1,j}^{*}}-g_{j}^{*}\|_{L^{1}(\mathcal{X})},
\end{align*}
where the first inequality is due to the triangle inequality, and the second inequality is due to the Lipschitz continuity of $\sigma_{\textup{ReLU}}$.
By the definition of $g_{\widetilde{\bm{W}}_{1,j}^{*},\widetilde{\bm{b}}_{1,j}^{*}}$ and its property~\eqref{eq:approximation property 1 j}, we can take some $L_{1}\in\mathbb{N}$, $J_{1}\lesssim \varepsilon^{-c}$, $S_{1}\lesssim \varepsilon^{-(K-1)/\alpha}$, $M_{1}\leq 1$, and $\bm{d}_{1,1}=(K,d_{\textup{NN},1},\cdots,d_{\textup{NN},L_{1}-1},1)\in \mathbb{N}^{L_{1}+1}$ for which it holds that $g_{\bm{W}_{1,j}^{*},\bm{b}_{1,j}^{*}}\in\mathcal{F}_{L_{1},J_{1},S_{1},M_{1},\bm{d}_{1,1}}^{\textup{NN}}$, and
\begin{align}
\label{eq:approximation error for indicator functions proposition}
\|g_{\bm{W}_{1,j}^{*},\bm{b}_{1,j}^{*}}-g_{j}^{*}\|_{L^{1}(\mathcal{X})}<\frac{\varepsilon}{4d_{1}}.
\end{align}

Regarding the first term in~\eqref{eq:lemma coordinate wise approximation eq 3}, by Lemma~\ref{citetheorem:bos theorem}, one can take deep ReLU networks $g_{\bm{W}_{2,1}^{*},\bm{b}_{2,1}^{*}},\cdots,g_{\bm{W}_{2,d_{1}}^{*},\bm{b}_{2,d_{1}}^{*}}\in\mathcal{F}_{L_{2},J_{2},S_{2},M_{2},\bm{d}_{1,2}}^{\textup{NN}}$ with $L_{2},J_{2},S_{2},M_{2},\bm{d}_{1,2}$ satisfying the conditions in Lemma~\ref{citetheorem:bos theorem} so that \eqref{eq:bos inequality for approximation error} is satisfied for the error $\frac{\varepsilon}{16d_{1}}$.
Since $|g_{\bm{W}_{1,i}^{*},\bm{b}_{1,i}^{*}}|\leq 1$ by the condition $M_{1}=1$, by Lemma~\ref{citetheorem:bos theorem}, it holds that
\begin{align}
\label{eq:bos inequality application}
\|\exp\circ g_{\bm{W}_{2,j}^{*},\bm{b}_{2,j}^{*}}\circ\sigma_{\textup{ReLU}}\circ g_{\bm{W}_{1,j}^{*},\bm{b}_{1,j}^{*}}-\sigma_{\textup{ReLU}}\circ g_{\bm{W}_{1,j}^{*},\bm{b}_{1,j}^{*}}\|_{L^{1}(\mathcal{X})}\leq \frac{\varepsilon}{4d_{1}}.
\end{align}

Define $g_{\bm{W}_{j}^{*},\bm{b}_{j}^{*}}=g_{\bm{W}_{2,j}^{*},\bm{b}_{2,j}^{*}}\circ\sigma_{\textup{ReLU}}\circ g_{\bm{W}_{1,j}^{*},\bm{b}_{1,j}^{*}}$ for each $j=1,\cdots,d_{1}$.
By Lemma~\ref{lem:basic operations of neural networks}, we can take networks $g_{\bm{W}^{*},\bm{b}^{*}}\in\mathcal{F}_{L^{*},J^{*},S^{*},M^{*},\bm{d}^{*}}^{\textup{NN}}$ for which
\begin{align*}
g_{\bm{W}^{*},\bm{b}^{*}}(x)=(g_{\bm{W}_{1}^{*},\bm{b}_{1}^{*}}(x),\cdots,g_{\bm{W}_{d_{1}}^{*},\bm{b}_{d_{1}}^{*}}(x))\quad \textup{ for every }x\in\mathcal{X},
\end{align*}
with $L^{*}\lesssim \log_{2}\varepsilon^{-1}+1\lesssim \log_{2}\varepsilon^{-1}$, $J^{*}\lesssim \varepsilon^{-c}\vee 1\lesssim \varepsilon^{-c}$, $S^{*}\lesssim \varepsilon^{-(K-1)/\alpha}(1\vee (\log_{2}\varepsilon^{-1}))\lesssim \varepsilon^{-(K-1)/\alpha}\log_{2}\varepsilon^{-1}$, $M^{*}\lesssim |\log(4\varepsilon)|\vee 1$, and $\bm{d}^{*}=(K,d_{\textup{NN},1}^{*},\cdots, d_{\textup{NN},L^{*}-1}^{*},d_{1})$.
By~\eqref{eq:lemma coordinate wise approximation eq 3}, \eqref{eq:approximation error for indicator functions proposition}, and \eqref{eq:bos inequality application}, we have
\begin{align*}
\|H_{i}\circ g_{\bm{W}^{*},\bm{b}^{*}}-g_{i}^{*}\|_{L^{1}(\mathcal{X})}\leq \varepsilon.
\end{align*}
Therefore, we obtain the claim.
\end{proof}

\subsubsection{Step 3}
\label{appsubsubsec:step 3 of the proof of the main theorem}

In this step, we show the following proposition:
\begin{proposition}
\label{thm:approximation and concentration}
Let $\xi=(R,K,d_{1},E,\theta_{\textup{NC}},\theta_{1},\theta_{2},\theta_{3})\in\Xi$, $\beta\in (0,D_{{\textup{proj}}})$, $\varepsilon>0$, and $P\in\mathcal{P}_{\xi}$.
Let $f^{*}=\sum_{i=1}^{d_{1}}g_{i}^{*}v_{i}$ be the contrastive function of $P$.
If $f=\sum_{i=1}^{d_{1}}g_{i}v_{i}\in\mathcal{F}\subset\mathcal{F}_{0}$ satisfies
\begin{align}
\label{eq:thm approximation concentration eq 1}
    \|g_{i}-g_{i}^{*}\|_{L^{1}(\mathcal{X})}\leq d_{1}^{-2}\|p_{X}\|_{L^{\infty}(\mathcal{X})}^{-1}\varepsilon\quad \forall i\in\{1,\cdots,d_{1}\},
\end{align}
then we have $f\in\mathscr{F}_{\beta,\beta^{-1}\varepsilon,P}(\mathcal{F})$.
\end{proposition}
\begin{proof}
Let $f\in\mathcal{F}$ be a function satisfying~\eqref{eq:thm approximation concentration eq 1}.
Then, we have
\begin{align}
& P_{X}(\{x\in\mathcal{X}\;|\;\|f(x)-f^{*}(x)\|_{2}\geq \beta\})\nonumber\\
\label{eq:thm approximation concentration eq 2}
&\leq 
P_{X}\left(\bigcup_{i=1}^{d_{1}}\{x\in\mathcal{X}\;|\;|g_{i}(x)-g_{i}^{*}(x)|\geq d_{1}^{-1}\beta\}\right)\\
\label{eq:thm approximation concentration eq 3}
&\leq
\sum_{i=1}^{d_{1}}d_{1}\beta^{-1}\mathbb{E}_{P_{X}}[|g_{i}-g_{i}^{*}|]\\
\label{eq:thm approximation concentration eq 4}
&\leq 
d_{1}\beta^{-1}\|p_{X}\|_{L^{\infty}(\mathcal{X})}\sum_{i=1}^{d_{1}}\|g_{i}-g_{i}^{*}\|_{L^{1}(\mathcal{X})}.
\end{align}
Here, in~\eqref{eq:thm approximation concentration eq 2} we note that for $x\in\{x'\in \mathcal{X}\;|\;\|f(x')-f^{*}(x')\|_{2}\geq \beta\}$, we have
\begin{align*}
    \beta\leq \|f(x)-f^{*}(x)\|_{2}\leq \sum_{i=1}^{d_{1}}|g_{i}(x)-g_{i}^{*}(x)|\|v_{i}\|_{2}=\sum_{i=1}^{d_{1}}|g_{i}(x)-g_{i}^{*}(x)|.
\end{align*}
Hence, there is at least one index $i\in\{1,\cdots,d_{1}\}$ such that $|g_{i}(x)-g_{i}^{*}(x)|\geq d_{1}^{-1}\beta$ holds.
In~\eqref{eq:thm approximation concentration eq 3} and \eqref{eq:thm approximation concentration eq 4}, we use Markov's inequality and H\"{o}lder's inequality, respectively.
By~\eqref{eq:thm approximation concentration eq 1} and~\eqref{eq:thm approximation concentration eq 4}, we have
\begin{align*}
    P_{X}(\{x\in\mathcal{X}\;|\; \|f(x)-f^{*}(x)\|_{2}<\beta\})\geq 1-\beta^{-1}\varepsilon,
\end{align*}
which shows the claim.
\end{proof}

\subsubsection{Step 4}
\label{appsubsubsec:step 4 of the proof of the main theorem}

To apply the results in~\citep{park2009convergence,kim2021fast}, we need to evaluate the covering numbers of the class of classifiers.
Given a compact subset $\mathcal{A}\subset\mathbb{R}^{s}$, let $\mathscr{N}(\delta,\mathcal{H},\|\cdot\|_{L^{\infty}(\mathcal{A})})$ denote the covering number of a class $\mathcal{H}$ of real-valued functions with respect to the pseudo-distance $\|\cdot\|_{L^{\infty}(\mathcal{A})}$ (see e.g.,~\citep[Definition~6.19]{steinwart2008support} for covering numbers).
Also, for any $\mathcal{F}\subset\mathcal{F}_{0}$, denote by $\rho(\mathcal{F})=\{\rho_{f}\;|\; f\in\mathcal{F}\}$.
Then, the following evaluation holds.
\begin{lemma}
\label{lem:excess risk bound lemma 1}
Let $\xi=(R,K,d_{1},E,\theta_{\textup{NC}},\theta_{1},\theta_{2},\theta_{3})\in\Xi$, $\beta>0$, $\beta_{0}\geq 0$, $P\in\mathcal{P}_{\xi}$, $L\in\mathbb{N}$, $J,S,M\geq 0$, and $\bm{d}=(K,d_{\textup{NN},1},\cdots,d_{\textup{NN},L-1},d_{1})\in\mathbb{N}^{L+1}$.
Consider a function class $\mathcal{F}'\subset \mathcal{F}_{L,J,S,M,\bm{d}}^{\Delta^{d}\textup{-NN}}$.
Then, for any $\delta\geq 0$, there exists a constant $C_{6}>0$ independent of $n$ and $\delta$ such that we have
\begin{align*}
    &\log\mathscr{N}(\delta,\psi\circ\rho (\mathscr{F}_{\beta,\beta_{0},P}(\mathcal{F}')),\|\cdot\|_{L^{\infty}(\mathcal{X}^{2})})\\
    &\leq d_{1}^{2}\log\mathscr{N}(C_{6}\delta,\mathcal{F}_{L,J,S,M,\bm{d}_{1}}^{\textup{NN}},\|\cdot\|_{L^{\infty}(\mathcal{X})}),
\end{align*}
where $\bm{d}_{1}=(K,d_{\textup{NN},1},\cdots,d_{\textup{NN},L-1},1)$.
Note that in the above statement, we can take $C_{6}=2^{-4}dd_{1}^{-1}e^{-4M}D_{\Delta^{d}}^{2}$.
\end{lemma}
\begin{proof}
Note that $\psi$ is Lipschitz continuous.
Thus, we have
\begin{align}
\label{eq:excess risk bound lemma 1 eq 2}
&\mathscr{N}(\delta,\psi\circ\rho (\mathscr{F}_{\beta,\beta_{0},P}(\mathcal{F}')),\|\cdot\|_{L^{\infty}(\mathcal{X}^{2})})\nonumber\\
&\leq
\mathscr{N}(2^{-1}D_{\Delta^{d}}^{2}\delta,\rho (\mathscr{F}_{\beta,\beta_{0},P}(\mathcal{F}')),\|\cdot\|_{L^{\infty}(\mathcal{X}^{2})}).
\end{align}
Then, we note that
\begin{align}
\label{eq:excess risk bound lemma 1 eq 3}
&\mathscr{N}(2^{-1}D_{\Delta^{d}}^{2}\delta,\rho (\mathscr{F}_{\beta,\beta_{0},P}(\mathcal{F}')),\|\cdot\|_{L^{\infty}(\mathcal{X}^{2})})\nonumber\\
&= \mathscr{N}(2^{-1}D_{\Delta^{d}}^{2}\delta,\{(x,x')\mapsto \|f(x)-f(x')\|_{2}^{2}\;|\;f\in\mathscr{F}_{\beta,\beta_{0},P}(\mathcal{F}')\},\|\cdot\|_{L^{\infty}(\mathcal{X}^{2})}).
\end{align}
Here, given any $f\in\mathcal{F}_{L,J,S,M,\bm{d}}^{\Delta^{d}\textup{-NN}}$, note that $f(x)=\sum_{j=1}^{d_{1}}(H_{j}\circ g(x)) v_{j}$ and $f(x')=\sum_{j=1}^{d_{1}}(H_{j}\circ g(x'))v_{j}$ for some $g\in\mathcal{F}_{L,J,S,M,\bm{d}}^{\textup{NN}}$.
Hence, we write as $c_{j}(f,x):=H_{j}\circ g(x)$, for convenience.
By Proposition~\ref{prop:simple fact for risks}--(i), we have
\begin{align*}
\|f(x)-f(x')\|_{2}^{2}
=\frac{d_{1}}{d}\sum_{j=1}^{d_{1}}(c_{j}(f,x)-c_{j}(f,x'))^{2}.
\end{align*}
Thus, we obtain
\begin{align}
\label{eq:excess risk bound lemma 1 eq 4}
&\mathscr{N}(2^{-1}D_{\Delta^{d}}^{2}\delta,\{(x,x')\mapsto \|f(x)-f(x')\|_{2}^{2}\;|\;f\in\mathscr{F}_{\beta,\beta_{0},P}(\mathcal{F}')\},\|\cdot\|_{L^{\infty}(\mathcal{X}^{2})})\nonumber\\
&\leq \nonumber\\
&\mathscr{N}(c_{6}\delta,\{(x,x')\mapsto\|\{c_{j}(f,x)-c_{j}(f,x')\}_{j=1}^{d_{1}}\|_{2}^{2} | f\in\mathscr{F}_{\beta,\beta_{0},P}(\mathcal{F}')\},\|\cdot\|_{L^{\infty}(\mathcal{X}^{2})}),
\end{align}
where $c_{6}=2^{-1}dd_{1}^{-1}D_{\Delta^{d}}^{2}$.
Here, $\{c_{j}(f,x)-c_{j}(f,x')\}_{j=1}^{d_{1}}$ denotes the sequence, and $\|\cdot\|_{2}$ denotes the 2-norm.
Since the function $|c_{j}(f,\cdot)|$ is bounded by the definition of $\Delta^{d}$, by the Lipschitz continuity of the function $s\mapsto s^{2}$ on a closed interval, we have
\begin{align}
\label{eq:excess risk bound lemma 1 eq 5}
&\mathscr{N}(c_{6}\delta,\{(x,x')\mapsto \|\{c_{j}(f,x)-c_{j}(f,x')\}_{j=1}^{d_{1}}\|_{2}^{2}\;|\;f\in\mathscr{F}_{\beta,\beta_{0},P}(\mathcal{F}')\},\|\cdot\|_{L^{\infty}(\mathcal{X}^{2})})\nonumber\\
&\leq
\prod_{j=1}^{d_{1}}\mathscr{N}(c_{6}'\delta,\{(x,x')\mapsto c_{j}(f,x)-c_{j}(f,x')\;|\;f\in\mathscr{F}_{\beta,\beta_{0},P}(\mathcal{F}')\},\|\cdot\|_{L^{\infty}(\mathcal{X}^{2})})\nonumber\\
&\leq
\prod_{j=1}^{d_{1}}\mathscr{N}(2^{-3}dd_{1}^{-2}D_{\Delta^{d}}^{2}\delta,\{x\mapsto c_{j}(f,x)\;|\;f\in\mathscr{F}_{\beta,\beta_{0},P}(\mathcal{F}')\},\|\cdot\|_{L^{\infty}(\mathcal{X})}),
\end{align}
where $c_{6}'=2^{-1}d_{1}^{-1}c_{6}$.
Combining~\eqref{eq:excess risk bound lemma 1 eq 2}, \eqref{eq:excess risk bound lemma 1 eq 3}, \eqref{eq:excess risk bound lemma 1 eq 4}, and \eqref{eq:excess risk bound lemma 1 eq 5}, we have 
\begin{align}
\label{eq:excess risk bound lemma 1 eq 7}
&\mathscr{N}(\delta,\psi\circ\rho (\mathscr{F}_{\beta,\beta_{0},P}(\mathcal{F}')),\|\cdot\|_{L^{\infty}(\mathcal{X}^{2})})\nonumber\\
&\leq
\prod_{j=1}^{d_{1}}\mathscr{N}(L_{2}\delta,\{x\mapsto c_{j}(f,x)\;|\;f\in\mathscr{F}_{\beta,\beta_{0},P}(\mathcal{F}')\},\|\cdot\|_{L^{\infty}(\mathcal{X})}),
\end{align}
where $L_{2}=2^{-3}dd_{1}^{-2}D_{\Delta^{d}}^{2}$.
Here, we note that for the softmax function on the domain $[-M,M]^{d_{1}}$ and any $x_{1},\cdots,x_{d_{1}},x_{1}',\cdots,x_{d_{1}}'\in [-M,M]$, we have
\begin{align*}
\left|\frac{e^{x_{j}}}{\sum_{i=1}^{d_{1}}e^{x_{i}}}-\frac{e^{x_{j}'}}{\sum_{i=1}^{d_{1}}e^{x_{i}'}}\right|
&\leq d_{1}^{-2}e^{2M}\sum_{i=1}^{d_{1}}|e^{x_{j}+x_{i}'}-e^{x_{j}'+x_{i}}|\\
&\leq 2d_{1}^{-2}e^{4M}\sum_{j=1}^{d_{1}}|x_{j}-x_{j}'|.
\end{align*}
Hence, we have
\begin{align}
\label{eq:excess risk bound lemma 1 eq 7.1}
&\log\mathscr{N}(L_{2}\delta,\{x\mapsto c_{j}(f,x)\;|\;f\in\mathscr{F}_{\beta,\beta_{0},P}(\mathcal{F}')\},\|\cdot\|_{L^{\infty}(\mathcal{X})})\nonumber\\
&\leq d_{1}\log\mathscr{N}(2^{-1}d_{1}e^{-4M}L_{2}\delta,\mathcal{F}_{L,J,S,M,\bm{d}_{1}}^{\textup{NN}},\|\cdot\|_{L^{\infty}(\mathcal{X})}).
\end{align}
Thus, setting $C_{6}=2^{-1}d_{1}e^{-4M}L_{2}$, by~\eqref{eq:excess risk bound lemma 1 eq 7} and \eqref{eq:excess risk bound lemma 1 eq 7.1}, we obtain
\begin{align}
\label{eq:excess risk bound lemma 1 eq 8}
&\log\mathscr{N}(\delta,\psi\circ\rho (\mathscr{F}_{\beta,\beta_{0},P}(\mathcal{F}')),\|\cdot\|_{L^{\infty}(\mathcal{X}^{2})})\nonumber\\
&\leq
d_{1}^{2}\log\mathscr{N}(C_{6}\delta,\mathcal{F}_{L,J,S,M,\bm{d}_{1}}^{\textup{NN}},\|\cdot\|_{L^{\infty}(\mathcal{X})}),
\end{align}
which shows the assertion.
\end{proof}

The following fact is proven by~\citet{nakada2020adaptive}.
\begin{lemma}[{Lemma~21 in~\citep{nakada2020adaptive}}]
\label{citelem:nakada lemma}
Given $K\in\mathbb{N}$, let $L\in\mathbb{N}$, $J,S,M\geq 0$, and $\bm{d}_{1}=(K,d_{\textup{NN},1},\cdots,d_{\textup{NN},L-1},1)\in\mathbb{N}^{L+1}$.
Then, for every $s>0$, it holds that
\begin{align}
\label{eq:nakada lemma}
\log\mathscr{N}(s,\mathcal{F}_{L,J,S,M,\bm{d}_{1}}^{\textup{NN}},\|\cdot\|_{L^{\infty}(\mathcal{X})})\leq S\log(2s^{-1}LJ^{L}(S+1)^{L}).
\end{align}
\end{lemma}

Given a compact subset $\mathcal{A}\subset \mathbb{R}^{s}$, the bracketing number of a class $\mathcal{H}$ of real-valued functions on $\mathcal{A}$ with respect to $\|\cdot\|_{L^{s}(\mathcal{A})}$ is denoted by
\begin{align*}
\mathscr{B}(\delta,\mathcal{H},\|\cdot\|_{L^{s}(\mathcal{A})})
\end{align*}
(see, e.g.,~\citep[Definition~2.2]{geer2000empirical} for bracketing numbers).
Now, we recall a useful fact on the excess risk, which is first proven by~\citet{park2009convergence}.
In what follows, we recall a simplified version proven by~\citet{kim2021fast}:
\begin{lemma}[Theorem~A.1 in~\citep{kim2021fast}]
\label{lem:lemma by kim et al}
Let $\mathcal{X}_{0}=[0,1]^{d'}$ for some $d'\in\mathbb{N}$.
For every $n\in\mathbb{N}$, consider a constant $M_{n}>0$ and a class $\mathcal{H}_{n}$ of $M_{n}$-uniformly bounded, measurable real-valued functions on $\mathcal{X}_{0}$ (namely, $\sup_{h\in\mathcal{H}_{n}}\|h\|_{\infty}\leq M_{n}$).
Let $\bar{\ell}:\mathbb{R}\times\mathcal{Y}\to\mathbb{R}$ be a loss function such that there is a $c_{0}$-Lipschitz continuous function $\tilde{\ell}:\mathbb{R}\to\mathbb{R}$ satisfying $\bar{\ell}(z,y)=\tilde{\ell}(yz)$ for any $y\in\mathcal{Y}$ and any $z\in\mathbb{R}$.
Let $P$ be a Borel probability measure in $\mathcal{X}_{0}\times\mathcal{Y}$ for which its marginal distribution in $\mathcal{X}_{0}$ is absolutely continuous for the Lebesgue measure, and its density is uniformly bounded.
Suppose that there is a measurable function $h^{*}:\mathcal{X}_{0}\to\mathbb{R}$ satisfying $\mathbb{E}_{P}[\bar{\ell}(h^{*}(x),y)]=\inf_{h}\mathbb{E}_{P}[\bar{\ell}(h(x),y)]$, where the infimum is taken over all the measurable functions on $\mathcal{X}_{0}$, and there are some $\kappa_{0}\in(0,1]$, positive constants $c$, $c'=c'(\kappa_{0},c_{0},c)$, and positive sequences $\{\varepsilon_{n}\}_{n\in\mathbb{N}}$ and $\{\widetilde{\varepsilon}_{n}\}_{n\in\mathbb{N}}$, such that the conditions below are satisfied for an arbitrary $n\in\mathbb{N}$:
\begin{enumerate}
    \item[(C1)] There is a function $h\in\mathcal{H}_{n}$ such that $\mathbb{E}_{P}[\bar{\ell}(h(x),y)]-\mathbb{E}_{P}[\bar{\ell}(h^{*}(x),y)]\leq \varepsilon_{n}$.
    \item[(C2)] The following inequality is satisfied for every $h\in\mathcal{H}_{n}$:
    \begin{align*}
        &\mathbb{E}_{P}[|\bar{\ell}(h(x),y)-\bar{\ell}(h^{*}(x),y)|^{2}]\\
        &\leq cM_{n}^{2-\kappa_{0}}(\mathbb{E}_{P}[\bar{\ell}(h(x),y)-\bar{\ell}(h^{*}(x),y)])^{\kappa_{0}}.
    \end{align*}
    \item[(C3)] It holds that $\log\mathscr{B}(\widetilde{\varepsilon}_{n},\mathcal{H}_{n},\|\cdot\|_{L^{2}(\mathcal{X}_{0})})\leq c'n\widetilde{\varepsilon}_{n}^{2-\kappa_{0}}M_{n}^{-2+\kappa_{0}}$.
\end{enumerate}
Then, for the estimator $\widehat{h}_{n}:(\mathcal{X}_{0}\times\mathcal{Y})^{n}\to\mathcal{H}_{n}$ such that for any sequence of pairs $(x_{1},y_{1}),\cdots,(x_{n},y_{n})\in\mathcal{X}_{0}\times\mathcal{Y}$, the empirical risk $\frac{1}{n}\sum_{i=1}^{n}\bar{\ell}(h(x_{i}),y_{i})$ is minimized at $\widehat{h}_{n}((x_{1},y_{1}),\cdots,(x_{n},y_{n}))$ in $\mathcal{H}_{n}$, there are positive universal constants $c_{1},c_{2}>0$ such that the following inequality holds:
\begin{align*}
&\mathbb{E}[\mathds{1}_{\{\mathbb{E}_{P}[\bar{\ell}(\widehat{h}_{n}(x),y)-\bar{\ell}(h^{*}(x),y)]\geq (2\varepsilon_{n})\vee 128c_{0}^{-1}\widetilde{\varepsilon}_{n}\}}]\\
&\leq c_{1}\exp(-c_{2}n((2\varepsilon_{n})\vee (128c_{0}^{-1}\widetilde{\varepsilon}_{n}))^{2-\kappa_{0}}M_{n}^{-2+\kappa_{0}}).
\end{align*}
\end{lemma}
We aim at applying this fact to a pairwise binary classification setting by checking conditions (C1)--(C3) in Lemma~\ref{lem:lemma by kim et al}.

We now prove Theorem~\ref{thm:estimation error for deep relu networks}.

\begin{proof}[Proof of Theorem~\ref{thm:estimation error for deep relu networks}]
Let $P\in\mathcal{P}_{\alpha,\tau,\xi}$.
First, note that by the condition $\theta_{1}\geq 1$ in Definition~\ref{def:main assumption}, we have $0<d_{1}^{-2}\theta_{1}^{-1}\leq 1$.
Also, $\varepsilon_{n}<\frac{1}{2}$ by the definition.
Thus, we have that $d_{1}^{-2}\theta_{1}^{-1}\varepsilon_{n}<\frac{1}{2}$, which implies that Proposition~\ref{prop:approximation error bound} is applicable.

By Proposition~\ref{prop:approximation error bound}, there are a constant $c>0$, parameters $L^{*}\lesssim \log_{2}\varepsilon_{n}^{-1}$, $1\leq J^{*}\lesssim \varepsilon_{n}^{-c}$, $S^{*}\lesssim \varepsilon_{n}^{-(K-1)/\alpha}\log_{2}\varepsilon_{n}^{-1}$, $M^{*}\lesssim |\log(4d_{1}^{-2}\theta_{1}^{-1}\varepsilon_{n})|\vee 1$, $\bm{d}^{*}=(K,d_{\textup{NN},1}^{*},\cdots,d_{\textup{NN},L^{*}-1}^{*},d_{1})$, and some $g_{\bm{W}^{*},\bm{b}^{*}}\in \mathcal{F}_{L^{*},J^{*},S^{*},M^{*},\bm{d}^{*}}^{\textup{NN}}$, such that we have
\begin{align}
\label{eq:example theorem eq 1}
\|H_{i}\circ g_{\bm{W}^{*},\bm{b}^{*}}-g_{i}^{*}\|_{L^{1}(\mathcal{X})}\leq d_{1}^{-2}\theta_{1}^{-1}\varepsilon_{n}.
\end{align}

Hereafter, the subclass $\mathcal{F}^{*}=\mathcal{F}_{L^{*},J^{*},S^{*},M^{*},\bm{d}^{*}}^{\Delta^{d}\textup{-NN}}$ is considered in this proof.
Denote by $\mathcal{F}_{\textup{local}}^{*}=\mathscr{F}_{\beta,\beta^{-1}\varepsilon_{n},P}(\mathcal{F}^{*})$.
Then, the set $\psi\circ\rho(\mathcal{F}_{\textup{local}}^{*})$ is $1$-uniformly bounded.
Also, it is well known that hinge loss $\max\{0,1-t\}$ is $1$-Lipschitz continuous (see, e.g.,~\citep[Example~2.27]{steinwart2008support}).
Thus, we can set $c_{0}=1$ and $M_{n}=1$ hereafter.

We consider to apply Lemma~\ref{lem:lemma by kim et al} to Theorem~\ref{thm:main result}.
To this end, we will check that conditions (C1) -- (C3) in Lemma~\ref{lem:lemma by kim et al} are satisfied.

\begin{claim}
\label{claim:check condition c1 to c3}
Given $\alpha>0$, $\tau\geq 1$, $\xi=(R,K,d_{1},E,\theta_{\textup{NC}},\theta_{1},\theta_{2},\theta_{3})\in\Xi$, $P\in\mathcal{P}_{\alpha,\tau,\xi}$, $\beta\in (0,D_{\textup{proj}})$, and $n\in\mathbb{N}\setminus \{1,2\}$ such that $\varepsilon_{n}=n^{-\tau\alpha/((2\tau-1)\alpha+\tau(K-1))}<1/2$, let $\mathcal{F}_{L^{*},J^{*},S^{*},M^{*},\bm{d}^{*}}^{\Delta^{d}\textup{-NN}}$ be a class of ReLU networks with a constant $c>0$, $L^{*}\lesssim \log_{2}\varepsilon_{n}^{-1}$, $1\leq J^{*}\lesssim \varepsilon_{n}^{-c}$, $S^{*}\lesssim \varepsilon_{n}^{-(K-1)/\alpha}\log_{2}\varepsilon_{n}^{-1}$, $M^{*}\lesssim |\log(4d_{1}^{-2}\theta_{1}^{-1}\varepsilon_{n})|\vee 1$, and $\bm{d}^{*}=(K,d_{\textup{NN},1}^{*},\cdots,d_{\textup{NN},L^{*}-1}^{*},d_{1})\in\mathbb{N}^{L^{*}+1}$, for which there is $g_{\bm{W}^{*},\bm{b}^{*}}\in\mathcal{F}_{L^{*},J^{*},S^{*},M^{*},\bm{d}^{*}}^{\textup{NN}}$ such that the following inequality is satisfied for every $i\in\{1,\cdots,d_{1}\}$:
\begin{align}
\label{eq:claim check condition c1 to c3 eq 1}
\|H_{i}\circ g_{\bm{W}^{*},\bm{b}^{*}}-g_{i}^{*}\|_{L^{1}(\mathcal{X})}\leq d_{1}^{-2}\theta_{1}^{-1}\varepsilon_{n},
\end{align}
where for $\mathscr{S}_{P}=\{\mathcal{K}_{i}\}_{i=1}^{d_{1}}$ we define $g_{i}^{*}:=\mathds{1}_{\mathcal{K}_{i}}$.
Then, there is a constant $C_{1}>0$ independent of $n$ and $P$ such that all the conditions \emph{(C1)} -- \emph{(C3)} in Lemma~\ref{lem:lemma by kim et al} are satisfied for both $\mathcal{F}_{\textup{local}}^{*}=\mathscr{F}_{\beta,\beta^{-1}\varepsilon_{n},P}(\mathcal{F}_{L^{*},J^{*},S^{*},M^{*},\bm{d}^{*}}^{\Delta^{d}\textup{-NN}})$ and $\mathcal{F}^{*}=\mathcal{F}_{L^{*},J^{*},S^{*},M^{*},\bm{d}^{*}}^{\Delta^{d}\textup{-NN}}$ with $\kappa_{0}=\tau^{-1}$, some positive constants $c,c'$, and the sequences $\{C_{1}\varepsilon_{n}\}_{n\in\mathbb{N}}$ and $\{\varepsilon_{n}\log^{3}{n}\}_{n\in\mathbb{N}}$.
\end{claim}
\begin{proof}[Proof of Claim~\ref{claim:check condition c1 to c3}]
We check condition (C1) in Lemma~\ref{lem:lemma by kim et al}.
Since $\|p_{X}\|_{L^{\infty}(\mathcal{X})}\leq \theta_{1}$ by condition (A3) in Definition~\ref{def:main assumption}, by \eqref{eq:claim check condition c1 to c3 eq 1}, for any $i\in\{1,\cdots,d_{1}\}$ we have
\begin{align}
\label{eq:claim check condition c1 to c3 eq 2}
\|H_{i}\circ g_{\bm{W}^{*},\bm{b}^{*}}-g_{i}^{*}\|_{L^{1}(\mathcal{X})}\leq d_{1}^{-2}\|p_{X}\|_{L^{\infty}(\mathcal{X})}^{-1}\varepsilon_{n}.
\end{align}
By~\eqref{eq:claim check condition c1 to c3 eq 2} and Proposition~\ref{thm:approximation and concentration}, for the function $f_{\bm{W}^{*},\bm{b}^{*}}=\sum_{i=1}^{d_{1}}(H_{i}\circ g_{\bm{W}^{*},\bm{b}^{*}})v_{i}$, we have
\begin{align}
\label{eq:example theorem eq 1.01}
f_{\bm{W}^{*},\bm{b}^{*}}\in\mathcal{F}_{\textup{local}}^{*}.
\end{align}

Let $\varepsilon_{\mathcal{F}_{\textup{local}}^{*},f^{*}}$ be the quantity defined in~\eqref{eq:inf sup approximation error}.
By Proposition~\ref{prop:approximation property of H}, there is a constant $C_{1}>0$ independent of $n$ such that
\begin{align}
\label{eq:example theorem eq 1.02}
P \textup{ is } (\psi,\mathcal{F}_{\textup{local}}^{*},C_{1}\varepsilon_{\mathcal{F}_{\textup{local}}^{*},f^{*}},1)\textup{-weak representable.}
\end{align}
Here, we note that by~\eqref{eq:claim check condition c1 to c3 eq 1} and~\eqref{eq:example theorem eq 1.01}, we have
\begin{align}
\label{eq:example theorem eq 1.021}
\varepsilon_{\mathcal{F}_{\textup{local}}^{*},f^{*}}\leq d_{1}^{-2}\theta_{1}^{-1}\varepsilon_{n}.
\end{align}
Hence, applying Lemma~\ref{lem:trade-off for parameters} and \eqref{eq:example theorem eq 1.021} to~\eqref{eq:example theorem eq 1.02}, we have that
\begin{align}
\label{eq:example theorem eq 1.03}
P \textup{ is }
(\psi,\mathcal{F}_{\textup{local}}^{*},C_{1}d_{1}^{-2}\theta_{1}^{-1}\varepsilon_{n},1)\textup{-weak representable.}
\end{align}
By~\eqref{eq:example theorem eq 1.03}, Lemma~\ref{prop:constructability for F0}, and Proposition~\ref{lem:connecting the upper bound to the excess risk}, there is some $f_{\bm{W},\bm{b}}\in\mathcal{F}_{\textup{local}}^{*}$ such that we have
\begin{align*}
\mathbb{E}_{P}[\ell_{f_{\bm{W},\bm{b}}}]-\mathbb{E}_{P}[\ell_{f^{*}}]\leq C_{1}d_{1}^{-2}\theta_{1}^{-1}\varepsilon_{n}\leq C_{1}\varepsilon_{n}.
\end{align*}
Thus, condition (C1) in Lemma~\ref{lem:lemma by kim et al} is satisfied for $\mathcal{F}_{\textup{local}}^{*}$.
We can check that condition (C1) in Lemma~\ref{lem:lemma by kim et al} is satisfied for $\mathcal{F}^{*}$, applying the arguments~\eqref{eq:example theorem eq 1.02} -- \eqref{eq:example theorem eq 1.03} to $\mathcal{F}^{*}$.

Note that one can verify that conditions (C2) and (C3) in Lemma~\ref{lem:lemma by kim et al} are satisfied in the case where $\mathcal{F}^{*}$ is considered, following almost the same arguments as those for $\mathcal{F}_{\textup{local}}^{*}$.
Hence, we present the detailed derivations for $\mathcal{F}_{\textup{local}}^{*}$ in the subsequent paragraphs.

We next check condition (C2) in Lemma~\ref{lem:lemma by kim et al}.
By Proposition~1 in~\citep{lecue2007optimal} (see Lemma~\ref{citelem:lecue lemma}), condition (C2) in Lemma~\ref{lem:lemma by kim et al} is satisfied with $\kappa_{0}=\tau^{-1}\in (0,1]$ (note that Lemma~6.1 in~\citep{steinwart2007fast} is also applicable, as in~\citep{kim2021fast}).

We finally check condition (C3) in Lemma~\ref{lem:lemma by kim et al}.
By the standard bracketing number bound shown in~\citep[Eq.~(A.1)]{kim2021fast}, which is a consequence of Lemma~2.1 in~\citep{geer2000empirical}, for any $s>0$ we have
\begin{align}
\label{eq:bracketing and covering numbers}
\log\mathscr{B}(s,\psi\circ\rho (\mathcal{F}_{\textup{local}}^{*}),\|\cdot\|_{L^{2}(\mathcal{X}^{2})})\leq \log\mathscr{N}(\frac{s}{2},\psi\circ\rho (\mathcal{F}_{\textup{local}}^{*}),\|\cdot\|_{L^{\infty}(\mathcal{X}^{2})}).
\end{align}
By Lemma~\ref{lem:excess risk bound lemma 1} and Lemma~\ref{citelem:nakada lemma}, for the constant $C_{6}=2^{-4}dd_{1}^{-1}e^{-4M^{*}} D_{\Delta^{d}}^{2}$, we have
\begin{align}
&\log\mathscr{N}(2^{-1}s,\psi\circ\rho (\mathcal{F}_{\textup{local}}^{*}),\|\cdot\|_{L^{\infty}(\mathcal{X}^{2})})\nonumber\\
&\leq
d_{1}^{2}\log\mathscr{N}(2^{-1}C_{6}s,\mathcal{F}_{L^{*},J^{*},S^{*},M^{*},\bm{d}_{1}^{*}}^{\textup{NN}},\|\cdot\|_{L^{\infty}(\mathcal{X})})\nonumber\\
\label{eq:condition (c3) eq 2}
&\leq d_{1}^{2}S^{*}\log(4s^{-1}C_{6}^{-1}L^{*}(J^{*})^{L^{*}}(S^{*}+1)^{L^{*}}),
\end{align}
where Lemma~\ref{lem:excess risk bound lemma 1} is applied in the first inequality, and Lemma~\ref{citelem:nakada lemma} is used in the second inequality.
Here, $\bm{d}_{1}^{*}=(K,d_{\textup{NN},1}^{*},\cdots,d_{\textup{NN},L^{*}-1}^{*},1)$.
Note that by the conditions of $L^{*},J^{*},S^{*}$, and $M^{*}$, we have
\begin{align*}
d_{1}^{2}S^{*}\log(4\widetilde{\varepsilon}_{n}^{-1}C_{6}^{-1}L^{*}(J^{*})^{L^{*}}(S^{*}+1)^{L^{*}})
\lesssim
\varepsilon_{n}^{-\frac{K-1}{\alpha}}((\log\widetilde{\varepsilon}_{n}^{-1})\vee (\log^{2}{n}))\log{n}.
\end{align*}
Hence, by~\eqref{eq:bracketing and covering numbers} and \eqref{eq:condition (c3) eq 2}, condition (C3) in Lemma~\ref{lem:lemma by kim et al} is satisfied with any $\widetilde{\varepsilon}_{n}$ satisfying
\begin{align}
\label{eq:condition (c3) eq 3}
\varepsilon_{n}^{-\frac{K-1}{\alpha}}((\log\widetilde{\varepsilon}_{n}^{-1})\vee(\log^{2}{n}))\log{n}
\leq 
c' n\widetilde{\varepsilon}_{n}^{2-\tau^{-1}},
\end{align}
where $c'>0$ is some constant independent of $n$ and $P$.
Hence, we define $\widetilde{\varepsilon}_{n}$ as
\begin{align*}
\widetilde{\varepsilon}_{n}=\varepsilon_{n}\log^{3}{n}.
\end{align*}
Then, $\widetilde{\varepsilon}_{n}$ satisfies~\eqref{eq:condition (c3) eq 3}.

Therefore, conditions (C1) -- (C3) in Lemma~\ref{lem:lemma by kim et al} are satisfied.
\end{proof}

Let $\widehat{g}_{n}^{\textup{LERM}}$ be the $(\beta,\varepsilon_{n},n,\mathcal{P}_{\alpha,\tau,\xi},\mathcal{F}^{*})$-local ERM estimator, and let $\widehat{f}_{n}^{\textup{LERM}}$ be the corresponding local estimator of vector-valued functions (see Definition~\ref{def:formal definition of erm}), where note that without loss of generality we may assume the existence of these estimators (see Remark~\ref{remark:local erm}).
By Theorem~\ref{thm:main result}, there are positive constants $C,C'$, and $C''$ independent of $n$ and $P$ such that
\begin{align}
\label{eq:example theorem eq 2}
\mathcal{R}(\widehat{g}_{n,P}^{\textup{LERM}};P)
\leq C(\log{n}) \mathbb{E}[\mathcal{E}(\widehat{f}_{n,P}^{\textup{LERM}}(U_{1}^{n});P)]^{\frac{1}{\tau}}+\frac{C'\varepsilon_{n}}{\beta}+\frac{C''}{n}.
\end{align}
Let $\widetilde{\varepsilon}_{n}=\varepsilon_{n}\log^{3}{n}$.
By Lemma~\ref{lem:lemma by kim et al}, there are positive universal constants $c_{1},c_{2}$ such that we have
\begin{align}
\label{eq:main theorem conclusion eq 1}
\mathbb{E}[\mathds{1}_{\{\mathcal{E}(\widehat{f}_{n,P}^{\textup{LERM}}(U_{1}^{n});P)\geq 128C_{1}\widetilde{\varepsilon}_{n}\}}]
\leq
c_{1}e^{-128^{2-\tau^{-1}}c_{2}n\widetilde{\varepsilon}_{n}^{2-\tau^{-1}}},
\end{align}
where $\varepsilon_{n}\leq \widetilde{\varepsilon}_{n}$ since $n>2$ by the definition.
Then, we have
\begin{align}
&\mathbb{E}[\mathcal{E}(\widehat{f}_{n,P}^{\textup{LERM}}(U_{1}^{n});P)]\nonumber\\
&\leq 128C_{1}\widetilde{\varepsilon}_{n}+\int_{128C_{1}\widetilde{\varepsilon}_{n}}^{2\vee (128C_{1}\widetilde{\varepsilon}_{n})}\mathbb{E}[\mathds{1}_{\{\mathcal{E}(\widehat{f}_{n,P}^{\textup{LERM}}(U_{1}^{n});P)\geq s\}}]ds\\
\label{eq:main theorem conclusion eq 2}
&\leq 128C_{1}\widetilde{\varepsilon}_{n}+2c_{1}e^{-128^{2-\tau^{-1}}c_{2}n\widetilde{\varepsilon}_{n}^{2-\tau^{-1}}},
\end{align}
where in~\eqref{eq:main theorem conclusion eq 2}, we used~\eqref{eq:main theorem conclusion eq 1}.
By the definition of $\widetilde{\varepsilon}_{n}$, there is a natural number $N$ such that for every $n\geq N$, we have $128C_{1}\widetilde{\varepsilon}_{n}\geq 2c_{1}e^{-128^{2-\tau^{-1}}c_{2}n\widetilde{\varepsilon}_{n}^{2-\tau^{-1}}}$.
Thus, by~\eqref{eq:example theorem eq 2} and \eqref{eq:main theorem conclusion eq 2}, for any $n\geq N$, we have
\begin{align*}
\mathcal{R}(\widehat{g}_{n,P}^{\textup{LERM}};P)\leq 256^{\frac{1}{\tau}}C_{1}^{\frac{1}{\tau}}C \varepsilon_{n}^{\frac{1}{\tau}}(\log{n})^{3\tau^{-1}+1}+\frac{C'}{\beta}\varepsilon_{n}+\frac{C''}{n}.
\end{align*}
Define
\begin{align*}
C^{*}=3\max\{(256C_{1})^{1/\tau}C,C'\beta^{-1},C''\}.
\end{align*}
Since $\varepsilon_{n}=n^{-\tau\alpha/((2\tau-1)\alpha+\tau(K-1))}$ and $\varepsilon_{n}<\frac{1}{2}$, if $n\geq N$, we obtain
\begin{align}
\label{eq:example theorem eq 7}
    \mathcal{R}(\widehat{g}_{n,P}^{\textup{LERM}};P)
    \leq C^{*} n^{-\frac{\alpha}{(2\tau-1)\alpha + \tau(K-1)}}\log^{3\tau^{-1}+1}{n}.
\end{align}
Therefore, we obtain the claim.
\end{proof}

\begin{remark}
\label{remark:local erm}
Consider the setting in the proof of Theorem~\ref{thm:estimation error for deep relu networks}.
Given $(u_{1},\cdots,u_{n})\in (\mathcal{X}^{2}\times\mathcal{Y})^{n}$ and $P\in\mathcal{P}_{\alpha,\tau,\xi}$, suppose that $\widehat{f}_{n,P}^{\textup{LERM}}(u_{1},\cdots,u_{n})$ does not exist in Definition~\ref{def:formal definition of erm}.
In this case, $\widehat{g}_{n,P}^{\textup{LERM}}(u_{1},\cdots,u_{n})$ is not defined.
To address this issue, one may consider to modify the definition of $\mathcal{F}_{L,J,S,M,\bm{d}}^{\Delta^{d}\textup{-NN}}$ so that every entry in $\bm{W}$ or $\bm{b}$ of each ReLU networks $g_{\bm{W},\bm{b}}$ of the modified class always belongs to a finite subset of $\mathbb{R}$, similarly to~\citep[Definition~2.9]{petersen2018optimal}.
Formally, let $\mathcal{W}$ be a finite subset of $\mathbb{R}$ such that for the ReLU networks $g_{\bm{W}^{*},\bm{b}^{*}}$ in~\eqref{eq:example theorem eq 1}, it holds that
\begin{align*}
\bigcup_{i,j_{1},j_{2}}\{W_{i,j_{1},j_{2}}^{*}\}\cup \bigcup_{i,j}\{b_{i,j}^{*}\}\subset \mathcal{W},
\end{align*}
where $\bm{W}^{*}=(W_{1}^{*},\cdots,W_{L^{*}}^{*})$, $W_{i}^{*}=(W_{i,j_{1},j_{2}}^{*})$ for each $i=1,\cdots,L^{*}$, $\bm{b}^{*}=(b_{1}^{*},\cdots,b_{L^{*}}^{*})$, and $b_{i}^{*}=(b_{i,j}^{*})$ for each $i=1,\cdots,L^{*}$.
Note that $\mathcal{W}$ may depend on $n$.
For each $n\in\mathbb{N}\setminus \{1,2\}$, define
\begin{align*}
\mathcal{F}_{L^{*},J^{*},S^{*},M^{*},\bm{d}^{*}}^{\Delta^{d}\textup{-NN},\mathcal{W}}=
\left\{
\begin{array}{@{}l|l@{}}
    \multirow{2}{*}{$f_{\bm{W},\bm{b}}\in\mathcal{F}_{L^{*},J^{*},S^{*},M^{*},\bm{d}^{*}}^{\Delta^{d}\textup{-NN}}$} & \bm{W}=(W_{1},\cdots,W_{L^{*}}) \textup{ and } \bm{b}=(b_{1},\cdots,b_{L^{*}}) \\
     & \textup{satisfy } \bigcup_{i,j_{1},j_{2}}\{W_{i,j_{1},j_{2}}\}\cup \bigcup_{i,j}\{b_{i,j}\}\subset \mathcal{W}
\end{array}
\right\}.
\end{align*}
Note that this definition is a slight generalization of~\citep[Definition~2.9]{petersen2018optimal}.
Note also that in the proof of Proposition~\ref{prop:approximation error bound}, one can take a finite subset $\mathcal{W}\subset \mathbb{R}$ independent of any given distribution $P\in\mathcal{P}_{\alpha,\tau,\xi}$ (see Lemma~\ref{lem:petersen voigtlaender lemma}).
Modifying the proof of Claim~\ref{claim:check condition c1 to c3} slightly, one can see that conditions (C1) -- (C3) in Lemma~\ref{lem:lemma by kim et al} are satisfied for $\mathcal{F}_{L^{*},J^{*},S^{*},M^{*},\bm{d}^{*}}^{\Delta^{d}\textup{-NN},\mathcal{W}}$ with the same constants and sequences as those in Claim~\ref{claim:check condition c1 to c3}.
Since $\mathcal{F}_{L^{*},J^{*},S^{*},M^{*},\bm{d}^{*}}^{\Delta^{d}\textup{-NN},\mathcal{W}}\subset \mathcal{F}_{L^{*},J^{*},S^{*},M^{*},\bm{d}^{*}}^{\Delta^{d}\textup{-NN}}$, the remained part of the proof is almost the same as the original proof.
\end{remark}

\subsection{Proof of Theorem~\ref{thm:minimax lower bound}}
\label{appsubsec:proof of theorem minimax lower bound}

We consider an approach using Assouad's lemma~\citep{assouad1983deux}, which is standard in the context of set estimation (see, e.g.,~\citep{mammen1995asymptotical,mammen1999smooth,tsybakov2004optimal,meyer2023optimal}).
In what follows, we recall a version shown in Lemma~2 of~\citep{yu1997assouad}, where a comment in p.427 of~\citep{yu1997assouad} is combined.
Note that Lemma~2 in~\citep{yu1997assouad} shows the lemma for a family of pseudo-distances, while it is mentioned in~\citep[p.427]{yu1997assouad} that this lemma can be extended to a setting where a family of non-negative, symmetric functions on a product parameter set $\Theta\times\Theta$ is used instead.
Specifically, \citet[Remark (i)]{yu1997assouad} introduces a non-negative symmetric function $\delta:\Theta\times\Theta\to\mathbb{R}$ that satisfies for any $\vartheta,\vartheta',\vartheta''\in\Theta$,
\begin{align}
\label{eq:a variant of triangle inequality}
c_{1}\delta(\vartheta,\vartheta')\leq \delta(\vartheta,\vartheta'')+\delta(\vartheta'',\vartheta'),
\end{align}
for some $c_{1}\in (0,1)$.
The following statement is due to~\citep[p.427]{yu1997assouad}, which is the combination of Lemma~2 and Remark~(i) in~\citep{yu1997assouad}.
\begin{lemma}[{Lemma~2 and Remark~(i) of~\citep{yu1997assouad}}]
\label{citelem:assouad lemma}
Given $t\in\mathbb{N}$, let $\overline{\mathcal{P}}$ be a set of probability measures in a measurable space $\mathcal{A}$ that are absolutely continuous for a given non-negative $\sigma$-finite measure $\nu$ in $\mathcal{A}$ and parameterized by the set $\{0,1\}^{t}$, namely
\begin{align*}
\overline{\mathcal{P}}=\{P_{w}\;|\; w:\{1,\cdots,t\}\to\{0,1\}\}.
\end{align*}
Let $\vartheta:\overline{\mathcal{P}}\to \Theta$ be a map from $\overline{\mathcal{P}}$ to the given parameter set $\Theta$.
Given a family $\{\delta_{I}:\Theta\times\Theta\to\mathbb{R}\;|\; I\in \{1,\cdots,t\}\}$ of non-negative symmetric functions that satisfy~\eqref{eq:a variant of triangle inequality} for some fixed $c_{1}\in (0,1)$, suppose that there is a non-negative number $s=s(t)$ that may depend on $t$ such that for any $I_{0}\in \{1,\cdots,t\}$,
\begin{align}
\label{eq:assouad lemma sufficient condition}
\delta_{I_{0}}(\vartheta(P_{w_{I_{0},1}}),\vartheta(P_{w_{I_{0},0}}))\geq s,
\end{align}
for any functions $w_{I_{0},1},w_{I_{0},0}:\{1,\cdots,t\}\to \{0,1\}$ satisfying $w_{I_{0},1}(I)=w_{I_{0},0}(I)$ for any $I\in \{1,\cdots,t\}\setminus \{I_{0}\}$, $w_{I_{0},1}(I_{0})=1$, and $w_{I_{0},0}(I_{0})=0$.
Then, it holds that
\begin{align*}
&\inf_{\widehat{\vartheta}}\sup_{P_{w}\in\overline{\mathcal{P}}}\mathbb{E}\left[\sum_{I\in \{1,\cdots,t\}}\delta_{I}(\widehat{\vartheta}(U_{1}^{n}),\vartheta(P_{w}))\right]\\
&\geq \frac{c_{1}ts}{2}\min_{\substack{I\in \{1,\cdots,t\},\\ w_{I,1},w_{I,0}}}\int_{\mathcal{A}^{n}} \min\left\{p_{w_{I,1}}^{\otimes n},p_{w_{I,0}}^{\otimes n}\right\}d\nu^{\otimes n},
\end{align*}
where the infimum is taken over all estimators $\widehat{\vartheta}$ in the given set of estimators, $U_{1}^{n}$ is any sequence of i.i.d. random variables drawn from the given $P_{w}$, $p_{w}$ denotes the Radon-Nikodym derivative of $P_{w}\in\overline{\mathcal{P}}$ with respect to $\nu$, $p_{w}^{\otimes n}$ denotes the tensor product of the function $p_{w}$, and $\nu^{\otimes n}$ denotes the product measure.
\end{lemma}

We are now in a position to prove Theorem~\ref{thm:minimax lower bound}.
The proof outline is similar to the standard one considered in the literature~\citep{mammen1995asymptotical,mammen1999smooth,tsybakov2004optimal,meyer2023optimal}.
We would like to emphasize that applying the standard approach directly to our problem setting is not straightforward from the previous results of~\citep{mammen1995asymptotical,mammen1999smooth,tsybakov2004optimal,meyer2023optimal}, as the construction of the subclass is complicated.

\begin{proof}[Proof of Theorem~\ref{thm:minimax lower bound}]
The proof of Theorem~\ref{thm:minimax lower bound} is divided in several steps.

\paragraph{Step 1 (Construction of the subclass $\mathcal{P}_{\alpha,1,\xi,N^{K-1}}$).}
Given $N\in\mathbb{N}$, we define a class $\mathcal{P}_{\alpha,1,\xi,N^{K-1}}$ of Borel probability measures in $\mathcal{X}^{2}\times\mathcal{Y}$ as follows:
\begin{itemize}
    \item Let $p_{X,\textsf{U}}$ be the probability density function of the uniform distribution in $\mathcal{X}$. The Borel probability measure corresponding to $p_{X,\textsf{U}}$ is denoted by $P_{X,\textsf{U}}$. In addition, define $p_{Y,\textsf{U}}(1)=\frac{1}{2}$ and $p_{Y,\textsf{U}}(-1)=\frac{1}{2}$. Note that $p_{Y,\textsf{U}}(1)> \theta_{3}/(1+\theta_{3})$ since $0\leq \theta_{3}< 1$.
    \item Let $h_{\textsf{base}}:\mathbb{R}^{K-1}\to [0,1]$ be an infinitely differentiable function such that $h_{\textsf{base}}(\bm{0})=1$ and $\textup{cl}(\{\widetilde{x}\in\mathbb{R}^{K-1}\;|\; h_{\textsf{base}}(\widetilde{x})\neq 0\})=[-1,1]^{K-1}$, where $\textup{cl}(\cdot)$ denotes the closure of the given set. Let $w:\{1,\cdots,N\}^{K-1}\to \{0,1\}$ be arbitrary.
    Similarly to~\citep[p.3651]{meyer2023optimal} (see also~\citep{mammen1995asymptotical,mammen1999smooth}), define $h_{w}:[0,1]^{K-1}\to \mathbb{R}$ as
    \begin{align}
    \label{eq:thm minimax lower bound eq 0}
        h_{w}(\widetilde{x})=1-\theta_{3}+c_{2}N^{-\alpha}\sum_{I\in \{1,\cdots,N\}^{K-1}} w(I) h_{\textsf{base}}\left(2N\left(\widetilde{x}-\frac{2I-\bm{1}}{2N}\right)\right),
    \end{align}
    where $\bm{1}=(1,\cdots,1)\in\mathbb{R}^{K-1}$, and $c_{2}>0$ is a constant independent of $N$ such that it is small enough to guarantee the following conditions (E1) -- (E3):
    \begin{enumerate}
    \item[(E1)] $\|h_{w}\|_{\mathcal{C}^{\alpha,K-1}}\leq R$. \label{condition:e1}
    \item[(E2)] $|h_{w}(\widetilde{x})|\leq \theta_{3}$ for any $\widetilde{x}\in [0,1]^{K-1}$. \label{condition:e2}
    \item[(E3)] $2^{-K+2}3c_{2}p_{Y,\textsf{U}}(1)\int_{[-1,1]^{K-1}}h_{\textsf{base}}(\widetilde{x})d\widetilde{x}\leq 1$. \label{condition:e3}
    \end{enumerate}
    
    \item Given a function $w:\{1,\cdots,N\}^{K-1}\to \{0,1\}$, define the subset $\mathcal{K}_{w}\subset \mathcal{X}$ as
    \begin{align*}
        \mathcal{K}_{w}=\{x\in\mathcal{X}\;|\; x_{K}< h_{w}(x_{\setminus K})\}.
    \end{align*}
    Here, recall the notation $x_{\setminus K}=(x_{1},\cdots,x_{K-1})$.
    Denote the complement of $\mathcal{K}_{w}$ by $\widetilde{\mathcal{K}}_{w}=\mathcal{X}\setminus \mathcal{K}_{w}$.
    Note that 
    \begin{align}
    \label{eq:thm minimax lower bound eq 1}
    P_{X,\textsf{U}}(\mathcal{K}_{w})=\int_{\mathcal{X}}\mathds{1}_{\mathcal{K}_{w}}(x)p_{X,\textsf{U}}(x)\mu(dx)
    &=\int_{[0,1]^{K-1}}\int_{0}^{h_{w}(x_{\setminus K})}dx_{K} dx_{\setminus K}\nonumber\\
    &=\int_{[0,1]^{K-1}}h_{w}(x_{\setminus K})dx_{\setminus K}\nonumber\\
    &\in [1-\theta_{3},\theta_{3}],
    \end{align}
    where in the second equality we use Fubini's theorem, and in~\eqref{eq:thm minimax lower bound eq 1} we note that $1-\theta_{3}\leq h_{w}(\widetilde{x})\leq \theta_{3}$ for any $\widetilde{x}\in [0,1]^{K-1}$ by condition (E2).
    Here, note also that $dx_{K}$ and $dx_{\setminus K}$ denote the Lebesgue measures in $[0,1]$ and $[0,1]^{K-1}$, respectively.
    By~\eqref{eq:thm minimax lower bound eq 1}, we also have
    \begin{align}
    \label{eq:thm minimax lower bound eq 2}
    P_{X,\textsf{U}}(\widetilde{\mathcal{K}}_{w})\in [1-\theta_{3},\theta_{3}].
    \end{align}
    \item Define the function $q_{w}:\mathcal{X}^{2}\to\mathbb{R}$ as
    \begin{align*}
        q_{w}(x,x')
        &= P_{X,\textsf{U}}(\mathcal{K}_{w})^{-1} p_{X,\textsf{U}}(x)p_{X,\textsf{U}}(x')\mathds{1}_{\mathcal{K}_{w}\times\mathcal{K}_{w}}(x,x')\\
        &\quad + P_{X,\textsf{U}}(\widetilde{\mathcal{K}}_{w})^{-1} p_{X,\textsf{U}}(x)p_{X,\textsf{U}}(x')\mathds{1}_{\widetilde{\mathcal{K}}_{w}\times\widetilde{\mathcal{K}}_{w}}(x,x').
    \end{align*}
    Note that $q_{w}$ is a probability density function in $\mathcal{X}^{2}$ since
    \begin{align*}
        \int_{\mathcal{X}\times\mathcal{X}}q_{w}(x,x')\mu(dx)\mu(dx')=\int_{\mathcal{K}_{w}}p_{X,\textsf{U}}(x)\mu(dx)+\int_{\widetilde{\mathcal{K}}_{w}}p_{X,\textsf{U}}(x)\mu(dx)=1,
    \end{align*}
    where in the first equality we use Fubini's theorem, and in the second equality we utilize $p_{X,\textsf{U}}(x)=1$ for any $x\in\mathcal{X}$.
    Note also that the definition of $q_{w}$ is also an example of the constructions in~\citep[Eq. (1) and (2)]{arora2019theoretical} and \citep[Assumption~3.1]{awasthi2022do}, as in the proof of Proposition~\ref{lem:universal partition}.
    \item Given a function $w:\{1,\cdots,N\}^{K-1}\to \{0,1\}$, define $p_{w}(x,x',y)$ as
    \begin{align*}
        &p_{w}(x,x',y)\\
        &=\mathds{1}_{\{1\}}(y)p_{Y,\textsf{U}}(1) q_{w}(x,x')+\mathds{1}_{\{-1\}}(y)p_{Y,\textsf{U}}(-1)p_{X,\textsf{U}}(x)p_{X,\textsf{U}}(x').
    \end{align*}
    Then, define the set $\mathcal{P}_{\alpha,1,\xi,N^{K-1}}$ as
    \begin{align*}
    \mathcal{P}_{\alpha,1,\xi,N^{K-1}}=
    \left\{
    \begin{array}{@{}l@{\;}|l@{}}
    \multirow{3}{*}{$P_{w}$} & P_{w}\textup{ is a Borel probability measure in }\mathcal{X}^{2}\times\mathcal{Y} \\
     & \textup{whose probability density function is }p_{w}(x,x',y)\\
     & \textup{for a function }w:\{1,\cdots,N\}^{K-1}\to \{0,1\}.
    \end{array}
    \right\}.
    \end{align*}
    Note that $\mathcal{P}_{\alpha,1,\xi,N^{K-1}}$ is a finite set.
\end{itemize}
We need to check the following claims:

\begin{claim}
    \label{claim:thm minimax lower bound claim existence of the constant}
    There is a constant $c_{2}>0$ such that $c_{2}$ is independent of $N$, and the conditions (E1) -- (E3) are satisfied.
    \end{claim}
    \begin{proof}[Proof of Claim~\ref{claim:thm minimax lower bound claim existence of the constant}]
    Recall that for any $I\in\{1,\cdots,N\}^{K-1}$ and any $\widetilde{x}\in [\frac{I_{1}-1}{N},\frac{I_{1}}{N}]\times \cdots \times [\frac{I_{K-1}-1}{N},\frac{I_{K-1}}{N}]$,
    \begin{align}
    \label{eq:thm minimax lower bound claim existence eq 0.0}
    h_{w}(\widetilde{x})=1-\theta_{3}+c_{2}N^{-\alpha}w(I)h_{\textsf{base}}\left(2N\left(\widetilde{x}-\frac{2I-\bm{1}}{2N}\right)\right),
    \end{align}
    where note that the support of the function $\widetilde{x}\mapsto h_{\textsf{base}}(2N\widetilde{x}-2I+\bm{1})$ is the set $[\frac{I_{1}-1}{N},\frac{I_{1}}{N}]\times \cdots \times [\frac{I_{K-1}-1}{N},\frac{I_{K-1}}{N}]$, and $h_{\textsf{base}}$ is continuous, which implies that $h_{\textsf{base}}(2N\widetilde{x}-2I+\bm{1})=0$ if $\widetilde{x}$ belongs to the boundary of $[\frac{I_{1}-1}{N},\frac{I_{1}}{N}]\times \cdots\times [\frac{I_{K-1}-1}{N},\frac{I_{K-1}}{N}]$.
    
    We have
    \begin{align}
    \label{eq:thm minimax lower bound claim existence eq 0.1}
        &\|h_{w}\|_{\mathcal{C}^{\alpha,K-1}}\nonumber\\
        &\leq 1-\theta_{3}+ c_{2}2N^{-\alpha}\max_{I\in\{1,\cdots,N\}^{K-1}}\|h_{\textsf{base}}(2N\widetilde{x}-2I+\bm{1})\|_{\mathcal{C}^{\alpha,K-1}}\\
        \label{eq:thm minimax lower bound claim existence eq 0.2}
        &\leq 1-\theta_{3}+c_{2}2^{\alpha+1}\|h_{\textsf{base}}\|_{\mathcal{C}^{\alpha,K-1}(\mathbb{R}^{K-1})},
    \end{align}
    where $\|\cdot\|_{\mathcal{C}^{\alpha,K-1}(\mathbb{R}^{K-1})}$ denotes the H\"{o}lder norm of functions on $\mathbb{R}^{K-1}$.
    The detailed derivations of the inequalities~\eqref{eq:thm minimax lower bound claim existence eq 0.1} and~\eqref{eq:thm minimax lower bound claim existence eq 0.2} are shown below, for completeness:
    
    (Derivation of~\eqref{eq:thm minimax lower bound claim existence eq 0.1})
    Define $\overline{\mathcal{M}}_{N,I}=[\frac{I_{1}-1}{N},\frac{I_{1}}{N}]\times \cdots\times [\frac{I_{K-1}-1}{N},\frac{I_{K-1}}{N}]$.
    Given any $\bm{s}\in (\mathbb{N}\cup \{0\})^{K-1}$ satisfying $\|\bm{s}\|_{1}= \lceil \alpha-1\rceil$, for any $\widetilde{x},\widetilde{x}'\in [0,1]^{K-1}$ such that $\widetilde{x}\neq \widetilde{x}'$, we have
    \begin{align}
    &\frac{|\partial_{\bm{s}}(\sum_{I}w(I)h_{\textsf{base}}(2N\widetilde{x}-2I+\bm{1}))-\partial_{\bm{s}}(\sum_{I}w(I)h_{\textsf{base}}(2N\widetilde{x}'-2I+\bm{1}))|}{\|\widetilde{x}-\widetilde{x}'\|_{\infty}^{\alpha-\lceil \alpha-1\rceil}}\nonumber\\
    \label{eq:thm minimax lower bound claim existence eq 0.02}
    &\leq
    2\max_{I\in\{1,\cdots,N\}^{K-1}}\frac{|\partial_{\bm{s}}(h_{\textsf{base}}(2N\widetilde{x}-2I+\bm{1}))-\partial_{\bm{s}}(h_{\textsf{base}}(2N\widetilde{x}'-2I+\bm{1}))|}{\|\widetilde{x}-\widetilde{x}'\|_{\infty}^{\alpha-\lceil\alpha-1\rceil}},
    \end{align}
    where in~\eqref{eq:thm minimax lower bound claim existence eq 0.02} we apply the triangle inequality and the property that $w(I)\in\{0,1\}$, and then we use the property that the support of $h_{\textsf{base}}$ is $[-1,1]^{K-1}$, implying that $h_{\textsf{base}}(2N\widetilde{x}-2I'+\bm{1})=0$ for any $\widetilde{x}\in\overline{\mathcal{M}}_{N,I}$ for which $I\neq I'$ is satisfied.
    Similarly, for any $\bm{s}\in (\mathbb{N}\cup\{0\})^{K-1}$ satisfying $\|\bm{s}\|_{1}\leq \lceil\alpha-1\rceil$ and an arbitrary $I_{0}\in\{1,\cdots,N\}^{K-1}$, we have
    \begin{align}
    \label{eq:thm minimax lower bound claim existence eq 0.03}
    &\|\partial_{\bm{s}}(\sum_{I}w(I)h_{\textsf{base}}(2N\widetilde{x}-2I+\bm{1}))\|_{\infty}\nonumber\\
    &\leq \sup_{\widetilde{x}\in [0,1]^{K-1}}\max_{I\in\{1,\cdots,N\}^{K-1}}|\partial_{\bm{s}}(h_{\textsf{base}}(2N\widetilde{x}-2I+\bm{1}))| \nonumber\\
    &= \max_{I\in\{1,\cdots,N\}^{K-1}}\|\partial_{\bm{s}}(h_{\textsf{base}}(2N\widetilde{x}-2I+\bm{1}))\|_{\infty}\nonumber\\
    &= \|\partial_{\bm{s}}(h_{\textsf{base}}(2N\widetilde{x}-2I_{0}+\bm{1}))\|_{\infty},
    \end{align}
    where in the inequality we note that the support of $\widetilde{x}\mapsto h_{\textsf{base}}(2N\widetilde{x}-2I+\bm{1})$ is $\overline{\mathcal{M}}_{N,I}$ for each $I\in\{1,\cdots,N\}^{K-1}$.
    By~\eqref{eq:thm minimax lower bound claim existence eq 0.02} and~\eqref{eq:thm minimax lower bound claim existence eq 0.03}, we have
    \begin{align*}
    &\|\sum_{I}w(I)h_{\textsf{base}}(2N\widetilde{x}-2I+\bm{1})\|_{\mathcal{C}^{\alpha,K-1}}\\
    &\leq 2\max_{I\in\{1,\cdots,N\}^{K-1}}\|h_{\textsf{base}}(2N\widetilde{x}-2I+\bm{1})\|_{\mathcal{C}^{\alpha,K-1}}.
    \end{align*}
    
    (Derivation of~\eqref{eq:thm minimax lower bound claim existence eq 0.2})
    In~\eqref{eq:thm minimax lower bound claim existence eq 0.2}, by differentiating the composite functions, we have
    \begin{align*}
    &\|h_{\textsf{base}}(2N\widetilde{x}-2I+\bm{1})\|_{\mathcal{C}^{\alpha,K-1}}\\
    &=
    \sum_{\bm{s}:\|\bm{s}\|_{1}\leq \lceil\alpha-1\rceil} \|\partial_{\bm{s}}(h_{\textsf{base}}(2N\widetilde{x}-2I+\bm{1}))\|_{\infty}\\
    &\quad +
    \sum_{\bm{s}:\|\bm{s}\|_{1}=\lceil\alpha-1\rceil}\sup_{\widetilde{x}\neq \widetilde{x}'}\frac{|\partial_{\bm{s}}(h_{\textsf{base}}(2N\widetilde{x}-2I+\bm{1}))-\partial_{\bm{s}}(h_{\textsf{base}}(2N\widetilde{x}'-2I+\bm{1}))|}{\|\widetilde{x}-\widetilde{x}'\|_{\infty}^{\alpha-\lceil\alpha-1\rceil}}\\
    &\leq 
    (2N)^{\lceil\alpha-1\rceil} \sum_{\bm{s}:\|\bm{s}\|_{1}\leq \lceil\alpha-1\rceil}\|\partial_{\bm{s}}h_{\textsf{base}}(2N\widetilde{x}-2I+\bm{1})\|_{\infty}\\
    &+(2N)^{\lceil\alpha-1\rceil}\sum_{\substack{\bm{s}:\\ \|\bm{s}\|_{1}=\lceil\alpha-1\rceil}} \sup_{\widetilde{x}\neq\widetilde{x}'}\frac{|\partial_{\bm{s}}h_{\textsf{base}}(2N\widetilde{x}-2I+\bm{1})-\partial_{\bm{s}}h_{\textsf{base}}(2N\widetilde{x}'-2I+\bm{1})|}{\|\widetilde{x}-\widetilde{x}'\|_{\infty}^{\alpha-\lceil\alpha-1\rceil}}\\
    &\leq
    (2N)^{\lceil\alpha-1\rceil}\sum_{\bm{s}:\|\bm{s}\|_{1}\leq \lceil\alpha-1\rceil}\|\partial_{\bm{s}}h_{\textsf{base}}\|_{L^{\infty}(\mathbb{R}^{K-1})}\\
    &\quad +(2N)^{\alpha}\sum_{\bm{s}:\|\bm{s}\|_{1}=\lceil\alpha-1\rceil} \sup_{\widetilde{x}\neq \widetilde{x}'}\frac{|\partial_{\bm{s}}h_{\textsf{base}}(2N\widetilde{x}-2I+\bm{1})-\partial_{\bm{s}}h_{\textsf{base}}(2N\widetilde{x}'-2I+\bm{1})|}{\|(2N\widetilde{x}-2I+\bm{1})-(2N\widetilde{x}'-2I+\bm{1})\|_{\infty}^{\alpha-\lceil\alpha-1\rceil}}\\
    &\leq
    (2N)^{\alpha}\|h_{\textsf{base}}\|_{\mathcal{C}^{\alpha,K-1}(\mathbb{R}^{K-1})}.
    \end{align*}
    Here, note that $\partial_{\bm{s}}(h_{\textsf{base}}(2N\widetilde{x}-2I+\bm{1}))$ denotes the $\bm{s}$-partial derivative of the composite function $\widetilde{x}\mapsto h_{\textsf{base}}(2N\widetilde{x}-2I+\bm{1})$, while $\partial_{\bm{s}}h_{\textsf{base}}(2N\widetilde{x}-2I+\bm{1})$ denotes the value of the derivative $x\mapsto \partial_{\bm{s}}h_{\textsf{base}}(x)$ at $x=2N\widetilde{x}-2I+\bm{1}$.
    
    By~\eqref{eq:thm minimax lower bound claim existence eq 0.2}, the condition (E1) is satisfied for the constant $c_{2,1}$ defined as
    \begin{align*}
    c_{2,1}=(2^{\alpha+1}\|h_{\textsf{base}}\|_{\mathcal{C}^{\alpha,K-1}(\mathbb{R}^{K-1})})^{-1}(R-1+\theta_{3}).
    \end{align*}
    
    For the second condition, note that
    \begin{align*}
    |h_{w}(\widetilde{x})|\leq 1-\theta_{3}+c_{2}N^{-\alpha}\sup_{\widetilde{x}\in \mathbb{R}^{K-1}}|h_{\textsf{base}}(\widetilde{x})|\leq 1-\theta_{3}+c_{2},
    \end{align*}
    where \eqref{eq:thm minimax lower bound claim existence eq 0.0} is used.
    Here, we recall the assumption $\theta_{3}>\frac{1}{2}$.
    Thus, the condition (E2) is satisfied with $c_{2,2}=2\theta_{3}-1$.

    The condition (E3) is satisfied with
    \begin{align*}
    c_{2,3}=2^{K-2}3^{-1}p_{Y,\textsf{U}}(1)^{-1}(\int_{[-1,1]^{K-1}}h_{\textsf{base}}(\widetilde{x})d\widetilde{x})^{-1}.
    \end{align*}

    Therefore, we can take $c_{2}=\min\{c_{2,1},c_{2,2},c_{2,3}\}$ to satisfy all the conditions.
    \end{proof}

\begin{claim}
\label{claim:minimax lower bound claim 1}
We have $\mathcal{P}_{\alpha,1,\xi,N^{K-1}}\subset \mathcal{P}_{\alpha,1,\xi}$.
\end{claim}
\begin{proof}[Proof of Claim~\ref{claim:minimax lower bound claim 1}]
We check whether any $P_{w}\in \mathcal{P}_{\alpha,1,\xi,N^{K-1}}$ satisfies all the conditions (A1) -- (A4) in Definition~\ref{def:main assumption}.
\begin{itemize}
    \item All the conditions in (A1) are satisfied by the definition of $p_{w}(x,x',y)$.
    \item By the definitions of $q_{w}$ and $p_{w}$, we have
    \begin{align*}
        \eta_{w}(x,x')&=p_{w}(y=1|x,x')\\
        &=
        \begin{cases}
            \frac{p_{Y,\textsf{U}}(1)}{(1-P_{X,\textsf{U}}(\mathcal{K}_{w}))p_{Y,\textsf{U}}(1)+P_{X,\textsf{U}}(\mathcal{K}_{w})} &\quad \textup{if }(x,x')\in\mathcal{K}_{w}\times\mathcal{K}_{w},\\
            \frac{p_{Y,\textsf{U}}(1)}{(1-P_{X,\textsf{U}}(\widetilde{\mathcal{K}}_{w}))p_{Y,\textsf{U}}(1)+P_{X,\textsf{U}}(\widetilde{\mathcal{K}}_{w})} &\quad \textup{if }(x,x')\in\widetilde{\mathcal{K}}_{w}\times\widetilde{\mathcal{K}}_{w},\\
            0&\quad \textup{otherwise}.
        \end{cases}
    \end{align*}
    Since $p_{Y,\textsf{U}}(1)\in (\theta_{3}/(1+\theta_{3}),1)$ by the definition, we obtain
    \begin{align}
        \label{eq:minimax lower bound claim 1 eq 1}
        p_{Y,\textsf{U}}(1)>\frac{\theta_{3}}{1+\theta_{3}}\geq \max\left\{\frac{P_{X,\textsf{U}}(\mathcal{K}_{w})}{1+P_{X,\textsf{U}}(\mathcal{K}_{w})},\frac{P_{X,\textsf{U}}(\widetilde{\mathcal{K}}_{w})}{1+P_{X,\textsf{U}}(\widetilde{\mathcal{K}}_{w})}\right\},
    \end{align}
    where the second inequality is due to~\eqref{eq:thm minimax lower bound eq 1}, \eqref{eq:thm minimax lower bound eq 2}, and the monotonicity of the function $s\mapsto s/(1+s)$ on $[0,1]$.
    By~\eqref{eq:minimax lower bound claim 1 eq 1} and the definition of $\eta_{w}$, we have
    \begin{align}
    \label{eq:minimax lower bound claim 1 eq 2}
    \begin{cases}
        \eta_{w}(x,x') >\frac{1}{2}&\quad\textup{if }(x,x')\in(\mathcal{K}_{w}\times\mathcal{K}_{w})\cup (\widetilde{\mathcal{K}}_{w}\times\widetilde{\mathcal{K}}_{w}),\\
        \eta_{w}(x,x')=0&\quad \textup{otherwise}.
    \end{cases}
    \end{align}
    Since $p_{Y,\textsf{U}}(1)=\frac{1}{2}$ and $P_{X,\textsf{U}}(\mathcal{K}_{w})\vee P_{X,\textsf{U}}(\widetilde{\mathcal{K}}_{w})\leq \theta_{3}$ by~\eqref{eq:thm minimax lower bound eq 1} and~\eqref{eq:thm minimax lower bound eq 2}, for any $s\in (0,\frac{1-\theta_{3}}{2(1+\theta_{3})})$, we have
    \begin{align*}
        P_{X,X',w}(\{(x,x')\in\mathcal{X}^{2}\;|\; |2\eta_{w}(x,x')-1|\leq s\})=0,
    \end{align*}
    where $P_{X,X',w}$ denotes the marginal distribution of $P_{w}$ with respect to the space $\mathcal{X}^{2}$.
    This shows that $P_{w}$ satisfies condition (A2).
    \item By the definition of $p_{w}$, it is clear that all the conditions in (A3) except $\|q\|_{L^{\infty}(\mathcal{X}^{2})}\leq \theta_{1}^{2}$ are satisfied immediately. To check the non-trivial part, note that $P_{X,\textsf{U}}(\mathcal{K}_{w})\wedge P_{X,\textsf{U}}(\widetilde{\mathcal{K}}_{w})\geq 1-\theta_{3}$, as shown in~\eqref{eq:thm minimax lower bound eq 1} and in~\eqref{eq:thm minimax lower bound eq 2}. Since $\theta_{1}(1-\theta_{3})^{\frac{1}{2}}\geq 1$ is satisfied, we have $\|p_{X,\textsf{U}}\|_{L^{\infty}(\mathcal{X})}\leq \theta_{1}(1-\theta_{3})^{\frac{1}{2}}$.
    Hence, we have $\|q\|_{L^{\infty}(\mathcal{X}^{2})}\leq (1-\theta_{3})^{-1}(\theta_{1}(1-\theta_{3})^{\frac{1}{2}})^{2}=\theta_{1}^{2}$.
    \item By~\eqref{eq:thm minimax lower bound eq 1} and~\eqref{eq:thm minimax lower bound eq 2}, we note that $\max\{P_{X,\textsf{U}}(\mathcal{K}_{w}),P_{X,\textsf{U}}(\widetilde{\mathcal{K}}_{w})\}\leq \theta_{3}$.
    We also note that for $\mathcal{K}_{3}=\cdots=\mathcal{K}_{d_{1}}=\varnothing$, $\{\mathcal{K}_{w},\widetilde{\mathcal{K}}_{w},\mathcal{K}_{3},\cdots,\mathcal{K}_{d_{1}}\}\in\mathscr{P}_{\alpha,R}^{K,d_{1},E}$ by condition (E1).
    By~\eqref{eq:minimax lower bound claim 1 eq 2},
    \begin{align}
    \label{eq:minimax lower bound claim 1 eq 3}
    \eta_{w}(x,x')\geq \frac{1}{2}\quad\textup{if and only if}\quad (x,x')\in (\mathcal{K}_{w}\times\mathcal{K}_{w})\cup (\widetilde{\mathcal{K}}_{w}\times\widetilde{\mathcal{K}}_{w}).
    \end{align}
    The relationship~\eqref{eq:minimax lower bound claim 1 eq 3} indicates that condition (A4) is satisfied for $P_{w}$.
\end{itemize}
We obtain the claim.
\end{proof}

\paragraph{Step 2 (Lower bound of the minimax risk).}

By Claim~\ref{claim:minimax lower bound claim 1}, we note that
\begin{align}
\label{eq:minimax lower bound step 2 eq 1}
\inf_{\widehat{g}_{n}}\sup_{P\in\mathcal{P}_{\alpha,1,\xi}}\mathcal{R}(\widehat{g}_{n};P)
\geq
\inf_{\widehat{g}_{n}}\sup_{P_{w}\in\mathcal{P}_{\alpha,1,\xi,N^{K-1}}}\mathcal{R}(\widehat{g}_{n};P_{w}),
\end{align}
where the infimum is taken over all the global estimators $\widehat{g}_{n}:(\mathcal{X}^{2}\times\mathcal{Y})^{n}\to\mathcal{G}_{0}$.
In addition, we note that
\begin{align}
\label{eq:minimax lower bound step 2 eq 2}
&\inf_{\widehat{g}_{n}}\sup_{P_{w}\in\mathcal{P}_{\alpha,1,\xi,N^{K-1}}}\mathcal{R}(\widehat{g}_{n};P_{w})\nonumber\\
&=\inf_{\widehat{g}_{n}}\sup_{P_{w}\in\mathcal{P}_{\alpha,1,\xi,N^{K-1}}}\mathbb{E}_{U_{1}^{n}}\left[\sum_{i=1}^{d_{1}}\left\|\widehat{g}_{n,i}(U_{1}^{n})-\mathds{1}_{\mathcal{K}_{i}}\right\|_{L^{2}(\mathcal{X},P_{X,\textsf{U}})}^{2}\right]\nonumber\\
&\geq
\inf_{\widehat{g}_{n,1}}\sup_{P_{w}\in\mathcal{P}_{\alpha,1,\xi,N^{K-1}}}\mathbb{E}_{U_{1}^{n}}\left[\left\|\widehat{g}_{n,1}(U_{1}^{n})-\mathds{1}_{\mathcal{K}_{w}}\right\|_{L^{2}(\mathcal{X},P_{X,\textsf{U}})}^{2}\right],
\end{align}
where for every $P_{w}\in\mathcal{P}_{\alpha,1,\xi,N^{K-1}}$, $U_{1}^{n}=(U_{1},\cdots,U_{n})$ are any sequence of i.i.d. random variables drawn from $P_{w}$, and without loss of generality\footnote{In general, $\mathscr{S}_{P_{w}}=\{\mathcal{K}_{i}\}_{i=1}^{d_{1}}$ with $\mathcal{K}_{\pi(1)}=\mathcal{K}_{w}$, $\mathcal{K}_{\pi(2)}=\widetilde{\mathcal{K}}_{w}$, and $\mathcal{K}_{\pi(3)}=\cdots=\mathcal{K}_{\pi(d_{1})}=\varnothing$ for some permutation $\pi$ on $\{1,\cdots,d_{1}\}$. In this general case, it suffices to replace $\widehat{g}_{n,1}$ with $\widehat{g}_{n,\pi(1)}$ in the remained part of this step.}, we may assume that $\mathscr{S}_{P_{w}}=\{\mathcal{K}_{i}\}_{i=1}^{d_{1}}$ with $\mathcal{K}_{1}=\mathcal{K}_{w}$, $\mathcal{K}_{2}=\widetilde{\mathcal{K}}_{w}$, and $\mathcal{K}_{3}=\cdots=\mathcal{K}_{d_{1}}=\varnothing$.
Note that in~\eqref{eq:minimax lower bound step 2 eq 2}, the infimum is taken over all the estimators $\widehat{g}_{n,1}:(\mathcal{X}^{2}\times\mathcal{Y})^{n}\to \{g_{1}:\mathcal{X}\to [0,1]\;|\; g_{1}\textup{ is measurable}\}$.

For every $I\in\{1,\cdots,N\}^{K-1}$, define
\begin{align}
\label{eq:minimax lower bound partition}
\mathcal{M}_{N,I}=\left[\frac{I_{1}-1}{N},\frac{I_{1}}{N}\right)\times \left[\frac{I_{2}-1}{N},\frac{I_{2}}{N}\right)\times\cdots \times \left[\frac{I_{K-1}-1}{N},\frac{I_{K-1}}{N}\right).
\end{align}
Note that the closure of $\mathcal{M}_{N,I}$ is equal to $\overline{\mathcal{M}}_{N,I}$ defined in Step 1 of this proof.
The right-hand side of the inequality~\eqref{eq:minimax lower bound step 2 eq 2} is calculated as
\begin{align}
\label{eq:minimax lower bound step 2 eq 3}
&\inf_{\widehat{g}_{n,1}}\sup_{P_{w}\in\mathcal{P}_{\alpha,1,\xi,N^{K-1}}}\mathbb{E}_{U_{1}^{n}}\left[\left\|\widehat{g}_{n,1}(U_{1}^{n})-\mathds{1}_{\mathcal{K}_{w}}\right\|_{L^{2}(\mathcal{X},P_{X,\textsf{U}})}^{2}\right]\nonumber\\
&=
\inf_{\widehat{g}_{n,1}}\sup_{P_{w}\in\mathcal{P}_{\alpha,1,\xi,N^{K-1}}}\mathbb{E}_{U_{1}^{n}}\left[\sum_{I\in\{1,\cdots,N\}^{K-1}}\int_{\mathcal{M}_{N,I}\times [0,1]}(\widehat{g}_{n,1}(U_{1}^{n})-\mathds{1}_{\mathcal{K}_{w}})^{2}d\mu\right],
\end{align}
where we note that $dP_{X,\textsf{U}}/d\mu=1$ almost everywhere, and for any $I,I'\in\{1,\cdots,N\}^{K-1}$ such that $I\neq I'$, $\mathcal{M}_{N,I}$ and $\mathcal{M}_{N,I'}$ are disjoint.

Given $I\in\{1,\cdots,N\}^{K-1}$, let $w_{I,0},w_{I,1}:\{1,\cdots,N\}^{K-1}\to \{0,1\}$ be any functions such that $w_{I,1}(I')=w_{I,0}(I')$ for any $I'\in\{1,\cdots,N\}^{K-1}\setminus \{I\}$, $w_{I,1}(I)=1$, and $w_{I,0}(I)=0$, hereafter.
We note the following claim:
\begin{claim}
\label{claim:minimax lower bound claim lower bound of l2 norm}
Define
\begin{align*}
C_{\textsf{\textup{L1}}}=c_{2}2^{-K+1}\int_{[-1,1]^{K-1}}h_{\textsf{\textup{base}}}(\widetilde{x})d\widetilde{x}.
\end{align*}
For every $I\in\{1,\cdots,N\}^{K-1}$ and any $w_{I,1}$ and $w_{I,0}$ defined above, we have
\begin{align*}
\int_{\mathcal{M}_{N,I}\times [0,1]}(\mathds{1}_{\mathcal{K}_{w_{I,1}}}-\mathds{1}_{\mathcal{K}_{w_{I,0}}})^{2}d\mu
=
C_{\textsf{\textup{L1}}}N^{-\alpha-K+1}.
\end{align*}
\end{claim}
\begin{proof}[Proof of Claim~\ref{claim:minimax lower bound claim lower bound of l2 norm}]
We have
\begin{align}
&\int_{\mathcal{M}_{N,I}\times [0,1]}(\mathds{1}_{\mathcal{K}_{w_{I,1}}}-\mathds{1}_{\mathcal{K}_{w_{I,0}}})^{2}d\mu\nonumber\\
&=
\int_{\mathcal{M}_{N,I}\times [0,1]}|\mathds{1}_{\mathcal{K}_{w_{I,1}}}-\mathds{1}_{\mathcal{K}_{w_{I,0}}}|d\mu\nonumber\\
\label{eq:claim lower bound of l2 norm eq 1}
&=
\int_{\mathcal{M}_{N,I}\times [0,1]}\mathds{1}_{\mathcal{K}_{w_{I,1}}\setminus \mathcal{K}_{w_{I,0}}}d\mu\\
&=
\int_{\mathcal{M}_{N,I}\times [0,1]}\mathds{1}_{\{x\in\mathcal{X}\;|\; h_{w_{I,0}}(x_{\setminus K})\leq x_{K}< h_{w_{I,1}}(x_{\setminus K})\}}d\mu\nonumber\\
\label{eq:claim lower bound of l2 norm eq 2}
&=
\int_{\mathcal{M}_{N,I}}\int_{h_{w_{I,0}}(x_{\setminus K})}^{h_{w_{I,1}}(x_{\setminus K})}dx_{K}dx_{\setminus K}\\
\label{eq:claim lower bound of l2 norm eq 3}
&=c_{2}N^{-\alpha}\int_{\mathcal{M}_{N,I}}h_{\textsf{base}}\left(2N\left(x_{\setminus K}-\frac{2I-\bm{1}}{2N}\right)\right)dx_{\setminus K}\\
\label{eq:claim lower bound of l2 norm eq 4}
&=
c_{2}2^{-K+1}N^{-\alpha-K+1}\int_{[-1,1]^{K-1}}h_{\textsf{base}}(x_{\setminus K})dx_{\setminus K}\\
\label{eq:claim lower bound of l2 norm eq 5}
&=
C_{\textsf{L1}}N^{-\alpha-K+1}.
\end{align}
Here, in~\eqref{eq:claim lower bound of l2 norm eq 1}, we note that $\mathcal{K}_{w_{I,0}}\subset \mathcal{K}_{w_{I,1}}$.
In~\eqref{eq:claim lower bound of l2 norm eq 2}, we used Fubini's theorem.
In~\eqref{eq:claim lower bound of l2 norm eq 3}, we directly compute the difference $h_{w_{I,1}}-h_{w_{I,0}}$ using the definitions of $h_{w_{I,1}}$ and $h_{w_{I,0}}$.
In~\eqref{eq:claim lower bound of l2 norm eq 4}, we used the change of variables formula for $\widetilde{x}=2Nx_{\setminus K}-2I+\bm{1}$.
In~\eqref{eq:claim lower bound of l2 norm eq 5}, we used the definition of $C_{\textsf{L1}}$.
We obtain the claim.
\end{proof}
Since the functional $\int_{\mathcal{M}_{N,I}\times [0,1]}(g_{1}-g_{2})^{2}d\mu$ is non-negative and symmetric and satisfies~\eqref{eq:a variant of triangle inequality} with $c_{1}=\frac{1}{2}$, Claim~\ref{claim:minimax lower bound claim lower bound of l2 norm} implies that Lemma~\ref{citelem:assouad lemma} is applicable to our problem setting, where note that $\{1,\cdots,N\}^{K-1}$ is a finite set.
Hence, by Lemma~\ref{citelem:assouad lemma}, we have
\begin{align}
\label{eq:minimax lower bound step 2 eq 4}
&\inf_{\widehat{g}_{n,1}}\sup_{P_{w}\in\mathcal{P}_{\alpha,1,\xi,N^{K-1}}}\mathbb{E}_{U_{1}^{n}}\left[\sum_{I\in\{1,\cdots,N\}^{K-1}}\int_{\mathcal{M}_{N,I}\times [0,1]}(\widehat{g}_{n,1}(U_{1}^{n})-\mathds{1}_{\mathcal{K}_{w}})^{2}d\mu\right]\nonumber\\
&\geq
C_{\textsf{L2}}N^{-\alpha}\min_{I\in\{1,\cdots,N\}^{K-1}, w_{I,1},w_{I,0}}\int_{(\mathcal{X}^{2}\times\mathcal{Y})^{n}}\min\left\{p_{w_{I,1}}^{\otimes n},p_{w_{I,0}}^{\otimes n}\right\}d(\mu\otimes \mu\otimes \chi)^{\otimes n},
\end{align}
where $C_{\textsf{L2}}=4^{-1}C_{\textsf{L1}}$.

Combining~\eqref{eq:minimax lower bound step 2 eq 1}, \eqref{eq:minimax lower bound step 2 eq 2}, \eqref{eq:minimax lower bound step 2 eq 3}, and~\eqref{eq:minimax lower bound step 2 eq 4}, we obtain
\begin{align}
\label{eq:minimax lower bound step 2 eq 5}
&\inf_{\widehat{g}_{n}}\sup_{P\in\mathcal{P}_{\alpha,1,\xi}}\mathcal{R}(\widehat{g}_{n};P)\nonumber\\
&\geq
C_{\textsf{L2}}N^{-\alpha}\min_{I\in\{1,\cdots,N\}^{K-1}, w_{I,1},w_{I,0}}\int_{(\mathcal{X}^{2}\times\mathcal{Y})^{n}}\min\left\{p_{w_{I,1}}^{\otimes n},p_{w_{I,0}}^{\otimes n}\right\}d(\mu\otimes \mu\otimes \chi)^{\otimes n}.
\end{align}

\paragraph{Step 3 (Evaluation of the integral).}

In this step, we derive a lower bound of the integral
\begin{align*}
\min_{I\in \{1,\cdots,N\}^{K-1},w_{I,1},w_{I,0}}\int_{(\mathcal{X}^{2}\times\mathcal{Y})^{n}} \min\left\{p_{w_{I,1}}^{\otimes n},p_{w_{I,0}}^{\otimes n}\right\}d(\mu\otimes \mu\otimes \chi)^{\otimes n}.
\end{align*}

Let $I\in\{1,\cdots,N\}^{K-1}$, $w_{I,1}$, and $w_{I,0}$ be arbitrary and fixed in this step.
By Fubini's theorem and the inequalities $p_{w_{I,1}}(x)\geq \min\{p_{w_{I,1}}(x),p_{w_{I,0}}(x)\}$ and $p_{w_{I,0}}(x)\geq \min\{p_{w_{I,1}}(x),p_{w_{I,0}}(x)\}$, we have
\begin{align}
\label{eq:thm minimax lower bound eq 8}
&\int_{(\mathcal{X}^{2}\times\mathcal{Y})^{n}}\min\left\{p_{w_{I,1}}^{\otimes n},p_{w_{I,0}}^{\otimes n}\right\}d(\mu\otimes \mu\otimes \chi)^{\otimes n}\nonumber \\
&\geq
\left(\int_{\mathcal{X}^{2}\times\mathcal{Y}}\min\left\{p_{w_{I,1}},p_{w_{I,0}}\right\}d(\mu\otimes\mu\otimes \chi)\right)^{n}.
\end{align}
By Scheffe's identity (see, e.g., Lemma~2.1 in~\citep{tsybakov2009introduction}), we have
\begin{align}
\label{eq:thm minimax lower bound eq 9}
&\int_{\mathcal{X}^{2}\times\mathcal{Y}}\min\{p_{w_{I,1}},p_{w_{I,0}}\}d(\mu\otimes\mu\otimes \chi)\nonumber\\
&=
1-\frac{1}{2}\int_{\mathcal{X}^{2}\times\mathcal{Y}}|p_{w_{I,1}}-p_{w_{I,0}}|d(\mu\otimes \mu\otimes \chi).
\end{align}
The second term in the right-hand side of~\eqref{eq:thm minimax lower bound eq 9} is bounded as follows:
\begin{claim}
\label{claim:thm minimax lower bound claim 5}
Define
\begin{align*}
C_{\textsf{\textup{L3}}}=2^{-K+2}3c_{2}p_{Y,\textsf{U}}(1)\int_{[-1,1]^{K-1}}h_{\textsf{\textup{base}}}(\widetilde{x})d\widetilde{x}.
\end{align*}
We have
\begin{align*}
\int_{\mathcal{X}^{2}\times\mathcal{Y}}|p_{w_{I,1}}-p_{w_{I,0}}|d(\mu\otimes \mu\otimes \chi)
\leq
C_{\textsf{\textup{L3}}}N^{-\alpha-K+1}.
\end{align*}
\end{claim}
\begin{proof}[Proof of Claim~\ref{claim:thm minimax lower bound claim 5}]
We have
\begin{align}
\label{eq:thm minimax lower bound eq 10}
&\int_{\mathcal{X}^{2}\times\mathcal{Y}}|p_{w_{I,1}}-p_{w_{I,0}}|d(\mu\otimes \mu\otimes \chi)\nonumber\\
&=
p_{Y,\textsf{U}}(1)\int_{\mathcal{X}^{2}}|q_{w_{I,1}}-q_{w_{I,0}}|d\mu^{\otimes 2}\nonumber\\
&\quad+
p_{Y,\textsf{U}}(-1)\int_{\mathcal{X}^{2}}|p_{X,\textsf{U}}(x)p_{X,\textsf{U}}(x')-p_{X,\textsf{U}}(x)p_{X,\textsf{U}}(x')|d\mu^{\otimes 2}\nonumber\\
&=p_{Y,\textsf{U}}(1)\int_{\mathcal{X}^{2}}|q_{w_{I,1}}-q_{w_{I,0}}|d\mu^{\otimes 2}\\
&=p_{Y,\textsf{U}}(1)\int_{\mathcal{X}^{2}}|P_{X,\textsf{U}}(\mathcal{K}_{w_{I,1}})^{-1}\mathds{1}_{\mathcal{K}_{w_{I,1}}\times\mathcal{K}_{w_{I,1}}}+P_{X,\textsf{U}}(\widetilde{\mathcal{K}}_{w_{I,1}})^{-1}\mathds{1}_{\widetilde{\mathcal{K}}_{w_{I,1}}\times \widetilde{\mathcal{K}}_{w_{I,1}}}\nonumber\\
\label{eq:thm minimax lower bound eq 11}
&\qquad\qquad\quad\; -P_{X,\textsf{U}}(\mathcal{K}_{w_{I,0}})^{-1}\mathds{1}_{\mathcal{K}_{w_{I,0}}\times \mathcal{K}_{w_{I,0}}}-P_{X,\textsf{U}}(\widetilde{\mathcal{K}}_{w_{I,0}})^{-1}\mathds{1}_{\widetilde{\mathcal{K}}_{w_{I,0}}\times\widetilde{\mathcal{K}}_{w_{I,0}}}|d\mu^{\otimes 2}\\
\label{eq:thm minimax lower bound eq 12}
&\leq
p_{Y,\textsf{U}}(1)\int_{\mathcal{X}^{2}}|P_{X,\textsf{U}}(\mathcal{K}_{w_{I,1}})^{-1}\mathds{1}_{\mathcal{K}_{w_{I,1}}\times \mathcal{K}_{w_{I,1}}}-P_{X,\textsf{U}}(\mathcal{K}_{w_{I,0}})^{-1}\mathds{1}_{\mathcal{K}_{w_{I,0}}\times \mathcal{K}_{w_{I,0}}}|d\mu^{\otimes 2}\nonumber\\
&+p_{Y,\textsf{U}}(1)\int_{\mathcal{X}^{2}}|P_{X,\textsf{U}}(\widetilde{\mathcal{K}}_{w_{I,0}})^{-1}\mathds{1}_{\widetilde{\mathcal{K}}_{w_{I,0}}\times \widetilde{\mathcal{K}}_{w_{I,0}}}-P_{X,\textsf{U}}(\widetilde{\mathcal{K}}_{w_{I,1}})^{-1}\mathds{1}_{\widetilde{\mathcal{K}}_{w_{I,1}}\times \widetilde{\mathcal{K}}_{w_{I,1}}}|d\mu^{\otimes 2},
\end{align}
where we note that $p_{X,\textsf{U}}(x)=1$ holds for any $x\in\mathcal{X}$ by the definition in~\eqref{eq:thm minimax lower bound eq 10}, the definitions of $q_{w_{I,1}}$ and $q_{w_{I,0}}$ are used in~\eqref{eq:thm minimax lower bound eq 11}, and the triangle inequality is used in~\eqref{eq:thm minimax lower bound eq 12}.
We note that
\begin{align}
\label{eq:thm minimax lower bound eq 13}
&\int_{\mathcal{X}^{2}}|P_{X,\textsf{U}}(\mathcal{K}_{w_{I,1}})^{-1}\mathds{1}_{\mathcal{K}_{w_{I,1}}\times \mathcal{K}_{w_{I,1}}}-P_{X,\textsf{U}}(\mathcal{K}_{w_{I,0}})^{-1}\mathds{1}_{\mathcal{K}_{w_{I,0}}\times \mathcal{K}_{w_{I,0}}}|d\mu^{\otimes 2}\nonumber\\
&=
\int_{\mathcal{K}_{w_{I,0}}\times \mathcal{K}_{w_{I,0}}}|P_{X,\textsf{U}}(\mathcal{K}_{w_{I,1}})^{-1}\cdot 1-P_{X,\textsf{U}}(\mathcal{K}_{w_{I,0}})^{-1}\cdot 1|d\mu^{\otimes 2}\nonumber\\
&\quad
+
\int_{\mathcal{X}^{2}\setminus (\mathcal{K}_{w_{I,0}}\times \mathcal{K}_{w_{I,0}})}|P_{X,\textsf{U}}(\mathcal{K}_{w_{I,1}})^{-1}\mathds{1}_{\mathcal{K}_{w_{I,1}}\times\mathcal{K}_{w_{I,1}}}-P_{X,\textsf{U}}(\mathcal{K}_{w_{I,0}})^{-1}\cdot 0|d\mu^{\otimes 2}\nonumber\\
&\leq
(P_{X,\textsf{U}}(\mathcal{K}_{w_{I,0}})^{-1}-P_{X,\textsf{U}}(\mathcal{K}_{w_{I,1}})^{-1})\mu(\mathcal{K}_{w_{I,0}})^{2}
+2\mu(\mathcal{K}_{w_{I,1}}\setminus \mathcal{K}_{w_{I,0}}).
\end{align}
Here, in~\eqref{eq:thm minimax lower bound eq 13} we use $\mathcal{K}_{w_{I,0}}\times \mathcal{K}_{w_{I,0}}\subset \mathcal{K}_{w_{I,1}}\times\mathcal{K}_{w_{I,1}}$, $(\mathcal{K}_{w_{I,1}}\times\mathcal{K}_{w_{I,1}})\setminus (\mathcal{K}_{w_{I,0}}\times\mathcal{K}_{w_{I,0}})=((\mathcal{K}_{w_{I,1}}\setminus \mathcal{K}_{w_{I,0}})\times \mathcal{K}_{w_{I,1}})\cup (\mathcal{K}_{w_{I,1}}\times (\mathcal{K}_{w_{I,1}}\setminus \mathcal{K}_{w_{I,0}}))$, and $P_{X,\textsf{U}}(\mathcal{K}_{w_{I,1}})=\int_{\mathcal{K}_{w_{I,1}}}p_{X,\textsf{U}}d\mu=\int_{\mathcal{K}_{w_{I,1}}}d\mu=\mu(\mathcal{K}_{w_{I,1}})$ to obtain
\begin{align*}
&\int_{\mathcal{X}^{2}\setminus (\mathcal{K}_{w_{I,0}}\times\mathcal{K}_{w_{I,0}})}|P_{X,\textsf{U}}(\mathcal{K}_{w_{I,1}})^{-1} \mathds{1}_{\mathcal{K}_{w_{I,1}}\times\mathcal{K}_{w_{I,1}}}-P_{X,\textsf{U}}(\mathcal{K}_{w_{I,0}})^{-1}\cdot 0|d\mu^{\otimes 2}\\
&=
\int_{(\mathcal{K}_{w_{I,1}}\times\mathcal{K}_{w_{I,1}})\setminus (\mathcal{K}_{w_{I,0}}\times\mathcal{K}_{w_{I,0}})}P_{X,\textsf{U}}(\mathcal{K}_{w_{I,1}})^{-1}d\mu^{\otimes 2}\\
&\leq
2\mu(\mathcal{K}_{w_{I,1}}\setminus \mathcal{K}_{w_{I,0}}).
\end{align*}
From~\eqref{eq:thm minimax lower bound eq 13}, we notice
\begin{align}
\label{eq:thm minimax lower bound eq 14}
&(P_{X,\textsf{U}}(\mathcal{K}_{w_{I,0}})^{-1}-P_{X,\textsf{U}}(\mathcal{K}_{w_{I,1}})^{-1})\mu(\mathcal{K}_{w_{I,0}})^{2}\nonumber\\
&=
\frac{P_{X,\textsf{U}}(\mathcal{K}_{w_{I,1}}\setminus \mathcal{K}_{w_{I,0}})}{P_{X,\textsf{U}}(\mathcal{K}_{w_{I,0}})P_{X,\textsf{U}}(\mathcal{K}_{w_{I,1}})}\mu(\mathcal{K}_{w_{I,0}})^{2}\nonumber\\
&=
P_{X,\textsf{U}}(\mathcal{K}_{w_{I,1}}\setminus \mathcal{K}_{w_{I,0}})\frac{\mu(\mathcal{K}_{w_{I,0}})}{\mu(\mathcal{K}_{w_{I,1}})}\nonumber\\
&\leq
\mu(\mathcal{K}_{w_{I,1}}\setminus \mathcal{K}_{w_{I,0}}),
\end{align}
where in~\eqref{eq:thm minimax lower bound eq 14} we used the monotonicity of the Lebesgue measure.
Similarly, we have
\begin{align}
&\int_{\mathcal{X}^{2}}|P_{X,\textsf{U}}(\widetilde{\mathcal{K}}_{w_{I,0}})^{-1}\mathds{1}_{\widetilde{\mathcal{K}}_{w_{I,0}}\times \widetilde{\mathcal{K}}_{w_{I,0}}}-P_{X,\textsf{U}}(\widetilde{\mathcal{K}}_{w_{I,1}})^{-1}\mathds{1}_{\widetilde{\mathcal{K}}_{w_{I,1}}\times\widetilde{\mathcal{K}}_{w_{I,1}}}|d\mu^{\otimes 2}\nonumber\\
\label{eq:thm minimax lower bound eq 15}
&\leq
(P_{X,\textsf{U}}(\widetilde{\mathcal{K}}_{w_{I,1}})^{-1}-P_{X,\textsf{U}}(\widetilde{\mathcal{K}}_{w_{I,0}})^{-1})\mu(\widetilde{\mathcal{K}}_{w_{I,1}})^{2}+2\mu(\widetilde{\mathcal{K}}_{w_{I,0}}\setminus \widetilde{\mathcal{K}}_{w_{I,1}})\\
\label{eq:thm minimax lower bound eq 16}
&\leq
3\mu(\widetilde{\mathcal{K}}_{w_{I,0}}\setminus \widetilde{\mathcal{K}}_{w_{I,1}}),
\end{align}
where in~\eqref{eq:thm minimax lower bound eq 15} and \eqref{eq:thm minimax lower bound eq 16} we use the same arguments as those used in~\eqref{eq:thm minimax lower bound eq 13} and~\eqref{eq:thm minimax lower bound eq 14}, respectively.

Note that
\begin{align}
\label{eq:thm minimax lower bound eq 17}
\mu(\widetilde{\mathcal{K}}_{w_{I,0}}\setminus \widetilde{\mathcal{K}}_{w_{I,1}})
&=\mu(\{x\in\mathcal{X}\;|\; h_{w_{I,0}}(x_{\setminus K})\leq x_{K} < h_{w_{I,1}}(x_{\setminus K})\})\nonumber\\
&=\mu(\mathcal{K}_{w_{I,1}}\setminus \mathcal{K}_{w_{I,0}}).
\end{align}
By~\eqref{eq:thm minimax lower bound eq 12}, \eqref{eq:thm minimax lower bound eq 13}, \eqref{eq:thm minimax lower bound eq 14},~\eqref{eq:thm minimax lower bound eq 16}, and~\eqref{eq:thm minimax lower bound eq 17}, we have
\begin{align}
\label{eq:thm minimax lower bound eq 18}
\int_{\mathcal{X}^{2}\times\mathcal{Y}}|p_{w_{I,1}}-p_{w_{I,0}}|d(\mu\otimes\mu\otimes \chi)\leq
6p_{Y,\textsf{U}}(1)\mu(\mathcal{K}_{w_{I,1}}\setminus \mathcal{K}_{w_{I,0}}).
\end{align}
The quantity $\mu(\mathcal{K}_{w_{I,1}}\setminus \mathcal{K}_{w_{I,0}})$ is calculated as
\begin{align}
\mu(\mathcal{K}_{w_{I,1}}\setminus \mathcal{K}_{w_{I,0}})
&=
\int_{[0,1]^{K-1}}\int_{h_{w_{I,0}}(x_{\setminus K})}^{h_{w_{I,1}}(x_{\setminus K})}dx_{K}dx_{\setminus K}\nonumber\\
&=
c_{2}N^{-\alpha}\int_{[0,1]^{K-1}} h_{\textsf{base}}\left(2N\left(\widetilde{x}-\frac{2I-\bm{1}}{2N}\right)\right)d\widetilde{x}\nonumber\\
\label{eq:thm minimax lower bound eq 19}
&=
c_{2}2^{-K+1}N^{-\alpha-K+1}\int_{\prod_{i=1}^{K-1}[-2I_{i}+1,2N-2I_{i}+1]}h_{\textsf{base}}(\widetilde{x})d\widetilde{x}\\
\label{eq:thm minimax lower bound eq 20}
&=c_{2}2^{-K+1}N^{-\alpha-K+1}\int_{[-1,1]^{K-1}}h_{\textsf{base}}(\widetilde{x})d\widetilde{x}
\end{align}
Here, we used the change of variables formula for $\widetilde{x}'=2N(\widetilde{x}-(2I-\bm{1})/2N)$ in~\eqref{eq:thm minimax lower bound eq 19}.
In~\eqref{eq:thm minimax lower bound eq 20}, we note that the support of $h_{\textsf{base}}$ is $[-1,1]^{K-1}$ by its definition, and $[-1,1]\subset [-2I_{i}+1,2N-2I_{i}+1]$ for any $i\in\{1,\cdots,K-1\}$ since $I_{1},\cdots,I_{K-1}\in \{1,\cdots,N\}$.

By~\eqref{eq:thm minimax lower bound eq 18} and~\eqref{eq:thm minimax lower bound eq 20}, we have
\begin{align*}
\int_{\mathcal{X}^{2}\times\mathcal{Y}}|p_{w_{I,1}}-p_{w_{I,0}}|d(\mu\otimes \mu\otimes \chi)\leq C_{\textsf{L3}} N^{-\alpha-K+1}.
\end{align*}
We obtain the claim.
\end{proof}

Thus, combining~\eqref{eq:thm minimax lower bound eq 8}, \eqref{eq:thm minimax lower bound eq 9}, and Claim~\ref{claim:thm minimax lower bound claim 5}, we have
\begin{align}
\label{eq:thm minimax lower bound eq 21}
\int_{(\mathcal{X}^{2}\times\mathcal{Y})^{n}}\min\left\{p_{w_{I,1}}^{\otimes n},p_{w_{I,0}}^{\otimes n}\right\}d(\mu\otimes \mu\otimes \chi)^{\otimes n}
&\geq 
\left(1-2^{-1}C_{\textsf{L3}}N^{-\alpha-K+1}\right)^{n}\\
\label{eq:thm minimax lower bound eq 22}
&\geq
(1-2^{-1}N^{-\alpha-K+1})^{n}.
\end{align}
where in~\eqref{eq:thm minimax lower bound eq 21} and \eqref{eq:thm minimax lower bound eq 22} we used $0\leq 2^{-1}C_{\textsf{L3}}N^{-\alpha-K+1}\leq \frac{1}{2}$, which is due to the condition (E3), that is, $C_{\textsf{L3}}\leq 1$.

By~\eqref{eq:minimax lower bound step 2 eq 5} and \eqref{eq:thm minimax lower bound eq 22}, we have
\begin{align}
\label{eq:thm minimax lower bound eq 23}
\inf_{\widehat{g}_{n}}\sup_{P\in\mathcal{P}_{\alpha,1,\xi}}\mathcal{R}(\widehat{g}_{n};P)
\geq C_{\textsf{L2}}N^{-\alpha}(1-2^{-1}N^{-\alpha-K+1})^{n}.
\end{align}

\paragraph{Step 4 (Concluding the proof).}

We now set $N=\lfloor n^{\frac{1}{\alpha+K-1}}\rfloor$.
By~\eqref{eq:thm minimax lower bound eq 23}, we have
\begin{align}
\label{eq:thm minimax lower bound eq 24}
\inf_{\widehat{g}_{n}}\sup_{P\in\mathcal{P}_{\alpha,1,\xi}}\mathcal{R}(\widehat{g}_{n};P)
\geq C_{\textsf{L2}}\left(1-\frac{1}{2(n-1)}\right)^{n} n^{-\frac{\alpha}{\alpha+K-1}}.
\end{align}
Since $(1-1/(2(n-1)))^{n}\geq \frac{1}{2}$ for any $n\in\mathbb{N}$, by~\eqref{eq:thm minimax lower bound eq 24} we have
\begin{align}
\label{eq:thm minimax lower bound eq 25}
\inf_{\widehat{g}_{n}}\sup_{P\in\mathcal{P}_{\alpha,1,\xi}}\mathcal{R}(\widehat{g}_{n};P)
\geq C_{\textsf{L4}}n^{-\frac{\alpha}{\alpha+K-1}},
\end{align}
where $C_{\textsf{L4}}=\frac{1}{2}C_{\textsf{L2}}$.
We obtain the claim of the theorem.
\end{proof}

\section{Consequences for Global Estimators}
\label{subsec:an auxiliary result for the global erm}

We show that Theorem~\ref{thm:estimation error for deep relu networks} has some implication to the global estimators.

\subsection{Results}
\label{supp:results of global estimators}

For convenience, given an estimator $\widehat{g}_{n}:(\mathcal{X}^{2}\times\mathcal{Y})^{n}\to\mathcal{G}_{0}$, we define
\begin{align*}
\widehat{\mathcal{R}}(\widehat{g}_{n};P,u_{1}^{n})=\sum_{i=1}^{d_{1}}\|\widehat{g}_{n,i}(u_{1},\cdots,u_{n})-\mathds{1}_{\mathcal{K}_{i}}\|_{L^{2}(\mathcal{X},P_{X})}^{2},
\end{align*}
where $u_{1}^{n}=(u_{1},\cdots,u_{n})\in(\mathcal{X}^{2}\times\mathcal{Y})^{n}$.
We define the global ERM estimator, similarly to Definition~\ref{def:formal definition of erm}.
\begin{definition}[$(n,\mathcal{F})$-global ERM]
\label{def:global empirical risk minimizers}
Given $n\in\mathbb{N}\setminus\{1,2\}$ and $\mathcal{F}\subset\mathcal{F}_{0}$, define $\widehat{f}_{n}^{\textup{ERM}}:(\mathcal{X}^{2}\times\mathcal{Y})^{n}\to\mathcal{F}$ as a map satisfying
\begin{align*}
\widehat{f}_{n}^{\textup{ERM}}(u_{1},\cdots,u_{n})\in\argmin_{f\in\mathcal{F}}\frac{1}{n}\sum_{i=1}^{n}\ell_{f}(u_{i}),
\end{align*}
where $u_{1},\cdots,u_{n}\in\mathcal{X}^{2}\times\mathcal{Y}$.
Then, the $(n,\mathcal{F})$\emph{-global ERM estimator} $\widehat{g}_{n}^{\textup{ERM}}:(\mathcal{X}^{2}\times\mathcal{Y})^{n}\to\mathcal{G}_{0}$ is defined as
\begin{align*}
\widehat{g}_{n}^{\textup{ERM}}=(\widehat{g}_{n,1}^{\textup{ERM}},\cdots,\widehat{g}_{n,d_{1}}^{\textup{ERM}})\textup{ such that }\widehat{f}_{n}^{\textup{ERM}}=\sum_{i=1}^{d_{1}}\widehat{g}_{n,i}^{\textup{ERM}}v_{i}.
\end{align*}
\end{definition}

In the following theorem, we prove an upper bound for a global estimator when $\tau=1$, while its generalization to any $\tau\geq 1$ is straightforward.
We show the case where $\tau=1$, for simplicity.
\begin{theorem}
\label{cor:result for the global erm}
Let $\alpha>0$, $\xi=(R,K,d_{1},E,\theta_{\textup{NC}},\theta_{1},\theta_{2},\theta_{3})\in \Xi$, $\beta\in (0,D_{{\textup{proj}}})$, and $n\in\mathbb{N}\setminus \{1,2\}$ for which $\varepsilon_{n}:=n^{-\frac{\alpha}{\alpha+K-1}}<2^{-1}$ is satisfied.
Given any $P\in\mathcal{P}_{\alpha,1,\xi}$, let $U_{1},\cdots,U_{n}$ be any sequence of i.i.d. random variables drawn from the distribution $P$.
Then, there are
\begin{itemize}
\item constants $C^{*},c_{1},c_{2}>0$ independent of $n$ and $P$,
\item $L^{*}\in\mathbb{N}$, $J^{*},S^{*},M^{*}\geq 0$, and $d^{*}=(K,d_{\textup{NN},1}^{*},\cdots,d_{\textup{NN},L^{*}-1}^{*},d_{1})\in\mathbb{N}^{L^{*}+1}$ depending on $n$, and
\item a global estimator $\widehat{f}_{n}:(\mathcal{X}^{2}\times\mathcal{Y})^{n}\to\mathcal{F}^{*}\subset\mathcal{F}_{0}$ with $\mathcal{F}^{*}:=\mathcal{F}_{L^{*},J^{*},S^{*},M^{*},\bm{d}^{*}}^{\Delta^{d}\textup{-NN}}$,
\end{itemize}
such that for the global estimator $\widehat{g}_{n}:(\mathcal{X}^{2}\times\mathcal{Y})^{n}\to \mathcal{G}_{0}$ satisfying that $\widehat{f}_{n}=\sum_{i=1}^{d_{1}}\widehat{g}_{n,i}v_{i}$ with $\widehat{g}_{n}=(\widehat{g}_{n,1},\cdots,\widehat{g}_{n,d_{1}})$, we have the following inequality:
With probability at least
\begin{align*}
1-c_{1}\exp(-128c_{2}n\varepsilon_{n}\log^{3}{n})-Q(\widehat{f}_{n}(U_{1}^{n})\notin \mathscr{F}_{\beta,\beta^{-1}\varepsilon_{n},P}(\mathcal{F}^{*})),
\end{align*}
we have
\begin{align*}
\widehat{\mathcal{R}}(\widehat{g}_{n};P,U_{1}^{n})\leq C^{*}\varepsilon_{n}\log^{4}{n}.
\end{align*}
Furthermore, we obtain
\begin{align*}
&\sup_{P\in\mathcal{P}_{\alpha,1,\xi}}\mathcal{R}(\widehat{g}_{n};P)\\
&\leq
2C^{*}\varepsilon_{n}\log^{4}{n}+4d_{1}C^{*}\sup_{P\in\mathcal{P}_{\alpha,1,\xi}}Q(\widehat{f}_{n}(U_{1}^{n})\notin \mathscr{F}_{\beta,\beta^{-1}\varepsilon_{n},P}(\mathcal{F}^{*})).
\end{align*}
\end{theorem}
The proof of Theorem~\ref{cor:result for the global erm} is deferred to Appendix~\ref{appsubsec:proof of corollary result for global erm}.
This corollary implies the existence of a global estimator whose convergence rate does not exceed $n^{-\frac{\alpha}{\alpha+K-1}}$.
This observation is similar to the results proven in the literature on nonparametric statistics~\citep{tsybakov2004optimal,kim2021fast,meyer2023optimal,imaizumi2019deep,imaizumi2022advantage}, although we consider a pairwise binary classification setting.
The global estimator used in this theorem is based on the ERM algorithm over the class $\mathcal{F}_{L^{*},J^{*},S^{*},M^{*},\bm{d}^{*}}^{\Delta^{d}\textup{-NN}}$, which does not depend on variable $P$ (see Definition~\ref{def:global empirical risk minimizers}).

In connection with Theorem~\ref{thm:minimax lower bound}, it might be natural to ask the following question:
\begin{question}
\label{conjecture: property of global erm}
In Theorem~\ref{cor:result for the global erm}, is it true that
\begin{align*}
\sup_{P\in\mathcal{P}_{\alpha,1,\xi}} Q(\widehat{f}_{n}^{\textup{ERM}}(U_{1}^{n})\notin \mathscr{F}_{\beta,\beta^{-1}\varepsilon_{n},P}(\mathcal{F}^{*}))\leq n^{-\frac{\alpha}{\alpha+K-1}}\;?
\end{align*}
\end{question}
The current work could not answer whether this hypothesis is true, while it might be intuitively reasonable when $n$ is sufficiently large, as the minimizer of the expected loss is the contrastive function $f^{*}$ (see Proposition~\ref{thm:main result basic case}), and the indicator functions $\mathds{1}_{\mathcal{K}_{1}},\cdots,\mathds{1}_{\mathcal{K}_{d_{1}}}$ with $\{\mathcal{K}_{i}\}_{i=1}^{d_{1}}\in\mathscr{P}_{\alpha,R}^{K,d_{1},E}$ can be approximated using some class $\mathcal{F}_{L,J,S,M,\bm{d}}^{\Delta^{d}\textup{-NN}}$ within error $n^{-\frac{\alpha}{\alpha+K-1}}$ under the $L^{1}(\mathcal{X})$-norm (see Proposition~\ref{prop:approximation error bound} and Proposition~\ref{thm:approximation and concentration}).
A technical obstacle might be that one needs to find some empirical risk minimizer that approximates the true smooth boundaries within an arbitrary error.
While the approximation error of some deep ReLU networks is analyzed in Proposition~\ref{prop:approximation error bound}, the investigation of the approximation property of the empirical risk minimizer seems more complicated than that of the usual deep ReLU networks since the boundaries generated by the minimizer might be influenced by the noise contained in the observed samples.

\subsection{Proof of Theorem~\ref{cor:result for the global erm}}
\label{appsubsec:proof of corollary result for global erm}

\begin{proof}
As in Remark~\ref{remark:local erm}, without loss of generality we may assume the existence of the $(n,\mathcal{F}^{*})$-global ERM estimator $\widehat{g}_{n}^{\textup{ERM}}$.

The proof is almost the same as that of Theorem~\ref{thm:estimation error for deep relu networks}.
Specifically, by Proposition~\ref{prop:simple fact for risks}--(ii) and~\eqref{eq:thm estimation bound eq 3} in the proof of Theorem~\ref{thm:main result}, there are positive constants $C,C',C''$ independent of $n$ and $P$ such that in the event $\{\omega\in\Omega\;|\; \widehat{f}_{n}^{\textup{ERM}}(U_{1}^{n}(\omega))\in \mathscr{F}_{\beta,\beta^{-1}\varepsilon_{n},P}(\mathcal{F}^{*})\}$ we have
\begin{align}
\label{eq:proof of corollary result for global erm eq 2}
&\widehat{\mathcal{R}}(\widehat{g}_{n}^{\textup{ERM}};P,U_{1}^{n}(\omega))\nonumber\\
&\leq
C(\log{n})\mathcal{E}(\widehat{f}_{n}^{\textup{ERM}}(U_{1}^{n}(\omega));P)+\frac{C'\varepsilon_{n}}{\beta}+\frac{C''}{n}.
\end{align}
We then consider to apply Lemma~\ref{lem:lemma by kim et al} due to~\citep{park2009convergence,kim2021fast}.
Using Proposition~\ref{prop:approximation error bound}, we can see that conditions (C1) -- (C3) in Lemma~\ref{lem:lemma by kim et al} are satisfied for $\mathcal{F}^{*}$, as shown in Claim~\ref{claim:check condition c1 to c3}.
Thus, by Lemma~\ref{lem:lemma by kim et al}, there are positive universal constants $c_{1},c_{2}$ such that with probability at least $1-c_{1}\exp(-128c_{2}n\varepsilon_{n}\log^{3}{n})$, we have
\begin{align*}
\mathcal{E}(\widehat{f}_{n}^{\textup{ERM}}(U_{1}^{n});P)\leq 128C_{1}\varepsilon_{n}\log^{3}{n},
\end{align*}
where $C_{1}$ is the constant introduced in Proposition~\ref{prop:approximation property of H}.

Therefore, we obtain the following:
With probability at least
\begin{align*}
1-c_{1}\exp(-128c_{2}n\varepsilon_{n}\log^{3}{n})-Q(\widehat{f}_{n}^{\textup{ERM}}(U_{1}^{n})\notin \mathscr{F}_{\beta,\beta^{-1}\varepsilon_{n},P}(\mathcal{F}^{*})),
\end{align*}
we have
\begin{align}
\label{eq:proof of corollary result for global erm}
\widehat{\mathcal{R}}(\widehat{g}_{n}^{\textup{ERM}};P,U_{1}^{n})
&\leq C^{*} (\varepsilon_{n}\log^{4}{n}+\varepsilon_{n}+\frac{1}{n})\nonumber \\
&\leq 3C^{*}\varepsilon_{n}\log^{4}{n},
\end{align}
where
\begin{align*}
C^{*}=\max\{128C_{1}C,C'\beta^{-1},C''\}.
\end{align*}
This inequality shows the first claim.

The second claim is obtained by evaluating the following integral for a sufficiently large $N\in\mathbb{N}$ and any natural number $n\geq N$:
\begin{align*}
&\mathbb{E}[\widehat{\mathcal{R}}(\widehat{g}_{n}^{\textup{ERM}};P,U_{1}^{n})]\\
&=\int_{0}^{4d_{1}}Q(\widehat{\mathcal{R}}(\widehat{g}_{n}^{\textup{ERM}};P,U_{1}^{n})>s)ds\\
&\leq 6C^{*}\varepsilon_{n}\log^{4}{n}+4d_{1}Q(\widehat{f}_{n}^{\textup{ERM}}(U_{1}^{n})\notin \mathscr{F}_{\beta,\beta^{-1}\varepsilon_{n},P}(\mathcal{F}^{*})).
\end{align*}
Therefore, the claims are proven.
\end{proof}

\section{Application to Multiclass Classification}
\label{subsec:consequence for multiclass classification}

The estimation problem we addressed in the current work is general in the sense that multiple subsets are estimated simultaneously.
To show the usefulness of our analyses, we investigate the following problem:
\begin{center}
\emph{Can one utilize the ERM estimator defined in this work to construct a consistent classifier in multiclass nonparametric classification?}
\end{center}
In this section, we show a consequence of Theorem~\ref{thm:estimation error for deep relu networks} for the non-asymptotic analysis of multiclass nonparametric classification.

\subsection{Results and Discussion}
\label{suppsubsec:results and discussion of application}

\paragraph{Problem setting.}
We briefly introduce the problem setting of multiclass classification.
We consider to predict the class label $z\in\mathcal{Z}_{d_{1}}=\{1,\cdots,d_{1}\}$ of covariate $x$, using data drawn from a given distribution $P\in\mathcal{P}_{\xi}$ characterized by the partition $\mathscr{S}_{P}$.

To do so, we recall a condition of~\citep{kim2021fast}, which controls the behavior of the distribution around the boundaries.
Indeed, while the Tsybakov noise condition is initially introduced in~\citep{mammen1999smooth,tsybakov2004optimal} under the settings where binary variables are considered, some extensions to multiclass cases are also studied in~\citep{tarigan2008moment,kim2021fast,bos2022convergence}.
In this analysis we focus on the version introduced in~\citep{kim2021fast}.

Let $P_{X,Z_{d_{1}}}$ be a Borel probability measure in $\mathcal{X}\times\mathcal{Z}_{d_{1}}$ such that it has the probability density function $p(x,z)$ in $\mathcal{X}\times\mathcal{Z}_{d_{1}}$.
With a slight abuse of notation, we write the marginal distribution of $P_{X,Z_{d_{1}}}$ as $P_{X}$.
Let $\mathcal{T}_{d_{1}}$ be the class defined as
\begin{align*}
\mathcal{T}_{d_{1}}=\{g:\mathcal{X}\to\mathcal{Z}_{d_{1}}\;|\; g\textup{ is measurable}\}.
\end{align*}

It is well known (see, e.g.,~\citep{tarigan2008moment,kim2021fast,mohri2018foundations}) that the Bayes classifier $z^{*}:\mathcal{X}\to\mathcal{Z}_{d_{1}}$ minimizing the risk $P_{X,Z_{d_{1}}}(g(x)\neq z)$ in the class $\mathcal{T}_{d_{1}}$ satisfies
\begin{align*}
z^{*}(x)\in\argmax_{i\in\mathcal{Z}_{d_{1}}}p(z=i|x).
\end{align*}
The condition of~\citep{kim2021fast} assumes that there are constants $\tau\geq 1$, $c>0$, and $t_{0}\in(0,1]$ such that for any Bayes classifier $z^{*}$ of $P_{X,Z_{d_{1}}}$ and every $t\in (0,t_{0}]$,
\begin{align}
\label{citeeq:kim condition}
P_{X}\left(\left\{x\in\mathcal{X}\;\Big|\; \min_{i\in\mathcal{Z}_{d_{1}}\setminus\{z^{*}(x)\}}|p(z=i|x)-p(z=z^{*}(x)|x)|\leq t\right\}\right)\leq c t^{\frac{1}{\tau-1}}.
\end{align}
\begin{definition}
\label{def:multiclass noise condition}
We say that the distribution $P_{X,Z_{d_{1}}}$ satisfies $\tau$-(MNC) if it satisfies the noise condition~\eqref{citeeq:kim condition} due to~\citep{kim2021fast}, with $\tau\geq 1$, $c>0$, and $t_{0}\in (0,1]$.
\end{definition}

\paragraph{Classifiers.}
Following~\citep[Definition~2.1]{arora2019theoretical}, we define a classifier using some vector-valued function in two steps: in the first step, for a function $h:\mathcal{X}\to\mathcal{S}^{d-1}$ we consider a classifier
\begin{align*}
\Upsilon_{0}(h)(\cdot)=\min\{j\in\mathcal{Z}_{d_{1}}\;|\;\langle h(\cdot),v_{j}\rangle=\max_{i=1,\cdots,d_{1}}\langle h(\cdot),v_{i}\rangle\},
\end{align*}
where $d=d_{1}-1$, and $\langle\cdot,\cdot\rangle$ denotes the usual inner product in $\mathbb{R}^{d}$.
The linear classifier $\Upsilon_{0}$ is defined with vector-valued function $h$ and vectors $v_{1},\cdots,v_{d_{1}}$, as in~\citep[Definition~2.1]{arora2019theoretical}.
In the second step, we use this functional to define a plug-in estimator $\Upsilon(\widehat{f}_{n}):(\mathcal{X}^{2}\times\mathcal{Y})^{n}\to \mathcal{T}_{d_{1}}$, where $\widehat{f}_{n}:(\mathcal{X}^{2}\times\mathcal{Y})^{n}\to\mathcal{F}_{0}$ denotes any estimator, and $\Upsilon(\widehat{f}_{n})$ is defined as
\begin{align*}
&\Upsilon(\widehat{f}_{n})(u_{1},\cdots,u_{n})(x)\\
&=
\begin{cases}
\Upsilon_{0}(\frac{\widehat{f}_{n}(u_{1},\cdots,u_{n})}{\|\widehat{f}_{n}(u_{1},\cdots,u_{n})(\cdot)\|_{2}})(x)&\quad \textup{if }\widehat{f}_{n}(u_{1},\cdots,u_{n})(x)\neq \bm{0},\\
1&\quad \textup{if } \widehat{f}_{n}(u_{1},\cdots,u_{n})(x)=\bm{0}.
\end{cases}
\end{align*}
\begin{remark}
\label{rem:remark on the multiclass classifier}
We need to define $\Upsilon$ since it is not necessarily true that $\widehat{f}_{n}(u_{1},\cdots,u_{n})(x)\neq \bm{0}$ for any $u_{1},\cdots,u_{n}\in\mathcal{X}^{2}\times\mathcal{Y}$ and any $x\in\mathcal{X}$.
In other words, the value of classifier $\Upsilon_{0}(\widehat{f}_{n}(u_{1},\cdots,u_{n})(\cdot)/\|\widehat{f}_{n}(u_{1},\cdots,u_{n})(\cdot)\|_{2})(x)$ is not defined if $x\in\mathcal{X}$ satisfies that $\widehat{f}_{n}(u_{1},\cdots,u_{n})(x)=\bm{0}$.
\end{remark}

\paragraph{Analyses.}
Here, for measurable subsets $\mathcal{K},\mathcal{K}'\subset\mathcal{X}$, let $D_{P_{X}}(\mathcal{K},\mathcal{K}')$ denote the volume of the symmetric difference measured by a probability measure $P_{X}$ in $\mathcal{X}$, namely
\begin{align*}
D_{P_{X}}(\mathcal{K},\mathcal{K}')=P_{X}((\mathcal{K}\cup\mathcal{K}')\setminus(\mathcal{K}\cap\mathcal{K}')).
\end{align*}
We note the following fact.
\begin{lemma}
\label{lem:approximation of partitions}
For any $\alpha>0$, $R\geq 1$, $K,d_{1},E\in\mathbb{N}$ for which $2^{E}\geq d_{1}$, and any partition $\mathscr{S}=\{\mathcal{K}_{i}\}_{i=1}^{d_{1}}\in\mathscr{P}_{\alpha,R}^{K,d_{1},E}$, there is a Borel probability measure $P_{X,Z_{d_{1}}}$ in $\mathcal{X}\times\mathcal{Z}_{d_{1}}$ satisfying 1-(MNC) such that for any Bayes classifier $z^{*}$ minimizing $P_{X,Z_{d_{1}}}(g(x)\neq z)$ in $\mathcal{T}_{d_{1}}$, we have $\sum_{i=1}^{d_{1}}D_{P_{X}}((z^{*})^{-1}(i),\mathcal{K}_{i})=0$.
\end{lemma}
The proof is deferred to Appendix~\ref{appsubsec:proof of lemma approximation of partitions}.
Based on this lemma, we define the class
\begin{align*}
\mathscr{U}_{\alpha,\xi}=
\left\{
\begin{array}{c|l}
\multirow{3}{*}{$P_{X,Z_{d_{1}}}$} & P_{X,Z_{d_{1}}}\textup{ is a Borel probability measure in}\\
& \mathcal{X}\times\mathcal{Z}_{d_{1}}\textup{ and satisfies (B1) and (B2)}\\
& \textup{with } \alpha \textup{ and } \xi=(R,K,d_{1},E,\theta_{\textup{NC}},\theta_{1},\theta_{2},\theta_{3})
\end{array}
\right\},
\end{align*}
where conditions (B1) and (B2) are defined as follows:
\begin{enumerate}
    \item[(B1)] $P_{X,Z_{d_{1}}}$ satisfies $1$-(MNC).
    \item[(B2)] There is some $P\in\mathcal{P}_{\alpha,1,\xi}$ such that the marginal distributions of $P$ and $P_{X,Z_{d_{1}}}$ with respect to $\mathcal{X}$ coincide with each other, and for the partition $\mathscr{S}_{P}=\{\mathcal{K}_{i}\}_{i=1}^{d_{1}}$, any Bayes classifier $z^{*}$ of $P_{X,Z_{d_{1}}}$ satisfies that $D_{P_{X}}((z^{*})^{-1}(i),\mathcal{K}_{i})=0$.
\end{enumerate}

Then, we have the following inequality between different risk functions.
\begin{lemma}
\label{thm:application to downstream task}
Given $\alpha>0$, $\xi=(R,K,d_{1},E,\theta_{\textup{NC}},\theta_{1},\theta_{2},\theta_{3})\in \Xi$ and a Borel probability measure $P_{X,Z_{d_{1}}}\in\mathscr{U}_{\alpha,\xi}$, let $P$ be an element of $\mathcal{P}_{\alpha,1,\xi}$ such that it satisfies condition (B2) for $P_{X,Z_{d_{1}}}$.
For any estimator $\widehat{f}_{n}:(\mathcal{X}^{2}\times\mathcal{Y})^{n}\to\mathcal{F}_{0}$, define $\widehat{z}_{n}:(\mathcal{X}^{2}\times\mathcal{Y})^{n}\to\mathcal{T}_{d_{1}}$ as
\begin{align}
\label{eq:multiclass classification estimator}
\widehat{z}_{n}=\Upsilon(\widehat{f}_{n}).
\end{align}
Also, let $\widehat{g}_{n}=(\widehat{g}_{n,1},\cdots,\widehat{g}_{n,d_{1}}):(\mathcal{X}^{2}\times\mathcal{Y})^{n}\to\mathcal{G}_{0}$ denote the estimator satisfying $\widehat{f}_{n}=\sum_{i=1}^{d_{1}}\widehat{g}_{n,i}v_{i}$.
Then, there is a constant $C>0$ independent of $n$ such that for any $\omega\in\Omega$, we have
\begin{align}
\label{eq:multiclass clsssification bound theorem}
P_{X,Z_{d_{1}}}(\widehat{z}_{n}(U_{1}^{n}(\omega))(x)\neq z)-\inf_{z^{*}}P_{X,Z_{d_{1}}}(z^{*}(x)\neq z)\leq C\widehat{\mathcal{R}}(\widehat{g}_{n};P,U_{1}^{n}(\omega)),
\end{align}
where the infimum in the left-hand side is taken over all measurable maps from $\mathcal{X}$ to $\mathcal{Z}_{d_{1}}$, and $U_{1}^{n}$ is any sequence of i.i.d. random variables drawn from $P$.
\end{lemma}
The proof is deferred to Appendix~\ref{appsubsec:proof of theorem application to downstream task modified}.

\paragraph{Result.}
We now show the result for the multiclass classification problem defined above.
\begin{theorem}
\label{cor:consequence to multiclass classification}
Let $\alpha>0$, $\xi=(R,K,d_{1},E,\theta_{\textup{NC}},\theta_{1},\theta_{2},\theta_{3})\in\Xi$, $\beta\in (0,D_{{\textup{proj}}})$, $n\in\mathbb{N}\setminus \{1,2\}$ for which $\varepsilon_{n}:= n^{-\frac{\alpha}{\alpha+K-1}}<2^{-1}$.
Given $P_{X,Z_{d_{1}}}\in\mathscr{U}_{\alpha,\xi}$, let $P\in\mathcal{P}_{\alpha,1,\xi}$ be any probability measure satisfying condition (B2).
Let $U_{1},\cdots,U_{n}$ be any sequence of i.i.d. random variables drawn from $P$.
Then, we have the following claims independent of each other:
\begin{itemize}
\item There are
\begin{itemize}
\item constants $C^{*},c_{1},c_{2}>0$ independent of $n$, and
\item $L^{*}\in\mathbb{N}$, $J^{*},S^{*},M^{*}\geq 0$, and $\bm{d}^{*}=(K,d_{\textup{NN},1}^{*},\cdots,d_{\textup{NN},L^{*}-1}^{*},d_{1})\in\mathbb{N}^{L^{*}+1}$ depending on $n$,
\end{itemize}
such that for the class $\mathcal{F}^{*}=\mathcal{F}_{L^{*},J^{*},S^{*},M^{*},\bm{d}^{*}}^{\Delta^{d}\textup{-NN}}$, the classifier $\widehat{z}_{n}^{\textup{ERM}}$ defined as in~\eqref{eq:multiclass classification estimator} using the $(n,\mathcal{F}^{*})$-global ERM estimator $\widehat{g}_{n}^{\textup{ERM}}$ and the corresponding estimator $\widehat{f}_{n}^{\textup{ERM}}$ of vector-valued functions, we have the following:
With probability at least
\begin{align*}
1-c_{1}e^{-128c_{2}n\varepsilon_{n}\log^{3}{n}}-Q(\widehat{f}_{n}^{\textup{ERM}}(U_{1}^{n})\notin \mathscr{F}_{\beta,\beta^{-1}\varepsilon_{n},P}(\mathcal{F}^{*})),
\end{align*}
we have
\begin{align*}
P_{X,Z_{d_{1}}}(\widehat{z}_{n}^{\textup{ERM}}(x)\neq z)-\inf_{z^{*}}P_{X,Z_{d_{1}}}(z^{*}(x)\neq z)
\leq C^{*}\varepsilon_{n}\log^{4}{n}.
\end{align*}
\item There are
\begin{itemize}
\item constants $C^{*}>0$ and $N\in\mathbb{N}$ independent of $n$, and
\item $L^{*}\in\mathbb{N}$, $J^{*},S^{*},M^{*}\geq 0$, and $\bm{d}^{*}=(K,d_{\textup{NN},1}^{*},\cdots,d_{\textup{NN},L^{*}-1}^{*},d_{1})\in \mathbb{N}^{L^{*}+1}$ depending on $n$,
\end{itemize}
such that for the class $\mathcal{F}^{*}=\mathcal{F}_{L^{*},J^{*},S^{*},M^{*},\bm{d}^{*}}^{\Delta^{d}\textup{-NN}}$ and the classifier $\widehat{z}_{n,P}^{\textup{LERM}}$ defined as in~\eqref{eq:multiclass classification estimator} with the $(\beta,\varepsilon_{n},n,\mathcal{P}_{\alpha,1,\xi},\mathcal{F}^{*})$-local ERM estimator $\widehat{g}_{n}^{\textup{LERM}}$, when $n\geq N$, we have
\begin{align*}
\mathbb{E}[P_{X,Z_{d_{1}}}(\widehat{z}_{n,P}^{\textup{LERM}}(x)\neq z)]-\inf_{z^{*}}P_{X,Z_{d_{1}}}(z^{*}(x)\neq z)\leq C^{*}n^{-\frac{\alpha}{\alpha+K-1}}\log^{4}{n}.
\end{align*}
\end{itemize}
\end{theorem}
\begin{proof}
Combine Lemma~\ref{lem:approximation of partitions}, Lemma~\ref{thm:application to downstream task}, and Theorem~\ref{thm:estimation error for deep relu networks} (or Theorem~\ref{cor:result for the global erm}).
\end{proof}
This result indicates that the general framework of nonparametric boundary estimation studied in the current work can apply to an other learning problem.

\paragraph{Discussion.}
Among the previous work~\citep{tarigan2008moment,kim2021fast,bos2022convergence}, a result shown by~\citet{kim2021fast} is closely related to the above analysis.
In Theorem~5.1 of~\citep{kim2021fast}, it is proven that a multiclass classifier defined with deep ReLU networks attains the convergence rate
$((\log^{3}{n})/n)^{\frac{\tau \alpha}{(2\tau-1)\alpha+(K-1)\tau}}$ under the excess risk,
where $\tau\in[1,\infty)$ represents the parameter of their noise condition, and they consider the case that the decision boundaries are parameterized by some $\alpha$-H\"{o}lder continuous functions.
Thus, the result of~\citet{kim2021fast} implies the convergence rate $n^{-\frac{\alpha}{\alpha+K-1}}$ up to a logarithmic factor when $\tau=1$.
The main difference is that \citet{kim2021fast} consider the conventional supervised learning setting.
On the other hand, we consider a pairwise binary classification setting, which requires a substantially different approach.
Additionally, we use a different estimator.

We discuss some technical differences from the results shown in~\citep{haochen2021provable}, in terms of multiclass classification.
As a special case, \citet{haochen2021provable} show the learnability on multiclass classification in the setting where the label is a function of covariate, by using both a vector-valued function trained via contrastive learning and some linear classifier trained in a supervised downstream task, where the best convergence rate is $O_{\mathbb{P}}(n^{-\frac{1}{2}})$ in the result of~\citep{haochen2021provable} ($O_{\mathbb{P}}(\cdot)$ denotes the rate of convergence in probability).
The main differences to the analysis of~\citep{haochen2021provable} are summarized in three points.
First, we additionally assume the smoothness of the boundaries, following~\citep{imaizumi2022advantage}.
In addition, the condition on the probability distribution $P_{X,Z_{d_{1}}}$ used in our analysis is more general.
Finally, Theorem~\ref{cor:consequence to multiclass classification} implies a faster rate when $\alpha>K-1$ and $n$ is sufficiently large, if the local ERM is used.

Recently, \citet{duan2024unsupervised} investigated transfer learning using contrastive learning and derived a convergence rate for multiclass classification.
Our problem setting, analyses, and results are largely different from those in~\citep{duan2024unsupervised}.
In addition, the purpose of the present work is to study the statistical learnability of smooth boundaries, while \citet{duan2024unsupervised} study transfer learning.

In Theorem~\ref{cor:consequence to multiclass classification}, we do not consider the setting where the marginal distributions of the given $P\in\mathcal{P}_{\xi}$ and $P_{X,Z_{d_{1}}}$ with respect to the variable $X$ can be different.
Since some problem settings are studied in the context of transfer learning~(see, e.g.,~\citep{cai2021transfer,duan2024unsupervised}), the investigation of this point is an interesting future work.
For instance, it might be possible to analyze some transfer learning setting, building on the theoretical analysis in~\citep{haochen2022beyond}.

\subsection{Proof of Lemma~\ref{lem:approximation of partitions}}
\label{appsubsec:proof of lemma approximation of partitions}

\begin{proof}
Given any Borel probability measure $P_{X,Z_{d_{1}}}$ satisfying $1$-(MNC), let $z^{*}$ be any Bayes classifier of $P_{X,Z_{d_{1}}}$.
Recall that the condition 1-(MNC) due to~\citep{kim2021fast} directly implies that there is a constant $t_{0}\in(0,1]$ such that we have
\begin{align*}
\min_{i\in\mathcal{Z}_{d_{1}}\setminus\{z^{*}(x)\}}|p(z=i|x)-p(z=z^{*}(x)|x)|>t_{0}\quad\textup{for }P_{X}\textup{-almost all }x\in\mathcal{X}.
\end{align*}
Hence, for $P_{X}$-almost all $x\in\mathcal{X}$, for every $i\in\mathcal{Z}_{d_{1}}\setminus\{z^{*}(x)\}$ we have
\begin{align*}
p(z=z^{*}(x)|x)>p(z=i|x)+t_{0}.
\end{align*}
Thus, we have
\begin{align}
\label{eq:thm application to downstream task eq 1}
\max_{i\in\mathcal{Z}_{d_{1}}\setminus\{z^{*}(x)\}}p(z=i|x)<p(z=z^{*}(x)|x)\quad\textup{for }P_{X}\textup{-almost all }x\in\mathcal{X}.
\end{align}
This inequality implies that the cardinality of $\argmax_{i=1,\cdots,d_{1}}p(z=i|x)$ is one for $P_{X}$-almost all $x\in\mathcal{X}$.
Therefore, for any Bayes classifiers $z_{1}^{*}$ and $z_{2}^{*}$ of $P_{X,Z_{d_{1}}}$, we have
\begin{align}
\label{eq:lem approximation of partitions eq 1}
\sum_{i=1}^{d_{1}}D_{P_{X}}((z_{1}^{*})^{-1}(i),(z_{2}^{*})^{-1}(i))=0.
\end{align}

We now construct a probability distribution.
Let $\mathscr{S}=\{\mathcal{K}_{i}\}_{i=1}^{d_{1}}\in\mathscr{P}_{\alpha,R}^{K,d_{1},E}$.
Let $P_{X,Z_{d_{1}}}^{(0)}$ be a Borel probability measure in $\mathcal{X}\times\mathcal{Z}_{d_{1}}$ such that it has density $p(x,z)$ in $\mathcal{X}\times\mathcal{Z}_{d_{1}}$ and satisfies
\begin{align*}
p(z=i|x)=
\begin{cases}
1&\quad\textup{if }x\in\mathcal{K}_{i},\\
0&\quad\textup{otherwise}.
\end{cases}
\end{align*}
Then, the function $z_{0}^{*}:\mathcal{X}\to\mathcal{Z}_{d_{1}}$ satisfying $(z_{0}^{*})^{-1}(i)=\mathcal{K}_{i}$ for every $i\in\mathcal{Z}_{d_{1}}$ is a Bayes classifier of $P_{X,Z_{d_{1}}}^{(0)}$.
Note that $P_{X,Z_{d_{1}}}^{(0)}$ satisfies $1$-(MNC) by the definition.
Thus, by~\eqref{eq:lem approximation of partitions eq 1}, for any Bayes classifier $z^{*}$ of $P_{X,Z_{d_{1}}}$ we have
\begin{align*}
\sum_{i=1}^{d_{1}}D_{P_{X}}((z^{*})^{-1}(i),\mathcal{K}_{i})=0.
\end{align*}
We obtain the claim.
\end{proof}

\subsection{Proof of Lemma~\ref{thm:application to downstream task}}
\label{appsubsec:proof of theorem application to downstream task modified}

\begin{proof}
We introduce the subset
\begin{align*}
\mathcal{J}_{i}= \{x\in\mathcal{X}\;|\;p(z=i|x)>\max_{j\in\mathcal{Z}_{d_{1}}\setminus\{i\}}p(z=j|x)\}.
\end{align*}
Let $z^{*}$ be any Bayes classifier of $P_{X,Z_{d_{1}}}$.
Let $\mathcal{J}=\bigcup_{i=1}^{d_{1}}\mathcal{J}_{i}\times\{i\}$.
Also, define
\begin{align*}
\mathcal{J}'=\{x\in\mathcal{X}\;|\;p(z=z^{*}(x)|x)\leq \max_{j\in\mathcal{Z}_{d_{1}}\setminus\{z^{*}(x)\}}p(z=j|x)\}.
\end{align*}
By~\eqref{eq:thm application to downstream task eq 1}, $P_{X}(\mathcal{J}')=0$.
Note that $\mathcal{X}=(\bigcup_{i=1}^{d_{1}}\mathcal{J}_{i})\cup \mathcal{J}'$.
Hence, we have
\begin{align*}
&P_{X,Z_{d_{1}}}((\mathcal{X}\times\mathcal{Z}_{d_{1}})\setminus \mathcal{J})\\
&=P_{X,Z_{d_{1}}}\left(\bigcup_{i=1}^{d_{1}}\mathcal{J}'\times\{i\}\right)
+P_{X,Z_{d_{1}}}\left(\bigcup_{i\neq j}\mathcal{J}_{j}\times\{i\}\right)\\
&\leq P_{X}(\mathcal{J}')+P_{X,Z_{d_{1}}}(z^{*}(x)\neq z)\\
&\leq P_{X,Z_{d_{1}}}(z^{*}(x)\neq z).
\end{align*}
Here, we notice that
\begin{align*}
\mathcal{J}\subset \bigcup_{i=1}^{d_{1}}(z^{*})^{-1}(i)\times \{i\}.
\end{align*}
From now on, we fix $\omega\in\Omega$ and abbreviate as $\widehat{z}_{n}=\widehat{z}_{n}(U_{1}^{n}(\omega))$.
Then, we have
\begin{align*}
&P_{X,Z_{d_{1}}}(\widehat{z}_{n}(x)\neq z)\\
&\leq
\sum_{i=1}^{d_{1}}\int_{(z^{*})^{-1}(i)}\mathds{1}_{\{\widehat{z}_{n}(x)\neq i\}}(x)p(z=i|x)p(x)\mu(dx)
+
P_{X,Z_{d_{1}}}(z^{*}(x)\neq z).
\end{align*}
Let $\varepsilon\geq 0$ be arbitrary.
Let $z_{0}^{*}$ be the Bayes classifier such that we have $P_{X,Z_{d_{1}}}(z_{0}^{*}(x)\neq z)-\inf_{z^{*}}P_{X,Z_{d_{1}}}(z^{*}(x)\neq z)\leq \varepsilon$.
Then, we obtain
\begin{align}
&P_{X,Z_{d_{1}}}(\widehat{z}_{n}(x)\neq z)-\inf_{z^{*}}P_{X,Z_{d_{1}}}(z^{*}(x)\neq z)\nonumber\\
\label{eq:thm application to downstream tasks eq 2}
&\leq
\sum_{i=1}^{d_{1}}\int_{(z_{0}^{*})^{-1}(i)}\mathds{1}_{\{\widehat{z}_{n}(x)\neq i\}}(x)P_{X}(dx)+\varepsilon,
\end{align}
where in~\eqref{eq:thm application to downstream tasks eq 2} we further use the inequality $p(z=i|x)\leq 1$.

Note that $D_{P_{X}}((z_{0}^{*})^{-1}(i),\mathcal{K}_{i})=0$ by condition (B2).
Also, we have
\begin{align*}
(z_{0}^{*})^{-1}(i) \subset \mathcal{K}_{i} \cup ((z_{0}^{*})^{-1}(i)\setminus\mathcal{K}_{i})
\subset \mathcal{K}_{i}\cup (((z_{0}^{*})^{-1}(i)\cup \mathcal{K}_{i})\setminus ((z_{0}^{*})^{-1}(i)\cap \mathcal{K}_{i}))
\end{align*}
Hence, we have
\begin{align}
\int_{(z_{0}^{*})^{-1}(i)}\mathds{1}_{\{\widehat{z}_{n}(x)\neq i\}}(x)P_{X}(dx)
&\leq
\int_{\mathcal{K}_{i}}\mathds{1}_{\{\widehat{z}_{n}(x)\neq i\}}(x)P_{X}(dx)\nonumber\\
\label{eq:application to downstream tasks eq 3}
&=\int_{\mathcal{K}_{i}}(1-\mathds{1}_{\{\widehat{z}_{n}(x)=i\}}(x)) P_{X}(dx).
\end{align}

We also abbreviate as $\widehat{f}_{n}=\widehat{f}_{n}(U_{1}^{n}(\omega))$.
Here note that
\begin{align}
\label{eq:proofs of theorem application eq 3}
\mathcal{D}_{i}:=\left\{x\in\mathcal{X}\;|\; \|\widehat{f}_{n}(x)-v_{i}\|_{2}< D_{\Delta^{d}}/2\right\}\subseteq \{x\in\mathcal{X}\;|\; \widehat{z}_{n}(x)=i\}.
\end{align}
Let
\begin{align*}
    \mathcal{D}=\left\{x\in\mathcal{X}\;|\; \|\widehat{f}_{n}(x)-f^{*}(x)\|_{2}< D_{\Delta^{d}}/2\right\}.
\end{align*}
Note that $\mathcal{D}_{i}\cap\mathcal{K}_{i}=\mathcal{D}\cap\mathcal{K}_{i}$ for any $i=1,\cdots,d_{1}$.
Then, by \eqref{eq:thm application to downstream tasks eq 2} and~\eqref{eq:application to downstream tasks eq 3} we have
\begin{align}
\label{eq:proofs of theorem application eq 4}
&P_{X,Z_{d_{1}}}(\widehat{z}_{n}(x)\neq z)-\inf_{z^{*}}P_{X,Z_{d_{1}}}(z^{*}(x)\neq z)\nonumber\\
&\leq 
\sum_{i=1}^{d_{1}}\int_{\mathcal{K}_{i}}(1-\mathds{1}_{\mathcal{D}_{i}}(x))P_{X}(dx)+\varepsilon\\
\label{eq:proofs of theorem application eq 5}
&=1-\sum_{i=1}^{d_{1}}\int_{\mathcal{K}_{i}}\mathds{1}_{\mathcal{D}}(x)P_{X}(dx)+\varepsilon\\
\label{eq:proofs of theorem application eq 6}
&=1-P_{X}(\mathcal{D})+\varepsilon\\
\label{eq:proofs of theorem application eq 7}
&\leq 4D_{\Delta^{d}}^{-2}\mathbb{E}_{P_{X}}[\|\widehat{f}_{n}(x)-f^{*}(x)\|_{2}^{2}]+\varepsilon,
\end{align}
where in~\eqref{eq:proofs of theorem application eq 4} we use \eqref{eq:thm application to downstream tasks eq 2} and~\eqref{eq:application to downstream tasks eq 3}, in~\eqref{eq:proofs of theorem application eq 5} we note that $f^{*}(x)=v_{i}$ for any $x\in\mathcal{K}_{i}$ by Definition~\ref{def:contrastive representations}, and in \eqref{eq:proofs of theorem application eq 7} we apply Markov's inequality.

Finally, since $\varepsilon\geq 0$ is arbitrary in~\eqref{eq:proofs of theorem application eq 7}, by applying Proposition~\ref{prop:simple fact for risks}--(ii), we obtain the claim.
\end{proof}

\addcontentsline{toc}{section}{References}

\end{document}